\documentclass[letterpaper,11pt,oneside,reqno]{amsart}
\usepackage{amsfonts,amsmath, amssymb,amsthm,amscd}
\usepackage{bm}
\usepackage{mathrsfs}
\usepackage{yfonts}
\usepackage{verbatim}
\usepackage{hyperref}
\usepackage{graphicx}
\usepackage[utf8]{inputenc}
\usepackage{latexsym}
\usepackage{upgreek}
\usepackage{relsize}
\usepackage{xcolor}
\usepackage{mathdots}
\usepackage{mathtools}
\usepackage{dsfont}
\usepackage[left=2.5cm, right=2.5cm, top=2.5cm, bottom=3cm]{geometry}
\usepackage{tikz}

\makeatletter
\renewcommand\section{\@startsection{section}{1}%
	\z@{3\linespacing\@plus1.5\linespacing}{2.5\linespacing}%
	{\large\normalfont\scshape\centering}}
\renewcommand\subsection{\@startsection{subsection}{2}%
	\z@{.5\linespacing\@plus.7\linespacing}{0.5\linespacing}%
	{\normalfont\bfseries}}
\renewcommand\subsubsection{\@startsection{subsubsection}{3}%
	\z@{.5\linespacing\@plus.7\linespacing}{-.5em}%
	{\normalfont\bfseries}}
\makeatother

\renewcommand{\emph}[1]{\textsf{\textit{#1}}}

\usepackage[width=.9\textwidth]{caption}

\let\oldtocsection=\tocsection
\let\oldtocsubsection=\tocsubsection
\renewcommand{\tocsection}[2]{\hspace{0em}\oldtocsection{#1}{#2}}
\renewcommand{\tocsubsection}[2]{\hspace{2em}\oldtocsubsection{#1}{#2}}

\begin{document}
\fontdimen8\textfont3=0.5pt  %default is 0.4

%%%%%%%%%%%%%%%%%%%%%%%%%%%%%%%%%%%%%%%%%%%Shortcuts%%%%%%%%%%%%%%%%%%%%%%%%%%%%%%%%%%%%%%%%%%%%%%%%%%%%%%%%
\def \Ai {{\rm Ai}}
\def \Pf {{\rm Pf}}
\def \sgn {{\rm sgn}}
\def \SS {\mathcal{S}}
\newcommand{\e}{\varepsilon}
\newcommand{\ve}{\mathcal{E}}
\newcommand{\EE}{\ensuremath{\mathbb{E}}}
\newcommand{\qq}[1]{(q;q)_{#1}}
\newcommand{\PP}{\ensuremath{\mathbb{P}}}
\newcommand{\R}{\ensuremath{\mathbb{R}}}
\newcommand{\C}{\ensuremath{\mathbb{C}}}
\newcommand{\Z}{\ensuremath{\mathbb{Z}}}
\newcommand{\N}{\ensuremath{\mathbb{N}}}
\newcommand{\Q}{\ensuremath{\mathbb{Q}}}
\newcommand{\T}{\ensuremath{\mathbb{T}}}
\newcommand{\Y}{\ensuremath{\mathbb{Y}}}
\newcommand{\I}{\ensuremath{\mathbf{i}}}
\newcommand{\Real}{\ensuremath{\mathfrak{Re}}}
\renewcommand{\Re}{\ensuremath{\mathfrak{Re}}}
\newcommand{\Imag}{\ensuremath{\mathfrak{Im}}}
\newcommand{\re}{\ensuremath{\mathfrak{Re}}}
\newcommand{\subs}{\ensuremath{\mathbf{Subs}}}
\newcommand{\Sym}{\ensuremath{\mathsf{Sym}}}
\newcommand{\phidist}{\ensuremath{\boldsymbol{\varphi}}}
\newcommand{\mskyl}{\ensuremath{\mathfrak{m}}}
\newcommand{\prech}{\ensuremath{\prec}}
\newcommand{\precv}{\ensuremath{\prec_{\mathrm{v}}}}
\newcommand{\gap}{\ensuremath{\mathrm{gap}}}
\newcommand{\diagq}{\ensuremath{a_{\circ}}}
\newcommand{\diag}{\ensuremath{\alpha_{\circ}}}
\newcommand{\whitenoise}{\ensuremath{\mathscr{\dot{W}}}}
\newcommand{\alphaW}[1]{\ensuremath{\mathbf{\alpha W}}_{(#1)}}
\newcommand{\alphaWM}[1]{\ensuremath{\mathbf{\alpha WM}}_{(#1)}}
\newcommand{\edge}{\textrm{edge}}
\newcommand{\dist}{\textrm{dist}}
\newcommand{\A}{\mathcal{A}}
\newcommand{\var}{{\rm var}}
\newcommand{\U}{\ensuremath{\mathcal{U}^{\mathlarger{\llcorner}}}}
\newcommand{\Udiag}{\ensuremath{\mathcal{U}^\angle}}
\newcommand{\rhodiag}{\ensuremath{\rho^{\swarrow}}}
\newcommand{\rholeft}{\ensuremath{\rho^{\leftarrow}}}
\newcommand{\rhoup}{\ensuremath{\rho^{\uparrow}}}
\newcommand{\Proj}{\ensuremath{\overrightarrow{\mathrm{Proj}}}}
\newcommand{\Projtilde}{\ensuremath{\overleftarrow{\mathrm{Proj}}}}
\newcommand{\PMM}{\ensuremath{\mathbb{P}^{q,t}}}
\newcommand{\PMP}{\ensuremath{\mathbb{P}^{q,t}}}
\newcommand{\EPMM}{\ensuremath{\mathbb{E}^{q,t}}}
\newcommand{\EPMP}{\ensuremath{\mathbb{E}^{q,t}}}
\newcommand{\PQWM}{\ensuremath{\mathbb{P}^{q}}}
\newcommand{\PQWP}{\ensuremath{\mathbb{P}^{q}}}
\newcommand{\EQWM}{\ensuremath{\mathbb{E}^{q}}}
\newcommand{\EQWP}{\ensuremath{\mathbb{E}^{q}}}
\newcommand{\PWM}{\ensuremath{\mathbb{P}}}
\newcommand{\PWP}{\ensuremath{\mathbb{P}}}
\newcommand{\EWM}{\ensuremath{\mathbb{E}}}
\newcommand{\EWP}{\ensuremath{\mathbb{E}}}
\newcommand{\PHL}{\ensuremath{\mathbb{P}^{t}}}
\newcommand{\EHL}{\ensuremath{\mathbb{E}^{t}}}
\newcommand{\PSM}{\ensuremath{\mathbb{P}^{q=t}}}
\newcommand{\EPSM}{\ensuremath{\mathbb{E}^{q=t}}}
\newcommand{\Res}[1]{\underset{{#1}}{\mathbf{Res}}}
\newcommand{\Resfrac}[1]{\mathbf{Res}_{{#1}}}
\newcommand{\Sub}[1]{\underset{{#1}}{\mathbf{Sub}}}
\newcommand{\la}{\lambda}
\newcommand{\ta}{\theta}
\newcommand{\Noumi}{\mathbf{N}}
\newcommand{\Moumi}{\overline{\Noumi}}
\newcommand{\MoumiA}{\mathbf{M}}
\newcommand{\Tshift}{\mathbf{T}}
\newcommand{\DD}{\mathbf{D}}
\newcommand{\disk}{\mathbb{D}}
\newcommand{\spaceanalytic}{\mathcal{A}^{\SS_n}(\disk^n)}
\newcommand{\bp}{b_{\circ}}
\newcommand{\BBone}{\overline{\mathbf{B}}}
\newcommand{\BBn}{\mathbf{B}}
\newcommand{\kernel}{\mathsf{K}}
\newcommand{\bel}[1]{b_{#1}^{\mathrm{el}}}
\newcommand{\fkernel}{\mathsf{f}}
\newcommand{\contourv}{\mathcal{C}^{\rm{steep}}}
\newcommand{\gqwhittn}{\mathcal{G}^{q}}
\newcommand{\gqwhittun}{\overline{\mathcal{G}}^q}
\newcommand{\admpath}{\Omega}
\newcommand{\pathh}{\omega}
\newcommand{\binomt}[2]{\left(\begin{matrix}
		#1 \\#2
	\end{matrix}\right)_{\!\!t}}
\newcommand{\p}{\mathsf{p}}
\newcommand{\q}{\mathsf{q}}
\newcommand{\eps}{\epsilon}
\newcommand{\ratealpha}{\upalpha}
\newcommand{\rategamma}{\upgamma}
\newcommand{\NN}{N}
\newcommand{\rrangle}{\rangle\!\rangle}
\newcommand{\llangle}{\langle\!\langle}
\newcommand{\labold}{\boldsymbol{\uplambda}}

\def\note#1{\textup{\textsf{\color{blue}(Guillaume: #1)}}}

%%%%%%%%%%%%%%%%%%%%%%%%%%Tikz Macros%%%%%%%%%%%%%%%%%%%%%%%%%%%%%%%%%%%%%%%%%%%%%%%%%%%%%%%%%

\usetikzlibrary{shapes.multipart}
\usetikzlibrary{patterns}
\usetikzlibrary{shapes.multipart}
\usetikzlibrary{arrows}
\usetikzlibrary{decorations.markings}
\tikzstyle{fleche}=[>=stealth', postaction={decorate}, thick]
\tikzstyle{axis}=[->, >=stealth', thick, gray]
\tikzstyle{grille}=[dotted, gray]
\tikzstyle{path}=[->, >=stealth', thick]
%\tikzstyle{path}=[->,>=latex, ultra thick]

\newcommand{\pathrr}{\raisebox{-6pt}{\begin{tikzpicture}[scale=0.3]
		\draw[thick] (-1,0) -- (1,0);
		\draw[dotted] (0,-1) -- (0,1);
		\end{tikzpicture}}}
\newcommand{\pathru}{\raisebox{-6pt}{\begin{tikzpicture}[scale=0.3]
		\draw[thick] (-1,0) -- (0,0) -- (0,1);
		\draw[dotted] (0,-1) -- (0,0) -- (1,0);
		\end{tikzpicture}}}
\newcommand{\pathuu}{\raisebox{-6pt}{\begin{tikzpicture}[scale=0.3]
		\draw[dotted] (-1,0) -- (1,0);
		\draw[thick] (0,-1) -- (0,1);
		\end{tikzpicture}}}
\newcommand{\pathur}{\raisebox{-6pt}{\begin{tikzpicture}[scale=0.3]
		\draw[dotted] (-1,0) -- (0,0) -- (0,1);
		\draw[thick] (0,-1) -- (0,0) -- (1,0);
		\end{tikzpicture}}}
\newcommand{\pathbrr}{\raisebox{2pt}{\begin{tikzpicture}[scale=0.4]
		\draw[thick, ->, >=latex] (-1,0) -- (0,0);
		\draw[dotted] (0,0) -- (0,1);
		\end{tikzpicture}}}
\newcommand{\pathbuu}{\raisebox{2pt}{\begin{tikzpicture}[scale=0.4]
		\draw[->, thick, >=latex] (0,0) -- (0,0.1);
		\draw[thick] (0,0) -- (0,1);
		\draw[dotted] (-1,0) -- (0,0);
		\end{tikzpicture}}}
\newcommand{\pathbru}{\raisebox{2pt}{\begin{tikzpicture}[scale=0.4]
		\draw[thick] (-1,0) -- (0,0);
		\draw[thick] (0,0) -- (0,1);
		\end{tikzpicture}}}
\newcommand{\pathbur}{\raisebox{2pt}{\begin{tikzpicture}[scale=0.4]
		\draw[dotted] (0,0) -- (0,1);
		\draw[dotted] (-1,0) -- (0,0);
		\end{tikzpicture}}}

%%%%%%%%%%%%%%%%%%%%%%%%%%%%%Theorems styling%%%%%%%%%%%%%%%%%%%%%%%%%%%%%%%%%%%%%%%%%%%%%%%%%%%%%%%%%%

\newtheorem{theorem}{Theorem}[section]
\newtheorem{conjecture}[theorem]{Conjecture}
\newtheorem{lemma}[theorem]{Lemma}
\newtheorem{proposition}[theorem]{Proposition}
\newtheorem{corollary}[theorem]{Corollary}
\newtheorem{prop}{Proposition}
\newtheorem{formal}{Formal asymptotics}
\newtheorem*{formal*}{Formal asymptotics}
\newtheorem{theoremintro}{Theorem}
\renewcommand*{\thetheoremintro}{\Alph{theoremintro}}

\theoremstyle{definition}
\newtheorem{remark}[theorem]{Remark}

\theoremstyle{definition}
\newtheorem{example}[theorem]{Example}
\theoremstyle{definition}
\newtheorem{definition}[theorem]{Definition}
\theoremstyle{definition}
\newtheorem{definitions}[theorem]{Definitions}

%%%%%%%%%%%%%%%%%%%%%%%%%%%%%%%%%%%%%Preamble%%%%%%%%%%%%%%%%%%%%%%%%%%%%%%%%%%%%%%%%%%%%%%%%%%%%%%%%%%%%%%%%%%

\title{\large Half-space Macdonald processes}

\author[G. Barraquand]{Guillaume Barraquand}
\address{G. Barraquand,
Laboratoire de Physique de l'\'Ecole Normale Supérieure, ENS, Universit\'e PSL, CNRS, Sorbonne Universit\'e, Universit\'e de Paris, Paris, France.}
\email{barraquand@math.cnrs.fr}
\author[A. Borodin]{Alexei Borodin}
\address{A. Borodin, Department of Mathematics, MIT, Cambridge, USA, and
	Institute for Information Transmission Problems, Moscow, Russia.}
\email{borodin@math.mit.edu}
\author[I. Corwin]{Ivan Corwin}
\address{I. Corwin, Columbia University,
	Department of Mathematics,
	2990 Broadway,
	New York, NY 10027, USA.}
\email{ivan.corwin@gmail.com}

\begin{abstract}
Macdonald processes are measures on sequences of integer partitions built using the Cauchy summation identity for Macdonald symmetric functions. These measures are a useful tool to uncover the integrability of  many probabilistic systems, including the Kardar-Parisi-Zhang (KPZ) equation and a number of other models in its universality class. In this paper we develop the structural theory behind half-space variants of these models and the corresponding half-space Macdonald processes. These processes are built using a Littlewood summation identity instead of the Cauchy identity, and their analysis is considerably harder than their full-space counterparts.

We compute moments and Laplace transforms of observables for general half-space Macdonald measures. Introducing new dynamics preserving this class of measures, we relate them to various stochastic processes, in particular the log-gamma polymer in a half-quadrant (they are also related to the stochastic six-vertex model in a half-quadrant and the half-space ASEP). For the polymer model, we provide explicit integral formulas for the Laplace transform of the partition function. Non-rigorous saddle point asymptotics yield convergence of the directed polymer free energy to either the Tracy-Widom GOE, GSE or the Gaussian distribution depending on the average size of weights on the boundary.
\end{abstract}

\maketitle

\setcounter{tocdepth}{2}

\newpage 
 \thispagestyle{plain}
\tableofcontents
%\hypersetup{linktocpage}
%%%%%%%%%%%%%%%%%%%%%%%%%%%%%%%%%%%%%%%%%%%%%%%%%%%%%%%%%%%%%%%%%%%%%%%%%%%%%%%%%%%%%%%%%%%%%%%%%%%%%%%%%%%%

\thispagestyle{plain}
\section{Introduction}
\label{sec:intro}

In commencing the investigation that resulted in this paper, our goal was to prove limit theorems for the Kardar-Parisi-Zhang (KPZ) stochastic PDE \cite{kardar1986dynamic,corwin2012kardar,quastel2012lectures} in a half-space \cite{corwin2016open}, as well as for the log-gamma directed polymer \cite{seppalainen2012scaling,corwin2014tropical} in a half-quadrant \cite{o2014geometric}. Half-space systems are considerably more complicated to analyze than their full-space counterparts, and to us, the proper framework in which to initiate our study seemed to be that of half-space Macdonald processes (which we introduce here). Based on results from earlier analysis of zero-temperature models like TASEP and last passage percolation in a half-space \cite{baik2001algebraic, baik2001asymptotics, baik2001symmetrized, sasamoto2004fluctuations, baik2018pfaffian, betea2018free} (solvable via methods of Pfaffian point processes \cite{borodin2005eynard}), one may predict a rich phase diagram detailing the effect of the boundary strength on the fluctuation scalings and statistics. Despite previous efforts \cite{tracy2013bose, tracy2013asymmetric, o2014geometric}, there were no  limit results known prior to our investigation (besides \cite{barraquand2018stochastic} which we developed with M. Wheeler in parallel to the present work). Even in the physics literature, the non-rigorous replica Bethe ansatz has proved difficult to apply, with results limited to two special boundary conditions for the KPZ equation (pure reflection \cite{borodin2016directed} or pure absorption \cite{gueudre2012directed}).

In this paper we develop the structural theory of half-space Macdonald processes and explore some of the rich hierarchy of limits and specializations. 
The theory of half-space Macdonald processes builts on the case of full-space 
Macdonald processes \cite{borodin2014macdonald} but also employs a number of novel ideas.
%The theory of half-space Macdonald processes relies on a combination of new ideas and some ideas which were present in the case of full-space Macdonald processes \cite{borodin2014macdonald}. 

\medskip
Before highlighting the new ideas and challenges which arise in this half-space setting, we briefly recall the major developments in the full-space theory of Macdonald process (see also the reviews \cite{borodin2012lectures, borodin2014integrableICM, borodin2014integrable, corwin2014macdonald}).
\begin{description}
	\item[Using operators to compute expectations] Applying operators which act diagonally on Macdonald polynomials to the normalizing constant for the measure yields a general mechanism to compute expectations of observables related to the operators' eigenvalues. This idea was introduced in \cite{borodin2014macdonald}, wherein Macdonald difference operators were used extensively, and it is  developed further in \cite{borodin2016observables, borodin2015general, dimitrov2016kpz, gorin2016interlacing}.  
	\vspace{.2cm}
	
	\item[$(2+1)$-dimensional Markovian dynamics] A general scheme to build Markovian dynamics on two dimensional triangular arrays preserving the class of Schur processes (introduced in \cite{okounkov2001infinite, okounkov2003correlation}) was proposed in \cite{borodin2008anisotropic, borodin2011schur}. These push-block dynamics were studied in \cite{borodin2014macdonald} in the Macdonald case, especially at the $q$-Whittaker level. Other dynamics preserving Macdonald processes connected to the RSK algorithm were studied  at the Whittaker level in continuous time in \cite{borodin2013nearest, o2013q} and  in discrete time  \cite{matveev2015q}. These RSK type dynamics were also studied earlier at the Whittaker level in  \cite{o2012directed, corwin2014tropical} and later at the Hall-Littlewood level \cite{borodin2016between,bufetov2015law,  bufetov2017hall}. 
	\vspace{.2cm}
	
	\item[Marginal Markov processes and their limits] 
	Some marginals of these $(2+1)$-dimensional dynamics are themselves Markov processes. Some of these processes were new, while others have been introduced earlier. This relation has provided some new tools in their studies. Let us mention the $q$-TASEP \cite{borodin2014macdonald, borodin2012duality, borodin2013discrete}, $q$-push(T)ASEP \cite{corwin2013q}, log-gamma directed polymer \cite{borodin2013log,  corwin2014tropical, seppalainen2012scaling}, strict-weak polymer \cite{corwin2014strict, o2014tracy}, O'Connell-Yor polymer \cite{ borodin2014macdonald, oconnell2001brownian, o2012directed}, KPZ equation \cite{alberts2014intermediate, bertini1997stochastic, borodin2012free, borodin2015height, kardar1986dynamic}, Stochastic six-vertex model \cite{ borodin2016stochasticsix, borodin2016between, gwa1992six}, Hall-Littlewood-pushTASEP \cite{ghosal2017hall}, ASEP \cite{borodin2012duality, borodin2016asep, bufetov2017hall}.  
	\vspace{.2cm}
	
	\item[Connections to random matrix theory] Relations between the coordinates of a random partition under the Macdonald measures (in particular Hall-Littlewood) and random matrices were explored in \cite{borodin1995limit, borodin1999law, borodin2015general, bufetov2015law, fulman2002random, gorin2015multilevel, gorin2014finite}. 
	\vspace{.2cm}
	
	\item[Gibbsian line ensembles] 
	After taking certain scaling limits, the algebra disappears but the integrability remains in the form of a Gibbs property; this is useful in extending one-point to process level asymptotics \cite{borodin2017stochastic, corwin2018transversal, corwin2014brownian, corwin2016kpz, corwin2016stationary}. 
	\vspace{.2cm}
	
	\item[Curious determinantal identities] In a few specific  cases, curious determinantal identities allow to relate certain functionals of the Macdonald measure with the Schur measure or other determinantal point processes \cite{aggarwal2018dynamical, borodin2016stochastic, aggarwal2019phase, barraquand2018stochastic, borodin2016asep, orr2017stochastic}. This typically relates non-free-fermionic models to fermionic ones, and greatly simplifies the asymptotic analysis. 
	\vspace{.2cm}
	
	\item[KPZ universality class asymptotics] For all the above mentioned models, the Laplace transform of observables of interest can be expressed as a Fredholm determinant whose asymptotic analysis leads to  KPZ type  limit theorems \cite{aggarwal2019phase, barraquand2014phase, borodin2014macdonald, borodin2012free,   borodin2013log, borodin2015height,  borodin2016asep, corwin2014strict, ferrari2013tracy, ghosal2017hall, krishnan2018tracy, o2014tracy}.
\end{description}
\medskip 
The story of (full-space) Macdonald processes is far from complete. Many challenges remain such as computing the asymptotic behavior for the entire measure, multipoint fluctuations (i.e. convergence to Airy type processes and line ensembles), asymptotics away from the edge (corresponding to bulk eigenvalue statistics).

\medskip
Given the success of Macdonald processes in studying systems like the KPZ equation, log-gamma polymer model, and ASEP, it is only natural to seek an appropriate half-space version of the measures and associated theory. There is a natural starting point based on the Macdonald polynomial version of the Littlewood identities. However, there are difficulties -- algebraic (we need new operators to compute expectations  of certain observables of interest), probabilistic/combinatorial (we need new dynamics to deal with the boundary), and analytic (our formulas do not organize themselves into Fredholm determinants or Pfaffians). We overcome all of these, except for the analytic ones where we still manage to obtain the expected phase diagram for fluctuations via formal steepest descent analysis. Despite the lack of rigor in this last step, it is the first time that this full phase diagram has been accessed for these models (even in the physics literature).
\medskip 

We now provide a few details on each of these novelties. The reader not familiar with Macdonald processes may skip this part on a first reading and jump to Section \ref{sec:beginning} below. 
\medskip 
\begin{description}
	\item[Using Littlewood type identities] The definition of half-space Macdonald process\footnote{The term ``process'' versus ``measure'' in ``Macdonald process" and ``Macdonald measure" distinguishes between measures on interlacing sequences of partitions and just on single partitions.} (Section \ref{sec:defhalfspaceMacdonald}) relies upon a Macdonald analog of the Littlewood symmetric function summation identity \eqref{eq:Littlewoodidentity} from \cite{macdonald1995symmetric}. 
	The ($q=t$) half-space (or Pfaffian) Schur process was defined much earlier in \cite{borodin2005eynard, sasamoto2004fluctuations} and studied at length in \cite{baik2018pfaffian} in connection to stochastic processes like half-space TASEP and last passage percolation. The half-space Whittaker measure was introduced in \cite{o2014geometric} (the corresponding Littlewood identity is due to \cite{stade2001mellin}). 
	\vspace{.2cm}
	
	\item[Markovian ``boundary'' dynamics] In order to relate our half-space Macdonald process to interesting stochastic processes, we construct local Markovian dynamics on interlacing sequences of partitions which preserve the class of \footnote{We will sometimes use ``Macdonald processes'' as short for `` half-space Macdonald processes''. To avoid any ambiguity, we will always precise ``full-space Macdonald processes'' when referring to the original Macdonald processes from \cite{borodin2014macdonald}.}Macdonald processes (i.e. applying the dynamics to a sequence distributed according to one Macdonald process yields, at a later time, another sequence distributed according to a Macdonald process with modified parameters) and which have Markovian projections when restricted to a few first or last parts of the partitions. 
	The existence of such dynamics is far from evident. The novelty here (explained in Section \ref{sec:Macdyn}) is finding appropriate dynamics at the boundary of the half-space.
	\vspace{.2cm}
	
	\item[Operators] In developing methods to compute distributional information about marginals of these measures, we construct a new operator (denoted $\MoumiA_n^z$ in the text, see Section \ref{sec:proofmaintheo}) which is an analytic continuation of a $q$-integral operator introduced by Noumi \cite{noumi2012direct} (denoted $\Noumi_n^z$ or $\Moumi_n^z$ in the text, see Section \ref{sec:Noumi}), and which acts diagonally on Macdonald polynomials. By applying this operator to the normalizing function for the measure, we are able to prove a $(q,t)$-Laplace transform formula for $\lambda_1$ and $\lambda_n$ (where $\lambda_1$ and $\lambda_n$ are the first and last parts of the partition under the half-space Macdonald measure). The original Noumi operator cannot be used in computing the Laplace transform of $\la_1$ since in doing so we must interchange an infinite summation in the definition of the operator with an infinite summation in the normalizing function. This interchange is not justifiable, and, in fact, leads to the wrong answer (see Remark \ref{rem:whymoredifficult}). By working with the analytic continuation operator $\MoumiA_n^z$ (which is encoded in terms of Mellin-Barnes type integrals with nice convergence properties), we may justifiably perform such an interchange. When the parameter $t=0$ (the case of $q$-Whittaker process), this yields a $q$-Laplace transform formula for $q^{-\lambda_1}$ which cannot be derived from moment formulas due to the ill-posedness of the moment problem for that random variable. 
	We additionally prove many other moment formulas as well as a Laplace transform formulas  using ideas that were present in \cite{borodin2014macdonald, borodin2016observables}.
	\vspace{.2cm}
	 
	\item[Use of Plancherel specialization and a proof of \cite{o2014geometric} conjectural formulas] Taking the Macdonald parameters $t=0$ and $q\to 1$, and performing appropriate scaling on partitions leads to the half-space Whittaker process (this limit procedure is similar to that used in the full-space case in \cite{borodin2014macdonald}). The dynamics we constructed at the top of the Macdonald hierarchy, when restricted to $\lambda_1$, converge to the recursion for partition functions satisfied by the half-space log-gamma polymer model. This model was studied previously in \cite{o2014geometric} by way of applying the geometric RSK correspondence to a symmetric weight matrix. That approach only related the polymer partition function on the diagonal (i.e. at the boundary of the domain) to the Whittaker measure. Our dynamic approach readily relates the partition function in the entire half-space to an appropriate Whittaker measure. This ability to work off the diagonal is, in fact, key -- it allows us to introduce a ``Plancherel'' specialization into our Whittaker measure which drastically improves the decay properties of the Whittaker measure density and ultimately allows us to rigorously derive Laplace transform formulas for the associated polymer model. The polymer model which comes from this Plancherel component is a mix of the log-gamma and O'Connell-Yor polymers (as also considered in \cite{borodin2015height} in the full-space setting). Having established formulas for the mixed model, we can then shrink the Plancherel component to zero and by continuity we arrive at formulas for the log-gamma polymer alone. Without the inclusion of a Plancherel component, the relevant  Whittaker measure formulas do not have sufficient decay to rigorously justify the derivation of the Laplace transform. In fact, in \cite{o2014geometric}, the authors performed formal (i.e. neglecting issues of convergence and the applicability of the Whittaker Plancherel theory) calculations to derive a Laplace transform formula (either formulas (5.15) and (5.16) in \cite{o2014geometric}). They remarked ``It seems reasonable to expect the integral formulas (5.15) and (5.16) to be valid, at least in some suitably regularized sense.'' We prove these formulas as Corollary \ref{cor:OSZproof2}.  
	\vspace{.2cm}
	
	\item[Two types of Laplace transform formulas] We prove two types of half-space log-gamma polymer Laplace transform formulas. The first type (Corollaries \ref{cor:OSZproof} and \ref{cor:OSZproof2}, coming from Theorem \ref{th:proofOSZ}), is in terms of a single $n$-fold contour integral (and is in the spirit of the speculative formulas from \cite{o2014geometric}, as well as formulas proved in the full-space case in \cite{corwin2014tropical}). The second type, Theorems \ref{theo:NoumiLaplace} and \ref{theo:MoumiLaplace}, is in terms of a (finite) series of increasing dimensional integrals. Though we are unable to write this as a Fredholm Pfaffian, it is the half-space version of the Fredholm determinant expansion formulas which arise in \cite{borodin2013log} and proved, therein, quite useful for asymptotics. We attempted to perform asymptotics (taking the system size $n\to \infty$) using these formulas. At the level of studying the term-by-term limit around the critical points of the integrals, we demonstrate convergence to the expected Fredholm Pfaffian expansions which govern the half-space KPZ universality class one-point fluctuation phase diagram. Unfortunately, the presence of certain gamma function ratios preclude establishing sufficient control over the tails of the integrals as well as the series, as would be necessary to rigorously prove our convergence results. Recently, asymptotics of flat initial data ASEP  \cite{ortmann2015pfaffian}  and multi-point formulas for polymers \cite{nguyen2016variants} have likewise been stymied by similar considerations which have prevented rigorous asymptotics, despite formal critical point results agreeing with predictions.
	\vspace{.2cm}
	
	\item[Pfaffian identities] Thus, we fall short of our initial goal of proving asymptotic limit theorems. There is, however, one exception. In joint work \cite{barraquand2018stochastic} with Wheeler -- which came as an outgrowth of the present project and \cite{borodin2016between, wheeler2016refined} --  we found that for a special case of the half-space Hall-Littlewood process (itself, a special case of the Macdonald process when $q=0$, which relates to the half-quadrant stochastic six-vertex model) the Laplace transform has an alternative expression in terms of a related Pfaffian point process (the half-space or Pfaffian Schur process \cite{borodin2005eynard}). Through this identity and known asymptotic techniques for Pfaffian point processes, we were able to rigorously prove the desired type of KPZ universality class asymptotics. The presence of Pfaffians in this Laplace transform representation encapsulates some crucial  cancellations  which are not apparent in our series expression. Unfortunately, this relationship to a Pfaffian point process is presently mysterious and it is unclear if it generalizes beyond the one special case (see Section \ref{sec:ASEP} for more details).
\end{description}

\definecolor{niceblue}{RGB}{50,50, 121}
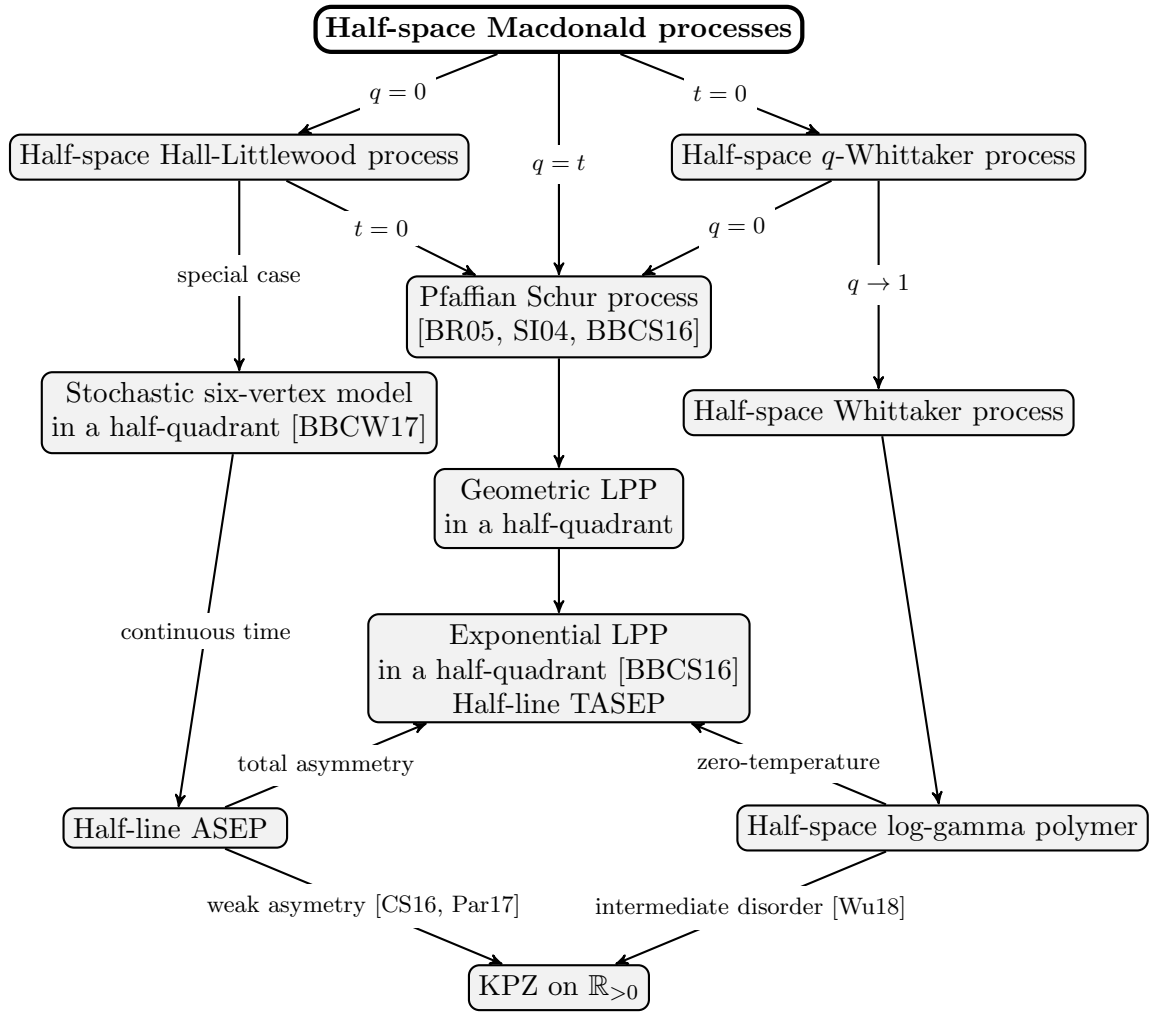
\begin{figure}
	\begin{center}
		\begin{tikzpicture}[scale=0.85, every text node part/.style={align=center}]
		\usetikzlibrary{arrows}
		\usetikzlibrary{shapes}
		\usetikzlibrary{shapes.multipart}
		\tikzstyle{structure}=[rectangle ,draw,thick ,rounded corners=4pt,fill=black!5 ]
		%		\tikzstyle{process}=[rounded rectangle, draw, thick , fill=black!20 ]
		\tikzstyle{process}=[structure]
		\tikzstyle{main}=[rectangle, draw, ultra thick, rounded corners=6pt ]
		%		\tikzstyle{rond}=[circle,draw, ultra thick, rounded corners=4pt, gray, text=gray]
		\tikzstyle{limit}=[->,>=stealth',thick]
		\tikzstyle{limitconj}=[->,>=stealth',thick,dashed]
		\tikzstyle{equiv}=[<->,>=stealth',thick]
		%		\tikzstyle{relation}=[ultra thick, gray]
		%		\tikzstyle{shortcut}=[<->,>=stealth', niceblue, thick, dashed, rounded corners=4pt]

		\node[main] (Macdonald) at (0,11) {\bf{Half-space Macdonald processes}};
		\node[structure]  (Schur) at (0,6.5) {Pfaffian Schur process \\ \cite{borodin2005eynard, sasamoto2004fluctuations, baik2018pfaffian}};
		\node[structure]  (HL) at (-5,9) {Half-space Hall-Littlewood process};
		\node[process]  (SV) at (-5,5) {Stochastic six-vertex model\\  in a half-quadrant \cite{barraquand2018stochastic}};
		\node[structure] (qW) at (5,9) {Half-space $q$-Whittaker process};
		\node[structure] (W) at (5,5) {Half-space Whittaker process};
		\node[process] (gLPP) at (0,3.5) {Geometric LPP\\in a half-quadrant};
		\node[process]  (LPP) at (0,1) {Exponential LPP\\ in a half-quadrant \cite{baik2018pfaffian}\\ Half-line TASEP};
		\node[process]  (ASEP) at (-6,-1.5) {Half-line  ASEP\  };
		%		\node[process]  (TASEP) at (0,-1.5) {Half-line TASEP\  };
		\node[process]  (KPZ) at (0,-4) {KPZ on $\R_{>0}$};
		\node[process] (polymer) at (6,-1.5) {Half-space log-gamma polymer};	
		%		
		%		\draw[relation] (refined) -- (Schur);
		%		\draw[relation] (refined) -- (Macdonald);
		%		\draw[relation] (refined) -- (HL);
		\draw[limit] (Macdonald) -- (Schur) node[midway,fill=white]{{\footnotesize $q=t$}};
		\draw[limit] (Macdonald)-- (HL) node[midway,fill=white]{\footnotesize {$q=0$}};
		\draw[limit] (Macdonald) -- (qW) node[midway,fill=white]{{\footnotesize $t=0$}};
		\draw[limit] (qW) -- (W) node[midway,fill=white]{{\footnotesize $q\to 1$}};
		\draw[limit] (qW) -- (Schur) node[midway,fill=white]{{\footnotesize $q=0$}};
		\draw[limit] (HL) -- (SV) node[midway,fill=white]{{\footnotesize special case}};
		\draw[limit] (HL) -- (Schur) node[midway,fill=white]{{\footnotesize $t=0$}};
		\draw[limit] (gLPP) -- (LPP);
		\draw[limit] (Schur) -- (gLPP);
		\draw[limit] (SV) -- (ASEP) node[midway,fill=white]{{\footnotesize continuous time}};
		%		\draw[limitconj] (SV) to (KPZ);
		\draw[limit] (ASEP) -- (KPZ) node[midway,fill=white]{{\footnotesize weak asymetry \cite{corwin2016open, parekh2017kpz}}};;
		\draw[limit] (ASEP) -- (LPP) node[midway,fill=white]{{\footnotesize total asymmetry}};
		%		\draw[equiv] (LPP) -- (TASEP);
		\draw[limit] (W) -- (polymer);
		\draw[limit] (polymer) -- (KPZ) node[midway,fill=white]{{\footnotesize intermediate disorder \cite{wu2018intermediate}}};
		\draw[limit] (polymer) -- (LPP) node[midway,fill=white]{{\footnotesize zero-temperature}};
		%		\draw[ultra thick, <->,>=stealth', niceblue] (-7,1) node[below right]{more specific models} -- (-7, 10) node[above right]{more general structures};
		\end{tikzpicture}
	\end{center}
	\caption{Hierarchy of half-space Macdonald processes and their degenerations. The arrows mean that one has to take a specialization of parameters or a scaling limit.}
	\label{fig:hierarchy}
\end{figure}

\subsection{Half-space Macdonald measures and processes}
\label{sec:beginning}

\emph{Half-space Macdonald measures}, defined more precisely in Section \ref{sec:defhalfspaceMacdonald}, are probability distributions on integer partitions $\la=(\la_1\geqslant  \la_2\geqslant \dots \geqslant 0)$ for which (the notation will be explained below)
$$ \PMM(\la) = \frac{1}{Z(\rhoup,\rhodiag)} P_{\la}(\rhoup ) \ve_{\la}(\rhodiag). $$
Here $P_{\la}$ are the Macdonald symmetric functions \cite[Chapter VI]{macdonald1995symmetric} depending on two parameters $q,t$ (we assume that these parameters take values in $[0,1)$ throughout) and $\ve_{\la}$ is another symmetric function defined  by
$$ \ve_{\la} = \sum_{\mu' \text{ even}} \bel{_{\mu}} Q_{\la/\mu},$$
where $P_{\la/\mu}, Q_{\la/\mu}$ are skew Macdonald symmetric functions, $\bel{_{\mu}}$ are explicit $(q,t)$-dependent coefficients and the sum is over all partitions which are dual even (meaning that $\mu_{2i-1}=\mu_{2i}$ for all $i$). The symbols $\rhodiag$ and $\rhoup$ represent specializations of the algebra of symmetric functions that can depend on many parameters, and $Z(\rhoup,\rhodiag)$ is the normalizing constant which is necessary to make $\PMM$ a probability measure. Section \ref{sec:backgroundSym} provides more details for all of these objects.

\emph{Half-space Macdonald processes} (Definition \ref{def:halfspaceMacdonaldprocess}) are probability measures on sequences of partitions whose marginals are half-space Macdonald measures. Macdonald symmetric functions usually take a different name when their parameters $q,t$ are specialized to certain values. We will name our half-space processes accordingly. The chart in Figure \ref{fig:hierarchy} depicts the hierarchy of these degenerations and the relations between most integrable half-space systems discussed in this paper. We will use $\PMP, \EPMP$ to denote the probability measure and expectation operator for the Macdonald process, and $\PMM, \EPMM$ for the Macdonald measure. Setting $t=0$ and $q$ general results in the \emph{$q$-Whittaker} case and we write $\PQWP, \EQWP$, where as setting $q=0$ and $t$ general results in the \emph{Hall-Littlewood} case and we likewise write $\PHL, \EHL$.  When we consider further degenerations (e.g. the Whittaker case, ASEP, etc) we will denote the probability measure and expectation by $\PP, \EE$.

\subsection{Computing expectations of observables}
At the Macdonald level we are able to compute integral formulas for various moments and  Laplace-type transforms. A general scheme for computing expectations of certain observables of Macdonald measures was introduced in \cite[Section 2.2.3]{borodin2014macdonald} and we develop this into the half-space setting. As alluded to earlier, in order to compute one of our $(q,t)$-Laplace transform formulas -- see Theorem \ref{theo:MoumiLaplaceintro} -- we need to introduce a new operator $\MoumiA_n^z$ which extends the action of $\Moumi^z$ from polynomials, but has better analytic properties. We briefly review our main Laplace transform results (we leave the moment formulas, which are in the spirit of the earlier full-space work of \cite{borodin2014macdonald}  to the main text  -- see Section \ref{sec:observables}).

Consider real variables $x_1,\ldots, x_n \in (0,1)$ and a specialization $\rho$. We have the (Littlewood) identity
$$
\sum_{\lambda} P_{\lambda}(x_1,\ldots, x_n) \ve_{\lambda}(\rho) = Z(x_1,\ldots x_n;\rho)
$$
where the function  $Z(x_1,\ldots x_n;\rho)$ has an explicit form (see Section \ref{sec:backgroundSym}). We recognize on the L.H.S. the unnormalized density of the half-space Macdonald measure.

Assume that we have a linear operator $\mathbf A_n$ acting on functions in variables $x_1, \dots, x_n$ which is diagonal in the basis of symmetric polynomials $\lbrace P_{\lambda}\rbrace$ with eigenvalues $\lbrace d_{\lambda} \rbrace$. Applying $\mathbf A_n$ to both sides of the above identity and subsequently dividing both sides by the normalizing constant $Z$,  we find that for the half-space Macdonald measure with $\rhoup=(x_1,\ldots,x_n)$ and $\rhodiag=\rho$,
\begin{equation}
\EPMM\big[d_{\lambda}]= \sum_{\lambda\in \Y} d_{\lambda} \frac{P_{\lambda}(x_1,\ldots, x_n) \ve_{\lambda}(\rho)}{Z(x_1,\ldots x_n;\rho) }  = \frac{ \mathbf A_n \ Z(x_1,\ldots x_n;\rho) }{ Z(x_1,\ldots x_n;\rho) },
\label{eq:computationobservables}
\end{equation}
though one needs to justify that the infinite summation commutes with $\mathbf A_n$.
We may also iterate the procedure to compute more complicated observables.

Such operators diagonalized by Macdonald polynomials are available. In particular, we will use Macdonald difference operators $\DD_n^r$,  and a variant of them that we denote  $\overline{\DD}_n^r$,
for which
$$ \DD_n^r P_{\la} = e_r(q^{\lambda_1}t^{n-1}, \dots, q^{\la_n}t^0) P_{\la}, \ \ \ \  \overline{\DD}_n^r P_{\la} = e_r(q^{-\lambda_1}t^{1-n}, \dots, q^{-\lambda_n}t^0)P_{\la},$$
where $e_r$ is the $r$th elementary symmetric function. 	We will also use Noumi's $q$-integral operator $\Noumi_n^z$ (see  Section \ref{sec:Noumi}), and a variant of it denoted  $\Moumi_n^z$, for which
$$ \Noumi_n^z P_{\la} =   \prod_{i=1}^n \frac{(q^{\la_i}t^{n-i+1}z; q)_{\infty}}{(q^{\la_i}t^{n-i}z; q)_{\infty}} P_{\la}, \ \ \ \ \Moumi_n^z P_{\la}  =   \prod_{i=1}^n \frac{(q^{-\la_i}t^{i}z; q)_{\infty}}{(q^{-\la_i}t^{i-1}z; q)_{\infty}} P_{\la}, $$
where $(x)_{\infty} = (x; q)_{\infty} = \prod_{i\geqslant 0} (1-q^ix)$.
Following arguments similar to \cite{borodin2014macdonald, borodin2016observables}, we can make the above approach entirely rigorous for operators $\DD_n^r$, $\overline{\DD}_n^r$ and $\Noumi_n^z $, and the action of these operators on the normalizing constant can be expressed in terms of contour integrals (see Sections \ref{sec:generalmomentsformulas} and \ref{sec:generalLaplace}). Regarding the  $q$-integral operator $\Moumi_n^z$, equation \eqref{eq:computationobservables} is not true since we cannot justify moving the operator inside the sum which defines $Z$ (in fact, this interchange is not true due to a lack of convergence when one tries to use Fubini). This is unfortunate, because the quantity
$$\EPMM\left[ \prod_{i=1}^n \frac{(q^{-\la_i}t^{i}z)_{\infty}}{(q^{-\la_i}t^{i-1}z)_{\infty}} \right]$$
is exactly what we need to compute in order to study the partition function of the log-gamma polymer.
This is why we introduce a different integral operator $\MoumiA_n^z$ which coincides with $\Moumi_n^z$ on polynomials,  but not on $Z(x_1, \dots, x_n; \rho)$, and has better analytic properties. This is the main technical novelty of the present paper regarding the computation of observables.

\medskip

Let $ a_1, \dots, a_n$ be parameters in  $(0,1)$.
Consider the half-space Macdonald measure with specializations $\rhoup=(a_1, \dots, a_n)$ and $\rhodiag=\rho$. We further assume that the parameters $ a_1, \dots, a_n$ are close enough to each other (see the statements of Theorem \ref{theo:NoumiLaplace} and Theorem \ref{theo:MoumiLaplace} for precise statements).
\begin{theoremintro}[Theorem \ref{theo:NoumiLaplace}]
	Let $z\in \C\setminus \R_{>0}$.  We have
	\begin{multline}
	\EPMM\left[ \prod_{i=1}^n \frac{(q^{\la_i}t^{n-i+1}z)_{\infty}}{(q^{\la_i}t^{n-i}z)_{\infty}} \right]= \sum_{k=0}^{n} \frac{1}{k!} \int_{R-\I\infty}^{R+\I\infty} \frac{\mathrm{d}s_1}{2\I\pi} \dots \int_{R-\I\infty}^{R+\I\infty} \frac{\mathrm{d}s_k}{2\I\pi}\
	\oint  \frac{\mathrm{d}w_1}{2\I\pi}\dots \oint \frac{\mathrm{d}w_k}{2\I\pi}  \mathcal{A}^{q,t}_{\vec s}(\vec w) \\ \times	\prod_{i=1}^k\Gamma(-s_i)\Gamma(1+s_i)  \prod_{i=1}^k \frac{\mathcal{G}^{q,t}(w_i)}{\mathcal{G}^{q,t}(q^{s_i}w_i)}\frac{\phi(w_i^2)(-z)^{s_i}}{\phi(q^{s_i}w_i^2) (q^{s_i}-1)w_i},
	\label{eq:LaplaceintroNoumi}
	\end{multline}
	where $R\in (0,1)$ is chosen such that $0<q^R<a_i/a_j$ for all $i,j$,  the positively oriented integration contours for the variables $w_j$ enclose all the $a_i$'s and no other singularity, and we have used the shorthand notations
	\begin{equation*}
	\mathcal{A}^{q,t}_{\vec s}(\vec w) := \prod_{1\leqslant i<j\leqslant k} \frac{ (q^{s_j}w_j-q^{s_i}w_i)(w_i-w_j)\phi(q^{s_i+s_j}w_i w_j)\phi(w_iw_j)}{(q^{s_i}w_i-w_j)(q^{s_j}w_j-w_i)\phi(q^{s_i}w_iw_j)\phi(q^{s_j}w_jw_i)}
	\end{equation*}
	and
	$$
	\mathcal{G}^{q,t}(w) = \prod_{j=1}^n\frac{\phi(w/a_j)}{\phi(wa_j)} \frac{1}{ \Pi(w; \rho)}, \ \ \ \phi(z) = \frac{(tz)_{\infty}}{(z)_{\infty}},
	$$
	with $\Pi(w; \rho)= \sum_{\la}P_{\la}(w)Q_{\la}(\rho)$ (see \eqref{eq:Pigeneral} for an explicit expression).
	\label{theo:NoumiLaplaceintro}
\end{theoremintro}

\begin{theoremintro}[Theorem \ref{theo:MoumiLaplace}]
	Let $z\in \C\setminus \R_{>0}$. Under mild assumptions on the specialization $\rho$, we have
	\begin{multline}
	\EPMM_{(a_1, \dots, a_n), \rho}\left[ \prod_{i=1}^n \frac{(q^{-\la_i}t^{i}z)_{\infty}}{(q^{-\la_i}t^{i-1}z)_{\infty}} \right] =  \sum_{k=0}^{n} \ \frac{1}{k!}\ \int_{R-\I\infty}^{R+\I\infty}\frac{\mathrm{d}s_1}{2\I\pi} \dots \int_{R-\I\infty}^{R+\I\infty}\frac{\mathrm{d}s_k}{2\I\pi}\   \oint\frac{\mathrm{d}w_1}{2\I\pi} \dots \oint\frac{\mathrm{d}w_k}{2\I\pi}    \mathcal{A}^{q,t}_{-\vec s}(\vec w) \\ \times
	\prod_{i=1}^k\Gamma(-s_i)\Gamma(1+s_i)  \frac{	\overline{\mathcal{G}}^{q,t}(w_i)}{	\overline{\mathcal{G}}^{q,t}(q^{-s_i}w_i)}  \frac{\phi(w_i^2) (-z)^{s_i}}{\phi(q^{-s_i}w_i^2) (1-q^{s_i})w_i} ,
	\label{eq:LaplaceintroMoumi}
	\end{multline}
	where $R\in (0,1)$ is chosen such that $a_i<q^R<a_i/a_j$ for all $i,j$, the positively oriented contours for the variables $w_j$ enclose all the $a_i$'s and no other singularity, and we have used the shorthand notation
	\begin{equation*}
	\overline{\mathcal{G}}^{q,t}(w) = \prod_{j=1}^n\frac{\phi(a_j/w)}{\phi(wa_j)}\frac{1}{\Pi(w; \rho)}.
	\label{eq:defGintro}
	\end{equation*}
	\label{theo:MoumiLaplaceintro}
\end{theoremintro}

The observables appearing in \eqref{eq:LaplaceintroNoumi} and \eqref{eq:LaplaceintroMoumi} above should be thought of as Laplace transforms. When $t=0$ and $q\to 1$ these will become exactly Laplace transforms of the random variables that we want to study. To be more precise, take $t=0$ for simplicity, using the $q$-binomial theorem \eqref{eq:qbinomial}, the L.H.S. of \eqref{eq:LaplaceintroNoumi} becomes
$$  \EQWM\left[ \frac{1}{(q^{\la_n}z)_{\infty}}  \right]  = \sum_{k=0}^{\infty}   \frac{z^k\EQWM[q^{k\la_n}]}{k!_q},$$
where $k!_q$ denotes the $q$-deformed factorial (see \eqref{eq:qfactorial}).

The proof of Theorem \ref{theo:MoumiLaplaceintro} is significantly more delicate than the proof of Theorem \ref{theo:NoumiLaplaceintro}. This is because, the L.H.S. of \eqref{eq:LaplaceintroMoumi} does not expand as a power series in $z$. Indeed, (we take again the case $t=0$ for simplicity of the exposition),
$$ \EQWM\left[ \frac{1}{(q^{-\la_1}z)_{\infty}}  \right]  \neq \sum_{k=0}^{\infty}   \frac{z^k\EQWM[q^{-k\la_1}]}{k!_q}.$$
Actually, the moments of $q^{-\la_1}$ grow too fast to determine  the distribution uniquely (and sometimes they do not even exist). Formally taking a moment generating series of moments would not yield the correct result (see Remarks \ref{rem:whymoredifficult} and  \ref{rem:subtleties}). This is why we need to work with the integral operator  $\MoumiA^z$ instead of the $q$-integral operator  $\Moumi^z$ in the proof of Theorem \ref{theo:MoumiLaplaceintro} in Section \ref{sec:generalLaplace}.

A similar moment problem issue came up in the study of full-space Macdonald processes and  \cite{borodin2015height} developed an involved argument (using formal power series in the variables of the $Q$ Macdonald polynomial) to prove the $q$-Laplace transform formula. That argument, however, cannot be applied here as there is no $Q$ polynomial or extra set of variables in which to expand. Thus, our new operator $\MoumiA_n^z$ provides the only apparent route to prove Theorem \ref{theo:MoumiLaplaceintro}. It also provides an alternative to the approach of \cite{borodin2015height} in the full-space case. 

\subsection{Models related to half-space Macdonald processes}
\subsubsection{Log-gamma polymer in a half-quadrant}

The \emph{log-gamma directed polymer} model was introduced in \cite{seppalainen2012scaling} and further studied in \cite{borodin2013log, corwin2014tropical, georgiou2013ratios, georgiou2013large, georgiou2016variational,  grange2017log, o2014geometric, nguyen2016variants, thiery2014log}. We consider a variant living in a half-quadrant of $\Z^2$.
\begin{definition}[Half-space log-gamma polymer] Let  $\alpha_1, \alpha_2, \dots $ be positive parameters and $\diag\in \R$ be such that $\alpha_i+\diag>0$ for all $i \geqslant 1$. The half-space log-gamma polymer is a probability measure on up-right paths confined in the half-quadrant $ \lbrace (i,j)\in \Z_{>0}^2 : i\geqslant j \rbrace $ (see Figure \ref{halfspaceloggamma}), where the probability of an admissible path $\pi$ between $(1,1)$ and $(n,m)$ is given by
	$$ \frac{1}{Z_{n,m}} \ \ \prod_{(i,j)\in \pi} w_{i,j},$$
	and where $\big(w_{i,j}\big)_{i\geqslant j}$ is a family of independent random variables such that for $i>j, w_{i,j}\sim \mathrm{Gamma}^{-1}(\alpha_i + \alpha_j)$ and $w_{i,i}\sim \mathrm{Gamma}^{-1}(\diag + \alpha_i)$. The notation  $\mathrm{Gamma}^{-1}(\theta)$ denotes the inverse of a Gamma distributed random variable with shape parameter $\theta$ (see Definition \ref{def:Gammadist}).
	The partition function $Z(n,m)$ is given by
	$$Z(n,m) = \sum_{\pi: (1,1) \to (n,m)} \prod_{(i,j)\in \pi} w_{i,j}. $$
	\label{def:LogGammapolymer}
\end{definition}
\begin{figure}
	\begin{tikzpicture}[scale=0.8]
	\draw[->, thick, >=stealth', gray] (0,0) -- (10, 0);
	\draw[->, thick, >=stealth', gray] (0,0) -- (0, 5.5);
	\draw[thick, gray]  (0,0) -- (5.5,5.5);
	\draw (-.5, -.2) node{{\footnotesize $(1,1)$}};
	\draw (9.6, 5.3) node{{\footnotesize $(n,m)$}};
	\draw[->, >=stealth'] (10,2.5) node[anchor=west]{{\footnotesize $w_{9,3} \sim \mathrm{Gamma}^{-1}(\alpha_9 + \alpha_3) $}} to[bend right] (8,2);
	\draw[->, >=stealth'] (1,4) node[anchor=south]{{\footnotesize $w_{4,4} \sim \mathrm{Gamma}^{-1}(\diag + \alpha_4) $}} to[bend right] (3,3);
	\clip (0,0) -- (9.5,0) -- (9.5,5.5) -- (5.5, 5.5) -- (0,0);
	\draw[dotted, gray] (0,0) grid (10,6);
	
	\draw[ultra thick] (0,0) -- (1,0) -- (2,0) -- (2,1) -- (2,2) -- (3,2) -- (4,2) -- (4,3) -- (5,3) -- (6,3) -- (6,4) -- (7,4) -- (8,4) -- (8,5) -- (9,5);
	\end{tikzpicture}
	\caption{An admissible path in the half space log-gamma polymer.}
	\label{halfspaceloggamma}
\end{figure}
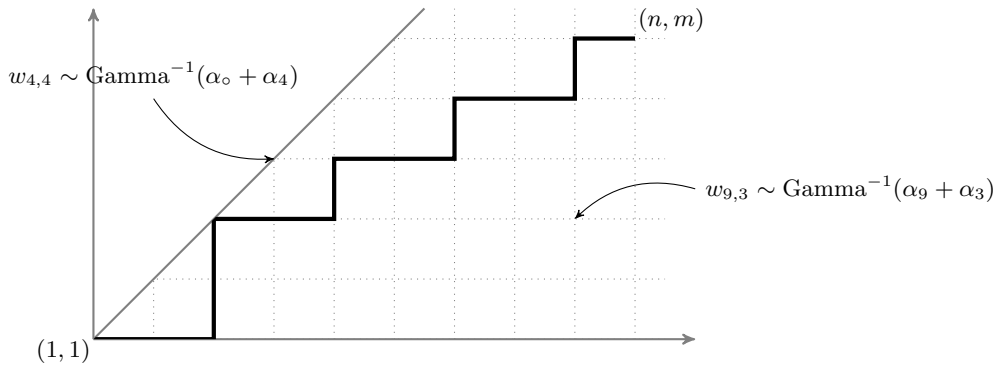

We show  (see Proposition \ref{prop:equalityZwhittaker} below) that the observable $q^{-\lambda_1}$ admits a limit to $Z(t,n)$ when the parameters of the  half-space  Macdonald measure are scaled correctly, $t=0$ and $q$ goes to $1$ in an appropriate manner. This limit corresponds to the half-space  Whittaker process  discussed in Section \ref{sec:Whittaker}. To prove this result, we study Markovian dynamics preserving half-space Macdonald processes, (actually,  $q$-Whittaker processes), and interpret them in terms of new integrable particle systems (see Section \ref{sec:newprocesses}). We then show (following arguments from \cite{matveev2015q}) that under these dynamics, $q^{-\lambda_1}$ satisfies a recurrence relation which, in the $q \to 1$ limit,  relates it to the half-space log-gamma polymer partition function.

Taking degenerations of integral formulas obtained for general Macdonald measures, we obtain the following moment formula (see Corollary \ref{cor:momentsZ} below). For $t\geqslant n$ and  $k\in\Z_{>0}$ such that $k<  \min\lbrace 2\alpha_i, \alpha_i+\diag \rbrace $,
\begin{multline}
\EWM[Z(t,n)^k] =\oint\frac{\mathrm{d}w_1}{2\I\pi}\cdots \oint\frac{\mathrm{d}w_k}{2\I\pi} \prod_{1\leqslant a<b\leqslant k} \frac{w_a-w_b}{w_a-w_b-1}\, \frac{1+w_a+w_b}{2+ w_a+w_b}\\ \times
\prod_{m=1}^{k}  \frac{1+2w_m}{1+w_m-\diag}\prod_{i=1}^t \left( \frac{1}{\alpha_i-w_m-1} \right) \prod_{j=1}^{n} \left(\frac{1}{w_m + \alpha_j}\right),
\label{eq:momentsZintro}
\end{multline}
where the contours are such that for all $1\leqslant c\leqslant k$, the contour for $w_c$ encloses $\lbrace - \alpha_j\rbrace_{1\leqslant j\leqslant n}$  and $\lbrace w_{c+1}+1, \dots, w_k+1\rbrace$, and excludes the poles of the integrand at $ \diag- 1$ and $ \alpha_j-1$ (for $1\leqslant j\leqslant t$).
Note that if $k>\alpha_i+\alpha_j$ or $k>\alpha_i+\diag$ for some $i<j$, 
the $k$-th moment of $Z(t,n)$ fails to exist.

In order to go around certain technical issues, we also define a hybrid polymer model corresponding to a sort of convolution of the half-space log-gamma polymer and the O'Connell-Yor semi-discrete Brownian polymer. We will not give its exact definition for the moment but refer to Definition \ref{def:hybridpolymer} below for the details.  We denote its partition function by $Z(t,n, \tau)$, where $\tau$ is a positive parameter (which correspond to the time in the O'Connell-Yor polymer). The random variable $Z(t,n, \tau)$ weakly converges to the log-gamma partition function $Z(t,n)$ as $\tau$ goes to zero, so that $Z(t,n, \tau)$ can be thought of as a regularization of $Z(t,n)$.

The moments of the partition function $Z(t,n)$ (and $Z(t,n, \tau)$ as well) grow too fast to determine its distribution uniquely. Nonetheless, by taking appropriate  degenerations  of Theorem \ref{theo:MoumiLaplaceintro}, we are able to characterize the distribution of $Z(t,n)$ via the following Laplace transform formula, which is proved as Corollary \ref{cor:LaplaceWhittaker} in the text (or more precisely a consequence of it stated as \eqref{eq:LaplaceZ2}).

If the parameters $\alpha_i>0$ are sufficiently close to each other, for any $t\geqslant n\geqslant 1$, $\tau>0$ and  $u>0$,
\begin{multline}
\EWM[e^{ -u Z(t,n, \tau)}] = \sum_{k=0}^{n} \ \frac{1}{k!}\ \int_{R-\I\infty}^{R+\I \infty}\frac{\mathrm{d}z_1}{2\I\pi} \dots \int_{R-\I\infty}^{R+\I \infty} \frac{\mathrm{d}z_k}{2\I\pi}\   \oint\frac{\mathrm{d}v_1}{2\I\pi} \dots \oint\frac{\mathrm{d}v_k}{2\I\pi}   \\  \times
\prod_{1\leqslant i<j\leqslant k} \frac{(z_i-z_j)(v_i-v_j)\Gamma(v_i+v_j)\Gamma(-z_i-z_j)}{(z_j+v_i)(z_i+v_j)\Gamma(v_j-z_i)\Gamma(v_i-z_j)}\\ \times
\prod_{i=1}^k \left[\frac{\pi}{\sin(\pi(v_i+z_i))} \frac{\overline{\mathcal{G}}(v_i)}{\overline{\mathcal{G}}(-z_i)}  \frac{\Gamma(2v_i)}{\Gamma(v_i-z_i)} \frac{u^{z_i+v_i}}{z_i+v_i} \right],
\label{eq:LaplaceWhittakerintro}
\end{multline}
where $R$ is chosen so that for all $i$,  $-\alpha_i<R<\min\lbrace  0 , \diag, 1-\alpha_i \rbrace$, the contours for each variable $v_i$ are positively oriented circles enclosing the poles $\lbrace \alpha_j\rbrace_{1\leqslant j\leqslant n}$ and no other singularity of the integrand, and
$$ \overline{\mathcal{G}}(v) = \frac{e^{-\tau v^2/2}}{\Gamma(\diag +v)}  \dfrac{\prod_{j=1}^{n} \Gamma(\alpha_j-v)   }{\prod_{j=1}^t\Gamma(\alpha_j+v)}  .$$

An important feature of \eqref{eq:LaplaceWhittakerintro} is the cross product
$$\prod_{1\leqslant i<j\leqslant k} \frac{\Gamma(v_i+v_j)\Gamma(-z_i-z_j)}{\Gamma(v_j-z_i)\Gamma(v_i-z_j)}.$$
If the Gamma functions were replaced by their rational approximation around zero, that is $\Gamma(z)\sim 1/z$, we could recast the right hand-side of \eqref{eq:LaplaceWhittakerintro} as the Fredholm Pfaffian of an explicit kernel and the asymptotic analysis would become much easier. Unfortunately, this cross products grows with $k$ as $e^{ck^2}$, which makes it difficult to control  the series \eqref{eq:LaplaceWhittakerintro} as the number of terms $n$ goes to infinity. Similar issues involving a cross product with Gamma factors have been encountered several times in exact formulas for models in the KPZ universality class, in particular in \cite{ortmann2015pfaffian, nguyen2016variants}.

\subsubsection{Relation to the work of O'Connell-Sepp\"al\"ainen-Zygouras}
\label{sec:OSZ}
A model equivalent to the half-space log-gamma polymer model was considered in \cite{o2014geometric}. It corresponds to a log-gamma polymer model where paths live in the first quadrant, as in the usual log-gamma polymer, but the weights $\tilde w_{i,j}$ are symmetric with respect to the first diagonal ($\tilde w_{i,j} = \tilde w_{j,i} $). Off diagonal weights are distributed as $w_{i,j}\sim\mathrm{Gamma}^{-1}(\alpha_i + \alpha_j)$ while the diagonal weights are distributed as $w_{i,i}\sim\frac 1 2 \mathrm{Gamma}^{-1}(\alpha_i + \diag)$. One can identify the weight of a path in this model with the weight of a path in the half-space log-gamma polymer from Definition \ref{def:LogGammapolymer} up to a factor $(1/2)^{k}$ where $k$ is the number of times the path hits the diagonal. Since there are $2^{k-1}$ path in the symmetrized model which correspond to the same path in the half-space model, the partition function  $\tilde Z(t,n)$ of the symmetrized model is such that  $\widetilde Z(t,n) = \frac 1 2 Z(t,n)$. When $t=n$,  the law of  $\widetilde Z(n,n)$ is a marginal of the push-forward of a symmetric matrix with inverse Gamma random variables by the geometric RSK algorithm.  \cite{o2014geometric} computed this pushforward (and hence the distribution of $\widetilde Z(n,n)$) as the Whittaker measure (with slightly different notations than in the present paper). By a formal (see Section \ref{sec:relationwithOSZ})  application of the Plancherel theorem for Whittaker functions, they derived a conjectural formula for the Laplace transform of $\widetilde Z(n,n)$ \cite[(5.15), (5.16)]{o2014geometric}. 
Though \cite{o2014geometric} was unable to prove this formula, they suggested that ``it seems reasonable to expect the integral formulas (5.15) and (5.16) to be valid, at least in some suitably regularized sense.'' In our present work, we show that our hybrid polymer provides such an appropriate regularization. The reason why this was inaccessible to \cite{o2014geometric} was that their results (and connection to Whittaker measures) were restricted to the diagonal and the hybrid polymer requires working off-diagonal,  hence the interest of our study of Markov dynamics on half-space Macdonald processes. Using Whittaker Plancherel theory for our hybrid model and letting $\tau$ go to zero, we obtain the following formula. 

For $t\geqslant n$ and any  $u>0$,  we have (see Corollary \ref{cor:OSZproof} below)
\begin{multline}
\EWM[e^{-u Z(t,n)}] = \frac{1}{n!} \int_{r-\I\infty}^{r+\I\infty}\frac{\mathrm{d}z_1}{2\I\pi} \dots \int_{r-\I\infty}^{r+\I\infty}\frac{\mathrm{d}z_n}{2\I\pi}  \prod_{i\neq j} \frac{1}{\Gamma(z_i-z_j)} \prod_{1\leqslant i<j\leqslant n} \frac{\Gamma(z_i+z_j)}{\Gamma(\alpha_i+\alpha_j)}
\prod_{i,j=1}^{n} \Gamma(z_i - \alpha_j)\\ \times
\prod_{i=1}^{n}\left(u^{ \alpha_i -z_i} \frac{\Gamma(\diag+z_i)}{\Gamma(\diag + \alpha_i)} \prod_{j=n+1}^t \frac{\Gamma(\alpha_j+z_i)}{\Gamma(\alpha_j+ \alpha_i)}\right)
\label{eq:OSZintro}
\end{multline}
where $r>0$ is such that  $r+\diag >0$ and  $r>\alpha_i$ for all $1\leqslant i \leqslant n$. We show in Corollary \ref{cor:OSZproof2} how to deduce rigorously \cite[(5.15), (5.16)]{o2014geometric} from the above formula. Note that \eqref{eq:OSZintro} is more general since we consider the partition function at any point $(t,n)$, not only when $t=n$. 

Applying the geometric RSK algorithm to inverse Gamma distributed matrices with other types of symmetries was further considered in \cite{bisi2019point} but the corresponding polymer models do not seem to be related to the present paper. 
Dynamics on Gelfand-Tsetlin patterns restricted by a wall were studied in Nteka's PhD thesis \cite{nteka2016positive}, but it is not clear if this is related to our present paper. 

\subsubsection{Half-space stochastic six vertex model}

The \emph{stochastic six-vertex model} was introduced in \cite{gwa1992six} and further studied in \cite{borodin2016stochasticsix}. It was related to (full-space) Hall-Littlewood processes in \cite{borodin2016stochastic} (see also \cite{borodin2016between, bufetov2017hall}). Half-space variants of the stochastic six-vertex model and half-space Hall-Littlewood processes were discussed in \cite{barraquand2018stochastic}.

\begin{figure}
	\begin{center}
		\begin{tikzpicture}[scale=0.7]
		\foreach \x in {1, ..., 7} {
			\draw[dotted] (0,\x) -- (\x,\x) -- (\x, 8);
			\draw[path] (0,\x) --(0.1,\x);
			%				\draw[gray] (\x, \x-0.5) node{$\x$};
			%				\draw[gray] (-0.5,\x) node{$\x$};
		}
		\draw[path] (0,1) --(1,1);
		\draw[path] (0,2) --(2,2);
		\draw[path] (0,3) --(3,3);
		\draw[path] (0,4) -- (1,4) -- (2,4) -- (2,5) -- (3,5) -- (3,6) -- (4,6) -- (4,7)--(7,7);
		\draw[path] (4,4) -- (4,6) -- (6,6);
		\draw[path] (5,5) -- (5,8);
		\draw[path] (0,5) --(1,5) -- (1,6) -- (2,6) -- (2,7) -- (3,7) -- (3,8);
		\draw[path] (0,6) --(1,6) -- (1,7) -- (2,7) -- (2,8) ;
		\draw[path] (0,7) --(1,7) -- (1,8)  ;
		\end{tikzpicture}
	\end{center}
	\caption{Sample configuration of the stochastic six-vertex model in a half-quadrant.}
	\label{fig:example6v}
\end{figure}
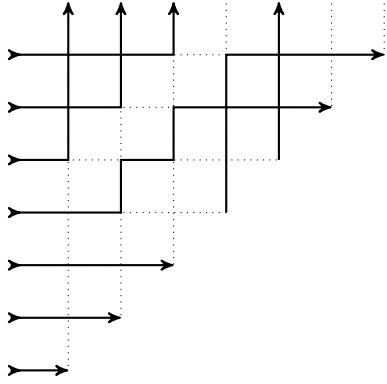
\begin{definition}
	Consider the half-quadrant $ \lbrace (x,y)\in \Z_{>0}^2: x\leqslant y \rbrace $.
	The stochastic six-vertex model in the half-quadrant is a probability measure on collections of up-right paths (see Figure \ref{fig:example6v}). We associate to each vertex a Boltzmann weight determined by the local configuration of adjacent paths. In the bulk, for a vertex $(x,y)$ with $x>y$,  there are six possible configurations and we choose the Boltzmann weights as
	$$ \PP\left(\pathrr\right) = \frac{1-a_xa_y}{1-ta_xa_y}, \ \ \PP\left(\pathru\right) = \frac{(1-t)a_xa_y}{1-ta_xa_y},\ \ \PP\left(\pathuu\right) = \frac{t(1-a_xa_y)}{1-ta_xa_y}, \ \ \PP\left(\pathur\right) =\frac{1-t}{1-ta_xa_y}. $$
	For a corner vertex of the form $(x,x)$,  we choose weights as 
	$$ \PP\left(\pathbrr\hspace{0.2cm}\right) =  \PP\left(\pathbuu\hspace{0.2cm}\right) = 1,\ \ \ \PP\left(\pathbru\hspace{0.2cm}\right) =  \PP\left(\pathbur\hspace{0.2cm}\right) =0.$$
    These weights are stochastic in the sense that
	\begin{align*}
	\PP\left(\pathrr\right) + \PP\left(\pathru\right) = \PP\left(\pathur\right)+\PP\left(\pathuu\right)=1,\qquad
	\PP\left(\pathbrr\right) + \PP\left(\pathbru\right) = \PP\left(\pathbuu\right)+\PP\left(\pathbur\right)=1.
	\end{align*}
	We define a probability measure on configurations of up-right paths as follows. We assume that there is an incoming   horizontal edge to each vertex $(1,y)$ on the left boundary. Assume that for some $n\geqslant 2$, the incoming edge states of the set of vertices $\lbrace (x,y)\rbrace_{x+y=n}$	are
	all determined. Choose the outgoing edge states of these vertices by sampling from the Bernoulli distribution imposed by the vertex weights above. This determines the incoming states of the set of vertices $\lbrace (x,y)\rbrace_{x+y=n+1}$, and  iterating this procedure defines the probability distribution of configurations on the whole half-quadrant. This implies that the probability distribution of the restriction of the configuration to a finite set of vertices near the origin such as in Figure \ref{fig:example6v} is given by the product of Boltzmann weights.  
	 We refer to  \cite[Section 3]{barraquand2018stochastic} for a more precise definition. 
	\label{def:sixvertex}
\end{definition}

Using the relation between half-space Hall-Littlewood measures and the stochastic six-vertex model established in \cite{barraquand2018stochastic} (see Theorem \ref{th:HL6V} below), we obtain moment formulas for the height function, stated as Corollary \ref{cor:moments6v} in the text:
\begin{multline*}
\EHL\left[ t^{-k\mathfrak{h}(x,y)} \right] =   t^{\frac{k(k-1)}{2}} \oint_{C_1} \frac{\mathrm{d}z_1}{2\I\pi} \dots  \oint_{C_k} \frac{\mathrm{d}z_k}{2\I\pi}  \  \prod_{1\leqslant i<j\leqslant k}\frac{z_i-z_j}{z_i- tz_j} \frac{1-tz_iz_j}{1-z_iz_j}  \\ \times 	\prod_{j=1}^k\left(  \frac{1}{z_j} \frac{1-tz_j^2}{1-z_j^2} \prod_{i=1}^{y} \frac{1-a_iz_j}{1-ta_iz_j} \prod_{i=1}^x \frac{z_j-a_i/t}{z_j-a_i} \right),
\end{multline*}
where the contours $C_1, \dots, C_m$ all enclose $0$ and the $a_i$ are contained in the open disk of radius $1$ around zero, and the contours are nested in such a way that for $i<j$ the contour $C_i$ does not include any part of $t C_j$.

The half-space six vertex model is a discrete time version of the half-line ASEP (see Definition \ref{def:halflineASEP}).  The formula above is similar to nested contours integral formulas obtained in \cite{borodin2012duality} for the full-space ASEP using coordinate Bethe ansatz. It is likely that the formula above can be obtained through coordinate Bethe ansatz as well, and we plan to study this further in future work.

Owing to a refined Littlewood identity originally conjectured in \cite{betea2015refined} and later proved in \cite{rains2014multivariate}, \cite{barraquand2018stochastic} determined -- for a certain initial data and a specific boundary condition -- the distribution of the height function at the boundary for ASEP and the KPZ equation in a half-space using a limit of the half-space stochastic six-vertex model. Note that with the techniques of \cite{barraquand2018stochastic} is was possible to characterize the distribution of $\mathfrak{h}(x,y)$ only when $x=y$.

\subsection{Asymptotics}

We turn to the asymptotic results that can be derived (at least formally) from our formulas.

\subsubsection{Log-gamma polymer}

In Section \ref{sec:asymptotics}, we perform an asymptotic analysis of our Laplace transform formula \eqref{eq:LaplaceWhittakerintro} as $n$ goes to infinity. We assume that the parameters $\alpha_i$ of the log-gamma polymer are all equal to some $\alpha>0$, and we keep the boundary parameter $\diag$ arbitrary (thus we have weights distributed as $\mathrm{Gamma}^{-1}(\diag+\alpha)$ on the boundary and $\mathrm{Gamma}^{-1}(2\alpha)$ in the bulk).  A non-rigorous application of Laplace's method yields the following limit laws\footnote{The digamma and polygamma functions are defined as $\Psi(z) = \frac{\rm d}{\mathrm{d} z} \log(\Gamma(z))$ and $\Psi_n(z) = \frac{\mathrm{d}^n}{\mathrm{d}z^n} \Psi(z)$.}.
When $\diag>0$,
\begin{equation}
\lim_{n\to \infty} \PP\left(\frac{\log(Z(n,n)) -f n}{\sigma n^{1/3}} \leqslant x\right)  = F_{\rm GSE}(x),
\label{eq:GSElimit}
\end{equation}
where the quenched free energy $f=-2\Psi(\alpha)$ and $\sigma = \sqrt[3]{\Psi_2(\alpha)}$.
When $\diag =0$,
\begin{equation}
\lim_{n\to \infty} \PP\left(\frac{\log(Z(n,n)) -f n}{\sigma n^{1/3}} \leqslant x\right)  = F_{\rm GOE}(x),
\label{eq:GOElimit}
\end{equation}
with the same free energy $f=-2\Psi(\alpha)$ and $\sigma = \sqrt[3]{\Psi_2(\alpha)}$.
When $\diag<0$,
\begin{equation}
\lim_{n\to \infty} \PP\left(\frac{\log(Z(n,n)) -f_{\diag} n}{\sigma_{\diag} n^{1/2}} \leqslant x\right)  = \int_{-\infty}^x \frac{e^{-t^2/2}}{\sqrt{2\pi}}\mathrm{d}t,
\label{eq:Gaussianlimit}
\end{equation}
where the free energy becomes  $f_{\diag} = -\Psi(\alpha-\diag) -\Psi(\alpha+\diag)$ and $\sigma_{\alpha} = \sqrt{\Psi_1(\alpha+\diag) -\Psi_1(\alpha-\diag)}$. Furthermore, if we scale $\diag$ close to the critical point as $\diag = n^{-1/3}\sigma^{-1} \varpi$, $F_{\rm GOE}$ would be replaced in \eqref{eq:GOElimit} by a crossover distribution $F(x; \varpi)$  such that  $F(x; 0) = F_{\rm GOE}(x)$ and $\lim_{\varpi\to\infty} F(x; \varpi) = F_{\rm GSE}(x)$. It was introduced in \cite[Definition 4]{baik2001asymptotics} in relation with asymptotics of half-space last passage percolation with geometric weights (see also \cite{forrester2006correlation, baik2018pfaffian, baik2018facilitated, betea2018free}).

Let us make clear that unlike all results stated previously, Equations \eqref{eq:GSElimit}, \eqref{eq:GOElimit} and \eqref{eq:Gaussianlimit} are not completely proved, our asympotics are non-rigorous at the level of neglecting convergence of tails of series and only focusing on critical points. Making these rigorous constitutes a significant challenge.

As $\alpha$ goes to zero, the free energy of the (half-space) log-gamma polymer converges to the last passage time in a model of last passage percolation with exponential weights in a half-quadrant. This model was considered in \cite{baik2018pfaffian}, where the analogues (as $\diag, \alpha\to 0$) of the limit laws \eqref{eq:GSElimit}, \eqref{eq:GOElimit} and \eqref{eq:Gaussianlimit} were proved.

It is reasonable to expect that when $\diag+\alpha$ is  close to zero, the boundary weights will be so large that their contribution to the free energy will dominate and fluctuations will be Gaussian on the $n^{1/2}$ scale. On the contrary, if $\diag \gtrsim \alpha$ we expect that the effect of the boundary should be limited, and fluctuations should occur on the scale $n^{1/3}$ by KPZ universality. We explain in Section \ref{sec:BaikRainsexplained} how to predict the critical $\diag$ between Gaussian and KPZ behaviour. We also provide heuristic arguments to explain the expression of the constants $f,f_{\diag}$ and  $\sigma_{\diag}$ above  in \eqref{eq:GSElimit}, \eqref{eq:GOElimit} and \eqref{eq:Gaussianlimit}.

\subsubsection{KPZ equation limit regime}

Among models in the KPZ universality class, a central object is the \emph{KPZ equation} -- a stochastic partial differential equation which reads
\begin{equation}
\partial_t h(t,x) = \tfrac 1 2 \Delta h(t,x) + \tfrac 1 2 \big[ \partial_x h(t,x) \big]^2 + \xi \ \  \ \ t>0,\  x\in \R,
\label{eq:KPZintro}
\end{equation}
where the $\xi$ is a Gaussian space-time white noise. This stochastic PDE plays an important role because
many models with a parameter controlling the asymmetry or the temperature converge to the KPZ equation under a certain  scaling \cite{alberts2014intermediate, dembo2016weakly, hairer2015class, diehl2017kardar}. This fact is generally referred to as weak universality. Exact formulae characterizing the distribution of ASEP or directed polymers yield, after appropriate scaling, information about the distribution of the solutions to \eqref{eq:KPZintro}. This approach have been successfully implemented in \cite{amir2011probability, sasamoto2010exact, calabrese2010free, dotsenko2010replica, borodin2015height} to determine the one-point distribution of the KPZ equation on the line $\R$, starting from several types of initial data. Since we know that statistics in the KPZ class usually depend on the geometry, it is natural to ask how the distribution would change for the KPZ equation on another spatial domain, for instance a circle, a segment, or a half-line. A partial answer is provided in \cite{barraquand2018stochastic} for the KPZ equation on $\R_{\geqslant 0}$ with Neumann type boundary condition $\partial_x h(t,0)=-1/2$.

As in the full-space case  \cite{borodin2014macdonald, borodin2015height}, our formulas for the half-space log-gamma should give, in the appropriate scaling regime, distributional information about the KPZ equation on $\R_{\geqslant 0}$ with Neumann type boundary condition $\partial_x h(t,0)=A$ (or, in other terms, the free energy of the continuous directed polymer model with a pinning at the boundary). It is not clear how to rigorously  take asymptotics  of  our Laplace transform formulas \eqref{eq:LaplaceWhittakerintro} or \eqref{eq:OSZintro} in the appropriate scaling regime. However, we consider in Section \ref{sec:KPZ} the limit of our moment formula \eqref{eq:momentsZintro}  and we do recover moment formulas obtained\footnote{The approach in \cite{borodin2016directed} requires  uniqueness of the system of ODEs  defining the delta Bose gas in a half-space \cite[Definition 3.2]{borodin2016directed}, which has not been proved.} in \cite{borodin2016directed} for the partition function of the half-space continuous directed polymer (see Corollary \ref{cor:momentsKPZ}). In particular, we relate the parameter $\diag$ of the half-space log-gamma polymer with the parameter $A$ involved in the boundary condition for the KPZ equation (we simply have $\diag=A+1/2$ with our scalings).

\subsection{Outline of the paper}

\begin{description}	
	%	\item[Section \ref{sec:intro}]
	\item[Section \ref{sec:Half-spaceMacdonaldprocess}] After providing some background on Macdonald symmetric functions, we define half-space Macdonald measures and processes. We also provide a general scheme to build dynamics on sequences of partitions preserving the class of half-space Macdonald processes.
	\item[Section \ref{sec:observables}] We use operators, in particular Macdonald difference operators and Noumi's $q$-integral operator, diagonalized by Macdonald symmetric functions, to produce integral formulas for moments and Laplace transforms of Macdonald measures for general $q,t$ parameters.
	\item[Section \ref{sec:qWhittaker}] We study in more detail the case $t=0$, $q\in (0,1)$, called half-space $q$-Whittaker process. We consider the degeneration of general formulas and study a particular class of dynamics related to $q$-deformations of the RSK algorithm, introduced in \cite{matveev2015q}. This allows us to relate the half-space $q$-Whittaker process to the distribution of certain $q$-deformed particle systems (Section \ref{sec:newprocesses}).
	\item[Section \ref{sec:HL}]  We study in more details the case $q=0$, $t\in (0,1)$, called Hall-Littlewood process. We provide moment formulas, which, using results from \cite{barraquand2018stochastic}, relate to the stochastic six-vertex model in a half-quadrant.
	\item[Section \ref{sec:Whittaker}] We consider the $q\to1$ degeneration of the half-space $q$-Whittaker process, called half-space Whittaker process. The $q\to 1$ degeneration of the dynamics studied in Section \ref{sec:qWhittaker} gives rise to the half-space log-gamma polymer. We also consider the degeneration of our integral formulas and relate them to the log-gamma directed polymer partition function.
	\item[Section \ref{sec:KPZ}] We define the KPZ equation on the positive reals and consider the scaling of the log-gamma directed polymer at high temperature which should lead to the continuous directed polymer, whose free energy solves the KPZ equation. We show that under these scalings, our moment formulas coincide with  moment formulas previously obtained for the continuous directed polymer in a half-space in \cite{borodin2016directed}.
	\item[Section \ref{sec:asymptotics}] We consider asymptotics of the free energy $\log\big(Z(n,n)\big)$ of the half-space log-gamma polymer. We first provide probabilistic heuristics to predict the constants arising in limit theorems, for different ranges of $\diag$.  Then, we explain how, for each possible range of $\diag$,  the Tracy-Widom GSE or GOE  or the Gaussian CLT arise from the Laplace transform formula \eqref{eq:LaplaceWhittakerintro}. These asymptotics are, however, non rigorous.
\end{description}

\subsection*{Acknowledgements}

We are grateful to the anonymous referees for their valuable suggestions. A. B. was partially supported by the National Science Foundation grants DMS-1056390, DMS-1607901 and DMS-1664619 and by Fellowships of the Radcliffe Institute for Advanced Study and the Simons
Foundation. I.C. was partially supported by the NSF grants DMS-1208998,  DMS-1664650 and DMS:1811143, the Clay Mathematics
Institute through a Clay Research Fellowship, the Poincar\'e Institute through the Poincar\'e chair,
and the Packard Foundation through a Packard Fellowship for
Science and Engineering. G.B. was partially supported by the Packard Foundation through I.C.’s fellowship.

% \newpage
 \thispagestyle{plain}
 \section{Half-space Macdonald processes}

\label{sec:Half-spaceMacdonaldprocess}
After fixing some useful notations and providing background on symmetric functions, we define in this Section  half-space Macdonald processes and explain a general scheme to build Markov dynamics preserving the set of such measures.  

\subsection{$q$-analogues}
\label{sec:qanalogues}
%In this preliminary section, we fix some useful notations. 
Throughout the paper, we assume that $0\leqslant q<1$. Recall the definition of the $q$-Pochhammer symbol 
$$(a; q)_n = (1-a)(1-aq) \dots (1-aq^{n-1}) \text{ and } (a; q)_{\infty} = \prod_{i=0}^{\infty}(1-aq^i).$$ 
When there is no ambiguity possible, we may write simply $(a)_{\infty}$ instead of $(a; q)_{\infty}$. Since 
$\frac{1-q^n}{1-q} $ goes to $n$ as $q$ goes to $1$, it is natural to define the $q$-integer $[n]_q$,  the $q$-factorial $k!_{q}$, and $q$-binomial coefficients $\binom{n}{k}_{\!\!q}$ as 
\begin{equation}
  [n]_q=\frac{1-q^n}{1-q}, \ \  k!_{q} = \frac{(q;q)_k}{(1-q)^k} = [1]_q \dots [k]_q, \ \ \binom{n}{k}_{\!\!q} = \frac{(q;q)_n}{(q; q)_k (q;q)_{n-k}}.
\label{eq:qfactorial}
\end{equation}
The $q$-binomial theorem states that for $\vert z \vert <1$, 
\begin{equation}
\sum_{k=0}^{\infty} \frac{z^k(a; q)_{k}}{(q;q)_k}  = \frac{(az; q)_{\infty}}{(z; q)_{\infty}}.
\label{eq:qbinomial}
\end{equation}
The $q$-exponential function is defined as 
\begin{equation}
e_q(z) = \frac{1}{(z(1-q); q)_{\infty}}.
\label{eq:qexponential}
\end{equation}
The $q$-binomial theorem shows that $e_q(z)$ converges as $q$ goes to $1$ to the usual exponential $e^x$ uniformly on any compact set in the complex plane. 
The $q$-Gamma function is defined by 
\begin{equation}
\Gamma_q(z) =  \frac{(q; q)_{\infty}(1-q)^{1-z}}{(q^z; q)_{\infty}}.
\label{eq:qGamma}  
\end{equation} 
When $z$ is not a negative integer, $\Gamma_q(z)$ converges to $\Gamma(z)$ as $q$ goes to $1$. We refer to \cite{andrews1999special} for more details.

\subsection{Background on Macdonald symmetric functions}
\label{sec:backgroundSym}
For a more comprehensive overview on (Macdonald) symmetric functions, see \cite{macdonald1995symmetric} or \cite[Section 2.1]{borodin2014macdonald}.

\subsubsection{Partitions and Gelfand-Tsetlin patterns}
A \emph{partition} $\lambda$ is a non-increasing sequence of non-negative integers $\lambda_1\geqslant  \lambda_2\geqslant \cdots$. The \emph{length} of $\lambda$ is the number of its non-zero parts and is denoted by $\ell(\lambda)$. The \emph{weight} of $\lambda$ is denoted as $|\lambda| := \sum_i \lambda_i$. If $|\lambda|=n$, one says that $\lambda$ partitions $n$, notation  $\lambda \vdash n$. The transpose $\lambda'$ of a partition is defined by $\lambda'_i = |\{j:\lambda_j\geqslant i\}|$. Let $\Y$ be the set of all partitions and let $\Y_k$ be its subset of partitions of length $k$. We will generally use Greek letters like $\lambda,\mu,\kappa,\nu,\pi,\tau$ to represent partitions. The empty partition (such that $\la_1=0$) is denoted by $\varnothing$. We will denote by $m_i(\la)$ the multiplicity of the integer $i$ in the sequence $\la$, and sometimes use the notation $\la=1^{m_1}2^{m_2}\dots$. 

A partition can be identified with a \emph{Young diagram} or with a particle configuration in which for each $i$, there is a corresponding particle at position $\lambda_i$ (see Figure \ref{fig:partitions}). For a box $\Box$ in a Young diagram, $\mathrm{leg}(\Box)$ is equal to the number of boxes in the diagram below it (the leg length) and $\mathrm{arm}(\Box)$ is equal to the number of boxes in the diagram to the right of it (the arm length). A partition is \emph{even} if all $\lambda_i$ are even. We write $\mu\subseteq \lambda$ if $\mu_i\leqslant \lambda_i$ for all $i$ and call $\lambda/\mu$ a \emph{skew Young diagram}.

A partition $\mu$ \emph{interlaces} with $\lambda$ if for all $i$, $\lambda_i \geqslant \mu_i \geqslant \lambda_{i+1}$. In the language of Young diagrams, this means that $\lambda$ can be obtained from $\mu$ by adding a \emph{horizontal strip} in which at most one box is added per column. We denote interlacing by $\mu \prech \lambda$.  
In terms of the particle representation, interlacing refers to the interlacing of the locations of the two sets of particles. See Figure \ref{fig:partitions} for illustrations of some of these definitions. 
\begin{figure}
	\begin{center}
		\begin{tikzpicture}[scale=0.52]
		\draw (3, 1) node{(i)};
		\foreach \y [count=\xi] in {5,3,1,1,0}
		\foreach \x in {0, ..., \y} 
		\draw (\x, -\xi) -- ++(1,0) -- ++(0,1) -- ++(-1,0)--cycle;
		\fill[black] (1, -1) -- ++(1,0) -- ++(0,1) -- ++(-1,0)--cycle;
		\draw[<->, >=stealth'] (2, -0.5) -- (6, -0.5) node[midway, fill=white, fill opacity=0.4, text opacity=1, anchor=north]{arm}; 
		\draw[<->, >=stealth'] (1.5, -1) -- (1.5, -4) node[midway, fill=white, fill opacity=0.4, text opacity=1, anchor=west]{leg}; 
		
		\begin{scope}[xshift=7cm]
		\draw (2.5, 1) node{(ii)};
		\foreach \y [count=\xi] in {4,3,1,1,0,0}
		\foreach \x in {0, ..., \y} 
		\draw (\x, -\xi) -- ++(1,0) -- ++(0,1) -- ++(-1,0)--cycle;
		\end{scope}
		
		\begin{scope}[xshift=13cm]
		\draw (3, 1) node{(iii)};
		\fill[gray] (6, -1) -- ++(1,0) -- ++(0,1) -- ++(-1,0)--cycle;
		\fill[gray] (7, -1) -- ++(1,0) -- ++(0,1) -- ++(-1,0)--cycle;
		\fill[gray] (3, -2) -- ++(1,0) -- ++(0,1) -- ++(-1,0)--cycle;
		\fill[gray] (4, -2) -- ++(1,0) -- ++(0,1) -- ++(-1,0)--cycle;
		\fill[gray] (1, -3) -- ++(1,0) -- ++(0,1) -- ++(-1,0)--cycle;
		\fill[gray] (0, -5) -- ++(1,0) -- ++(0,1) -- ++(-1,0)--cycle;
		
		\foreach \y [count=\xi] in {7,4,1,0,0}
		\foreach \x in {0, ..., \y} 
		\draw (\x, -\xi) -- ++(1,0) -- ++(0,1) -- ++(-1,0)--cycle;
		\end{scope}

		\begin{scope}[xshift=10cm, yshift=-17cm]
		\draw (-7,7) node{$ \mu_{n-1}$};
		\draw (-3,7) node{$  \mu_{n-2}$};
		\draw (0,7)  node{$\ldots$};
		\draw (3,7) node{$  \mu_{2}$};
		\draw (7,7) node{$  \mu_{1}$};
		\draw (-8,8) node[rotate=-45]{$ \leqslant$};
		\draw (-4,8) node[rotate=-45]{$ \leqslant$};
		\draw (-6,8) node[rotate=45]{$ \leqslant$};
		\draw (-2,8) node[rotate=45]{$ \leqslant$};
		\draw (2,8) node[rotate=-45]{$ \leqslant$};
		\draw (6,8) node[rotate=-45]{$ \leqslant$};
		\draw (4,8) node[rotate=45]{$ \leqslant$};
		\draw (8,8) node[rotate=45]{$ \leqslant$};
		\draw (-9,9) node{$ \lambda_{n}$};
		\draw (-5,9) node{$  \lambda_{n-1}$};
		\draw (0,9)  node{$\ldots$};
		\draw (5,9) node{$  \lambda_{2}$};
		\draw (9,9) node{$  \lambda_{1}$};
		\end{scope}
		\end{tikzpicture}
	\end{center}
	\caption{ (i) Young diagram corresponding to the partition $\lambda = (6,4,2,2,1)$;  
		the black box has arm length $ a(\blacksquare)=4 $ and leg length $\ell(\blacksquare)=3$. (ii) Young diagram corresponding to  $\lambda$'s transpose $\lambda' = (5,4,2,2,1,1)$. (iii) The diagram contains a horizontal strip in grey added to the diagram $\mu=(6,3,1,1)$; the grey boxes are also the skew diagram $\kappa/\mu$ where $\kappa=(8,5,2,1,1)$.}
	\label{fig:partitions}
\end{figure}
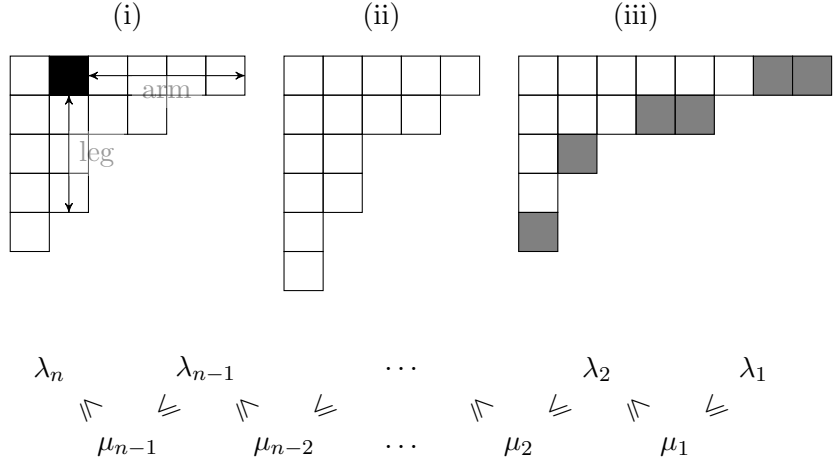

\subsubsection{Symmetric functions}
 \emph{Symmetric functions} are defined with respect to an infinite number of formal variables (we will generally use arguments like $x=(x_1,x_2,\ldots)$ or $y=(y_1,y_2,\ldots)$ although the order of variables does not matter, or simply leave off the argument of a symmetric function when it is not important). We denote the algebra of symmetric functions by $\Sym$.  It can be seen as a commutative algebra $\R[p_1, p_2, \dots]$ where $p_k(x) = x_1^k+x_2^k+\cdots$ are the Newton power sum symmetric functions,  and we refer to \cite[I]{macdonald1995symmetric} or \cite[Section 2.1.2]{borodin2014macdonald} for more details. 

The \emph{skew Macdonald $P$ ($Q$) functions} $P_{\lambda/\mu}$ ($Q_{\lambda/\mu}$) (see \cite[Chapter VI]{macdonald1995symmetric}) are symmetric functions indexed by skew partitions $\lambda/\mu$ that have coefficients in $\Q(q, t)$, which is the space of rational functions in  two auxiliary parameters $q,t$ (we will assume them to be in $[0,1)$). For $\lambda \in \Y$, define symmetric functions
\begin{equation}
\ve_{\lambda} = \sum_{\mu'\in \Y \textrm{ even}} b^{\textrm{el}}_{\mu} Q_{\lambda/\mu}
\end{equation}
where ``el'' stands for ``even leg'', and $b^{\textrm{el}}_{\mu}\in \mathbb{Q}(q,t)$ is given by
\begin{equation}
b^{\textrm{el}}_{\mu} = \prod_{\substack{\Box\in \mu\\ \mathrm{leg}(\Box)\textrm{ even}}} b_\mu(\Box), \qquad b_\mu(\Box) = \begin{cases} \dfrac{1-q^a t^{\ell+1}}{1-q^{a+1}t^{\ell}} &\Box\in \mu\\ 1&\Box\notin \mu \end{cases}
\end{equation}
with $\ell=\mathrm{leg}(\Box)$ and $a=\mathrm{arm}(\Box)$ in the definition of $b_{\mu}(\Box)$ (see Figure \ref{fig:partitions} for the definitions of $\mathrm{leg}(\Box)$ and $\mathrm{arm}(\Box)$). 

Macdonald symmetric functions satisfy the following combinatorial formula \cite[VI, (7.13)]{macdonald1995symmetric}. For two partitions $\lambda, \mu$ such that $\la/\mu$ is a horizontal strip,  define coefficients 
\begin{eqnarray}
\varphi_{\lambda/\mu} &= \prod_{1\leqslant i\leqslant j\leqslant \ell(\la)} \dfrac{f(q^{\la_i-\la_j}t^{j-i})f(q^{\mu_i-\mu_{j+1}}t^{j-i})}{f(q^{\la_i-\mu_j}t^{j-i})f(q^{\mu_i-\la_{j+1}}t^{j-i})},\\
\psi_{\lambda/\mu} &= \prod_{1\leqslant i\leqslant j\leqslant \ell(\mu)} \dfrac{f(q^{\mu_i-\mu_j}t^{j-i})f(q^{\la_i-\la_{j+1}}t^{j-i})}{f(q^{\la_i-\mu_j}t^{j-i})f(q^{\mu_i-\la_{j+1}}t^{j-i})},
\label{eq:defphipsi}
\end{eqnarray}
where $f(u) = (tu; q)_{\infty}/(qu; q)_{\infty}$. Then, we have that 
\begin{equation}
P_{\la/\mu}(x_1, \dots, x_n)  = \sum_{\lambda^{(1)}, \dots,\lambda^{(n-1)} }  \prod_{i=1}^n \psi_{\la^{(i)}/\la^{(i-1)}} \ x_i^{\vert \la^{(i)} \vert - \vert \la^{(i-1)}\vert},
\label{eq:combinatorialP}
\end{equation}
where the sum runs over sequences of partitions such that  
$$ \mu =\lambda^{(0)}  \subset \lambda^{(1)} \subset \dots \subset \lambda^{(n)} = \lambda,$$
where for all $1\leqslant i \leqslant n$,  $\la^{(i)}/\la^{(i-1)}$ is a horizontal strip. Similarly, 
\begin{equation} Q_{\la/\mu}(x_1, \dots, x_n)  = \sum_{\lambda^{(1)}, \dots,\lambda^{(n-1)} }  \prod_{i=1}^n \varphi_{\la^{(i)}/\la^{(i-1)}} \ x_i^{\vert \la^{(i)} \vert - \vert \la^{(i-1)}\vert}.
\label{eq:combinatorialQ}
\end{equation}

\subsubsection{Identities}
\label{sec:identities}
We recall certain identities involving symmetric functions which will be utilized in the remainder of the paper. In this section, all summations run over the set $\Y$ of all partitions, unless otherwise specified. 

The skew Cauchy identity \cite[VI.7]{macdonald1995symmetric} holds for two sets of formal variables $x$ and $y$:
\begin{equation}
\label{eq:skewCauchy}
\sum_{\kappa} P_{\kappa/\nu}(x)Q_{\kappa/\lambda}(y) =\Pi(x;y)\sum_{\tau}Q_{\nu/\tau}(y)P_{\lambda/\tau}(x),
\end{equation}
where $\Pi(x;y)$ is given by \cite[VI, (2.5)]{macdonald1995symmetric}
\begin{equation}
\label{eq:PI}
\Pi(x;y) := \sum_{\kappa} P_{\kappa}(x)Q_{\kappa}(y) = \prod_{i,j\geqslant 1} \phi(x_iy_j) \qquad\text{where}\qquad
\phi(x) = \frac{(tx;q)_{\infty}}{(x;q)_{\infty}}.
\end{equation}

Macdonald $P$ and $Q$ functions also satisfy a sort of semi-group property called branching rule whereby \cite[VI.7]{macdonald1995symmetric}
\begin{equation}
\label{eq:branchingrule}
\sum_{\mu} P_{\nu/\mu}(x) P_{\mu/\lambda}(y) = P_{\nu/\lambda}(x,y) \qquad\text{and} \qquad \sum_{\mu} Q_{\nu/\mu}(x) Q_{\mu/\lambda}(y) = Q_{\nu/\lambda}(x,y).
\end{equation}

Turning to the $\ve_{\lambda}$ function, from \cite[VI.7, Ex. 4(i)]{macdonald1995symmetric} we have
\begin{equation}\label{eq:Littlewoodidentity}
\Phi(x) := \sum_{\nu'\in \Y \textrm{ even}} b^{\textrm{el}}_{\nu}\ P_{\nu}(x) = \prod_{i<j} \phi(x_ix_j). 
\end{equation}
It follows from the definition of $\ve_{\mu}$ along with \eqref{eq:branchingrule} that
\begin{equation}
\label{eq:branchingve}
\sum_{\mu} Q_{\lambda/\mu}(x) \ve_{\mu}(y) = \ve_{\lambda}(x,y).
\end{equation}
From there, one can show (see e.g. \cite[Proposition 2.1]{barraquand2018stochastic}) that 
\begin{equation}
\ve_{\mu}(x) = \Phi(x)^{-1} \sum_{\nu'\in\Y \textrm{ even}} b^{\textrm{el}}_{\nu} P_{\nu/\mu}(x).
\label{eq:skewLittlewood}
\end{equation}
Combining \eqref{eq:skewCauchy} with \eqref{eq:skewLittlewood} yields
\begin{equation}
\label{eq:iii}
\sum_{\mu} \ve_{\mu}(x) P_{\mu/\lambda}(y) = \Pi(x; y) \, \Phi(y)\, \ve_{\lambda}(x,y),
\end{equation}
 and in particular 
\begin{equation}
\label{eq:CauchyLittlewood}
\sum_{\mu} \ve_{\mu}(x) P_{\mu}(y) = \Pi(x; y) \, \Phi(y).
\end{equation}
In the following, we will refer to \eqref{eq:CauchyLittlewood} as generalized Littlewood identity.

\subsubsection{Specializations}
\label{sec:specializations}
A \emph{specialization} $\rho$ of $\Sym$ is an algebra homomorphism of $\Sym$ to $\C$ -- see \cite[Section 2.2.1]{borodin2014macdonald} for a more involved discussion. We denote the application of $\rho$ to $f\in \Sym$ as $f(\rho)$. The \emph{trivial} specialization $\rho=\varnothing$ takes the value 1 for the constant function $1\in \Sym$ and $0$ for all homogeneous functions $f\in\Sym$ of higher degree. The \emph{union} of two specializations $\rho_1,\rho_2$ is defined via the relation
$$
p_k(\rho_1,\rho_2) = p_{k}(\rho_1)+p_{k}(\rho_2).
$$
Since the power sums $p_k$ span $\Sym$ (algebraically), one can extend the definition of the union to any symmetric function. Notationally, we will write the union of $\rho_1,\rho_2$ by putting a comma between them. 

We say a specialization $\rho$ is \emph{Macdonald nonnegative} if for every skew diagram $\lambda/\mu$, $P_{\lambda/\mu}(\rho)\geqslant 0$. For nonnegative numbers $\alpha=\{\alpha\}_{i\geqslant 1}$, $\beta=\{\beta_i\}_{i\geqslant 1}$ and $\gamma$ such that $\sum_i \alpha_i+\beta_i <\infty$, we define the specialization $\rho=\rho(\alpha,\beta,\gamma)$ by
\begin{equation}
\Pi(u;\rho)= \sum_{n\geqslant 0} u^n Q_{(n)}(\rho) = \exp(\gamma u) \prod_{i\geqslant 1} \frac{(t\alpha_i u;q)_{\infty}}{(\alpha_i u;q)_{\infty}} (1+\beta_i u).
\label{eq:Pigeneral}
\end{equation}
It has recently been proved in \cite[Theorem 1.4]{matveev2019macdonald} that a specialization $\rho$ is  Macdonald nonnegative if and only if $\rho=\rho(\alpha, \beta, \gamma)$ for some $\alpha, \beta,  \gamma$ as above (the \textit{if} part is not hard to show, the \textit{only if} part was conjectured by Kerov in 1992).  Notice that we have used the same notation as $\Pi(x;y)$ where $x$ is specialized into a single variable $u$ and $y$ is specialized into $\rho$. The $\alpha_i$ variables are called \emph{usual} (because they correspond to the usual notion of evaluation of a polynomial into some variables), the $\beta_i$ are called \emph{dual}, and the $\gamma$ is called \emph{Plancherel}. When $\rho$ only involves usual variables $\alpha_1,\ldots, \alpha_k$ (and all other $\alpha,\beta,\gamma$ are zero), $P_{\lambda}$ is supported on partitions $\lambda$ of $\ell(\lambda)\leqslant k$ and is a polynomial in the variables $\alpha_i$; and when $\rho$ only involves dual variables $\beta_1,\ldots, \beta_k$ (and all other $\alpha,\beta,\gamma$ are zero), $P_{\lambda}$ is supported on partitions $\lambda$ of $\ell(\lambda')\leqslant k$ (in other words, all $\lambda_i\leqslant k$).

Specializations of $\Sym$ allow to turn the formal summation identities of Section \ref{sec:identities} into analytic ones. In particular, if there exists $0<R<1$ and specializations $\rho_1, \rho_2$ such that for all $k\geqslant 1$,
$$ \vert p_k(\rho_1)\vert <R^k \text{ and }  \vert p_k(\rho_1)p_k(\rho_2)\vert <R^k, $$
then the formal identity \eqref{eq:CauchyLittlewood} becomes, after specializing $x$ into $\rho_2$ and $y$ into $\rho_1$, 
\begin{equation}
 \sum_{\la\in \Y} \ve_{\la}(\rho_2) P_{\la}(\rho_1)  = \Pi(\rho_1, \rho_2)\Phi(\rho_1),
 \label{eq:CauchyLittlewoodspecialized}
\end{equation}
where the sum is absolutely convergent.   

\subsubsection{Orthogonality}
\label{sec:orthogonality} 

Macdonald symmetric functions $P_{\la}$ and $Q_{\la}$ form a basis of 
$\Sym$, and they are orthogonal with respect to the scalar product $\langle \cdot, \cdot \rangle_{q,t}$ defined by 
$$ \langle p_{\mu}, p_{\la}\rangle_{q,t} = \mathds{1}_{\la=\mu} \prod_i i^{m_i(\la)} m_i(\la)! \prod_{i=1}^{\ell(\la)} \frac{1-q^{\la_i}}{1-t^{\la_i}}.$$

When specialized into $n$ usual variables, Macdonald symmetric functions are polynomials in these variables, and they are orthogonal with respect to another scalar product, introduced in \cite[VI, (9.10)]{macdonald1995symmetric}. 
In \cite[Section 2.1.5]{borodin2014macdonald}, it is written as 
$$ \langle f,g \rangle'  = 
\int_{\mathbb{T}^n} f(z)\overline{g(z)} \mskyl_{n}^{q,t}(z)\prod_{i=1}^n\frac{\mathrm{d}z_i}{z_i},  \ \ \ \ \ \mskyl_{n}^{q,t}(z):= \frac{1}{(2\I \pi)^n n!}  \prod_{i\neq j=1}^{n} \frac{(z_i/z_j; q)_{\infty}}{(tz_i/z_j; q)_{\infty}},$$
where $\mathbb{T}^n$ is the n-fold torus $ \left( \lbrace e^{2\I \pi \theta} \rbrace_{\theta\in [0, 1)}\right)^n$. Using the identity \eqref{eq:PI}, we may write 
$$ Q_{\la}(x) = \frac{\langle \Pi(\cdot, x), P_{\la}(\cdot) \rangle'}{\langle P_{\la}, P_{\la}\rangle'}.$$
Similarly \eqref{eq:CauchyLittlewood} would suggest  
$$ \ve_{\la}(x) \overset{?}{=} \frac{\langle \Pi(\cdot, x) \Phi(\cdot), P_{\la}(\cdot) \rangle'}{\langle P_{\la}, P_{\la}\rangle'},$$
but this does not make sense because $\Phi$ has singularities on the torus, and \eqref{eq:CauchyLittlewood} is not valid with such arguments. 
Thus, we rewrite the scalar product as\footnote{This is actually the original definition in \cite[VI.9]{macdonald1995symmetric}.}
\begin{equation}
\llangle f,g \rrangle   = 
\int_{\mathbb{T}^n} f(z)g(z^{-1}) \mskyl_{n}^{q,t}(z)\prod_{i=1}^n\frac{\mathrm{d}z_i}{z_i}
\label{eq:innerproduct}
\end{equation}
where $z^{-1}=(1/z_1, \dots, 1/z_n)$. Since for $w\in \mathbb{T}$, $\overline{w}=1/w $, the scalar products $\langle \cdot, \cdot\rangle'$ and $\llangle \cdot , \cdot \rrangle$ coincide on polynomials, and Macdonald symmetric polynomials are orthogonal with respect to $\llangle \cdot , \cdot \rrangle$ as well, with the same norm. In \eqref{eq:innerproduct}, the integrand is analytic as long as $f$ and $g$ are, so that we may use Cauchy's theorem and deform the contour. It is particularly convenient for us to take the contour as $c\mathbb{T}$ where $0<c<1$. We obtain using \eqref{eq:CauchyLittlewood} that
\begin{equation}
 \ve_{\la}(x) = \frac{1}{\llangle P_{\la}, P_{\la}\rrangle} \int_{(c\mathbb{T})^n} P_{\la}(z^{-1})\Pi(z, x)\Phi(z) \mskyl_{n}^{q,t}(z)\prod_{i=1}^n\frac{\mathrm{d}z_i}{z_i}.
 \label{eq:veintegral}
\end{equation} 
\begin{remark}
A very similar scalar product $\llangle \cdot , \cdot \rrangle$ appears in  \cite{borodin2014spectral, borodin2015spectral, borodin2017family, borodin2016higher}. More precisely,  taking $s=0$ in \cite[Theorem 7.2]{borodin2017family}, we recover the orthogonality of  Hall-Littlewood polynomials with respect to $\llangle \cdot , \cdot \rrangle$. 
\end{remark}

\subsection{Definition of half-space Macdonald processes} 
\label{sec:defhalfspaceMacdonald}
We define half-space Macdonald processes in terms of a certain type of paths in the sector $\{(i,j)\in \Z^2: 0\leqslant j\leqslant i\}$ decorated with Macdonald nonnegative specializations. An analogous definition of Pfaffian Schur processes -- which are a particular case $q=t$ of the following  -- was described in \cite[Section 3.2]{baik2018pfaffian}. An alternative but equivalent definition of half-space Macdonald processes was provided in in \cite[Definition 2.3]{barraquand2018stochastic}. The paths we consider are half-infinite and oriented, starting at $(+\infty,0)$ and proceeding to some $(i,0)$ before proceeding by unit steps along upward and leftward edges until the diagonal, at which point there is a final edge connecting that point to $(0,0)$ where the path terminates. We will call such a path $\pathh$ and denote its set of vertices as $V(\pathh)$ and edges as $E(\pathh)$. Denote the set of vertical edges by $E^{\uparrow}(\pathh)$, of horizontal edges by $E^{\leftarrow}(\pathh)$ (we do not include edges along the $x$ axis) and the singleton containing the diagonal edge by $E^{\swarrow}(\pathh)$. The set of all such paths will be denoted by $\admpath$. Note that the last diagonal edge is a single edge, not a union of all of the $\sqrt{2}$ length edges between consecutive diagonal lattice points. Likewise, $V(\pathh)$ does not include these intermediate diagonal points. We introduce a natural ordering on vertices: $v<v'$ if $v$ comes before $v'$ in $\pathh$; likewise define similar precedence ordering on edges as well as between vertices and edges.

We label each edge $e\in E(\pathh)$ with a Macdonald specialization $\rho_e$. We label each vertex $v\in V(\pathh)$ by a partition $\lambda^v$ with the convention that $\lambda^{v}\equiv \varnothing$ for all $v$ with $y$-coordinate equal to 0. Figure \ref{path} provides an example of one such path (the $\varnothing$ vertices have been labeled, but all other vertices and edge labels are not present in the figure).

\begin{figure}
\begin{center}
\begin{tikzpicture}[scale=0.7]
\draw[axis] (0,0) -- (0,5.6) ;
\draw[axis] (0,0) -- (12.6, 0) ;
\begin{scope}[decoration={
    markings,
    mark=at position 0.5 with {\arrow{<}}}]
\draw[grille] (0,0) grid(12.5, 5.5);
\draw[fleche] (0,0) -- (4,4);
\draw[fleche] (4,4) -- (5,4);
\draw[fleche] (5,4) -- (5,3);
\draw[fleche] (5,3) -- (6,3);
\draw[fleche] (6,3) -- (7,3);
\draw[fleche] (7,3) -- (7,2);
\draw[fleche] (7,2) -- (7,1);
\draw[fleche] (7,1) -- (8,1);
\draw[fleche] (8,1) -- (9,1);
\draw[fleche] (9,1) -- (9,0);
\draw[fleche] (9,0) -- (10,0);
\draw[fleche] (10,0) -- (11,0);
\draw[fleche] (11,0) -- (12,0);
\fill (0,0) circle(0.1);
\fill (4,4) circle(0.1);
\fill (5,4) circle(0.1);
\fill (5,3) circle(0.1);
\fill (6,3) circle(0.1);
\fill (7,3) circle(0.1);
\fill (7,2) circle(0.1);
\fill (7,1) circle(0.1);
\fill (8,1) circle(0.1);
\fill (9,1) circle(0.1);
\fill (9,0) circle(0.1);
\fill (10,0) circle(0.1);
\fill (11,0) circle(0.1);
\fill (12,0) circle(0.1);
\draw (-0.2, -0.4) node{$\varnothing$};
\draw (9, -0.4) node{$\varnothing$};
\draw (10, -0.4) node{$\varnothing$};
\draw (11, -0.4) node{$\varnothing$};
\draw (12, -0.4) node{$\varnothing$};
\end{scope}
\end{tikzpicture}
\end{center}
\caption{A possible path $\pathh\in\admpath$. The vertices with $y$-coordinate 0 are labeled with trivial partitions, while all others are (not shown) labeled with partitions and all edges are labeled with Macdonald nonnegative specializations.}\label{path}
\end{figure}
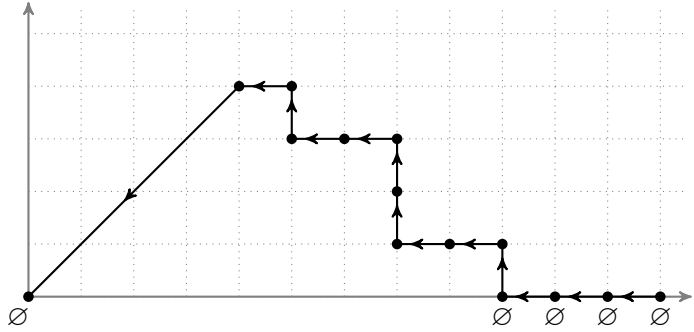

For a given path $\pathh\in \admpath$ and set of specializations $\bm\uprho = \{\rho_{e}\}_{e\in E(\pathh)}$, we associate a weight to the sequence of partitions $\labold = \{\lambda^{v}\}_{v\in V(\pathh)}$:
\begin{equation}\label{eq:Weight}
\mathcal{W}\left(\labold\right) := \prod_{e\in E(\pathh)} \mathcal{W}(e),
\end{equation}
where the weight of an edge $e$ is given as follows. Let $\kappa$ denote the partition at the start of $e$ and $\mu$ the partition at the end of $e$. Then, 
$$ \mathcal{W}(e) = \begin{cases}
\ve_{\kappa}(\rho_e) &\text{ if }e\in E^{\swarrow}(\pathh),\\
Q_{\kappa/\mu}(\rho_e) &\text{ if }e\in E^{\leftarrow}(\pathh),\\
P_{\mu/\kappa}(\rho_e) &\text{ if }e\in E^{\uparrow}(\pathh).
\end{cases} $$
We use the convention $P_{\varnothing/\varnothing}=1$, so that  if $e$ is a leftward edge with $y$-coordinate $0$, then $\mathcal{W}(e)=1$. 
\begin{proposition} Assuming each term on the right-hand side is finite,
\begin{equation}
\label{eq:Z}
\sum_{\labold}\  \mathcal{W}\left(\labold\right)\ \   = \!\!\!\prod_{\substack{e<e':\\ e\in E^{\uparrow}(\pathh),\\ e'\in E^{\leftarrow}(\pathh)\cup E^{\swarrow}(\pathh)}}\!\!\!  \Pi(\rho_e;\rho_{e'})  \Phi(\cup_{e\in E^{\uparrow}(\pathh)} \rho_e).
\end{equation}
\end{proposition}
\begin{proof}
This can be proved through applying (specializations of) identities \eqref{eq:skewCauchy}, \eqref{eq:branchingrule}, \eqref{eq:branchingve}, and \eqref{eq:iii}. We provide a pictorial proof which explains in which order these identities must be used. Figure \ref{interpretation} provides a graphical representation for the meaning of each identity. 
\begin{figure}
	\begin{center}
		\begin{tikzpicture}[scale=0.8]
		\begin{scope}[decoration={
			markings,
			mark=at position 0.5 with {\arrow{<}}}]
		
		\draw (-3,0) node{(a)};
		\fill (0,0) circle(0.1);
		\draw[fleche] (-1,0) -- (0,0) node[midway, anchor=south]{$\rho_2$};
		\draw[fleche] (0,0) -- (1,0) node[midway, anchor=south]{$\rho_1$};
		\draw[ultra thick, gray, dotted] (-1.5,0) -- (-1,0);
		\draw[ultra thick, gray, dotted] (1.5,0) -- (1,0);
		\draw[thick, gray] (-1.5,0) ellipse (.5 and .3);
		\draw[thick, gray] (1.5,0) ellipse (.5 and .3);
		\draw (-.2,-.2)  -- ++(0.4,0) -- ++(0,0.4) -- ++(-0.4,0) -- cycle; 
		\draw (3,0) node{$=$}; 
		\draw[fleche] (5,0) -- (6,0) node[midway, anchor=south]{$\rho_1,\rho_2$};
		\draw[ultra thick, gray, dotted] (4.5,0) -- (5,0);
		\draw[ultra thick, gray, dotted] (6.5,0) -- (6,0);
		\draw[thick, gray] (4.5,0) ellipse (.5 and .3);
		\draw[thick, gray] (6.5,0) ellipse (.5 and .3);
		
		\begin{scope}[yshift=-2cm]
		\draw (-3,0) node{(b)};
		\fill (0,0) circle(0.1);
		\draw[fleche] (0,0) -- (0,-1) node[midway, anchor=west]{$\rho_2$};
		\draw[fleche] (0,1) -- (0,0) node[midway, anchor=west]{$\rho_1$};
		\draw[ultra thick, gray, dotted] (0.5,-1) -- (0,-1);
		\draw[ultra thick, gray, dotted] (-0.5,1) -- (0,1);
		\draw[thick, gray] (-.5,1) ellipse (.5 and .3);
		\draw[thick, gray] (.5,-1) ellipse (.5 and .3);
		\draw (-.2,-.2)  -- ++(0.4,0) -- ++(0,0.4) -- ++(-0.4,0) -- cycle; 
		\draw (3,0) node{$=$}; 
		\draw[fleche] (5,.5) -- (5,-.5) node[midway, anchor=west]{$\rho_1,\rho_2$};
		\draw[ultra thick, gray, dotted] (4.5,.5) -- (5,.5);
		\draw[ultra thick, gray, dotted] (5.5,-.5) -- (5,-.5);
		\draw[thick, gray] (4.5,.5) ellipse (.5 and .3);
		\draw[thick, gray] (5.5,-.5) ellipse (.5 and .3);
		\end{scope}
		
		\begin{scope}[yshift=-4cm]
		\draw (-3,-.5) node{(c)};
		\fill (0,0) circle(0.1);
		\draw[fleche] (-1,0) -- (0,0) node[midway, anchor=south]{$\rho_2$};
		\draw[fleche] (0,0) -- (0,-1) node[midway, anchor=west]{$\rho_1$};
		\draw[ultra thick, gray, dotted] (-1.5,0) -- (-1,0);
		\draw[ultra thick, gray, dotted] (0,-1) -- (.5,-1);
		\draw[thick, gray] (-1.5,0) ellipse (.5 and .3);
		\draw[thick, gray] (.5,-1) ellipse (.5 and .3);
		\draw (-.2,-.2)  -- ++(0.4,0) -- ++(0,0.4) -- ++(-0.4,0) -- cycle; 
		\draw (3,-.5) node{$=$}; 
		\fill (5,-1) circle(0.1);
		\draw[fleche] (5,-1) -- (6,-1) node[midway, anchor=north]{$\rho_2$};
		\draw[fleche] (5,0) -- (5,-1) node[midway, anchor=east]{$\rho_1$};
		\draw[ultra thick, gray, dotted] (4.5,0) -- (5,0);
		\draw[ultra thick, gray, dotted] (6,-1) -- (6.5,-1);
		\draw[thick, gray] (4.5,0) ellipse (.5 and .3);
		\draw[thick, gray] (6.5,-1) ellipse (.5 and .3);
		\draw (4.8,-1.2)  -- ++(0.4,0) -- ++(0,0.4) -- ++(-0.4,0) -- cycle; 
		\draw (8.5,-.5) node{$\times\ \Pi(\rho_1, \rho_2)$};
		\end{scope}
		
		\begin{scope}[yshift=-6cm]
		\draw (-3,-.5) node{(d)};
		\fill (0,0) circle(0.1);
		\draw[fleche] (-1,-1) -- (0,0) node[midway, anchor=south east]{$\rho_2$};
		\draw[fleche] (0,0) -- (1,0) node[midway, anchor=south]{$\rho_1$};
		%\draw[ultra thick, gray, dotted] (-1.5,0) -- (-1,0);
		\draw[ultra thick, gray, dotted] (1.5,0) -- (1,0);
		%\draw[thick, gray] (-1.5,0) ellipse (.5 and .3);
		\draw[thick, gray] (1.5,0) ellipse (.5 and .3);
		\draw (-.2,-.2)  -- ++(0.4,0) -- ++(0,0.4) -- ++(-0.4,0) -- cycle; 
		\draw (3,-.5) node{$=$}; 
		\draw[fleche] (5, -1) -- (6,0) node[midway, anchor=south east]{$\rho_1,\rho_2$};
		%\draw[ultra thick, gray, dotted] (4.5,0) -- (5,0);
		\draw[ultra thick, gray, dotted] (6.5,0) -- (6,0);
		%\draw[thick, gray] (4.5,0) ellipse (.5 and .3);
		\draw[thick, gray] (6.5,0) ellipse (.5 and .3);
		\end{scope}
		
		\begin{scope}[yshift=-8cm]
		\draw (-3,-.5) node{(e)};
		\fill (0,0) circle(0.1);
		\draw[fleche] (-1,-1) -- (0,0) node[midway, anchor=south east]{$\rho_2$};
		\draw[fleche] (0,0) -- (0,-1) node[midway, anchor=west]{$\rho_1$};
		%\draw[ultra thick, gray, dotted] (-1.5,0) -- (-1,0);
		\draw[ultra thick, gray, dotted] (.5,-1) -- (0,-1);
		%\draw[thick, gray] (-1.5,0) ellipse (.5 and .3);
		\draw[thick, gray] (.5,-1) ellipse (.5 and .3);
		\draw (-.2,-.2)  -- ++(0.4,0) -- ++(0,0.4) -- ++(-0.4,0) -- cycle; 
		\draw (3,-.5) node{$=$}; 
		\draw[fleche] (5, -1) -- (6,0) node[midway, anchor=south east]{$\rho_1,\rho_2$};
		%\draw[ultra thick, gray, dotted] (4.5,0) -- (5,0);
		\draw[ultra thick, gray, dotted] (6.5,0) -- (6,0);
		%\draw[thick, gray] (4.5,0) ellipse (.5 and .3);
		\draw[thick, gray] (6.5,0) ellipse (.5 and .3);
		\draw (9,-.5) node{$\times\ \Pi(\rho_1, \rho_2) \Phi(\rho_1)$};
		\end{scope}
		\end{scope}
		\end{tikzpicture}
		
	\end{center}
	\caption{Graphical representations of summation identities. The boxes represent vertices whose partitions are being summed over; the directed edges are labeled by Macdonald nonnegative specializations; the blobs represent other terms which may arise in the weight of a path $\pathh$ which are not involved in these identities. Graphics (a) represents the branching rule \eqref{eq:branchingrule} for Macdonald polynomials $Q$, (b) represents the branching rule for polynomials $P$, (c) represents the skew Cauchy identity \eqref{eq:skewCauchy}, (d) represents \eqref{eq:branchingve}, and (e) represents \eqref{eq:iii}.}
	\label{interpretation}
\end{figure}
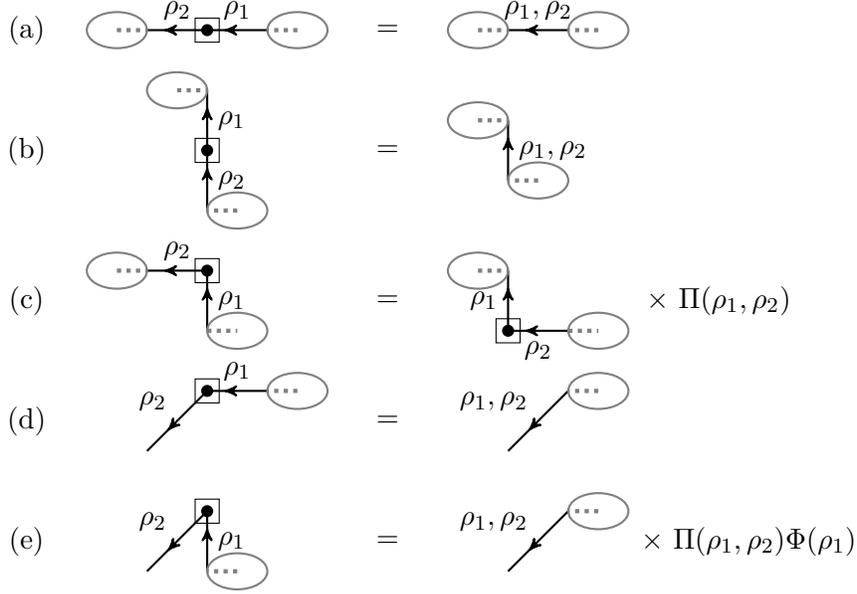

Starting from any path $\pathh$, one may apply these elementary moves until $\pathh$ is reduced to the trivial path (with $x$-coordinate always equal to zero) which assigns weight 1 to trivial partitions and 0 otherwise. Figure \ref{deconstruct} provides a step-by-step illustration of this reduction process. Keeping track of the products of $\Pi$ and $\Phi$ terms yields the desired formula.
\begin{figure}
\begin{center}
\begin{tikzpicture}[scale=0.7]
\draw[axis] (0,0) -- (0,3.5) ;
\draw[axis] (0,0) -- (5.5, 0) ;
\begin{scope}[decoration={
    markings,
    mark=at position 0.5 with {\arrow{<}}}]
\draw[grille] (0,0) grid(5.5, 3.5);
\draw[fleche] (0,0) -- (2,2);
\draw[fleche] (2,2) -- (3,2);
\draw[fleche] (3,2) -- (3,1);
\draw[fleche] (3,1) -- (4,1);
\draw[fleche] (4,1) -- (4,0);
\fill (2,2) circle(0.1);
\fill (3,2) circle(0.1);
\fill (3,1) circle(0.1);
\fill (4,1) circle(0.1);
\fill (4,0) circle(0.1);
\draw (1.8,1.8)  -- ++(0.4,0) -- ++(0,0.4) -- ++(-0.4,0) -- cycle;
\draw (2.8,1.8)  -- ++(0.4,0) -- ++(0,0.4) -- ++(-0.4,0) -- cycle;
\draw (2.8,0.8)  -- ++(0.4,0) -- ++(0,0.4) -- ++(-0.4,0) -- cycle;
\draw (3.8,0.8)  -- ++(0.4,0) -- ++(0,0.4) -- ++(-0.4,0) -- cycle;
\end{scope}

\begin{scope}[yshift=0cm, xshift=7cm]
\draw[axis] (0,0) -- (0,3.5) ;
\draw[axis] (0,0) -- (4.5, 0) ;
\begin{scope}[decoration={
    markings,
    mark=at position 0.5 with {\arrow{<}}}]
\draw[grille] (0,0) grid(4.5, 3.5);
\draw[fleche] (0,0) -- (2,2);
\draw[fleche] (2,2) -- (3,2);
\draw[fleche] (3,2) -- (3,1);
\draw[fleche] (3,1) -- (3,0);
\fill (2,2) circle(0.1);
\fill (3,2) circle(0.1);
\fill (3,1) circle(0.1);
\fill (3,0) circle(0.1);
\draw (1.8,1.8)  -- ++(0.4,0) -- ++(0,0.4) -- ++(-0.4,0) -- cycle;
\draw (2.8,1.8)  -- ++(0.4,0) -- ++(0,0.4) -- ++(-0.4,0) -- cycle;
\draw (2.8,0.8)  -- ++(0.4,0) -- ++(0,0.4) -- ++(-0.4,0) -- cycle;
\end{scope}
\end{scope}

\begin{scope}[yshift=0cm, xshift=13cm]
\draw[axis] (0,0) -- (0,3.5) ;
\draw[axis] (0,0) -- (4.5, 0) ;
\begin{scope}[decoration={
    markings,
    mark=at position 0.5 with {\arrow{<}}}]
\draw[grille] (0,0) grid(4.5, 3.5);
\draw[fleche] (0,0) -- (2,2);
\draw[fleche] (2,2) -- (2,1);
\draw[fleche] (2,1) -- (3,1);
\draw[fleche] (3,1) -- (3,0);
\fill (2,2) circle(0.1);
\fill (2,1) circle(0.1);
\fill (3,1) circle(0.1);
\fill (3,0) circle(0.1);
\draw (1.8,1.8)  -- ++(0.4,0) -- ++(0,0.4) -- ++(-0.4,0) -- cycle;
\draw (1.8,0.8)  -- ++(0.4,0) -- ++(0,0.4) -- ++(-0.4,0) -- cycle;
\draw (2.8,0.8)  -- ++(0.4,0) -- ++(0,0.4) -- ++(-0.4,0) -- cycle;
\end{scope}
\end{scope}

\begin{scope}[yshift=-5cm, xshift=0cm]
\draw[axis] (0,0) -- (0,3.5) ;
\draw[axis] (0,0) -- (4.5, 0) ;
\begin{scope}[decoration={
    markings,
    mark=at position 0.5 with {\arrow{<}}}]
\draw[grille] (0,0) grid(4.5, 3.5);
\draw[fleche] (0,0) -- (2,2);
\draw[fleche] (2,2) -- (2,1);
\draw[fleche] (2,1) -- (2,0);
\fill (2,2) circle(0.1);
\fill (2,1) circle(0.1);
\fill (2,0) circle(0.1);
\draw (1.8,1.8)  -- ++(0.4,0) -- ++(0,0.4) -- ++(-0.4,0) -- cycle;
\draw (1.8,0.8)  -- ++(0.4,0) -- ++(0,0.4) -- ++(-0.4,0) -- cycle;
\end{scope}
\end{scope}

\begin{scope}[yshift=-5cm, xshift=5.5cm]
\draw[axis] (0,0) -- (0,3.5) ;
\draw[axis] (0,0) -- (3.5, 0) ;
\begin{scope}[decoration={
    markings,
    mark=at position 0.5 with {\arrow{<}}}]
\draw[grille] (0,0) grid(3.5, 3.5);
\draw[fleche] (0,0) -- (1,1);
\draw[fleche] (1,1) -- (2,1);
\draw[fleche] (2,1) -- (2,0);
\fill (1,1) circle(0.1);
\fill (2,1) circle(0.1);
\fill (2,0) circle(0.1);
\draw (0.8,0.8)  -- ++(0.4,0) -- ++(0,0.4) -- ++(-0.4,0) -- cycle;
\draw (1.8,0.8)  -- ++(0.4,0) -- ++(0,0.4) -- ++(-0.4,0) -- cycle;
\end{scope}
\end{scope}

\begin{scope}[yshift=-5cm, xshift=10cm]
\draw[axis] (0,0) -- (0,3.5) ;
\draw[axis] (0,0) -- (3.5, 0) ;
\begin{scope}[decoration={
    markings,
    mark=at position 0.5 with {\arrow{<}}}]
\draw[grille] (0,0) grid(3.5, 3.5);
\draw[fleche] (0,0) -- (1,1);
\draw[fleche] (1,1) -- (1,0);
\fill (1,1) circle(0.1);
\fill (1,0) circle(0.1);
\draw (0.8,0.8)  -- ++(0.4,0) -- ++(0,0.4) -- ++(-0.4,0) -- cycle;
\end{scope}
\end{scope}

\begin{scope}[yshift=-5cm, xshift=14.5cm]
\draw[axis] (0,0) -- (0,3.5) ;
\draw[axis] (0,0) -- (3.5, 0) ;
\begin{scope}[decoration={
    markings,
    mark=at position 0.5 with {\arrow{<}}}]
\draw[grille] (0,0) grid(3.5, 3.5);
\end{scope}
\end{scope}

\end{tikzpicture}
\end{center}
\caption{To compute the normalizing constant for a path $\pathh$, one sums over all non-trivial partitions labeled by vertices. The boxes represent these summations and the figure shows the sequential application of the identities in Figure \ref{interpretation}. The multiplicative factors of $\Pi$ and $\Phi$ which arise from each summation are not shown, nor are the specializations.}
\label{deconstruct}
\end{figure}
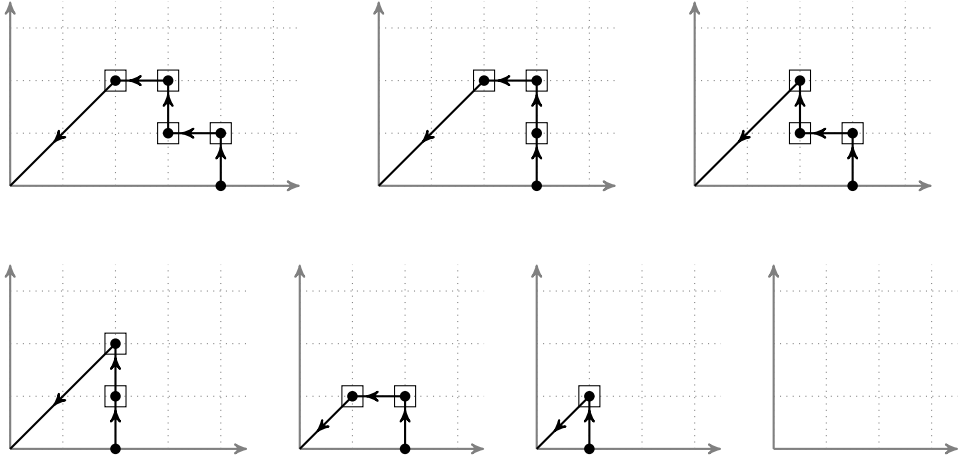
\end{proof}

\begin{definition}
For a given path $\pathh\in \admpath$ and a set of Macdonald nonnegative specializations $\bm\uprho = \{\rho_{e}\}_{e\in E(\pathh)}$ such that \eqref{eq:Z} is finite, the \emph{half-space Macdonald process} $\PMP_{\pathh;\rho}$ is a measure on the sequence of partitions $\labold = \{\lambda^{v}\}_{v\in V(\pathh)}$ given by
$$
\PMP_{\pathh;\bm\uprho}\left(\labold\right) \,:=\, \frac{\mathcal{W}\left(\labold\right)}{\!\!\!\!\prod\limits_{\substack{e<e':\\ e\in E^{\uparrow}(\pathh),\\ e' \in E^{\leftarrow}(\pathh)\cup E^{\swarrow}(\pathh)}}\!\!\!\!\!\!\!\!\!\!\!\!\ \Pi(\rho_e;\rho_{e'})  \Phi(\cup_{e\in E^{\uparrow}(\pathh)} \rho_e)}.
$$

The \emph{half-space Macdonald measure} is a measure on a single partition. It is a special case  of the process  when the path  $\pathh= \tikz[baseline, scale=0.3]{\draw[thick] (0,0) -- (1,1) -- (1,0);}$\  , i.e. the path with only one vertex above the $x$-axis at position $(1,1)$. In that case, the only non-trivial partition is $\lambda^{(1,1)}$ (which we will just write as $\lambda$) and the only specializations that matter are those for the upward edge into $(1,1)$ which we denote $\rhoup$, and that of the diagonal edge out of $(1,1)$ which we denote $\rhodiag$. The half-space Macdonald measure is then simply written as $\PMM_{\rhoup,\rhodiag}$ and is explicitly given by the formula
$$
\PMM_{\rhoup,\rhodiag}(\lambda) = \frac{P_{\lambda}(\rhoup) \ve_{\lambda}(\rhodiag)}{\Pi(\rhoup;\rhodiag)\Phi(\rhoup)}.
$$
It is convenient to also introduce an expectation operator, which for the half-space Macdonald process is denoted  by $\EPMP_{\pathh;\bm\uprho}$ and for the measure is denoted by $\EPMM_{\rhoup,\rhodiag}$.
\label{def:halfspaceMacdonaldprocess}
\end{definition}

It is easy to show using the relations in Figure \ref{interpretation} that various marginals of half-space Macdonald processes to subsequences of $\lambda$ remain half-space Macdonald processes. In particular, for a particular vertex $v\in \pathh$, the marginal distribution of $\lambda^v$ under the half-space Macdonald process $\PMP_{\pathh;\rho}$ is given by the half-space Macdonald measure $\PMM_{\rhoup,\rhodiag}(\lambda)$ where
$$
\rhoup = \!\!\!\!\!\bigcup_{e<v: e\in E^{\uparrow}(\pathh)} \rho_e \qquad \text{and} \qquad \rhodiag = \!\!\!\!\!\bigcup_{e>v: e\in E(\pathh)}\rho_e.
$$
\begin{remark}
(Full space) Macdonald processes \cite{borodin2014macdonald} can be defined in a similar manner. The difference is that the paths $\pathh$ which index Macdonald processes start at $(+\infty,0)$ and end at $(0,+\infty)$. The weight of a collection of partitions on vertices of the path is still given by the product of specialized skew Macdonald $P$ or $Q$ functions along the edges.
\end{remark}
\begin{remark}
Macdonald (and in particular Schur) processes are commonly defined as measures on
sequences $\labold = (\lambda^{(1)},\ldots, \lambda^{(N)})$ and $\upmu= (\mu^{(1)},\ldots, \mu^{(N-1)})$ of partitions satisfying the interlacing condition
\begin{equation*}
\varnothing\subset \la^{(1)}\supset \mu^{(1)}\subset
\la^{(2)}\supset \mu^{(2)} \subset \dots \supset \mu^{(N-1)}\subset
\la^{(N)}\supset \varnothing.
\end{equation*}
In the context of half-space Macdonald processes, we could similarly define our measure on such a set of partitions by fixing Macdonald nonnegative specializations $\rho^+_0,\ldots \rho^+_{N-1},\rho^-_1,\ldots, \rho^-_N$ and defining a weight
\begin{equation*}
\mathcal{W}(\labold,\boldsymbol{\upmu}):=P_{\la^{(1)}}(\rho_0^+)\,Q_{\la^{(1)}/\mu^{(1)}}(\rho_1^-)
P_{\la^{(2)}/\mu^{(1)}}(\rho_1^+)\,\cdots
P_{\la^{(N)}/\mu^{(N-1)}}(\rho_{N-1}^+)\, \ve_{\la^{(N)}}(\rho_N^-).
\end{equation*}
This was the definition employed in \cite[Definition 3.2]{barraquand2018stochastic}  and \cite{borodin2005eynard} in the Schur case. It is easy to match this measure to a half-space Macdonald process indexed by a  particular choice of path $\pathh$ (which maximally zig-zags from the $x$-axis to the diagonal). More general choices of $\pathh$ come from choosing trivial specializations (which force equality of consecutive partitions). Thus Definition \ref{def:halfspaceMacdonaldprocess} is equivalent to \cite[Definition 3.2]{barraquand2018stochastic} -- see also \cite[Remark 3.6]{baik2018pfaffian} about the equivalence between both formulations.
\end{remark}

The next proposition is a useful identity in law valid only when the diagonal specialization is the evaluation into a single variable (see  Section \ref{sec:BaikRainsexplained} for an application). In the Schur degeneration, one recovers Corollary 7.6 in \cite{baik2001algebraic} (see also Proposition 3.4 in \cite{baik2018pfaffian}).
\begin{proposition}
Let $\mu$ be distributed according to the Macdonald measure $\PMM_{\rho, \alpha}$ where $\alpha$ is a single variable specialization,  and let $\lambda$ be distributed according to the Macdonald measure $\PMM_{\rho', 0}$ with $\rho'=(\rho, \alpha)$. Then $(\lambda_1, \lambda_3, \dots)$ and  $(\mu_1, \mu_3, \dots)$ have the same distribution. 
\label{prop:identityinlaw}
\end{proposition}
\begin{proof}
Since $\ve_{\lambda}(0)$ is supported on partitions with even dual, 
$$ \PMM_{\rho', 0}\big( \lambda_1\leqslant \ell_1, \lambda_3\leqslant \ell_3, \dots\big) =  \sum_{\lambda' \textrm{even}} b^{\textrm{el}}_{\lambda} P_{\lambda}(\rho'),$$
where the sum runs over partitions $\lambda$ such that $ \lambda_i\leqslant \ell_i$ for all odd $i$. Using \eqref{eq:branchingrule}, the sum can be rewritten as 
\begin{equation}
\label{eq:sumexchange}
\sum_{\lambda' \textrm{even}}\sum_{\mu\subset\lambda} b^{\textrm{el}}_{\lambda} P_{\lambda/\mu}(\alpha)P_{\mu}(\rho) = \sum_{\mu\prec\lambda} \sum_{\lambda' \textrm{even}} b^{\textrm{el}}_{\lambda} \psi_{\lambda/\mu}\alpha^{\vert \lambda -\mu\vert } P_{\mu}(\rho).
\end{equation}
where in the R.H.S.,  the first sum runs over partitions $\mu$ such that  $ \mu_i\leqslant \ell_i$ for all odd $i$. Now we will use the fact that if $\mu\prec\la $ and $\la'$ is even, then we have (see \cite[VI.7 Ex. 4, Eq. (4)]{macdonald1995symmetric})
$$b^{\textrm{el}}_{\lambda} \psi_{\lambda/\mu} =  b^{\textrm{el}}_{\nu} \varphi_{\mu/\nu},$$
where $\nu$ is the only partition such that $\nu\prec\mu$ and $\nu'$ is even, and the coefficients $ \psi_{\lambda/\mu},  \varphi_{\mu/\nu}$ are defined in \eqref{eq:defphipsi}. 
Since we also have $ \vert \la-\mu\vert = \vert\mu-\nu\vert$, \eqref{eq:sumexchange} equals 
$$ \sum_{\mu}  b^{\textrm{el}}_{\nu} \varphi_{\mu/\nu}\alpha^{\vert \mu -\lambda\vert } P_{\mu}(\rho) .$$
We recognize $ \ve_{\mu}(\alpha) = b^{\textrm{el}}_{\nu} \varphi_{\mu/\nu}\alpha^{\vert \mu -\lambda\vert } $ and conclude that 
$$ \PMM_{\rho', 0}\big( \lambda_1\leqslant \ell_1, \lambda_3\leqslant \ell_3, \dots\big) = \PMM_{\rho, \alpha}\big( \mu_1\leqslant \ell_1, \mu_3\leqslant \ell_3, \dots\big). $$
\end{proof}

\subsection{Markov dynamics on half-space Macdonald processes}
\label{sec:Macdyn} 
\subsubsection{Bulk and boundary transition operators} We consider here Markov transition operators which map half-space Macdonald processes with one set of parameters to those processes with an updated set of parameters.
We will leverage the graphical representation of the half-space Macdonald process so as to describe a general mechanism through which to `grow' such a measure. We refer to \cite[Section 3.3]{baik2018pfaffian} where such a procedure is explained in the case of the Pfaffian Schur processes. In terms of the path $\pathh$, there are two elementary moves (see Figure \ref{elemoves}) which can be used to transition between any $\pathh$ and $\pathh'$ where $\pathh'$ contains $\pathh$ (in the sense that $\pathh$ sits entirely to the bottom left of $\pathh'$). The first move -- bulk growth -- takes a piece $\mathlarger{\llcorner}$ in $\pathh$ with corner coordinates $(i,j)$ with $i>j$ and inverts it into piece $_\mathlarger{\urcorner}$ with corner coordinates $(i+1,j+1)$. The second move -- boundary growth -- takes the piece composed of the diagonal from $(i,i)$ to $(0,0)$ and the leftward edge immediately preceding it and replaces it with the diagonal from $(i+1,i+1)$ to $(0,0)$ and an upward edge immediately preceding it. To each type of moves, we associate an operator. For the bulk growth, out of corner $(i,j)$ with $i>j$, we associate the operator $\U_{i,j}$ and for the boundary growth, out of corner $(i,i)$ we associate the operator $\Udiag_{i,i}$. 

The specializations will be carried out from the initial path to the new one as follows. The operator $\U_{i,j}$ will encode a Markov transition from $\lambda^{(i,j)}$ to $\lambda^{(i+1,j+1)}$ and will be chosen so as to map the half-space Macdonald process to a new half-space Macdonald process on the new path $\pathh'$ where the specializations all remain the same, except the leftward one into $(i,j)$ becomes the leftward one out of $(i+1,j+1)$ and likewise the upward one out of $(i,j)$ becomes the upward one into $(i+1,j+1)$. The operator $\Udiag_{i,i}$ will encode a Markov transition from $\lambda^{(i,i)}$ to $\lambda^{(i+1,i+1)}$ and should map the half-space Macdonald process to a new half-space Macdonald process on the new path $\pathh'$ where the specializations all remain the same (including the diagonal edge from $(i+1, i+1)$ to $(0,0)$), except the leftward one into $(i,i)$ becomes the upward one into $(i+1,i+1)$.

Our use of the operator $\U$ implies that we use the same specializations on  all vertical edges at the same ordinate, and on all horizontal edges at the same abscisse. Let us denote by $\rho_i^{\mathrm{h}}$ the specialization carried by any horizontal edge between a point $(i-1, j)$ to a point $ (i,j)$; and $\rho_j^{\mathrm{v}}$ the specialization carried by any vertical edge between a point $(i, j-1)$ to a point $ (i,j)$. Moreover, our use of the operator $\Udiag$ implies that for every $i\in \Z_{\geqslant 0}$, $ \rho_i^{\mathrm{h}} = \rho_i^{\mathrm{v}}$. We will henceforth use the notation 
$ \rho_i := \rho_i^{\mathrm{h}} = \rho_i^{\mathrm{v}}$. We will also denote by $\rho_{\circ}$ the specialization on  the diagonal edge.  

Note that despite this seemingly restrictive choice of specializations, any half-space Macdonald process -- defined by any admissible path $ \pathh\in \admpath$ and arbitrary specializations on $ E(\pathh)$ -- can be realized as the output of the above transition operators along a sequence of elementary moves from the empty path to the path $\pathh$.

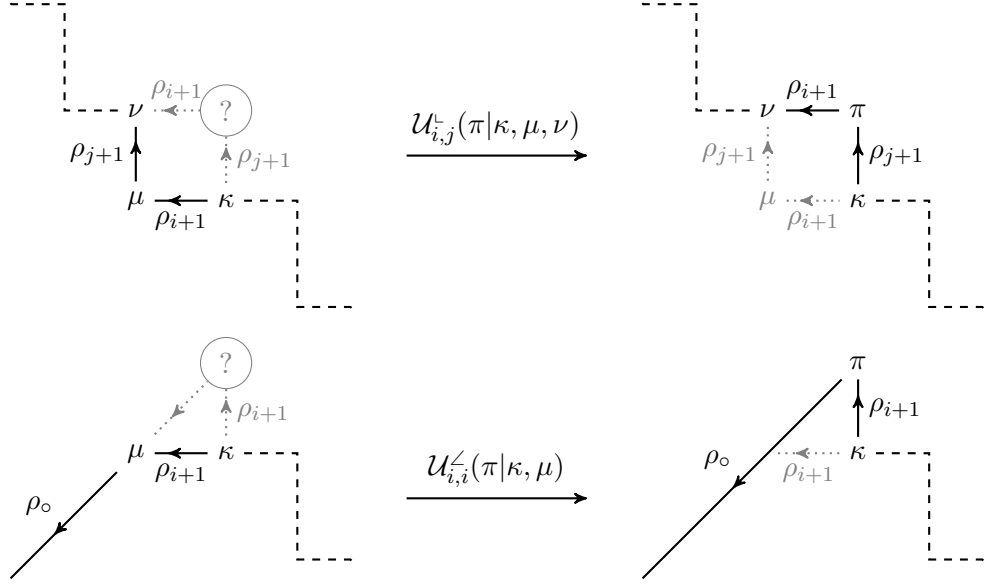
\begin{figure}
\begin{center}
\vspace{1cm}
\begin{tikzpicture}[scale=1.2]
\usetikzlibrary{shapes}
\usetikzlibrary{snakes}
\begin{scope}[decoration={
	markings,
	mark=at position 0.5 with {\arrow{<}}}]
\node (v1) at (0,1) {$\nu$};
\node (m) at (0,0) {$\mu$};
\node (k1) at (1,0) {$\kappa$};
\node[gray, circle, draw] (p1) at (1,1) {$?$};
\draw[fleche] (v1) -- (m) node[midway, anchor=east]{$\rho_{j+1}$};
\draw[fleche] (m) -- (k1) node[midway, anchor=north]{$\rho_{i+1}$};
\draw[fleche, dotted, gray] (v1) -- (p1) node[midway, anchor=south]{$\rho_{i+1}$};
\draw[fleche, dotted, gray] (p1) -- (k1) node[midway, anchor=west]{$\rho_{j+1}$};

\draw[thick, ->, >=stealth'] (3,0.5) --(5,0.5) node[midway, anchor=south] {$\mathcal{U}^{\llcorner}_{i,j}(\pi \vert \kappa, \mu, \nu)$};

\draw[thick, dashed, snake=zigzag, segment amplitude=0.5cm, segment length=2cm] (1.2,0) -- ++(1.2, -1.2);
\draw[thick, dashed,snake=zigzag,  segment amplitude=0.5cm, segment length=2cm] (-0.2,1) -- ++(-1.2, 1.2);
\draw[thick,dashed, snake=zigzag,  segment amplitude=0.5cm, segment length=2cm] (8.2,0) -- ++(1.2, -1.2);
\draw[thick, dashed,snake=zigzag,  segment amplitude=0.5cm, segment length=2cm] (6.8,1) -- ++(-1.2, 1.2);

\node (v2) at (7,1) {$\nu$};
\node (p) at (8,1) {$\pi$};
\node (k2) at (8,0) {$\kappa$};
\node[gray] (m2) at (7,0) {$\mu$};
\draw[fleche] (v2) -- (p) node[midway, anchor=south]{$\rho_{i+1}$};
\draw[fleche] (p) -- (k2) node[midway, anchor=west]{$\rho_{j+1}$};
\draw[fleche, dotted, gray] (v2) -- (m2) node[midway, anchor=east]{$\rho_{j+1}$};
\draw[fleche, dotted, gray] (m2) -- (k2) node[midway, anchor=north]{$\rho_{i+1}$};

\begin{scope}[yshift=-2.8cm]
\node (m) at (0,0) {$\mu$};
\node (bas) at (-1.5,-1.5){};
\node (k1) at (1,0) {$\kappa$};
\node[gray, circle, draw] (p) at (1,1) {$?$};
\draw[fleche] (bas) -- (m) node[midway, anchor=south east]{$\rho_{\circ}$};
\draw[fleche] (m) -- (k1) node[midway, anchor=north]{$\rho_{i+1}$};
\draw[fleche, dotted, gray] (m) -- (p) node[midway, anchor=south east]{};
\draw[fleche, dotted, gray] (p) -- (k1) node[midway, anchor=west]{$\rho_{i+1}$};

\draw[thick, ->, >=stealth'] (3,-0.5) --(5,-0.5) node[midway, anchor=south] {$\mathcal{U}^{\angle}_{i,i}(\pi \vert \kappa, \mu)$};

\node (p) at (8,1) {$\pi$};
\node (bas) at (5.5,-1.5){};
\node (k1) at (8,0) {$\kappa$};
\node[gray] (m) at (7,0) {};
\draw[fleche] (bas) -- (p) node[midway, anchor=south east]{$\rho_{\circ}$};
\draw[fleche] (p) -- (k1) node[midway, anchor=west]{$\rho_{i+1}$};
\draw[fleche, dotted, gray] (m) -- (k1) node[midway, anchor=north]{$\rho_{i+1}$};

\draw[thick, dashed,snake=zigzag, segment amplitude=0.5cm, segment length=2cm] (1.2,-0) -- ++(1.2, -1.2);
\draw[thick, dashed,snake=zigzag, segment amplitude=0.5cm, segment length=2cm] (8.2,0) -- ++(1.2, -1.2);
\end{scope}
\end{scope} 
\end{tikzpicture}
\end{center}
\caption{Operators $\U_{(i,j)}$ and $\Udiag_{(i,i)}$ map a half-space Macdonald process to a new half-space Macdonald process indexed by a new path with specializations chosen as above. The dashed part of the path does not play any role. }\label{elemoves}
\end{figure}

There may be many ways to choose the operators $\U$ and $\Udiag$. We will restrict our attention to those $\U_{(i,j)}$ which only depend on the partitions at the three vertices $(i,j+1),(i,j),(i+1,j)$ as well as the two specializations between these vertices; and those $\Udiag_i$ which only depend on the two vertices at $(i,i)$ and $(i+1,i)$ as well as the specialization between these vertices and $(0,0$).
\begin{lemma} 
Let $i,j$ be nonnegative  integers. The transition operators $\U_{i,j}$ and $\Udiag_{i,i}$ map a half-space Macdonald process to another half-space Macdonald process as specified above if and only if they satisfy the following two equations: For all partitions $\kappa,\mu,\nu,\pi$ and all specializations $\rho_{\circ}, \rho_{i+1}, \rho_{j+1}$ for which normalizations remain finite,
\begin{align}\label{eq:definingeq}
\sum_{\mu} Q_{\kappa/\mu}(\rho_{i+1}) P_{\nu/\mu}(\rho_{j+1}) \U(\pi|\kappa,\mu,\nu)
&=
 \frac{ P_{\pi/\kappa}(\rho_{j+1})Q_{\pi/\nu}(\rho_{i+1})}{\Pi(\rho_{j+1},\rho_{i+1})}, \\ 
\sum_{\mu} \frac{Q_{\kappa/\mu}(\rho_{i+1})  \ve_{\mu}(\rho_{\circ})}{\Pi(\bigcup_{e\in E^{\uparrow}(\pathh)} \rho_e; \rho_{i+1})} \Udiag(\pi|\kappa,\mu)
&= 
\frac{P_{\pi/\kappa}(\rho_{i+1})\ve_{\pi}(\rho_{\circ})}{\Phi(\bigcup_{e\in E^{\uparrow}} \rho_e)^{-1}\Phi(\bigcup_{e\in E^{\uparrow}} \rho_e, \rho_{i+1})}. \label{eq:definingeq2}
\end{align}
\end{lemma}
\begin{proof} 
Let us explain how these equations are derived. It will be clear from the proof why they are sufficient to define bona fide bulk and boundary operators preserving the half-space Macdonald process. 
We focus only on  $\Udiag$. The analogous result for $\U$ can be proved similarly and was essentially already obtained in the literature (see \cite[Eq. (2.24)]{matveev2015q} and references therein, in the context of dynamics preserving the full-space Macdonald process).

 Let $\pathh\in \admpath$ such that $(i,i)\in\pathh$ and the edge $(i+1,i) \rightarrow (i,i)\in E(\pathh)$ and let $\pathh'$ be the path obtained from $\pathh$ by changing the leftward edge into $(i,i)$ to an upward edge into $(i+1,i+1)$. The fact that 
$\Udiag_{i,i}$
maps a half-space Macdonald process on $\pathh$ to a half-space Macdonald process on $\pathh'$ with specializations chosen according  Figure \ref{elemoves}, is equivalent to 
\begin{equation}
 \sum_{\labold} \PMP_{\pathh, \bm\uprho}\left(\labold\right)\ \ \Udiag_{i,i}(\la^{(i+1,i+1)}|\la^{(i+1,i)},\la^{(i,i)}) = \PMP_{\pathh', \bm\uprho'}(\labold'),
 \label{eq:definingeqtrue}
\end{equation}
where $\labold$ and  $\labold'$ (resp. $\bm\uprho$ and $\bm\uprho'$) are the sequences of partitions  (resp. specializations) along $\pathh$ and $\pathh'$.
Removing on both sides of \eqref{eq:definingeqtrue} the weights corresponding to the edges shared by $\pathh$ and $\pathh'$, one is left with
\begin{multline*}
\sum_{\la^{(i,i)}} \frac{Q_{\la^{(i+1,i)}/\la^{(i,i)}}(\rho_{i+1})  \ve_{\la^{(i,i)}}(\rho_{\circ})}{\Pi(\bigcup_{e\in E^{\uparrow}(\pathh)} \rho_e; \rho_{i+1})}\ \  \Udiag_{i,i}(\la^{(i+1,i+1)}|\la^{(i+1,i)},\la^{(i,i)})
= 
\frac{P_{\la^{(i+1,i+1)}/\la^{(i+1,i)}}(\rho_{i+1})\ve_{\la^{(i+1,i+1)}}(\rho_{\circ})}{\Phi(\bigcup_{e\in E^{\uparrow}} \rho_e)^{-1}\Phi(\bigcup_{e\in E^{\uparrow}} \rho_e, \rho_{i+1})},
\end{multline*}
which must  hold for any partitions $\la^{(i+1,i)}$ and $\la^{(i+1,i+1)}$.
\end{proof}

\subsubsection{Push-block dynamics}
\label{sec:pushblock}
It is non-trivial to solve equations \eqref{eq:definingeq} and \eqref{eq:definingeq2} in general. It simplifies things considerably if we assume that $\U(\pi|\kappa,\mu,\nu) = \U(\pi|\kappa,\nu)$
 and $\Udiag(\pi|\kappa,\mu)=\Udiag(\pi|\kappa)$ (i.e., they do not depend on the partition which is being summed over in \eqref{eq:definingeq} and \eqref{eq:definingeq2}). 
In such a case, $\U(\pi|\kappa,\nu)$ and  $\Udiag(\pi|\kappa)$ factor out of the left-hand side of these equations and we can compute the summation over $\mu$ as
\begin{align}
\sum_{\mu} Q_{\kappa/\mu}(\rho_{i+1}) P_{\nu/\mu}(\rho_{j+1})
&=
\sum_{\pi} \frac{ P_{\pi/\kappa}(\rho_{j+1})Q_{\pi/\nu}(\rho_{i+1})}{\Pi(\rho_{j+1},\rho_{i+1})}, \\
\sum_{\mu} \frac{Q_{\kappa/\mu}(\rho_{i+1})  \ve_{\mu}(\rho_{\circ})}{\Pi(\bigcup_{e\in E^{\uparrow}(\pathh)} \rho_e; \rho_{i+1})}
&=\frac{\ve_{\kappa}(\rho_{\circ},\rho_{i+1})}{\Pi(\bigcup_{e\in E^{\uparrow}(\pathh)} \rho_e; \rho_{i+1})}.\label{eq:Udiagsimplif}
\end{align}
Thus, we may choose $\U_{i,j}(\pi|\kappa,\mu, \nu)$ as 
$$ \U_{i,j}(\pi|\kappa, \mu, \nu) = \U_{i,j}(\pi|\kappa,\nu)  =    \frac{ P_{\pi/\kappa}(\rho_{j+1})Q_{\pi/\nu}(\rho_{i+1})}{\sum_{\pi}  P_{\pi/\kappa}(\rho_{j+1})Q_{\pi/\nu}(\rho_{i+1}) }.$$
This transition operator was introduced in \cite[Section 2.3]{borodin2014macdonald} following an approach introduced in the work of Diaconis-Fill \cite{diaconis1990strong} for general Markov chains and developed by Borodin-Ferrari \cite{borodin2014anisotropic} in the Schur process case. The term \emph{push-block dynamics} comes from  \cite{borodin2014anisotropic} where the dynamics on Gelfand-Tsetlin patterns (that give a way of encoding a sequence of interlacing partitions), interpreted as particle system dynamics, are described using nearest neighbour interaction of particles with pushing and blocking mechanisms. At the Macdonald process level of generality, these push-block dynamics do not admit a particularly nice interpretation in terms of local moves on a Gelfand-Tsetlin pattern. In the special case of $t=0$ (Section \ref{sec:qWhittaker}), we will recall two other choices of dynamics from \cite{matveev2015q} which have nicer marginals. 

Using \eqref{eq:Udiagsimplif}, the operator $\Udiag_{i,i}$  may be chosen as 
$$ \Udiag_{i,i}(\pi\vert \kappa, \mu)  = \Udiag_{i,i}(\pi\vert \kappa) = \frac{P_{\pi/\kappa}(\rho_{i+1})\ve_{\pi}(\rho_{\circ})}{\ve_\kappa(\rho_{i+1}, \rho_{\circ})\Pi(\rho_{i+1}, \rho_{\circ})\Phi(\rho_{i+1})}.$$

\begin{remark}
The above construction also provides a way to grow usual Macdonald processes \cite{borodin2014macdonald}. In that case we would consider path starting from the horizontal axis and ending on the vertical axis, the growth mechanism would involve only the bulk transition operator $\U$, and the set of specializations $\left\lbrace \rho_i^{\mathrm h}\right\rbrace_{i\geqslant 1} $ on horizontal edges and $\left\lbrace \rho_i^{\mathrm v}\right\rbrace_{i\geqslant 1} $ on vertical edges can be different. 
\end{remark}
\begin{remark}
The degeneration when $q=t$ of these dynamics preserve the Pfaffian Schur process. They are studied in \cite{baik2018pfaffian} to justify that the first coordinates $ \la_1$ of random partitions in a Pfaffian Schur process have the same law as last passage times along a down-right path in a model of last passage percolation in a half quadrant. 
\end{remark}

\thispagestyle{plain}
 \section{Observables of half-space Macdonald processes}
\label{sec:observables}
 Consider positive variables $x_1,\ldots, x_n$ and a Macdonald nonnegative specialization $\rho$. The generalized Littlewood  summation identity \eqref{eq:CauchyLittlewoodspecialized}
$$
\sum_{\lambda\in \Y} P_{\lambda}(x_1,\ldots, x_n) \ve_{\lambda}(\rho) = \Pi(x_1,\ldots x_n;\rho) \Phi(x_1,\ldots, x_n)
$$
holds as a numeric identity under some assumptions on $x_1, \dots, x_n$ and $\rho$. 
Let us assume that we have a linear operator $\mathbf A_n$ which acts on $n$-variable symmetric functions and which is diagonal in the basis $\lbrace P_{\lambda}\rbrace$ of symmetric polynomials, with eigenvalues $\lbrace d_{\lambda}\rbrace$. Then, applying $\mathbf A_n$ to both sides of the above equality in the variables $x_1, \dots, x_n$ and subsequently dividing both sides by $\Pi(x_1,\ldots x_n;\rho) \Phi(x_1,\ldots, x_n)$ we find that
$$
\EPMM_{(x_1,\ldots,x_n),\rho}\big[d_{\lambda}]= \sum_{\lambda\in \Y} d_{\lambda} \frac{P_{\lambda}(x_1,\ldots, x_n) \ve_{\lambda}(\rho)}{\Pi(x_1,\ldots x_n;\rho) \Phi(x_1,\ldots, x_n)}  = \frac{ \mathbf A_n \ \Pi(x_1,\ldots x_n;\rho) \Phi(x_1,\ldots, x_n)}{ \Pi(x_1,\ldots x_n;\rho) \Phi(x_1,\ldots, x_n)}
$$
where the numerator on the right-hand side $\mathbf A_n \ \Pi(x_1,\ldots x_n;\rho) \Phi(x_1,\ldots, x_n)$ stands for the application of  $\mathbf A_n$ to the function $(x_1,\ldots,x_n)\mapsto \Pi(x_1,\ldots x_n;\rho) \Phi(x_1,\ldots, x_n)$ and then subsequent evaluation at the point $(x_1,\ldots,x_n)$. If there are many operators which are all mutually diagonalized by the $P_{\lambda}$ then applying them sequentially yields formulas for moments involving the products of their eigenfunctions, cf \cite[Section 2.2.3]{borodin2014macdonald} where this scheme was first realized in a similar context. Note that one must check the validity of exchanging the action of $\mathbf A_n$ with the summation over $\lambda$. 

\subsection{Macdonald difference operators}
\label{sec:differenceoperators}
\begin{definition}
	For any $u\in\R$ and $1\leqslant i\leqslant n$, define the \emph{shift operator} $\Tshift_{u,x_i}$ by
	\begin{equation*}
	(\Tshift_{u,x_i}F)(x_1,\dots,x_n)=F(x_1,\dots,ux_i,\dots,x_n),
	\end{equation*}
	and for any subset $I\subset\{1,\dots,n\}$ with $r$ elements, define
	\begin{equation*}
	A_I(x;t)=t^{\frac{r(r-1)}2}\prod_{i\in I,\,j\notin I}\frac{tx_i-x_j}{x_i-x_j}\,.
	\end{equation*}
	For $r=1,2,\ldots,n$, define the \emph{$r^{th}$ Macdonald difference operator}
	\begin{equation*}
	\DD_n^r=\sum_{\substack{I\subset\{1,\ldots,n\}\\|I|=r}} A_I(x;t)\prod_{i\in I} \Tshift_{q,x_i}.
	\end{equation*}
	Also, define a variant of the Macdonald difference operator
	\begin{equation*}
	\overline{\DD}_n^r = t^{-\frac{n(n-1)}{2}} \DD_n^{n-r} \Tshift_{q^{-1}},
	\end{equation*}
	where the operator $\Tshift_{q^{-1}}$ multiplies all variables by $q^{-1}$, $T_{q^{-1}} = T_{q^{-1}, x_1}, \dots, T_{q^{-1}, x_n}$.
\end{definition}

\begin{proposition}
	For any partition $\lambda$ with $\ell(\lambda)\leqslant  n$
	\begin{align}
	\DD_n^r P_\lambda(x_1,\dots,x_n)&=e_r(q^{\lambda_1}t^{n-1},q^{\lambda_2}t^{n-2},\dots,q^{\lambda_n}t^0) P_\lambda(x_1,\dots,x_n),\\
	\overline{\DD}_n^r P_{\lambda}(x_1,\ldots,x_n) &= t^{-\frac{n(n-1)}{2}} q^{-|\lambda|} e_{n-r}(q^{\lambda_1}t^{n-1},\ldots, q^{\lambda_n}t^0)P_{\lambda}(x_1,\ldots,x_n)\\
	&= e_r(q^{-\lambda_1}t^{1-n},q^{-\lambda_2}t^{2-n},\ldots,q^{-\lambda_n}t^0)P_{\lambda}(x_1,\ldots, x_n).
	\label{eq:eigenrelationD}
	\end{align}
	Here $e_r$ is the elementary symmetric function, $e_r(x_1,\ldots,x_n) = \sum\limits_{1\leqslant i_1<\cdots<i_r\leqslant n} x_{i_1}\cdots x_{i_r}$.
	\label{prop:Dreigenrelation}
\end{proposition}
\begin{proof}
	The first identity is from \cite[VI(4.15)]{macdonald1995symmetric} and the next two follow easily, as explained in \cite[Remark 2.2.12]{borodin2014macdonald}.
\end{proof}

\begin{proposition}\label{prop:intfor}
	Fix $k\geqslant 1$ and consider functions $H(u_1,\ldots, u_n) = \prod_{i=1}^{n} h(u_i)$ and $\Phi(u_1,\ldots, u_n) = \prod_{i<j} \phi(u_iu_j)$, where $h(\cdot)$ and $\phi(\cdot)$ are univariate locally holomorphic functions as specified below. Then
	\begin{align*}
	\frac{\big((\DD^1_n)^k (H\Phi)\big)(x_1,\ldots, x_n)}{(H\Phi)(x_1,\ldots, x_n)} &= \frac{(t-1)^{-k}}{(2\pi \I)^k} \oint \cdots \oint \prod_{1\leqslant a<b\leqslant k} \frac{(tw_a-qw_b)(w_a-w_b)}{(w_a-qw_b)(tw_a-w_b)}\, \frac{\phi(q^2 w_aw_b)\phi(w_a w_b)}{\phi(qw_a w_b)^2} \\ 
	&\hspace {2.6cm} \times \prod_{m=1}^{k} \prod_{j=1}^{n}\bigg( \frac{tw_m-x_j}{w_m -x_j}\, \frac{\phi(q x_jw_m)}{\phi(x_j w_m)}\bigg) \, \frac{h(q w_m) \phi(w_m^2)}{h(w_m) \phi(q w_m^2)} \, \frac{dw_m}{w_m},
	\end{align*}
	and
	\begin{align*}
	\frac{\big((t^{n-1} \overline{\DD}^1_n)^k (H\Phi)\big)(x_1,\ldots, x_n)}{(H\Phi)(x_1,\ldots, x_n)} &= \frac{(t-1)^{-k}}{(2\pi \I)^k} \oint \cdots \oint \prod_{1\leqslant a<b\leqslant k} \frac{(tw_a-qw_b)(w_a-w_b)}{(w_a-qw_b)(tw_a-w_b)}\, \frac{\phi\big(\frac{1}{q^2 w_aw_b}\big)\phi\big(\frac{1}{w_a w_b}\big)}{\phi\big(\frac{1}{qw_a w_b}\big)^2} \\
	&\hspace {2.6cm}\times \prod_{m=1}^{k} \prod_{j=1}^{n}\bigg( \frac{1-tw_m x_j}{1 -w_m x_j}\, \frac{\phi\big(\frac{x_j}{q w_m}\big)}{\phi\big(\frac{x_j}{w_m}\big)}\bigg) \, \frac{h\big(\frac{1}{q w_m}\big) \phi\big(\frac{1}{w_m^2}\big)}{h\big(\frac{1}{w_m}\big) \phi\big(\frac{1}{q w_m^{2}}\big)} \, \frac{dw_m}{w_m}.
	\end{align*}
	%I corrected a little typo \phi\big(\frac{1}{(q w_m)^{2}}\big) becomes \phi\big(\frac{1}{q w_m^{2}}\big)
	There are some assumptions we must make on the contours and functions $h, \phi$ for the above equalities to hold. In the first formula, we assume that the (positively oriented) contour for $w_c$, $1\leqslant c\leqslant k$, contains $\{x_1,\ldots, x_n\}$ and the image of the $w_{c+1},\ldots, w_{k}$ contours multiplied by $q$, and does not contain any other singularities of the integrand. In the second formula, we assume that the (positively oriented) contour for $w_c$, $1\leqslant c\leqslant k$, contains $\{x_1^{-1},\ldots, x_n^{-1}\}$ and the image of the $w_{c+1},\ldots, w_{k}$ contours multiplied by $q$, and does not contain any other singularities of the integrand. In both cases, we also assume that $h$ and $\phi$ are holomorphic and nonzero on a suitably large complex neighborhood so as not to yield any singularities when the integrals are evaluated through residues by  shrinking the contours.
\end{proposition}
\begin{proof}
	We prove the first equality (the second follows similarly) via the approach used in the proofs of \cite[Propositions 2.2.11, 2.2.14]{borodin2014macdonald} (which correspond to the special case when $\phi$ is assumed to be constant). Let us consider the effect of the $\phi$ terms. Observe that due to the multiplicative structure of $\Phi$, 
	$$
	\frac{(\Tshift_{q,x_j} \Phi)(x_1,\ldots, x_n)}{ \Phi(x_1,\ldots, x_n)} = \subs_{z=x_j} \frac{\phi(z^2)}{\phi(q z^2)} \prod_{i=1}^{n} \frac{\phi(q x_i z)}{\phi(x_i z)}
	$$
	where $\subs_{z=x}$ means the substitution of  $z=x$ into the expression that follows. Combining with the proof of \cite[Propositions 2.2.11]{borodin2014macdonald} we arrive at 
	$$\big(\DD^1_n (H\Phi)\big)(x_1,\ldots, x_n)  =  \frac{(t-1)^{-1}}{2\pi \I} \oint \prod_{j=1}^{n} \left(  \frac{tw_1-x_j}{w_1 -x_j}\, \frac{\phi(q x_jw_1)}{\phi(x_j w_1)} \right) \frac{h(q w_1) \phi(w_1^2)}{h(w_1) \phi(q w_1^2)} \, \frac{dw_1}{w_1}(H\Phi)(x_1,\ldots, x_n),  
	$$
	which is 
	the $k=1$ case of the proposition we seek to prove.
	
	Observe that inside the integral formula for $\big(\DD^1_n (H\Phi)\big)(x_1,\ldots, x_n)$, the variables $x_1,\ldots, x_n$ arise in the terms
	$$
	(H\Phi)(x_1,\ldots, x_n) \prod_{j=1}^{n} \frac{tw_1-x_j}{w_1-x_j}\, \frac{\phi(q x_j w_1)}{\phi(x_jw_1)},
	$$
	which is equal to $H_1 \Phi$ where $H_1(x_1,\ldots, x_n) := \prod_{i=1}^{N} h_1(x_i)$ with
	$$
	h_1(x) := h(x)\, \frac{tw_1-x}{w_1-x}\, \frac{\phi(q x w_1)}{\phi(x w_1)}.
	$$
	Applying $\DD^1_n$, we find that by linearity it can be taken into the integrand and it applies just to $H_1\Phi$ which introduces a second integral (appealing to the $k=1$ case we have already proved). We obtain 
	\begin{multline*}
\big((\DD^1_n)^2 (H\Phi)\big)(x_1,\ldots, x_n) = 
\frac{(t-1)^{-2}}{(2\pi \I)^2} \oint \oint  \frac{(tw_1-qw_2)(w_1-w_2)}{(w_1-qw_2)(tw_1-w_2)}   \frac{\phi(q^2 w_1w_2)\phi(w_1 w_2)}{\phi(qw_1 w_2)^2}  \\ \times 
\prod_{j=1}^{n} \left(\frac{tw_2-x_j}{w_2 -x_j}\, \frac{\phi(q x_jw_2)}{\phi(x_j w_2)}\right) 
\frac{h(q w_2) \phi(w_2^2)}{h(w_2) \phi(q w_2^2)} \, \frac{dw_2}{w_2} \\ \times 
	\prod_{j=1}^{n}\left( \frac{tw_1-x_j}{w_1 -x_j}\, \frac{\phi(q x_jw_1)}{\phi(x_j w_1)}\right)
	 \frac{h(q w_1) \phi(w_1^2)}{h(w_1) \phi(q w_1^2)} \, \frac{dw_1}{w_1} (H\Phi)(x_1,\ldots, x_n), 
	\end{multline*}
	where the contour for $w_1$ encircles the $x_i$'s and $q w_2$, and the contour for $w_2$ encircles  the $x_i$'s. 
	Now the variables $x_1, \dots, x_n$ arise inside the double integral  as $H_2 \Phi$, where $H_2(x_1, \dots, x_n) = \prod_{i=1}^n h_2(x_i)$ with 
	$$h_2(x)  = h(x) \frac{tw_2 - x}{w_2-x}\frac{tw_1 - x}{w_1-x}\frac{\phi(q x w_1)}{\phi(x w_1)}\frac{\phi(q x w_2)}{\phi(x w_2)}.$$ 
	Repeating $k$ times leads to the claimed formula.
\end{proof}
\begin{proposition}
	Fix $r\geqslant 1$, $H(u_1, \dots, u_n)= \prod_{i=1}^{n} h(u_i)$ and $\Phi(u_1,\ldots, u_n) = \prod_{i<j} \phi(u_iu_j)$, where $h$ and $\phi$ are holomorphic and nonzero on a suitably large complex neighborhood so as not to yield any singularities when both integrals below are evaluated through residues.  Then
	\begin{multline}
	\frac{\DD_n^r(H\Phi)(x_1, \dots, x_n)}{(H\Phi)(x_1, \dots, x_n)}= \frac{1}{ r!} \oint\frac{\mathrm{d}z_1}{2\I\pi}\cdots \oint\frac{\mathrm{d}z_r}{2\I\pi} \det\left[ \frac{1}{tz_k-z_l}\right]_{k,l=1}^r \prod_{j=1}^{r} \left( \frac{h(qz_j)}{h(z_j)} \prod_{m=1}^n \frac{tz_j-x_m}{z_j-x_m}\right)\\ \times 
	\prod_{1\leqslant i<j\leqslant r} \frac{\phi(q^2z_iz_j)}{\phi(z_iz_j)} \prod_{i=1}^n\prod_{j=1}^{r} \frac{\phi(qz_j x_i)}{\phi(z_j x_i)} \prod_{i,j=1}^r \frac{\phi(z_iz_j)}{\phi(qz_iz_j)}.
	\end{multline}
	where the contours are positively oriented contours encircling $\lbrace x_1, \dots, x_n\rbrace$ and no other singularity of the integrand. 
	
	We also have 
	\begin{multline}
	\frac{\overline{\DD}_n^r(H\Phi)(x_1, \dots, x_n)}{(H\Phi)(x_1, \dots, x_n)}= \frac{1}{ r!} \oint\frac{\mathrm{d}w_1}{2\I\pi}\cdots \oint\frac{\mathrm{d}w_r}{2\I\pi} \det\left[ \frac{1}{tw_k-w_l}\right]_{k,l=1}^r \prod_{j=1}^{r} \left( \frac{h((qw_j)^{-1})}{h(w_j^{-1})} \prod_{m=1}^n \frac{tw_j-x_m^{-1}}{w_j-x_m^{-1}}\right)\\ \times 
	\prod_{1\leqslant i<j\leqslant r} \frac{\phi((q^2w_iw_j)^{-1})}{\phi((w_iw_j)^{-1})} \prod_{i=1}^n\prod_{j=1}^{r} \frac{\phi(x_i/(qw_j))}{\phi( x_i/w_j)} \prod_{i,j=1}^r \frac{\phi((w_iw_j)^{-1})}{\phi((qw_iw_j)^{-1})}.
	\end{multline}
     where the contours are positively oriented circles encircling $\lbrace 1/x_1, \dots, 1/x_n\rbrace$ and no other singularity.
	\label{prop:highermacdo}
\end{proposition}
\begin{proof}
	Notice that for $\vert I\vert =r$, one has that 
	$$ \prod_{i\in I} \frac{\Tshift_{q, x_i} \Phi}{\Phi}(x_1, \dots, x_n) = \underset{\substack{
		z_1 = x_{i_1}\\ \dots \\ z_r = x_{i_r}} }{\subs} \left[ \prod_{1\leqslant <i<j\leqslant r} \frac{\phi(q^2z_iz_j)}{\phi(z_iz_j)} \prod_{i=1}^n\prod_{j=1}^{r} \frac{\phi(qz_j x_i)}{\phi(z_j x_i)} \prod_{i,j=1}^r \frac{\phi(z_iz_j)}{\phi(qz_iz_j)}\right].
	$$
	Then, one proceeds as in Proposition 2.2.11 in \cite{borodin2014macdonald} evaluating the  integral through the residues at $ x_1, \dots, x_n $. 
\end{proof}

\subsection{Moment integral formulas}
\label{sec:generalmomentsformulas}
In this section we consider the half-space Macdonald measure with specializations $\rhoup = (a_1, \dots, a_n)$, where $a_i\in (0,1)$ and $\rhodiag = \rho$ where $\rho$ is a Macdonald non-negative specialization such that  $\Pi\big( \vec a ;\rho\big)\Phi(\vec{a})$ can be expanded via the generalized Littlewood identity \eqref{eq:CauchyLittlewoodspecialized}.
\begin{proposition} For any positive integer $k$, 
		\begin{multline}
		\EPMM_{(a_1,\ldots, a_n),\rho}\Big[\big(q^{\lambda_1}t^{n-1} + q^{\lambda_2} t^{n-2} + \cdots + q^{\lambda_n}\big)^k\Big]  \\  = 
 \frac{1}{(t-1)^k} \oint\frac{\mathrm{d}w_1}{2\I\pi} \cdots \oint\frac{\mathrm{d}w_k}{2\I\pi} \prod_{1\leqslant a<b\leqslant k} \frac{(tw_a-qw_b)(w_a-w_b)}{(w_a-qw_b)(tw_a-w_b)}\, \frac{(1-q w_aw_b)(1-t w_a w_b)}{(1-q t w_a w_b)(1-w_aw_b)} \\ \times
		\prod_{m=1}^{k} \bigg(\prod_{j=1}^{n} \frac{tw_m-a_j}{w_m -a_j}\, \frac{1- a_jw_m}{1-t a_j w_m}\bigg) \, \frac{\Pi(q w_m;\rho) (1-tw_m^2)}{\Pi(w_m;\rho) (1-w_m^2)} \, \frac{1}{w_m},
		\label{eq:momentsMacdonald1}
		\end{multline}
		where the contour for $w_c$, $1\leqslant c\leqslant k$, contains $\{a_1,\ldots, a_n\}$ and the image of the $w_{c+1},\ldots, w_{k}$ contours multiplied by $q$, and does not contain any other singularities of the integrand (this may restrict the choice of admissible specializations $\rho$). 

Similarly, assume that $\rho$ and $k$ are chosen so that $\Pi\big(q^{-k} \vec a;\rho\big)\Phi(q^{-k} \vec{a})$ can be expanded via the generalized Littlewood identity \eqref{eq:CauchyLittlewoodspecialized}. Then we also have 
		\begin{multline}
		\EPMM_{(a_1,\ldots, a_n),\rho}\Big[\big(q^{-\lambda_1} + q^{-\lambda_2} t+ \cdots + q^{-\lambda_n} t^{n-1}\big)^k\Big]   \\ =
\frac{1}{(t-1)^k} \oint\frac{\mathrm{d}w_1}{2\I\pi} \cdots \oint\frac{\mathrm{d}w_k}{2\I\pi} \prod_{1\leqslant a<b\leqslant k} \frac{(tw_a-qw_b)(w_a-w_b)}{(w_a-qw_b)(tw_a-w_b)}\, \frac{(q^2 w_aw_b-t)(qw_a w_b-1)}{(q^2 w_a w_b-1)(qw_aw_b-t)}\\
	\times \prod_{m=1}^{k} \bigg(\prod_{j=1}^{n} \frac{1-tw_m a_j}{1 -w_m a_j}\, \frac{qw_m - t a_j}{q w_m - a_j}\bigg) \, \frac{\Pi\big((q w_m)^{-1};\rho\big) (q w_m^2-1)}{\Pi\big((w_m)^{-1};\rho\big) (qw_m^2-t)} \, \frac{1}{w_m},
		\label{eq:momentsMacdonald2}
		\end{multline}
		where the contour for $w_c$, $1\leqslant c\leqslant k$, contains $\{1/a_1,\ldots, 1/a_n\}$ and the image of the $w_{c+1},\ldots, w_{k}$ contours multiplied by $q$, and does not contain any other singularities of the integrand  (this may restrict the choice of admissible specializations $\rho$). All contours above are positively oriented. 
	\label{prop:momentsMacdonald}
\end{proposition}
\begin{proof}
	Plug in $h(u) = \Pi(u;\rho)$ and $\phi(u) = \frac{(tu;q)_{\infty}}{(u;q)_{\infty}}$ in Proposition \ref{prop:intfor}. This yields integral formulas for $\frac{\big( \DD_n^1 \big)^k \ \Pi(\vec a, \rho)\Phi(\vec a) }{\Pi(\vec a, \rho)\Phi(\vec a)}$ and $\frac{\big( \overline{\DD}_n^1 \big)^k \ \Pi(\vec a, \rho)\Phi(\vec a) }{\Pi(\vec a, \rho)\Phi(\vec a)}$, where the operators act on the variables $a_i$. These quantities are related to the desired expectations by Proposition \ref{prop:Dreigenrelation}. In \eqref{eq:momentsMacdonald2}, we need the extra assumptions on $\rho$ that $\Pi\big(q^{-k} \vec a;\rho\big)\Phi(q^{-k} \vec{a})$ can be expanded via the generalized Littlewood identity \eqref{eq:CauchyLittlewoodspecialized} to ensure that one can commute the action of the operator $\big(\overline{\DD}_n^1 \big)^k$ with the expectation $\EPMM_{(a_1,\ldots, a_n),\rho}$ and one obtains using Fubini theorem that  
	$$ \frac{\big( \overline{\DD}_n^1 \big)^k \ \Pi(\vec a, \rho)\Phi(\vec a) }{\Pi(\vec a, \rho)\Phi(\vec a)} = \EPMM_{(a_1,\ldots, a_n),\rho}\Big[\big(q^{-\lambda_1} + q^{-\lambda_2} t+ \cdots + q^{-\lambda_n} t^{n-1}\big)^k\Big] .$$
\end{proof}
\begin{remark}
	If $\rho$ is of the form $\rho=\rho(\alpha, \beta, \gamma)$ as in Section \ref{sec:specializations}, the hypothesis that $\Pi\big(q^{-k} \vec a);\rho\big)\Phi(q^{-k} \vec{a})$ can be expanded via the generalized Littlewood identity \eqref{eq:CauchyLittlewoodspecialized} is equivalent to the fact  that $ \max_i\lbrace \alpha_i \rbrace \max_i\lbrace a_i \rbrace <q^k$ and $ \big( \max_i\lbrace a_i \rbrace  \big)^2<q^k$. Otherwise, the expectation in the left-hand-side in \eqref{eq:momentsMacdonald2} would fail to exist (see Section \ref{sec:specializations}). 
\end{remark}

\begin{proposition}
	Let $r$ be a positive integer.  We have 
			\begin{multline} 
		\EPMM_{(a_1,\ldots, a_n),\rho}\Big[e_r\big(q^{\lambda_1}t^{n-1} ,  q^{\lambda_2} t^{n-2} ,  \cdots,  q^{\lambda_n}\big)\Big]   \\ =
		\frac{1}{r!} \oint\frac{\mathrm{d}z_1}{2\I\pi}\cdots \oint\frac{\mathrm{d}z_r}{2\I\pi} \det\left[ \frac{1}{tz_k-z_l}\right]_{k,l=1}^r \prod_{j=1}^{r} \left( \frac{\Pi(qz_j ; \rho)}{\Pi(z_j ; \rho)} \prod_{i=1}^n \left( \frac{tz_j-a_i}{z_j-a_i}  \frac{1-z_ja_i}{1-tz_ja_i}\right)\right)\\ \times  
		\prod_{i=1}^r \frac{1-tz_i^2}{1-z_i^2}
		\prod_{1\leqslant i<j\leqslant r} \frac{1-qz_iz_j}{1-tq z_iz_j}\frac{1-tz_iz_j}{1-z_iz_j},
		\end{multline}
		where the contours are positively oriented closed curves around $\lbrace a_1, \dots, a_n\rbrace$ and no other singularity of the integrand. 

Moreover, assume that $\rho$  and $r$ are chosen so that $\Pi\big(q^{-r} \vec{a};\rho\big)\Phi(q^{-r} \vec{a})$ can be expanded via the generalized Littlewood identity \eqref{eq:CauchyLittlewoodspecialized}. Then,
		\begin{multline}
		\EPMM_{(a_1,\ldots, a_n),\rho}\Big[e_r\big(q^{-\lambda_1}t^{1-n} ,  q^{-\lambda_2} t^{2-n} ,  \cdots,  q^{-\lambda_n}\big)\Big]   \\ =
		\frac{1}{r!} \oint\frac{\mathrm{d}w_1}{2\I\pi}\cdots \oint\frac{\mathrm{d}w_r}{2\I\pi} \det\left[ \frac{1}{tw_k-w_l}\right]_{k,l=1}^r \prod_{j=1}^{r} \left( \frac{\Pi((qw_j)^{-1}; \rho)}{\Pi(w_j^{-1}; \rho)} \prod_{i=1}^n\left(  \frac{1-a_itw_j}{1-a_i w_j}\frac{qw_j-ta_i}{qw_j-a_i}\right)\right)\\ \times
		\prod_{i=1}^r \frac{qw_i^2-1}{qw_i^2-t}   \prod_{1\leqslant i<j\leqslant r}
		\frac{(q^2w_iw_j-t)(qw_iw_j-1)}{(q^2w_iw_j-1)(qw_iw_j-t)},
		\end{multline}
		where the contours are positively oriented closed curves around $\lbrace 1/a_1, \dots, 1/a_n\rbrace$ and no other singularity.
	\label{prop:Dr}
\end{proposition}
\begin{proof}
	The proof is similar as the proof of Proposition \ref{prop:momentsMacdonald}, using Proposition \ref{prop:highermacdo} instead of Proposition \ref{prop:intfor}.
\end{proof}

\subsection{Noumi's $q$-integral operator and variants}
\label{sec:Noumi}
In order to lighten some notations, we set in this Section $ (a)_n=(a;q)_{n} $ for any $n\in \Z_{\geqslant 0}\cup\lbrace \infty\rbrace$,  and we will use the function $\phi(x) = (tx)_{\infty}/(x)_{\infty}$ from \eqref{eq:PI} and the function $f(x) = (tx)_{\infty}/(qx)_{\infty}$ from \eqref{eq:defphipsi}.

Define the operator $\Noumi_n^z$ acting on analytic functions in $x_1, \dots, x_n$ by 
$$ \Noumi_n^z = \sum_{\eta_1, \dots, \eta_n=0}^{\infty} z^{\vert \eta\vert } h_{\eta}(x;q,t) \prod_{i=1}^n \big(\Tshift_{q, x_i}\big)^{\eta_i},$$
where $\vert \eta\vert  = \eta_1+\dots+\eta_n$ and 
$$ h_{\eta}(x;q,t) = \prod_{i=1}^n t^{(i-1)\eta_i}\frac{(t)_{\infty}(q^{\eta_i+1})_{\infty}}{(tq^{\eta_i})_{\infty}(q)_{\infty}}\prod_{i<j} \frac{(q^{\eta_i-\eta_j}x_i/x_j)_{\infty}(q^{1-\eta_j}t^{-1}x_i/x_j)_{\infty}(tx_i/x_j)_{\infty}(q^{1+\eta_i}x_i/x_j)_{\infty}}{(q^{-\eta_j}x_i/x_j)_{\infty}(qt^{-1}x_i/x_j)_{\infty}(tq^{\eta_i}x_i/x_j)_{\infty}(q^{1+\eta_i-\eta_j}x_i/x_j)_{\infty}}. 
$$ 
Proposition 2.2.17 in \cite{borodin2014macdonald} states\footnote{More precisely, \cite[Proposition 2.2.17]{borodin2014macdonald} states the identity as a formal power series in $z$ and one can see that the R.H.S. is a convergent series when $\vert z\vert <1$ using the $q$-Binomial theorem \eqref{eq:qbinomial}.} that for $\vert z\vert <1$, 
\begin{equation}
\Noumi_n^z P_{\la}(x)=\prod_{i=1}^n\frac{(q^{\la_i}t^{n-i+1}z)_{\infty}}{(q^{\la_i}t^{n-i}z)_{\infty}}P_{\la}(x).
\label{eq:Noumieigenrelation}
\end{equation}

\begin{proposition}
	The operator $\Noumi^z$ can be rewritten as 
	$$\Noumi_n^z =\sum_{\eta_1, \dots, \eta_n=0}^{\infty} z^{\vert \eta\vert }\prod_{i<j} \frac{q^{\eta_j}x_j - q^{\eta_i}x_i}{x_j-x_i} \prod_{i,j} \frac{(tx_i/x_j)_{\eta_i}}{(qx_i/x_j)_{\eta_i}} \prod_{i=1}^n \big(\Tshift_{q, x_i}\big)^{\eta_i}.$$
	\label{prop:equivalencewithNoumi}
\end{proposition}
\begin{remark}
	Proposition \ref{prop:equivalencewithNoumi} shows that $\Noumi_n^z$ coincides with an operator known as \emph{Noumi's q-integral operator}. The eigenrelation \eqref{eq:Noumieigenrelation} first appeared in \cite{feigin2009commutative} where it  is attributed to \cite{noumiinfinite}. Additional properties of $\Noumi_n^z$ can be found in \cite{noumi2012direct}, which in particular gives a proof of \eqref{eq:Noumieigenrelation} (see around equation  (5.6)).
	Note that Section 5 of  \cite{borodin2016observables} provides yet another proof of \eqref{eq:Noumieigenrelation} by E. Rains.
\end{remark}
\begin{proof}
	We need to show that 
	\begin{equation}
	h_{\eta}(x;q,t) = \prod_{i<j} \frac{q^{\eta_j}x_j - q^{\eta_i}x_i}{x_j-x_i} \prod_{i,j} \frac{(tx_i/x_j)_{\eta_i}}{(qx_i/x_j)_{\eta_i}}.
	\label{eq:equalityh}
	\end{equation}
 On the L.H.S., we have 
	\begin{align*} h_{\eta}(x;q,t) &= \prod_{i=1}^n \frac{f(1)}{f(q^{\eta_i})} \prod_{i<j} \frac{t^{\eta_j}(1-q^{\eta_i-\eta_j}x_i/x_j)(tx_i/x_j)_{\infty}}{f(q^{-\eta_j}t^{-1}x_i/x_j)f(q^{\eta_i}x_i/x_j)(qt^{-1}x_i/x_j)_{\infty}}\\
	&=  \prod_{i=1}^n \frac{f(1)}{f(q^{\eta_i})} \prod_{i<j} \frac{t^{\eta_j}(1-q^{\eta_i-\eta_j}x_i/x_j)f(x_i/x_j)f(t^{-1}x_i/x_j)}{(1-x_i/x_j)f(q^{-\eta_j}t^{-1}x_i/x_j)f(q^{\eta_i}x_i/x_j)}.
	\end{align*}
	On the other hand, $$\frac{(tx_i/x_j)_{\eta_i}}{(qx_i/x_j)_{\eta_i}} = \frac{f(x_i/x_j)}{f(q^{\eta_i}x_i/x_j)},$$
	so that 
	$$ R.H.S. \eqref{eq:equalityh} = \prod_{i=1}^n \frac{f(1)}{f(q^{\eta_i})}\prod_{i<j}\frac{q^{\eta_j}(1-q^{\eta_i-\eta_j}x_i/x_j)}{1-x_i/x_j}  \frac{f(x_i/x_j)}{f(q^{\eta_i}x_i/x_j)}\frac{f(x_j/x_i)}{f(q^{\eta_j}x_j/x_i)}.$$
	Thus, it is sufficient to show that 
	$$
	\frac{t^{\eta_j}f(t^{-1}x_i/x_j)}{f(q^{-\eta_j}t^{-1}x_i/x_j)} = \frac{q^{\eta_j}f(x_j/x_i)}{f(q^{\eta_j}x_j/x_i)}.
	$$
	which is equivalent to 
	\begin{equation}
	\frac{t^{\eta_j}(q^{1-\eta_j}t^{-1}x_i/x_j)_{\eta_j}}{(q^{-\eta_j}x_i/x_j)_{\eta_j}} = \frac{q^{\eta_j}(tx_j/x_i)_{\eta_j}}{(q x_j/x_i)_{\eta_j}}.
	\label{eq:passepasse}
	\end{equation} 
	It is easy to check that 
	$$ \frac{t^N(q^{1-N}t^{-1}X)_N}{q^N(q^{-N}X)_N} = \frac{(tX^{-1})_N}{(qX^{-1})_N},$$
	so that \eqref{eq:passepasse} is established by setting $X=x_i/x_j, N=\eta_j$.
\end{proof}

In Section \ref{sec:differenceoperators} we discussed two types of Macdonald difference operators $\DD^r$ and $\overline{\DD}^r$.   We define now an operator $\Moumi_n^z$, similar to $\Noumi_n^z$, but having eigenvalue $\prod_{i=1}^n \frac{(q^{-\la_i}t^{i}z)_{\infty}}{(q^{-\la_i}t^{i-1}z)_{\infty}}$
  instead of 
  $\prod_{i=1}^n\frac{(q^{\la_i}t^{n-i+1}z)_{\infty}}{(q^{\la_i}t^{n-i}z)_{\infty}}$. 
  Let 
$$ \Moumi_n^z = \sum_{\eta_1, \dots, \eta_n=0}^{\infty} z^{\vert \eta\vert } h_{\eta}(x_1^{-1}, \dots, x_n^{-1};q,t) \prod_{i=1}^n \big(\Tshift_{q^{-1}, x_i}\big)^{\eta_i}.$$
Equivalently, using Proposition \ref{prop:equivalencewithNoumi}, 
$$ \Moumi_n^z  =\sum_{\eta_1, \dots, \eta_n=0}^{\infty} z^{\vert \eta\vert }\prod_{i<j} \frac{q^{\eta_j}x_i - q^{\eta_i}x_j}{x_i-x_j} \prod_{i,j} \frac{(tx_i/x_j)_{\eta_j}}{(qx_i/x_j)_{\eta_j}} \prod_{i=1}^n \big(\Tshift_{q^{-1}, x_i}\big)^{\eta_i}.$$

\begin{proposition}[\cite{noumi2012direct}]We have the following formal power series identity in the variable $z$:   
	\begin{equation}
	\Moumi_n^z P_{\la}(x)=\prod_{i=1}^n\frac{(q^{-\la_i}t^{i}z)_{\infty}}{(q^{-\la_i}t^{i-1}z)_{\infty}}P_{\la}(x).
	\label{eq:othereigenrelation}
	\end{equation}
	Moreover, when $z$ is such that for all $i=1, \dots n$, $\vert zq^{-\lambda_i}t^{i-1}\vert <1$, the identity above is an equality of absolutely convergent series. 
	\label{prop:othereigenrelation} 
\end{proposition}
Although this eigenrelation can be deduced from (5.13) in \cite{noumi2012direct}, we give another proof in the spirit of \cite[Proposition 2.2.17]{borodin2014macdonald}.
\begin{proof}
%Let us first establish \eqref{eq:othereigenrelation} as a formal power series in the variable $z$. 
	\textbf{Step 1:} Let us show that identity \eqref{eq:othereigenrelation} holds for a certain set of specializations of the $x_i$. For the moment we work with formal power series in the variable $z$. 
	Let $u_{\mu}$ be the specialization (of the ring of rational functions) that substitutes $q^{\mu_i}t^{n-i}$ in place of $x_i$ for all $i$. Let $\bar u_{\mu}$ be the specialization that substitutes $q^{-\mu_i}t^{i-n}$ in place of the variable $x_i$. The eigenvalue appearing in \eqref{eq:othereigenrelation} can be rewritten as
	\begin{equation}
	\prod_{i=1}^n\frac{(q^{-\la_i}t^{i}z)_{\infty}}{(q^{-\la_i}t^{i-1}z)_{\infty}} = \bar u_{\la}\Big( \Pi(zt^{n-1}, x)\Big) =  \bar u_{\la}\Big( \sum_{m\geqslant 0} g_m(x; q,t) z^m t^{(n-1)m}\Big),
	\label{eq:decompositioneigenvalue}
	\end{equation} 
	where $g_m$ is the $q,t$ analogue of complete homogeneous symmetric functions, that is $g_m=Q_{(m)}$. 
	
	We now show that the equality \eqref{eq:othereigenrelation} holds under specialization $\bar u_{\mu}$ for any $\mu$.  Macdonald symmetric polynomials satisfy an index-variable duality relation 
	$$ u_{\mu}(P_{\la}) = \frac{u_0(P_{\la})}{u_0(P_{\mu})}u_{\la}(P_{\mu}).$$
	Since it is true for any $q,t$, we also have that 
	$$\bar u_{\mu}(P_{\la}(x; q^{-1}, t^{-1}))\bar u_0(P_{\mu}(x; q^{-1}, t^{-1})) = \bar u_0(P_{\la}(x; q^{-1}, t^{-1}))\bar u_{\la}(P_{\mu}(x; q^{-1}, t^{-1}))$$
	and since $P_{\la}(x; q^{-1}, t^{-1})=P_{\la}(x; q, t)$ (cf. \cite[p.324]{macdonald1995symmetric}), we have another index-variable duality involving the specialization $\bar{u}$, which reads
	$$ \bar u_{\mu}(P_{\la}) = \frac{\bar u_0(P_{\la})}{\bar u_0(P_{\mu})}\bar u_{\la}(P_{\mu}).$$
	Moreover, using the combinatorial formula for Macdonald polynomials $P$ (see \eqref{eq:combinatorialP}), 
	$$ \frac{\bar u_0(P_{\la})}{\bar u_0(P_{\mu})} = t^{(1-n)\vert\la/\mu\vert} \frac{u_0(P_{\la})}{ u_0(P_{\mu})}.$$
	As in the proof of Proposition 2.2.17 in \cite{borodin2014macdonald}, we can specialize the Pieri rule (for $P$) as 
	\begin{equation}
	\bar{u}_{\la}(g_m(x; q,t))\bar u_{\mu}(P_{\la}) = \sum_{\nu \succ \mu\ :\ \vert \nu/\mu\vert =m } \frac{\bar u_0(P_{\nu})}{\bar  u_0(P_{\mu})}\varphi_{\nu/\mu} \bar u_{\nu}(P_{\la}).
	\label{eq:Pierispecialized}
	\end{equation}
	We need to show that the application of $\bar{u}_{\mu}$ to the R.H.S of \eqref{eq:othereigenrelation} equals the application of $\bar{u}_{\mu}$ to the L.H.S. Using \eqref{eq:decompositioneigenvalue} and  \eqref{eq:Pierispecialized}, applying $\bar{u}_{\mu}$ to the R.H.S of \eqref{eq:othereigenrelation} yields
	$$ \sum_{\nu \succ \mu } z^{\vert \nu/\mu\vert }t^{(n-1)\vert \nu/\mu\vert}\frac{\bar u_0(P_{\nu})}{\bar  u_0(P_{\mu})}\varphi_{\nu/\mu} \bar u_{\nu}(P_{\la}).$$
	On the other hand, the application of $\bar{u}_{\mu}$ to the L.H.S of \eqref{eq:othereigenrelation} is
	$$ \bar{u}_{\mu}\left( \sum_{\eta} z^{\vert \eta \vert} h_{\eta}(x_1^{-1}, \dots, x_n^{-1};q,t) \big(T_{q^{-1}, x_i}\big)^{\eta_i} P_{\lambda}(x) \right).$$
	
	If $\eta_i=\nu_i-\mu_i$ for all $i$, then  $ \bar{u}_{\mu}\left(  \big(T_{q^{-1}, x_i}\big)^{\eta_i} P_{\lambda} \right) = \bar u_{\nu}(P_{\la}) $. So, it is enough to show that for $\eta_i=\nu_i-\mu_i$, 
	$$ \bar u_{\mu}\Big(  h_{\eta}(x_1^{-1}, \dots, x_n^{-1};q,t) \Big)  = \frac{\bar u_0(P_{\nu})}{\bar  u_0(P_{\mu})}\varphi_{\nu/\mu}t^{(n-1)\vert \nu/\mu\vert}.$$
	This holds true since it is shown in the proof of Proposition 2.2.17 in \cite{borodin2014macdonald} that 
	$$  u_{\mu}\Big(  h_{\eta}(x_1, \dots, x_n;q,t) \Big)  = \frac{ u_0(P_{\nu})}{ u_0(P_{\mu})}\varphi_{\nu/\mu}.$$
	
		\textbf{Step 2:} Let us extend the result to any $x$. Let us denote by $[z^m]\ \Moumi_n^z$ the operator corresponding to the $z^m$ coefficient, that is 
 $$ [z^m]\ \Moumi_n^z = \sum_{\eta_1, \dots, \eta_n\geqslant 0, \vert \eta\vert =m}  h_{\eta}(x_1^{-1}, \dots, x_n^{-1};q,t) \prod_{i=1}^n \big(\Tshift_{q^{-1}, x_i}\big)^{\eta_i}.$$
We have shown that for any partition $\mu$, 
$$  \bar u_{\mu} \left( [z^m]\ \Moumi_n^z P_{\la}(x) \right)  =  g_m(q^{-\lambda_1}t^{0}, \dots, q^{-\lambda_n}t^{n-1}; q,t)  \bar u_{\mu} \left(P_{\la}(x)\right).$$
By definition $[z^m]\ \Moumi_n^z P_{\la}(x)$ is a finite sum of rational functions, so it is a rational function. Let us fix  $x_2= q^{-\mu_2}t^{2-n}, \dots, x_n= q^{-\mu_n}t^{0}$ for some integers $\mu_2\geqslant  \dots \geqslant \mu_n$. The following equality of rational functions in the variable $X$,
\begin{equation}
[z^m]\ \Moumi_n^z P_{\la}(X, q^{-\mu_2}t^{2-n}, \dots, q^{-\mu_n}t^{0})  =  g_m(q^{-\lambda_1}t^{0}, \dots, q^{-\lambda_n}t^{n-1}; q,t)   P_{\la}(X, q^{-\mu_2}t^{2-n},\dots , q^{-\mu_n}t^{0} )
\label{eq:rationalfunctionequality}
\end{equation}
is satisfied for any $X$ of the form $q^{-\mu_1}t^{1-n}$ where $\mu_1$ is a nonnegative integer such that $\mu_1\geqslant \mu_2$, hence \eqref{eq:rationalfunctionequality} is true as an equality between rational functions in the variable $X$. We may iterate this procedure  for each variable $x_2, \dots, x_n$ in this order,  and we obtain that 
$$  [z^m]\ \Moumi_n^z P_{\la}(x)  =  g_m(q^{-\lambda_1}t^{0}, \dots, q^{-\lambda_n}t^{n-1}; q,t)  P_{\la}(x).$$
holds as an identity of rational functions in the variables $x_1, \dots, x_n$. Thus, we have established \eqref{eq:decompositioneigenvalue} as a formal power series in the variable $z$. 

\textbf{Step 3:} When $z$ is such that for all $i=1, \dots n$, $\vert zq^{-\lambda_i}t^{i-1}\vert <1$, the expansion \eqref{eq:Pierispecialized} is absolutely convergent. Alternatively, this can also be seen by expanding $\prod_{i=1}^n\frac{(q^{-\la_i}t^{i}z)_{\infty}}{(q^{-\la_i}t^{i-1}z)_{\infty}}$ using the $q$-binomial theorem \eqref{eq:qbinomial}. Since we have already established \eqref{eq:decompositioneigenvalue} as a formal power series, the identity holds as a numeric identity for any  $z$ such that for all $i=1, \dots n$, $\vert zq^{-\lambda_i}t^{i-1}\vert <1$.
\end{proof}

\subsection{$(q,t)$-Laplace transforms}
\label{sec:generalLaplace}
In this section we consider the half-space Macdonald measure with specializations $\rhoup = (a_1, \dots, a_n)$ for $a_i\in (0,1)$ and $\rhodiag = \rho$ where $\rho$ is a Macdonald non-negative specialization of the form $\rho=\rho(\alpha, \beta, \gamma)$ as in Section \ref{sec:specializations}. We also assume that  $\Pi\big( \vec a;\rho\big)\Phi(\vec{a})$ can be expanded via the generalized Littlewood identity \eqref{eq:CauchyLittlewoodspecialized}. The observable of the half-space Macdonald measure appearing in Theorem \ref{theo:NoumiLaplace} can be understood as a $q,t$ analogue of a Laplace transform formula, coming from the properties of the Noumi operator $\Noumi_n^z$. Actually, in the limit $t=0$, $q\to 1$, this becomes exactly a Laplace transform of the rescaled smallest coordinate of the half-space Macdonald random partition (see Section \ref{sec:rigorousconvwithplancherel}). 
\begin{definition}
	For $r\in \R$, we define the contour $\mathcal{D}_r$ to be the vertical line  $r+\I\R$, oriented from bottom to top. 
\end{definition}
\begin{theorem}
	Let $z\in \C\setminus \R_{>0}$. Assume that:
	\begin{enumerate}
		\item[(i)] The parameters  $ a_1, \dots, a_n\in(0,1) $ are chosen such that for all $i,j$,  $\vert t a_i/a_j\vert <1$ and $\vert q a_i/a_j\vert <1$.  
		\item[(ii)] $R\in (0,1)$ is chosen such that $0<q^R<a_i/a_j$ for all $i,j$. 
	\end{enumerate}
	Then we have 
	\begin{multline}
	\EPMM_{(a_1, \dots, a_n), \rho}\left[ \prod_{i=1}^n \frac{(q^{\la_i}t^{n-i+1}z)_{\infty}}{(q^{\la_i}t^{n-i}z)_{\infty}} \right]= \sum_{k=0}^{n} \frac{1}{k!} \int_{\mathcal{D}_{R}}\frac{\mathrm{d}s_1}{2\I\pi} \dots \int_{\mathcal{D}_{R}} \frac{\mathrm{d}s_k}{2\I\pi}\  
\oint  \frac{\mathrm{d}w_1}{2\I\pi}\dots \oint \frac{\mathrm{d}w_k}{2\I\pi}  \mathcal{A}^{q,t}_{\vec s}(\vec w) \\ \times	\prod_{i=1}^k\Gamma(-s_i)\Gamma(1+s_i)  \prod_{i=1}^k \frac{\mathcal{G}^{q,t}(w_i)}{\mathcal{G}^{q,t}(q^{s_i}w_i)}\frac{\phi(w_i^2)(-z)^{s_i}}{\phi(q^{s_i}w_i^2) (q^{s_i}-1)w_i}, 
	 \label{eq:expansionNoumi}
	\end{multline}
	where the integration contours for the variables $w_i$ enclose all the $a_i$ and no other singularity; and we have used the shorthand notations
	\begin{equation}
	\mathcal{A}^{q,t}_{\vec s}(\vec w) := \prod_{1\leqslant i<j\leqslant k} \frac{ (q^{s_j}w_j-q^{s_i}w_i)(w_i-w_j)\phi(q^{s_i+s_j}w_i w_j)\phi(w_iw_j)}{(q^{s_i}w_i-w_j)(q^{s_j}w_j-w_i)\phi(q^{s_i}w_iw_j)\phi(q^{s_j}w_jw_i)}
	\label{eq:defA}
	\end{equation}  
	and 
	\begin{equation}
	\mathcal{G}^{q,t}(w) = \prod_{j=1}^n\frac{\phi(w/a_j)}{\phi(wa_j)} \frac{1}{ \Pi(w; \rho)}.
	\label{eq:defH}
	\end{equation}  
	\label{theo:NoumiLaplace}
\end{theorem}
This theorem is the half-space analogue of \cite[Theorem 4.8]{borodin2016observables} which deals with full-space Macdonald processes. 
\begin{remark}
	We can use  variables $(u_1, \dots, u_{2k}) := (w_1^{-1}, q^{\nu_1}w_1, \dots , w_k^{-1}, q^{\nu_k}w_k)$, so that 
	\begin{align}
	\mathcal{A}^{q,t}_{\nu}(\vec w) &= \prod_{1\leqslant i<j\leqslant 2k} \frac{u_j-u_i}{1-u_iu_j} \prod_{1\leqslant i<j\leqslant k} \frac{ f(q^{\nu_i+\nu_j}w_i w_j)f(w_iw_j)}{f(q^{\nu_i}w_iw_j)f(q^{\nu_j}w_jw_i)}\prod_{i=1}^k\frac{1-q^{\nu_i}}{q^{\nu_i} w_i-w_i^{-1}}\\
	&= \Pf\left[ \frac{u_j-u_i}{1-u_iu_j}\right] \prod_{1\leqslant i<j\leqslant k} \frac{ f(q^{\nu_i+\nu_j}w_i w_j)f(w_iw_j)}{f(q^{\nu_i}w_iw_j)f(q^{\nu_j}w_jw_i)}\prod_{i=1}^k \frac{1-q^{\nu_i}}{q^{\nu_i} w_i-w_i^{-1}}
	\end{align}
	where $\phi(u)=(tu)_{\infty}/(u)_{\infty}$ and $f(u)=(tu)_{\infty}/(qu)_{\infty}$ as before, and we have used Schur's Pfaffian identity \eqref{eq:SchurPfaffian}. If $q=t$, the function $f$ is constant, and the whole integrand in \eqref{eq:actionNouminu} can be written as a Pfaffian. This is coherent with the fact that the Pfaffian Schur process determines a Pfaffian point process -- see \cite{borodin2005eynard, ghosal2017correlation}. 
\end{remark} 
\begin{proof}[Proof of Theorem \ref{theo:NoumiLaplace}]
	For $\vert z\vert <1$, the result follows from the combination of Proposition \ref{prop:NoumiLaplace}, Proposition \ref{prop:actionNoumi} and Lemma \ref{lem:MellinBarnes} below. 
	
	Once the result is established for $\vert z\vert <1$, one can analytically continue to any $z\in \C\setminus \R_{>0}$ (the reason why both sides are analytic is similar to the proof of \cite[Theorem 3.2.11]{borodin2014macdonald}).  
\begin{proposition}
	As a formal series in $z$, and as a numeric equality for $\vert z \vert <1$, we have 
	\begin{align}
	\EPMM_{(a_1, \dots, a_n), \rho}\left[ \prod_{i=1}^n \frac{(q^{\la_i}t^{n-i+1}z)_{\infty}}{(q^{\la_i}t^{n-i}z)_{\infty}} \right]&=\frac{\Noumi^z_{n}\Pi(x_1, \dots, x_n; \rho)\Phi(x_1, \dots, x_n) }{\Pi(x_1, \dots, x_n; \rho)\Phi(x_1, \dots, x_n)} \Bigg\vert_{x_1=a_1, \dots, x_n=a_n} \label{eq:Laplacelan}
	\end{align}
	where $\la$ is distributed according to the half-space  Macdonald measure with specializations 
	$(a_1, \dots, a_n)$ and $\rho$, and $\Noumi^z_{n}$ acts on $(x_1, \dots, x_n)\mapsto \Pi(x_1, \dots, x_n; \rho)\Phi(x_1, \dots, x_n)$. 
	\label{prop:NoumiLaplace}
\end{proposition}
\begin{proof}
	We have 
	$$  \prod_{i=1}^n \frac{(q^{\la_i}t^{n-i+1}z)_{\infty}}{(q^{\la_i}t^{n-i}z)_{\infty}}  = \sum_{k=0}^{+\infty} z^k g_{k}(q^{\la_1}t^{n-1}, \dots, q^{\la_n}),$$
	where we recall that $g_k = Q_{(k)}$. 
	Thus the formal series on the L.H.S. of \eqref{eq:Laplacelan} is by definition such that the coefficient of $z^k$ is  $\EPMM[g_k(q^{\la_1}t^{n-1}, \dots, q^{\la_n})]$. Since $g_k(q^{\la_1}t^{n-1}, \dots, q^{\la_n})$ is bounded by $g_k(1, \dots, 1)$, it is absolutely integrable with respect to the half-space Macdonald measure. One obtains \eqref{eq:Laplacelan} by multiplying both sides of Noumi's eigenrelation \eqref{eq:Noumieigenrelation} by $\ve_{\la}(\rho)$ and summing over $\la$. Moreover, the series expansion is absolutely convergent when $\vert z\vert <1$.
\end{proof}
Let us now examine the quantity in the R.H.S. of \eqref{eq:Laplacelan}.
\begin{proposition}
	Consider a formal variable $z$, a function $H(u_1, \dots, u_n)= \prod_{i=1}^{n} h(u_i)$ with  $h$ being a meromorphic function whose poles are away from  the $a_i$'s and $\Phi(u_1,\ldots, u_n) = \prod_{i<j} \phi(u_iu_j)$ with $\phi(u) = (tu)_{\infty}/(u)_{\infty}$. Then
	\begin{multline}\frac{\Noumi^z_{n}H\Phi(x_1, \dots, x_n) }{H\Phi(x_1, \dots, x_n)} \Bigg\vert_{x_1=a_1, \dots, x_n=a_n}    \\ = \sum_{k=0}^{n} \frac{1}{k!} \sum_{\nu_1, \dots, \nu_k=1}^{\infty} \oint\frac{\mathrm{d}w_1}{2\I\pi} \dots \oint\frac{\mathrm{d}w_k}{2\I\pi}  \mathcal{A}^{q,t}_{\nu}(\vec w) \prod_{i=1}^k \frac{\mathcal{H}^{q,t}_n(w_i)}{\mathcal{H}^{q,t}_n(q^{\nu_i}w_i)}\frac{\phi(w_i^2)z^{\nu_i}}{\phi(q^{\nu_i}w_i^2) (q^{\nu_i}-1)w_i}, 
	\label{eq:Noumiexpansiondiscrete}
	\end{multline}
	where the (positively oriented) integration contours enclose all the $a_i$'s and no other singularity, and 
	$$  
	\mathcal{H}^{q,t}(w) = \prod_{j=1}^n\frac{\phi(w/a_j)}{\phi(wa_j)} \frac{1}{ h(w)}. $$
	\label{prop:actionNoumi}
\end{proposition}
\begin{proof}
	Let us denote 
	$$\Noumi_{n, \eta}^z = z^{\vert \eta\vert }\prod_{i<j} \frac{q^{\eta_j}x_j - q^{\eta_i}x_i}{x_j-x_i} \prod_{i,j} \frac{(tx_i/x_j)_{\eta_i}}{(qx_i/x_j)_{\eta_i}} \prod_{i=1}^n \big(\Tshift_{q, x_i}\big)^{\eta_i}.$$
	Observe that 
	\begin{equation}
	\Noumi^z_n = \sum_{k=0}^n \sum_{\nu_1, \dots, \nu_k=1}^{\infty} \frac{n!}{(n-k)!k!}\frac{1}{n!} \sum_{\sigma\in \SS_n} \Noumi^u_{n, \sigma(\nu)},
	\label{eq:decompositionNoumi}
	\end{equation} 
	where in the R.H.S. $\nu=(\nu_1, \dots, \nu_k, 0\dots, 0)$ and $\sigma\in \SS_n$ acts by permuting  the coordinates of $\nu$. Hence the proof 
	\begin{lemma} Assume $\nu_1, \dots, \nu_k\geqslant 1$ and $\nu_{k+1} = \dots = \nu_n=0$. Under the hypotheses of Proposition \ref{prop:actionNoumi}, we have 
		\begin{multline}
		\frac{1}{(n-k)!} \sum_{\sigma \in \SS_n}  
		\frac{\Noumi^z_{n, \sigma(\nu)}H\Phi(x_1, \dots, x_n) }{H\Phi(x_1, \dots, x_n)} \Bigg\vert_{x_1=a_1, \dots, x_n=a_n} \\  = 
	\oint\frac{\mathrm{d}w_1}{2\I\pi} \dots \oint\frac{\mathrm{d}w_k}{2\I\pi} \prod_{1\leqslant i<j\leqslant k} \frac{ (q^{\nu_j}w_j-q^{\nu_i}w_i)(w_i-w_j)\phi(q^{\nu_i+\nu_j}w_i w_j)\phi(w_iw_j)}{(q^{\nu_i}w_i-w_j)(q^{\nu_j}w_j-w_i)\phi(q^{\nu_i}w_iw_j)\phi(q^{\nu_j}w_jw_i)} \\ \times 
		\prod_{i=1}^k\left(\prod_{j=1}^n\frac{\phi(w_i/a_j)\phi(q^{\nu_i }w_i a_j)}{\phi(q^{\nu_i} w_i/a_j)\phi(w_i a_j)}\right) \frac{\phi(w_i^2)h(q^{\nu_1}w_i)z^{\nu_i}}{\phi(q^{\nu_i}w_i^2)h(w_i) (q^{\nu_i}-1)w_i},
		\label{eq:actionNouminu}
		\end{multline}
		where the contours are small positively oriented circles enclosing the $a_i$'s and no other singularity. 
		\label{lem:actionNouminu}
	\end{lemma}
	\begin{proof}
		The proof is modelled after Lemma 4.12 in \cite{borodin2016observables}. We have 
		$$  \prod_{i=1}^n\left(T_{q, x_i} \right)^{\nu_i}\ H\Phi(\vec x) = \prod_{i=1}^n h(q^{\nu_i}x_i)\prod_{i<j}\phi(q^{\nu_i+\nu_j}x_ix_j).$$
		The L.H.S in \eqref{eq:actionNouminu} equals 
		$$ \frac{1}{(n-k)!} \sum_{\sigma \in \SS_n} z^{\vert \nu\vert} \prod_{i<j} \frac{q^{\nu_j}a_{\sigma(j)}-q^{\nu_i}a_{\sigma(i)}}{a_{\sigma(j)}-a_{\sigma(i)}} \prod_{i,j=1}^n \frac{(ta_{\sigma(i)}/a_{\sigma(j)})_{\nu_i}}{(qa_{\sigma(i)}/a_{\sigma(j)})_{\nu_i}}  \prod_{i=1}^n \frac{h(q^{\nu_i}a_{\sigma(i)})}{h(a_{\sigma(i)})}\prod_{i<j}\frac{\phi(q^{\nu_i+\nu_j}a_{\sigma(i)}a_{\sigma(j)})}{\phi(a_{\sigma(i)}a_{\sigma(j)})},$$
		because permuting the $\nu_i$ is equivalent to permuting the $a_i$. Since $\nu_{k+1}=\dots=\nu_n=0$, the summand is invariant with respect to permutation of the $\lbrace \nu_i\rbrace_{i>k}$. Hence one can absorb the factor $1/(n-k)!$ and sum on permutations in 
		$\SS_n^k := \left\lbrace \sigma\in \SS_n \ : \  \sigma(1+k)<\dots <\sigma(n)\right\rbrace.$ 
		Thus, L.H.S in \eqref{eq:actionNouminu} equals 
		\begin{multline}\sum_{\sigma\in\SS_n^k} \underset{\underset{\forall 1\leqslant i\leqslant k}{w_i=a_{\sigma(i)}}}{\subs} \left\lbrace 
		\prod_{i<j} \frac{q^{\nu_j}w_j-q^{\nu_i}w_i}{w_j-w_i} \prod_{i=1}^k \prod_{j=k+1}^n \frac{a_{\sigma(j)}-q^{\nu_i}w_i}{a_{\sigma(j)}-w_i}
		\prod_{i=1}^k\prod_{j=1}^n \frac{(tw_i/a_j)_{\nu_i}}{(qw_i/a_j)_{\nu_i}}  \prod_{i=1}^n \frac{h(q^{\nu_i}w_i)}{h(w_i)}\right.\\ \left. \times 
		\prod_{1\leqslant i<j\leqslant k }\frac{\phi(q^{\nu_i+\nu_j}w_iw_j)}{\phi(w_iw_j)}
		\prod_{i=1}^k \prod_{j=1}^n \frac{\phi(q^{\nu_i}w_ia_j)}{\phi(w_ia_j)}
		\prod_{i=1}^k\prod_{j=1}^k \frac{\phi(w_iw_j)}{\phi(q^{\nu_i}w_iw_j)}
		\right\rbrace.
		\label{eq:sumresidues}
		\end{multline}
		Notice that we have  
		$$ \underset{\underset{\forall 1\leqslant i\leqslant k}{w_i=a_{\sigma(i)}}}{\subs} \left\lbrace    \prod_{i=1}^k \prod_{j=k+1}^n \frac{a_{\sigma(j)}-q^{\nu_i}w_i}{a_{\sigma(j)}-w_i} \right\rbrace  = \Res{\underset{\forall 1\leqslant i\leqslant k}{w_i=a_{\sigma(i)}}} \left\lbrace  \prod_{i=1}^k \prod_{j=1}^n \frac{q^{\nu_i}w_i-a_j}{w_i-a_j} \prod_{i,j=1}^k \frac{1}{q^{\nu_i}w_i-w_j}\prod_{1\leqslant i\neq j \leqslant k} (w_i-w_j)\right\rbrace,$$
		where the notation $\Res{\underset{\forall 1\leqslant i\leqslant k}{w_i=x_i}} \left\lbrace \cdot  \right\rbrace  $ denotes the residue of the function inside brackets at $w_1=x_1, \dots , w_k=x_k$.  
		This  implies that 
		\begin{multline} \eqref{eq:sumresidues} = \oint\frac{\mathrm{d}w_1}{2\I\pi} \dots \oint\frac{\mathrm{d}w_k}{2\I\pi} \prod_{i<j} \frac{q^{\nu_j}w_j-q^{\nu_i}w_i}{w_j-w_i} \prod_{i=1}^k \prod_{j=1}^n \frac{q^{\nu_i}w_i-a_j}{w_i-a_j} \prod_{i,j=1}^k \frac{1}{q^{\nu_i}w_i-w_j}\prod_{i\neq j=1}^k (w_i-w_j)\\ \times \prod_{i=1}^k\prod_{j=1}^n \frac{(tw_i/a_j)_{\nu_i}}{(qw_i/a_j)_{\nu_i}}  \prod_{i=1}^n \frac{h(q^{\nu_i}w_i)}{h(w_i)} 
		\prod_{1\leqslant i<j\leqslant k }\frac{\phi(q^{\nu_i+\nu_j}w_iw_j)}{\phi(w_iw_j)}
		\prod_{i=1}^k \prod_{j=1}^n \frac{\phi(q^{\nu_i}w_ia_j)}{\phi(w_ia_j)}
		\prod_{i=1}^k\prod_{j=1}^k \frac{\phi(w_iw_j)}{\phi(q^{\nu_i}w_iw_j)},
		\end{multline}
		where the contours enclose the $a_i$'s and no other singularity. Indeed, the integral of the right-hand side can be evaluated by summing all residues at $(w_i)_{1\leqslant i\leqslant k}=(a_{p(i)})_{1\leqslant i\leqslant k}$, for all functions $p:\lbrace 1, \dots, k \rbrace\to \lbrace 1, \dots, n \rbrace$, and it is clear that the residue is zero when $p$ is not injective because of the product of $(w_i-w_j)$. Hence, the integral is a sum of residues at $w_i=a_{\sigma(i)}$ over all $\sigma\in \SS_n^k$. Finally,  one has 
		$$  \frac{q^{\nu_i}w_i-a_j}{w_i-a_j} \frac{(tw_i/a_j)_{\nu_i}}{(qw_i/a_j)_{\nu_i}} = \frac{\phi(w_i/a_j)}{\phi(q^{\nu_i}w_i/a_j)},$$
		so that the integrand can be arranged to match with the R.H.S of \eqref{eq:actionNouminu}. 
	\end{proof}
	Let us finish the proof of Proposition \ref{prop:actionNoumi}.  Recalling the notations in \eqref{eq:defA} and \eqref{eq:defH}, the application of  Lemma \ref{lem:actionNouminu} in each term of the sum in \eqref{eq:decompositionNoumi} yields the desired result. Furthermore, if $\vert z\vert <1$, all sums are absolutely convergent so that \eqref{eq:Noumiexpansiondiscrete} holds as a numeric equality. 
\end{proof}
The expansion on the right of \eqref{eq:Noumiexpansiondiscrete} is not exactly as in Theorem \ref{theo:NoumiLaplace}. In order to replace the discrete sums over the $\nu_i$ in Proposition \ref{prop:actionNoumi} by contour integrals over vertical contours, we use the following Lemma. 
\begin{lemma}
Let $a\in (0,1)$ and  $h$ be a  holomorphic function on $ \lbrace z\in \C : \Real[z]\geqslant a \rbrace $, such that there exist a constant $C>0$ with 
$$ \vert h(z) \vert \leqslant C/\vert z\vert^2,  $$
for any $z$ in $ \lbrace z\in \C : \Real[z]\geqslant a \rbrace $. 
 Then, for $\zeta\in \C\setminus\R_{>0}$ we have 
$$ \sum_{k=1}^{\infty}  h(k) \zeta^k = \int_{\mathcal{D}_{a}}\frac{\mathrm{d}s}{2\I\pi}\ (-\zeta)^s\Gamma(-s)\Gamma(1+s) h(s),$$
whenever both sides are absolutely convergent (recall that $\mathcal{D}_a$ is oriented upwards). 
\label{lem:MellinBarnes}
\end{lemma}
\begin{proof}
	It is clear by the residue theorem that 
	$$ \sum_{k=1}^{\infty}  h(k) \zeta^k = \int \frac{\mathrm{d}s}{2\I\pi}\ (-\zeta)^s\Gamma(-s)\Gamma(1+s) h(s),$$
	where the contour is chosen so as to enclose all the positive integers and no other singularity, e.g.,  the union of small positively oriented  circles around each positive integer. Using the fact that $h$ is holomorphic, and using the decay bound, one can deform the integration contour to be  $\mathcal{D}_{a}$ without changing the value of the integral. 
\end{proof}
To conclude the proof of Theorem \ref{theo:NoumiLaplace}, we need to check that the hypotheses of Lemma \ref{lem:MellinBarnes} are satisfied for its $k$-fold application in the $k$th term of \eqref{eq:Noumiexpansiondiscrete} for all $k\geqslant 1$. In the R.H.S. of \eqref{eq:expansionNoumi}, the integrand in each of the variables $s_i$ can be represented as $\Gamma(-s_i)\Gamma(1+s_i)(-\zeta)^{s_i}g(q^{s_i})$, where the function $g$ is an analytic function with isolated singularities. One may inspect all of these singularities one by one and check that for a specialization $\rho=\rho(\alpha, \beta, \gamma)$ with $\alpha_i\in (0,1)$ and assuming that the variables $w_i$ are integrated along contours very close to the $a_i$, there are no singularities in $s_i$ lying on the right of $\mathcal{D}_{R}$.
% (this needs to be carefully checked, and we will see later an example of a similar formula where this is not the case, see e.g. Remark \ref{rem:presencepoles}). 
Then, since $q^s$ stays bounded for $s$ on the right of $\mathcal{D}_{R}$, $g(q^s)$ stays bounded as well. For $z$ such that $\vert z \vert <1$,  the integrand satisfies the decay bound of Lemma \ref{lem:MellinBarnes}. Thus, we have established Theorem \ref{theo:NoumiLaplace} for $\vert z \vert <1$ and it can be then analytically continued to $z \in \C\setminus\R_{>0}$. This completes the proof of Theorem \ref{theo:NoumiLaplace}. 
\end{proof}

We now seek to prove an analogous result using the operator $\Moumi_z^n$ instead of $\Noumi_z^n$. 
\begin{definition}
	We define for $z\in \C\setminus \R_{>0}$
	\begin{multline}
	E_R^z(a_1, \dots, a_n; \rho):=  \sum_{k=0}^{n} \ \frac{1}{k!}\ \int_{\mathcal{D}_R}\frac{\mathrm{d}s_1}{2\I\pi} \dots \int_{\mathcal{D}_R}\frac{\mathrm{d}s_k}{2\I\pi}\   \oint\frac{\mathrm{d}w_1}{2\I\pi} \dots \oint\frac{\mathrm{d}w_k}{2\I\pi}     \mathcal{B}_{\vec s}^{q,t}(\vec w) \\ \times
	\prod_{i=1}^k\Gamma(-s_i)\Gamma(1+s_i)  \frac{	\overline{\mathcal{G}}^{q,t}(w_i)}{	\overline{\mathcal{G}}^{q,t}(q^{-s_i}w_i)}  \frac{\phi(w_i^2) (-z)^{s_i}}{\phi(q^{-s_i}w_i^2) (1-q^{s_i})w_i} ,
	\label{eq:MoumiMellinBarnes}
	\end{multline}
	where the $w$ contours are small positively oriented  circles enclosing the $a_i$ and no other singularity, 
	\begin{equation*}
	\mathcal{B}_{\vec s}^{q,t}(\vec w): = \prod_{1\leqslant i<j\leqslant k} \frac{ (q^{s_j}w_i-q^{s_i}w_j)(w_j-w_i)\phi(q^{-s_i-s_j}w_i w_j)\phi(w_iw_j)}{(q^{s_i}w_j-w_i)(q^{s_j}w_i-w_j)\phi(q^{-s_i}w_iw_j)\phi(q^{-s_j}w_jw_i)},
	\end{equation*}
	and 
	\begin{equation}
	\overline{\mathcal{G}}^{q,t}(w) = \prod_{j=1}^n\frac{\phi(a_j/w)}{\phi(wa_j)}\frac{1}{\Pi(w; \rho)}.
	\label{eq:defG}
	\end{equation}  
	\label{def:Moumiexpansion}
\end{definition}
\begin{remark}
	\label{rem:presencepoles}
	Unlike the expansion \eqref{eq:expansionNoumi} occurring in Theorem \ref{theo:NoumiLaplace}, in \eqref{eq:MoumiMellinBarnes} the integrand has many poles in the variables $s_i$ lying on the right of the contour $\mathcal{D}_{R}$. This means that the integral cannot be turned into a discrete sum using Lemma \ref{lem:MellinBarnes}. 
\end{remark}

\begin{theorem}  Let $z\in \C\setminus \R_{>0}$. Assume that:
	\begin{enumerate}
		\item[(i)] The parameters  $ a_1, \dots, a_n\in(0,1) $ are chosen such that for all $i,j$, $\vert t a_i/a_j\vert <1$, $\vert q a_i/a_j\vert <1$ and $\max\lbrace a_i\rbrace < a_i/a_j$.
		\item[(ii)]  $R\in (0,1)$ is chosen so that for all $i,j$,  $a_i<q^R<a_i/a_j$. 
		\item[(iii)] The specialization $\rho=\rho(\alpha, \beta, \gamma)$ is such that for all $i,j$,  $q^R>a_i\alpha_j$. 
	\end{enumerate}
 Then, 
	\begin{equation}
	\EPMM_{(a_1, \dots, a_n), \rho}\left[ \prod_{i=1}^n \frac{(q^{-\la_i}t^{i}z)_{\infty}}{(q^{-\la_i}t^{i-1}z)_{\infty}} \right] =E_R^z(a_1, \dots, a_n; \rho). \label{eq:Laplacela1}
	\end{equation}
	\label{theo:MoumiLaplace}
\end{theorem}
The expectation in the L.H.S. of \eqref{eq:Laplacela1} should again be thought of as a Laplace transform. A full-space analogue of this result could be proved by adapting the proof that we present below to the case of the full-space Macdonald processes. Note that in the $q$-Whittaker case ($t=0$), a full-space analogue is already available \cite[Theorem 3.3]{borodin2015height}. 

\begin{remark}
It will be useful to note that for fixed $n\geqslant 1$ and $z\in \C\setminus \R_{>0}$ the observable  $\prod_{i=1}^n \frac{(q^{-\la_i}t^{i}z)_{\infty}}{(q^{-\la_i}t^{i-1}z)_{\infty}}$ is bounded as a function of $\lambda$. Indeed, as long as $t\in (0,1)$, the ratio  $\frac{(q^{-m}tx)_{\infty}}{(q^{-m}x)_{\infty}}$ goes to zero as $m$ goes to $+\infty$ for any fixed $x\in \C \setminus \R_{>0}$. \label{rem:observablebounded}
\end{remark}
\begin{remark}
Since the L.H.S in \eqref{eq:Laplacela1} is analytic in the parameters $a_1, \dots, a_n$ and the parameters of the specialization $\rho$, the formula can be extended to a range of parameters forbidden by the hypotheses of Theorem \ref{theo:MoumiLaplace}, at the expense of choosing more complicated contours for the variables $s_i$. 
\end{remark}
\begin{remark} 
	The proof of Theorem \ref{theo:MoumiLaplace} is substantially different -- and more difficult -- than that of Theorem \ref{theo:NoumiLaplace}. The L.H.S of \eqref{eq:Laplacela1} is by definition 
	$$ \frac{1}{\Pi(a; \rho)\Phi(a)} \sum_{\la} \ve_{\la}(\rho) P_{\la}(a) \prod_{i=1}^n \frac{(q^{-\la_i}t^{i}z)_{\infty}}{(q^{-\la_i}t^{i-1}z)_{\infty}}.$$
	It is tempting to replace the quantity $P_{\la}(a) \prod_{i=1}^n \frac{(q^{-\la_i}t^{i}z)_{\infty}}{(q^{-\la_i}t^{i-1}z)_{\infty}}$ above by 
	$\Moumi_n^z P_{\la}(a)$ using Proposition \ref{prop:othereigenrelation}, and exchange the action of the operator $\Moumi_n^z$ with the summation over $\la$, to arrive at an analogue of Proposition \ref{prop:NoumiLaplace}. But this exchange of summations is forbidden because Proposition \ref{prop:othereigenrelation} requires $\vert zq^{-\lambda_i}t^{i-1}\vert <1$, which can never be true simultaneously for all $\la$.

	A natural roundabout way would be to work with formal power series. It is reasonable to expect both sides of \eqref{eq:Laplacela1} to be analytic in the variables $a_1, \dots, a_n$. For a function $F$
	analytic and symmetric in the variables $(a_1, \dots, a_n)=:\vec a$, let us denote by $[P_{\lambda}(\vec a)]\left\lbrace F\right\rbrace$ the coefficient of $P_{\la}(\vec a)$ when the quantity $F$ is expanded in $a_1, \dots, a_n$ using the basis of  Macdonald polynomials $P_{\la}$. Similarly, for a formal power series $F$ in the variable $z$, we likewise  denote by  $[z^k]\left\lbrace F\right\rbrace$ the coefficient of $z^k$. 
	
	Then, we clearly have 
	$$ [z^k]\left\lbrace [P_{\la}(\vec a)] \left\lbrace   \EPMM_{(a_1, \dots, a_n), \rho }\left[ \prod_{i=1}^n \frac{(q^{-\la_i}t^{i}z)_{\infty}}{(q^{-\la_i}t^{i-1}z)_{\infty}} \right]  \right\rbrace \right\rbrace  = \frac{\ve_{\la}(\rho) g_k(q^{-\la_1}t^0, \dots,q^{-\la_n}t^{n-1})}{ \Pi(\vec a; \rho)\Phi(\vec a)}.$$
	On the other hand, 
	$$ [P_{\la}(\vec a)] \left\lbrace [z^k] \Big\lbrace   \Moumi_n^z  \Pi(\vec a; \rho)\Phi(\vec a) \Big\rbrace \right\rbrace  = \ve_{\la}(\rho) g_k(q^{-\la_1}t^0, \dots,q^{-\la_n}t^{n-1}).$$
	Unlike the proof of Proposition \ref{prop:NoumiLaplace}, it turns out to be  impossible to justify the exchange of  summations over $k$ and $\lambda$, and we actually have that in general 
	$$ \EPMM_{(a_1, \dots, a_n), \rho}\left[ \prod_{i=1}^n \frac{(q^{-\la_i}t^{i}z)_{\infty}}{(q^{-\la_i}t^{i-1}z)_{\infty}} \right]  \neq \frac{ \Moumi_n^z  \Pi(\vec a; \rho)\Phi(\vec a)}{\Pi(\vec a; \rho)\Phi(\vec a)},$$
even for $z$ and $\vec a$ very close to $0$	(this can be checked explicitly when, e.g., $n=1$ and $\rho$ is a single variable pure alpha specialization, see also Remark \ref{rem:subtleties}). To resolve this issue, we will construct another operator $\MoumiA_n^z$, such that 
	$$ [z^k] \left\lbrace  [P_{\la}(\vec a)] \Big\lbrace   \MoumiA_n^z  \Pi(\vec a; \rho)\Phi(\vec a) \Big\rbrace \right\rbrace  = \ve_{\la}(\rho) g_k(q^{-\la_1}t^0, \dots,q^{-\la_n}t^{n-1}).$$
	The definition of  $\MoumiA_n^z$ will be suggested by the analytic continuation of \eqref{eq:othereigenrelation} to $z\in \C\setminus \R_{>0}$. 
	\label{rem:whymoredifficult}
\end{remark} 

\begin{remark}
There is another approach to obtain formulas for the observable appearing in \eqref{eq:Laplacela1}. We thank an anonymous referee for this suggestion. Acting on the generalized Littlewood identity with the operator $\Noumi^z_n \prod_{i=1}^n \Tshift_{u,x_i}$ instead of $\Noumi^z_n$ as in the proof of Theorem \ref{theo:NoumiLaplace}, we can compute the observable 
\begin{equation}
\mathbb E\left[  \prod_{i=1}^n u^{\lambda_i}\frac{(q^{\lambda_i}t^{n-i+1}z)_{\infty}}{(q^{\lambda_i}t^{n-i}z)_{\infty}}   \right].
\label{eq:moregeneralobservable}
\end{equation} 
Replacing $(z,u)$ by $(\frac{q}{t^n z}, tu)$, we obtain, after appropriate renormalization, the observable  $$ \mathbb E\left[  \prod_{i=1}^n u^{\lambda_i}\frac{(q^{-\lambda_i}t^{i}z)_{\infty}}{(q^{-\lambda_i}t^{i-1}z)_{\infty}}   \right].$$ 
This approach has the advantage of avoiding analytic continuations and the introduction of the operator $\MoumiA_n^z$. However, the formula produced by this method is different from the result of Theorem \ref{theo:MoumiLaplace}, and it is not clear how to show the equivalence of both formulas. Because of the substitution $(z,u)\to (\frac{q}{t^n z}, tu)$, the formula obtained appears singular at $t=0$. It would be interesting to manipulate the formula so as to remove all the singularities at $t=0$. We leave this for future consideration.  
\end{remark}

\subsection{Proof of Theorem \ref{theo:MoumiLaplace}}
\label{sec:proofmaintheo}
\begin{definition}
Let $\disk$ be the open unit disk $\disk := \lbrace z\in \C: \vert z\vert <1\rbrace$ and  $\spaceanalytic$ be the space of analytic  symmetric functions\footnote{Here we do not mean elements of $\Sym$ but analytic functions in $n$ variables that are symmetric in these variables.} $f(x_1, \dots, x_n)$ on $\mathbb{D}^n$.  Such a function $f$ admits an absolutely convergent expansion in Macdonald symmetric polynomials on $\disk^n$. More precisely, for all $x=( x_1, \dots, x_n) \in \disk^n$,
	$$ f(x_1, \dots, x_n)  = \sum_{\la} c_{\la}(f)  P_{\la}(x), \text{ with } \sum_{\la} \big\vert c_{\la}(f)  P_{\la}(x)\big\vert <\infty .$$
\end{definition}
\begin{definition}
	We define an operator $\MoumiA_n^z : \spaceanalytic \to \spaceanalytic$ by 
	\begin{multline}
	\MoumiA_n^z f(x_1, \dots, x_n) =   \int_{\mathcal{D}_{-\e}} \frac{\mathrm{d}s_1}{2\I\pi}\dots \int_{\mathcal{D}_{-\e}}\frac{\mathrm{d}s_n}{2\I\pi} (-z)^{s_1 + \dots + s_n }\prod_{i<j} \frac{q^{s_j}x_i - q^{s_i}x_j}{x_i-x_j} \\ \times  \prod_{i,j} \frac{(tx_i/x_j)_{\infty}}{(qx_i/x_j)_{\infty}}\frac{(q^{s_j+1}x_i/x_j)_{\infty}}{(tq^{s_j}x_i/x_j)_{\infty}} \prod_{i=1}^n \Gamma(-s_i)\Gamma(1+s_i) f(q^{-s_1}x_1, \dots, q^{-s_n}x_n), 
	\end{multline}
	where $\e(x_1, \dots, x_n)$ is chosen small enough so that there are no poles of the integrand with real part  between $-\e$ and $0$.  
\end{definition}
It is not yet clear why $\MoumiA_n^z f$ must belong to $\spaceanalytic$. The next Proposition can be interpreted as an analytic continuation of Proposition \ref{prop:othereigenrelation}. It implies, in particular, that  $\MoumiA_n^z$ preserves $\spaceanalytic$.
\begin{proposition} For $z\in \C \setminus \R_{>0}$ and $ x_1, \dots, x_n\in\disk $ such that for all $i,j$, $\vert t x_i/x_j\vert <1$, 
	\begin{equation}
\MoumiA_n^z P_{\la}(x_1, \dots, x_n)  = \prod_{i=1}^n\frac{(q^{-\la_i}t^{i}z)_{\infty}}{(q^{-\la_i}t^{i-1}z)_{\infty}}P_{\la}(x). 
\label{eq:analyticeigenrelation}
	\end{equation}
	\label{prop:analyticeigenrelation}
\end{proposition}
\begin{proof}
	When $z$ is such that for all $i=1, \dots n$, $\vert zq^{-\lambda_i}t^{i-1}\vert <1$  and $ x_1, \dots, x_n\in\disk$, \eqref{eq:analyticeigenrelation} is a reformulation of Proposition \ref{prop:othereigenrelation} using Mellin-Barnes integrals via Lemma \ref{lem:MellinBarnes} (The condition $\vert t x_i/x_j\vert <1$ ensures that there are no unwanted singularities to the right of the contour $\mathcal{D}_{-\e}$). 
	
	Both sides of \eqref{eq:analyticeigenrelation} are analytic in $z$, so that  one can analytically extend the identity to any $z\in \C \setminus \R_{>0}$.  Analyticity is clear for the right-hand-side. For the left-hand-side, it follows from the exponential decay of $ \Gamma(-s_i)\Gamma(1+s_i)$ on the contour $\mathcal{D}_{-\e}$ so that all integrals are absolutely convergent.
\end{proof}
\begin{proposition}
	Let $f \in \spaceanalytic$ with $ f(x_1, \dots, x_n)  = \sum_{\la} c_{\la}(f) P_{\la}(x)$. Then for  $ x_1, \dots, x_n\in(0,1) $ such  that for all $i,j$, $\vert t x_i/x_j\vert <1$, we have 
	$$  \MoumiA_n^z f(x_1, \dots, x_n)  = \sum_{\la} c_{\la}(f) P_{\la}(x) \prod_{i=1}^n\frac{(q^{-\la_i}t^{i}z)_{\infty}}{(q^{-\la_i}t^{i-1}z)_{\infty}} . $$
	\label{prop:generalactionO}
\end{proposition}
\begin{proof}
	The operator $\MoumiA_n^z$ is linear on $\spaceanalytic$, so that 
	$$  \MoumiA_n^z f(x_1, \dots, x_n) = \sum_{\la} c_{\la}(f) \MoumiA_n^z  P_{\la}(x),$$
	and one can apply Proposition \ref{prop:analyticeigenrelation} on each summand. When applying the operator $\MoumiA_n^z$, one has to choose the parameter $\e$ involved in the integration contour in \eqref{eq:analyticeigenrelation} in such a way that for all $i$, $\vert q^{-\e}x_i \vert <1$. 
\end{proof}
\begin{corollary} For  $ x_1, \dots, x_n\in(0,1) $ such  that for all $i,j$, $\vert t x_i/x_j\vert <1$,
\begin{equation}
	\frac{\MoumiA_n^z \left(\Pi(x, \rho) \Phi(x)\right)}{\Pi(x, \rho) \Phi(x)}  = \EPMM_{(a_1, \dots, a_n), \rho}\left[ \prod_{i=1}^n \frac{(q^{-\la_i}t^{i}z)_{\infty}}{(q^{-\la_i}t^{i-1}z)_{\infty}} \right].
	\label{eq:LaplacewithMoumiA}
\end{equation}  
	\label{cor:analyticoperatorandLaplace}
\end{corollary}
\begin{proof}
	This is a direct application of Proposition \ref{prop:generalactionO} with $f(x) = \Pi(x, \rho) \Phi(x)$.
\end{proof}
To conclude the proof of Theorem \ref{theo:MoumiLaplace}, we need to compute the L.H.S in \eqref{eq:LaplacewithMoumiA}.
\begin{proposition}	Under the hypotheses of Theorem \ref{theo:MoumiLaplace}, 
	$$\frac{\MoumiA_n^z \left(\Pi(x, \rho) \Phi(x)\right)}{\Pi(x, \rho) \Phi(x)}  =   E_R^z(a_1, \dots, a_n; \rho) . $$
	\label{prop:actionMoumiAonZ}
\end{proposition}
\begin{proof}By definition, 
\begin{multline}
\MoumiA_n^z \left(\Pi(x, \rho) \Phi(x)\right) = \int_{\mathcal{D}_{-\e}} \frac{\mathrm{d}s_1}{2\I\pi}\dots \int_{\mathcal{D}_{-\e}}\frac{\mathrm{d}s_n}{2\I\pi} (-z)^{s_1 + \dots + s_n }\prod_{i<j} \frac{q^{s_j}x_i - q^{s_i}x_j}{x_i-x_j} \\ \times   \prod_{i,j} \frac{(tx_i/x_j)_{\infty}}{(qx_i/x_j)_{\infty}}\frac{(q^{s_j+1}x_i/x_j)_{\infty}}{(tq^{s_j}x_i/x_j)_{\infty}} \prod_{i=1}^n \Gamma(-s_i)\Gamma(1+s_i)  \Pi(q^{-s_1}x_1, \dots, q^{-s_n}x_n, \rho) \Phi(q^{-s_1}x_1, \dots, q^{-s_n}x_n).
\label{eq:MoumiAonZ}
\end{multline}
The first step is to transform this $n$-fold integral into a sum of $k$-fold integrals for $k=1, \dots, n$ by taking the residues when some of the $s_i$ are zero. 
Let us denote the integrand in \eqref{eq:MoumiAonZ} by $I_{\vec s}^z$, where $\vec s=(s_1, \dots, s_n)$. Deforming the contours $\mathcal{D}_{-\e}$ to $\mathcal{D}_{\e}$ yields
	\begin{equation}
	\MoumiA_n^z \left(\Pi(x, \rho) \Phi(x)\right) = \sum_{k=0}^n  \int_{\mathcal{D}_{\e}} \frac{\mathrm{d}s_1}{2\I\pi}\dots \int_{\mathcal{D}_{\e}}\frac{\mathrm{d}s_k}{2\I\pi}  \frac{n!}{(n-k)!k!}  \frac{1}{n!} \sum_{\sigma\in \SS_n} I_{\sigma(\vec s)}^z,
	\label{eq:decompositionMoumiA}
	\end{equation} 
	where in the $k$-th summand, $\vec s = (s_1, \dots, s_k, 0, \dots, 0)$ and $\sigma$ acts by permuting the coordinates of $\vec s$. The only singularity that we cross from $\mathcal{D}_{-\e}$ to $\mathcal{D}_{\e}$ is  at $s=0$ (it comes from the factor $\Gamma(-s)$). 
	
        We further deform the integration contour from $\mathcal{D}_{\e}$ to $\mathcal{D}_{R}$. The next Lemma ensures that when the $x_i$'s satisfy the same hypotheses as the $a_i$ in the statement of Theorem \ref{theo:MoumiLaplace},  we do not cross any singularity. 
		\begin{lemma}
			Under the assumptions (i), (ii) and (iii) of Theorem \ref{theo:MoumiLaplace}, 
			the poles in variables $s_1, \dots, s_n$ in 
			\begin{equation} 
			\prod_{i,j}\frac{(q^{s_j+1}a_i/a_j)_{\infty}}{(tq^{s_j}a_i/a_j)_{\infty}} \Pi(q^{-s_1}a_1, \dots, q^{-s_n}a_n, \rho) \Phi(q^{-s_1}a_1, \dots, q^{-s_n}a_n)
			\label{eq:termwithpoles}
			\end{equation}
			do not lie between $\mathcal{D}_{0}$ and $\mathcal{D}_R$. 
		\end{lemma}
		\begin{proof}
When $\rho$ is of the form $\rho=\rho(\alpha, \beta, \gamma)$ as in Section \ref{sec:specializations}, the poles of $\Pi(q^{-s_j}a_j; \rho)$ in the variable $s_j$ are the same as those of 
			$$ \prod_{i}\frac{1}{(\alpha_i q^{-s_j}a_j)_{\infty}}.$$
			Hence, the poles in \eqref{eq:termwithpoles} correspond to the following cases: 
			\begin{enumerate}
				\item $q^{s_i+s_j} = q^k a_ia_j$ for some integer $k\geqslant 0$ and $1\leqslant i\neq j\leqslant n$.  
				\item $q^{s_j} = q^{-k} a_j  t^{-1}a_i^{-1}$ for some integer $k\geqslant 0$ and $1\leqslant i,j\leqslant n$.  
				\item $q^{s_j} = \alpha_i a_j q^k$ for some integer $k\geqslant 0$, $1\leqslant  j\leqslant n$ and any $i$. 
			\end{enumerate}
Since $q^R>a_i$ for all $1\leqslant  j\leqslant n$, the poles in the case (1) all lie to the right of $\mathcal{D}_R$. Since $\vert ta_i/a_j\vert <1$, we have $1< \vert a_j  t^{-1}a_i^{-1}\vert < t^{-2}$ and the poles in the case (2) all have negative real part. Since $q^R>a_j\alpha_i$ for $1\leqslant  j\leqslant n$ and any $i$, the poles in the case (3) all lie to the right of $\mathcal{D}_R$.
		\end{proof}
We have arrived at 
	\begin{equation}
	\frac{\MoumiA_n^z \left(\Pi(x, \rho) \Phi(x)\right)}{\Pi(x, \rho) \Phi(x)} = \sum_{k=0}^n  \int_{\mathcal{D}_{R}} \frac{\mathrm{d}s_1}{2\I\pi}\dots \int_{\mathcal{D}_{R}}\frac{\mathrm{d}s_k}{2\I\pi}  \frac{1}{k!} \frac{1}{(n-k)!}   \sum_{\sigma\in \SS_n} \frac{ I_{\sigma(\vec s)}^z }{\Pi(x, \rho) \Phi(x)},
	\label{eq:decompositionMoumiA2}
	\end{equation} 
	where again $\vec s = (s_1, \dots, s_k, 0, \dots, 0)$.
	
The last step is to rewrite the sum over permutations in the R.H.S. of \eqref{eq:decompositionMoumiA2} as some contour integral.  
\begin{lemma} Assume $\Real[s_1], \dots, \Real[s_k]>0$ and $s_{k+1} = \dots = s_n=0$. Then
		\begin{equation}
		\frac{1}{(n-k)!} \sum_{\sigma \in \SS_n}  
		\frac{ I_{\sigma(\vec s)}^z }{   \Pi(x, \rho) \Phi(x)  } =
	\oint\frac{\mathrm{d}w_1}{2\I\pi} \dots \oint\frac{\mathrm{d}w_k}{2\I\pi} \mathcal{B}_{\vec s}^{q,t}(\vec w)
		\prod_{i=1}^k\frac{\mathcal{G}^{q,t}_n(w_i)}{\mathcal{G}^{q,t}_n(q^{-s_i}w_i)}  \frac{\phi(w_i^2)\Pi(q^{-s_i}w_i; \rho)(-z)^{s_i}\mathrm{d}w_i}{\phi(q^{-s_i}w_i^2)\Pi(w_i; \rho) (1-q^{s_i})w_i},
		\label{eq:actionMoumiAnu}
		\end{equation}
		where the contours are small positively oriented  circles enclosing the $x_i$ and no other singularity. 
		\label{lem:actionMoumiAnu}
	\end{lemma}
	\begin{proof}
		We adapt Lemma \ref{lem:actionNouminu}. 
		The L.H.S in \eqref{eq:actionMoumiAnu} equals 
		\begin{multline} \frac{1}{(n-k)!} \sum_{\sigma \in \SS_n} z^{s_1+\dots+s_k} \prod_{i<j} \frac{q^{s_j}x_{\sigma(i)}-q^{s_i}x_{\sigma(j)}}{x_{\sigma(i)}-x_{\sigma(j)}} \prod_{i,j=1}^n \frac{(tx_{\sigma(j)}/x_{\sigma(i)})_{\infty}}{(qx_{\sigma(j)}/x_{\sigma(i)})_{\infty}}\frac{(q^{s_i+1}x_{\sigma(j)}/x_{\sigma(i)})_{\infty}}{(tq^{s_i}x_{\sigma(j)}/x_{\sigma(i)})_{\infty}}  \\ \times 
		\prod_{i=1}^n \frac{\Pi(q^{-s_i}x_{\sigma(i)}; \rho)}{\Pi(x_{\sigma(i)}; \rho)}\prod_{i<j}\frac{\phi(q^{-s_i-s_j}x_{\sigma(i)}x_{\sigma(j)})}{\phi(x_{\sigma(i)}x_{\sigma(j)})}.\end{multline}
		Since $s_{k+1}=\dots=s_n=0$, the summand is invariant with respect to permutation of the $\lbrace s_i\rbrace_{i>k}$. Hence one can absorb the factor $1/(n-k)!$ and sum over permutations in 
		$$\SS_n^k := \left\lbrace \sigma\in \SS_n \ : \  \sigma(1+k)<\dots <\sigma(n)\right\rbrace.$$ 
		Thus, the L.H.S in \eqref{eq:actionMoumiAnu} equals 
		\begin{multline}\sum_{\sigma\in\SS_n^k} \underset{\underset{\forall 1\leqslant i\leqslant k}{w_i=x_{\sigma(i)}}}{\subs} \left\lbrace 
		z^{s_1+\dots+s_k} \prod_{i<j} \frac{q^{s_j}w_i-q^{s_i}w_j}{w_i-w_j} \prod_{i=1}^k \prod_{j=k+1}^n \frac{w_i-q^{s_i}x_{\sigma(j)}}{w_i-x_{\sigma(j)}}
		\right.\\ \left. \times 
		\prod_{i=1}^k\prod_{j=1}^n 
		\frac{(tx_{\sigma(j)}/w_i)_{\infty}}{(qx_{\sigma(j)}/w_i)_{\infty}}\frac{(q^{s_i+1}x_{\sigma(j)}/w_i)_{\infty}}{(tq^{s_i}x_{\sigma(j)}/w_i)_{\infty}}
		 \prod_{i=1}^k \frac{\Pi(q^{-s_i}w_i;\rho)}{\Pi(w_i;\rho)}
		 \right.\\ \left.  \times 
		\prod_{1\leqslant i<j\leqslant k }\frac{\phi(q^{-s_i-s_j}w_iw_j)}{\phi(w_iw_j)}
		\prod_{i=1}^k \prod_{j=1}^n \frac{\phi(q^{-s_i}w_i x_j)}{\phi(w_i x_j)}
		\prod_{i=1}^k\prod_{j=1}^k \frac{\phi(w_iw_j)}{\phi(q^{-s_i}w_iw_j)}
		\right\rbrace.
		\label{eq:sumresiduesM}
		\end{multline}
		Notice that we have  
		$$ \underset{\underset{\forall 1\leqslant i\leqslant k}{w_i=x_{\sigma(i)}}}{\subs} \left\lbrace    \prod_{i=1}^k \prod_{j=k+1}^n \frac{w_i-q^{s_i}x_{\sigma(j)}}{w_i-x_{\sigma(j)}}\right\rbrace  = \Res{\underset{\forall 1\leqslant i\leqslant k}{w_i=x_{\sigma(i)}}} \left\lbrace  \prod_{i=1}^k \prod_{j=1}^n \frac{w_i-q^{s_i}x_j}{w_i-x_j} \prod_{i,j=1}^k \frac{1}{w_i-q^{s_i} w_j}\prod_{i\neq j=1}^k (w_i-w_j)\right\rbrace.$$
		It implies that 
		\begin{multline} \eqref{eq:sumresiduesM} = \oint\frac{\mathrm{d}w_1}{2\I\pi} \dots \oint\frac{\mathrm{d}w_k}{2\I\pi} \prod_{i<j} \frac{q^{s_j}w_i-q^{s_i}w_j}{w_i-w_j} \prod_{i=1}^k \prod_{j=1}^n \frac{w_i-q^{s_i}x_j}{w_i-x_j} \prod_{i,j=1}^k \frac{1}{w_i-q^{s_i} w_j}\prod_{i\neq j=1}^k (w_i-w_j)\\  \times 
\prod_{i=1}^k\prod_{j=1}^n 
\frac{(tx_j/w_i)_{\infty}}{(qx_j/w_i)_{\infty}}\frac{(q^{s_i+1}x_j/w_i)_{\infty}}{(tq^{s_i}x_j/w_i)_{\infty}}
	\prod_{i=1}^k \frac{\Pi(q^{-s_i}w_i;\rho)}{\Pi(w_i;\rho)} \\ \times 
		\prod_{1\leqslant i<j\leqslant k }\frac{\phi(q^{-s_i-s_j}w_iw_j)}{\phi(w_iw_j)}
		\prod_{i=1}^k \prod_{j=1}^n \frac{\phi(q^{-s_i}w_i x_j)}{\phi(w_i x_j)}
		\prod_{i=1}^k\prod_{j=1}^k \frac{\phi(w_iw_j)}{\phi(q^{-s_i}w_iw_j)},
		\end{multline}
		where the contours enclose the $x_i$ and no other singularity. The assumption (cf. assumption (ii) in Theorem \ref{theo:MoumiLaplace}) that the real part of the variables $s_i$ is such that $\vert q^{s_i}\vert <x_j/x_i$ for all $i,j$ ensures that one can choose these contours as small circles enclosing the $x_i$.   
		Finally, using the fact that $\phi(u) = (tu)_{\infty}/(u)_{\infty}$, one has 
		$$ \frac{w_i-q^{s_i}x_j}{w_i-x_j} \frac{(tx_j/w_i)_{\infty}}{(qx_j/w_i)_{\infty}}\frac{(q^{s_i+1}x_j/w_i)_{\infty}}{(tq^{s_i}x_j/w_i)_{\infty}} = \frac{\phi(x_j/w_i)}{\phi(q^{s_i}x_j/w_i)},$$
		so that the integrand can be arranged to match with the R.H.S of \eqref{eq:actionMoumiAnu}. 
	\end{proof}
	The application of  Lemma \ref{lem:actionMoumiAnu} to each term of the sum in \eqref{eq:decompositionMoumiA2} yields 
the desired formula, which concludes the proof of Proposition \ref{prop:actionMoumiAonZ} and Theorem \ref{theo:MoumiLaplace}.
\end{proof}

 \thispagestyle{plain}\section{Half-space $q$-Whittaker processes}
\label{sec:qWhittaker}
We assume now that $t=0$ and $q\in (0,1)$. We will, however, use the same notations $P,Q,\ve$ as before. These ($P$ and $Q$) are now called \emph{$q$-Whittaker functions} \cite{gerasimov2009q}.
We define the \emph{half-space $q$-Whittaker measure} $\PQWM_{\rhoup, \rhodiag}$ as the measure on partitions $\la\in \Y$ such that 
$$ \PQWM_{\rhoup, \rhodiag}(\lambda) = \frac{P_{\lambda}(\rhoup) \ve_{\lambda}(\rhodiag)}{\Pi(\rhoup;\rhodiag)\Phi(\rhoup)},$$
and denote by $\PQWP_{\pathh, \bm\uprho}$, for a sequence of specializations $\bm\uprho$, the \emph{half-space $q$-Whittaker process}, i.e., the $t=0$ degeneration of the half-space Macdonald process $\PMP_{\pathh, \bm\uprho}$.

\subsection{Observables and integral formulas}

Consider a $q$-Whittaker measure where $\rhoup=(a_1, \dots, a_n) \in( 0,1)^n$, and $\rhodiag= \rho(\alpha, \beta, \gamma)$ as defined in Section \ref{sec:specializations}. We will now degenerate the results from Sections \ref{sec:generalmomentsformulas} and \ref{sec:generalLaplace} to the $q$-Whittaker case.  We further assume that the beta component of the specialization $\rhodiag$ is trivial (that is $\beta_i\equiv 0$), and all the parameters $\alpha_i$ are such that $\max\lbrace a_i\rbrace \max\lbrace \alpha_j\rbrace <1$, since this is the case which matters in our applications.

It is convenient to define functions 
$$ \gqwhittn(w)  = e^{-\gamma w} \prod_{j=1}^n\frac{(wa_j)_{\infty}}{(w/a_j)_{\infty}} \prod_{j=1}^{\ell} (w \alpha_j)_{\infty} \ \ \text{ and } \ \  \gqwhittun(w) = e^{-\gamma w} \prod_{j=1}^n\frac{(wa_j)_{\infty}}{(a_j/w)_{\infty}} \prod_{j=1}^{\ell} (w \alpha_j)_{\infty}.$$

\subsubsection{Moment formulas}

Let us write explicitly  the $q$-Whittaker degeneration of moment formulas from Section \ref{sec:generalmomentsformulas}. The moments  appearing in the following Corollary \ref{cor:momentsqWhittaker} can also be written in a different form as in Corollary \ref{cor:smallcontoursmoments1} and Corollary \ref{cor:smallcontoursmomentsn}. While both statements could in principle be deduced from Corollary \ref{cor:momentsqWhittaker}, it is much simpler to deduce them from the $q$-Laplace transform formulas, hence they will appear in the next section.
\begin{corollary}  Under the $q$-Whittaker measure $\PQWM_{(a_1,\ldots, a_n), ((\alpha_1, \dots, \alpha_{\ell}), \gamma)}$, the two following moment formulas hold.   	For any $k\in\Z_{>0}$, 
		\begin{equation}
		\EQWM[q^{k\lambda_N}] = (-1)^k \oint\frac{\mathrm{d}w_1}{2\I\pi}\cdots \oint\frac{\mathrm{d}w_k}{2\I\pi} \prod_{1\leqslant a<b\leqslant k}\frac{w_a-w_b}{q^{-1} w_a-w_b}\,  \frac{1 - q w_aw_b}{1 - w_aw_b}
		\prod_{m=1}^{k} \frac{1}{1-w^2_m} \frac{\gqwhittn(w_m)}{\gqwhittn(qw_m)} \frac{1}{w_m},
		\label{eq:qmomentslambdan}
		\end{equation}
		where the positively oriented contours are such that for all $1\leqslant c\leqslant k$, the contour for $w_c$ encloses $\lbrace a_j\rbrace_{1\leqslant j\leqslant n}$  and $q \lbrace w_{c+1}, \dots, w_k\rbrace$, and excludes the poles of the integrand at $1$ and $0$.

For all $k\in\Z_{>0}$ such that $q^{k}> \big(\max \lbrace a_j \rbrace\big)^2$ and $q^{k}> \max \lbrace \alpha_j \rbrace \max \lbrace a_j \rbrace$, 
		\begin{equation}
		\EQWM[q^{-k\lambda_1}] = (-1)^k\oint\frac{\mathrm{d}w_1}{2\I\pi}\cdots \oint\frac{\mathrm{d}w_k}{2\I\pi} \prod_{1\leqslant a<b\leqslant k} \frac{w_a-w_b}{q^{-1} w_a-w_b}\, \frac{q^{-1} - w_aw_b}{q^{-2} - w_aw_b}
		\prod_{m=1}^{k}\frac{q w_m^2-1}{q w_m^2}  \frac{\gqwhittun(w_m^{-1})}{\gqwhittun(q^{-1}w_m^{-1})}\frac{1}{w_m},
		\label{eq:qmomentslambda1}
		\end{equation}
		where the positively oriented contours are such that for all $1\leqslant c\leqslant k$, the contour for $w_c$ encloses $\lbrace 1/a_j\rbrace_{1\leqslant j\leqslant n}$  and $q\lbrace w_{c+1}, \dots, w_k\rbrace$, and excludes the poles of the integrand at $0$, $\alpha_i/q$ (for $1\leqslant i\leqslant \ell$) and $a_j/q$ (for $1\leqslant j\leqslant n$). 
\label{cor:momentsqWhittaker}
\end{corollary}
\begin{proof}
This is the $t=0$ degeneration of Proposition \ref{prop:momentsMacdonald}. Note that 
$$ \frac{\gqwhittn(w)}{\gqwhittn(qw)} =  \exp\big((q-1)\gamma w\big) \prod_{j=1}^{n}\left(\frac{1- w a_j}{1-w/a_j}\right) \prod_{i=1}^{\ell}\left(1-\alpha_i w\right) $$ 
and 
$$ \frac{\gqwhittun(w^{-1})}{\gqwhittun(q^{-1}w^{-1})} =   \exp\big((q^{-1}-1)\gamma w^{-1}\big) \prod_{j=1}^{n} \left(\frac{q w}{(1-w a_j)(q w - a_j)} \right) \prod_{i=1}^{\ell} \left(\frac{q w}{qw - \alpha_i} \right).$$
\end{proof}

\begin{corollary}
Under the $q$-Whittaker measure $\PQWM_{(a_1,\ldots, a_n), ((\alpha_1, \dots, \alpha_{\ell}), \gamma)}$, the two following formulas hold.  For $r\geqslant 1$,
	\begin{multline*} 
	\EQWM\Big[q^{\lambda_N+\ \ldots\ +\lambda_{N-r+1}}\Big] \\  = 
	\frac{(-1)^{\frac{r(r+1)}{2}}}{ r!} \oint\frac{\mathrm{d}z_1}{2\I\pi}\cdots \oint\frac{\mathrm{d}z_r}{2\I\pi} \prod_{j=1}^{r} \left( \frac{1}{(z_j)^r} \frac{1}{1-z^2_j} \frac{\gqwhittn(z_j)}{\gqwhittn(q z_j)} \right)
		\prod_{1\leqslant i<j\leqslant r} \left((z_i-z_j)^2 \  \frac{1-qz_iz_j}{1-z_iz_j}\right),
	\end{multline*}
	where the positively oriented contours encircle $\lbrace a_1, \dots, a_N\rbrace$ and no other singularity of the integrand. 

 For $r\geqslant 1$, and assuming that  $q^r> \big(\max \lbrace a_j \rbrace\big)^2$ and $q^r> \max \lbrace \alpha_j \rbrace \max \lbrace a_j \rbrace$, 
	\begin{multline*}
	\EQWM\Big[q^{-\lambda_1-\ \ldots\ -\lambda_r}\Big]  \\ = 
	\frac{(-1)^{\frac{r(r+1)}{2}}}{ r!} \oint\frac{\mathrm{d}w_1}{2\I\pi}\cdots \oint\frac{\mathrm{d}w_r}{2\I\pi}   \prod_{j=1}^{r} \left( \frac{1}{(w_i)^r}\frac{qw_j^2-1}{qw_j^2}  \frac{\gqwhittun(w_j^{-1})}{\gqwhittun(q^{-1}w_j^{-1})} \right) 
	\prod_{1\leqslant i<j\leqslant r}\left( (w_i-w_j)^2\ 
	\frac{q^2w_iw_j-q}{q^2w_iw_j-1} \right),
	\end{multline*}
	where the positively oriented contours encircle $\lbrace 1/a_1, \dots, 1/a_N\rbrace$ and no other singularity. 

\end{corollary}
\begin{proof}
	This is the $t=0$ degeneration of Proposition \ref{prop:Dr}. 
\end{proof}

\subsubsection{$q$-Laplace transform formulas}

We consider first the $q$-Laplace transform of $q^{\la_n}$. 
\begin{corollary}
	Let $z\in \C\setminus \R_{>0}$. Assume that the parameters  $ a_1, \dots, a_n\in(0,1) $ are chosen so that for all $i,j$,  $\vert q a_i/a_j\vert <1$,  and let $0<R<1$ be such that $0<q^R<a_i/a_j$ for all $i,j$. Under the $q$-Whittaker measure $\PQWM_{(a_1,\ldots, a_n), ((\alpha_1, \dots, \alpha_{\ell}), \gamma)}$, we have 
	\begin{multline}
	\EQWM\left[  \frac{1}{(q^{\la_n}z)_{\infty}} \right]= \sum_{k=0}^{n} \frac{1}{k!} \int_{\mathcal{D}_{R}}\frac{\mathrm{d}s_1}{2\I\pi} \dots \int_{\mathcal{D}_{R}} \frac{\mathrm{d}s_k}{2\I\pi}\  
\oint\frac{\mathrm{d}w_1}{2\I\pi} \dots \oint\frac{\mathrm{d}w_k}{2\I\pi}   \\ \times 
	\prod_{1\leqslant i<j\leqslant k} \frac{ (q^{s_j}w_j-q^{s_i}w_i)(w_i-w_j)(q^{s_i}w_iw_j)_{\infty}(q^{s_j}w_jw_i)_{\infty}}{(q^{s_i}w_i-w_j)(q^{s_j}w_j-w_i)(q^{s_i+s_j}w_i w_j)_{\infty}(w_iw_j)_{\infty}} \\
	\times	\prod_{i=1}^k \left[  \Gamma(-s_i)\Gamma(1+s_i)   \frac{\gqwhittn(w_i)}{\gqwhittn(q^{s_i}w_i)}  
	\frac{ (q^{s_i}w_i^2)_{\infty} (-z)^{s_i}}{(w_i^2)_{\infty} (q^{s_i}-1)w_i}
	\right], 
	\label{eq:expansionNoumiqWhittaker}
	\end{multline}
	where the positively oriented integration contours for the variables $w_i$ enclose all the $a_i$ and no other singularity.
	\label{cor:LaplaceqWhittakerlambdan}
\end{corollary}

\begin{proof}
	This is the $t=0$ degeneration of Theorem \ref{theo:NoumiLaplace}. 
\end{proof}
It is possible to extract the coefficient of $z^k$ in the above formula, which yields the following alternative moment formula. 
\begin{corollary} For any $r\in \Z_{>0}$, under the $q$-Whittaker measure $\PQWM_{(a_1,\ldots, a_n), ((\alpha_1, \dots, \alpha_{\ell}), \gamma)}$,
\begin{multline*}
\EQWM[q^{r\lambda_n}] = (q;q)_{r}  \sum_{{\tiny \begin{matrix} \mu\vdash r \\ \mu=1^{m_1}2^{m_2}\dots \end{matrix}}} \frac{1}{m_1! m_2! \dots }    \oint\frac{\mathrm{d}w_1}{2\I\pi} \dots  \oint\frac{\mathrm{d}w_{\ell(\mu)}}{2\I\pi} \det\left[ \frac{1}{w_iq^{\mu_i } - w_j}  \right]_{i,j=1}^{\ell(\mu)} \\ \times 
 \prod_{1\leqslant a<b\leqslant \ell(\mu)} \frac{(q^{\mu_a}w_a w_b)_{\infty}(q^{\mu_b}w_a w_b)_{\infty}}{(w_a w_b)_{\infty}(q^{\mu_a+ \mu_b}w_a w_b)_{\infty}} \prod_{j=1}^{\ell(\mu)} \frac{(q^{\mu_m}w_j^2)_{\infty}}{(w_j^2)_{\infty}}\frac{ \gqwhittn(w_j)}{\gqwhittn(q^{\mu_j}w_j)}   \mathrm{d}w_j,
\end{multline*}	
where the positively oriented contours enclose the $a_j$ and no other singularity. 
\label{cor:smallcontoursmomentsn}
\end{corollary}
\begin{proof}[Proof 1]
To extract the coefficient of $z^r$ in \eqref{eq:expansionNoumiqWhittaker}, we use the $q$-binomial theorem \ref{eq:qbinomial} on the left-hand-side. On the right-hand-side, we may transform the integration over $s_i$ into a discrete sum using Lemma \ref{lem:MellinBarnes} and  collect the terms of degree $k$ in $z$. We obtain 
\begin{multline*}
\frac{1}{(q;q)_{r}} \EQWM[q^{r\lambda_n}] =  \sum_{k=0}^{\infty} \frac{1}{k!} \sum_{\substack{\nu_1, \dots, \nu_k \geqslant 1\\ \nu_1+\dots+\nu_k=r}} \oint\frac{\mathrm{d}w_1}{2\I\pi} \dots  \oint\frac{\mathrm{d}w_{k}}{2\I\pi} \det\left[ \frac{1}{w_iq^{\nu_i } - w_j}  \right]_{i,j=1}^{k} \\ \times 
\prod_{1\leqslant a<b\leqslant k} \frac{(q^{\nu_a}w_a w_b)_{\infty}(q^{\nu_b}w_a w_b)_{\infty}}{(w_a w_b)_{\infty}(q^{\nu_a+ \nu_b}w_a w_b)_{\infty}} \prod_{j=1}^{k} \frac{(q^{\nu_m}w_j^2)_{\infty}}{(w_j^2)_{\infty}}\frac{ \gqwhittn(w_j)}{\gqwhittn(q^{\nu_j}w_j)}   \mathrm{d}w_j.
\end{multline*}
We may relabel the $w_j$'s since they  are integrated on the same contour, or equivalently we may relabel the $\nu_j$'s. Hence, the sum over compositions $\nu$ can be rewritten as a sum over partitions $\mu$ (the factor $k!/(m_1! m_2! \dots)$ corresponds to the number of compositions corresponding to the same partition).  
\end{proof}

\begin{proof}[Proof 2]
We present a second proof deriving Corollary \ref{cor:smallcontoursmomentsn} directly from Corollary \ref{cor:momentsqWhittaker}. Although this proof is longer, it is useful to compare our present results with \cite{borodin2014macdonald} and \cite{borodin2016directed}, and a similar approach will be useful in Section \ref{sec:HL}. We first recall a useful result in the theory of Macdonald processes. It was first stated in \cite{borodin2014macdonald} as Proposition 3.2.1, but we use a slightly more general version from \cite{borodin2015spectral}. 
\begin{proposition}[{\cite[Proposition 7.4]{borodin2015spectral}}]
 Let $\gamma_1, \dots, \gamma_k$ be positively oriented closed contours and a function $F(z_1, \dots, z_k)$ be such that 
 \begin{itemize}
 	\item The contour $\gamma_k$ is a small circle around $1$, small enough so as to not contain $q$. 
 	\item For all $A<B$, $\gamma_A$ enclose $q\gamma_B$. 
 	\item For all $1\leqslant j\leqslant k$, one can deform $\gamma_j$ to $\gamma_k$ in such a way that the function 
 	$$ \frac{F(z_1, \dots, z_k)}{z_1 \dots z_k}\prod_{i<j}(z_i-z_j) $$
 	is analytic in $z_j$ (when all variables lie on their respective contours) in a neighbourhood of the area swept by the deformation.  
 \end{itemize}
 Then we have 
 \begin{equation}
\oint\frac{\mathrm{d}z_1}{2\I\pi} \dots \oint\frac{\mathrm{d}z_k}{2\I\pi} \prod_{A<B} \frac{z_A-z_B}{z_A-qz_B} \frac{F(z_1, \dots, z_k)}{z_1 \dots z_k} = \sum_{\lambda\vdash k}  \oint\frac{\mathrm{d}w_1}{2\I\pi} \dots \oint\frac{\mathrm{d}w_{\ell(\la)}}{2\I\pi} \mathrm{d}\mu_{\lambda}(\vec{w}) E^q(\vec{w}\circ \lambda),  
 \label{eq:residuebookkeeping}
 \end{equation}
 where
 \begin{equation}
 \vec{w}\circ \lambda = (w_1, qw_1, \dots, q^{\lambda_1-1}w_1, w_2, \dots, q^{\la_2-1}w_2, \dots, q^{\la_{\ell(\la)}-1}w_{\ell(\la)}),
 \label{eq:defwcirclambda}
 \end{equation}
 $$E^q(\vec z)  = \sum_{\sigma\in \mathcal{S}_k}\sigma\left( \prod_{A<B} \frac{z_A-qz_B}{z_A-z_B}F(z_1, \dots, z_k) \right),$$
 and 
 $$\mathrm{d}\mu_{\lambda}(\vec w) = \frac{(q-1)^{k}q^{-\frac{k(k-1)}{2}}}{m_1!m_2!\dots} \det\left[\frac{1}{w_iq^{\la_i}-w_j} \right]_{i,j=1}^{\ell(\la)} %\prod_{j=1}^{\ell(\la)} w_j^{\lambda_j}q^{\frac{\lambda_j(\la_j-1)}{2}}
  \prod_{j=1}^{\ell(\la)}\mathrm{d}w_j.$$ 
 \label{prop:contourshift}
\end{proposition}
We may  apply Proposition \ref{prop:contourshift} to \eqref{eq:qmomentslambdan} with the benign modification that the contour $\gamma_k$ is around the $a_i$ instead of around $1$. Note that the factor 
$$ \prod_{1\leqslant a<b\leqslant k}\frac{1 - q z_az_b}{1 - z_az_b} \prod_{i=1}^k \frac{1}{1-z_i^2}$$
is symmetric and can be taken outside of the sum over permutations in $E^q$. The evaluation into $\vec w\circ \lambda$ is given by the following.

\begin{lemma} For a partition $\la\vdash k$ with $m=\ell(\la)$, if  $\vec z = \vec w \circ \lambda$ (the notation was introduced in \eqref{eq:defwcirclambda}) then we have 
	\begin{equation}
	\prod_{1\leqslant a<b\leqslant k}\frac{1 - q z_az_b}{1 - z_az_b} \prod_{i=1}^k \frac{1}{1-z_i^2} =   \prod_{1\leqslant a<b\leqslant m} \frac{(q^{\la_a}w_a w_b)_{\infty}(q^{\la_b}w_a w_b)_{\infty}}{(w_a w_b)_{\infty}(q^{\la_a+ \la_b}w_a w_b)_{\infty}} \prod_{j=1}^{m} \frac{(q^{\lambda_j}w_j^2)_{\infty}}{(w_j^2)_{\infty}}.
	\label{eq:simplification}
	\end{equation}
	\label{lem:evaluationstring}
\end{lemma}
\begin{proof} For $\vec z = \vec w \circ \lambda$, 
\begin{align*}
	L.H.S \ \eqref{eq:simplification} &= \prod_{1\leqslant i<j\leqslant m} \prod_{r=0}^{\la_i-1}\prod_{s=0}^{\la_j-1} \frac{1-q^{1+r+s}w_iw_j}{1-q^{r+s}w_iw_j}  \prod_{i=1}^m \prod_{0\leqslant r<s\leqslant \la_i-1} \frac{1-q^{1+r+s}w_i^2}{1-q^{r+s}w_i^2} \prod_{t=0}^{\lambda_i-1} \frac{1}{1-q^{2t}w_i^2} \\ 
&= \prod_{1\leqslant i<j\leqslant m} \prod_{r=0}^{\la_i-1} \frac{1-q^{r+\lambda_j}w_iw_j}{1-q^{r}w_iw_j} 
\prod_{i=1}^m \prod_{r=0}^{\la_i-1} \frac{1-q^{r+\la_i}w_i^2}{1-q^{2r+1}w_i^2} \frac{1}{1-q^{2r}w_i^2}\\ 
&= \prod_{1\leqslant i<j\leqslant m} \frac{(q^{\la_j}w_iw_j)_{\infty}/(q^{\la_i+\la_j}w_iw_j)_{\infty}}{(w_iw_j)_{\infty}/(q^{\la_i}w_iw_j)_{\infty}} 
\prod_{i=1}^m \prod_{r=0}^{\la_i-1}  \frac{1}{1-q^{r}w_i^2} 
=  R.H.S\  \eqref{eq:simplification}.
\end{align*}
\end{proof}
$E^q$ can then be computed using the symmetrization identity \cite[VII, (1.4)]{macdonald1995symmetric}
\begin{equation}
\sum_{\sigma\in \mathcal{S}_k} \sigma\left( \prod_{1\leqslant i<j\leqslant k} \frac{u_i-q u_j}{u_i-u_j} \right) = \frac{(q;q)_k}{(1-q)^k}. 
\label{eq:symidentity}
\end{equation}
Thus, Proposition \ref{prop:contourshift} applied  to Corollary \ref{cor:momentsqWhittaker} using Lemma \ref{lem:evaluationstring}  yields Corollary \ref{cor:smallcontoursmomentsn}. 
\end{proof}

Now we provide formulas characterizing the distribution of $q^{-\la_1}$. 
\begin{corollary} 
	Let $z\in \C\setminus \R_{>0}$. Assume that
	\begin{enumerate}
		\item[(i)] The parameters  $ a_1, \dots, a_n\in(0,1) $ are chosen such that for all $i,j$, $\max\lbrace a_i\rbrace< a_i/a_j$ and $q a_i/a_j <1$.
		\item[(ii)] $R\in (0,1)$, $q^{R}> \max\lbrace a_i\rbrace $, and for all $i,j$ $q^R< a_i/a_j$.
		\item[(iii)]  The $\alpha_i$'s are chosen so that $q^{R}> \max\lbrace a_i\rbrace \max\lbrace \alpha_j\rbrace$. 
	\end{enumerate}
	  Then, under the $q$-Whittaker measure $\PQWM_{(a_1,\ldots, a_n), ((\alpha_1, \dots, \alpha_{\ell}), \gamma)}$, we have 
	\begin{multline}
	\EQWM\left[ \frac{1}{(z q^{-\la_1})_{\infty}}\right] =  \sum_{k=0}^{n} \ \frac{1}{k!}\ \int_{\mathcal{D}_R}\frac{\mathrm{d}s_1}{2\I\pi} \dots \int_{\mathcal{D}_R} \frac{\mathrm{d}s_k}{2\I\pi}\   \oint\frac{\mathrm{d}w_1}{2\I\pi} \dots \oint\frac{\mathrm{d}w_k}{2\I\pi}    \\  \times 
	\prod_{1\leqslant i<j\leqslant k} \frac{(q^{s_j}w_i-q^{s_i}w_j)(w_j-w_i)(q^{-s_i}w_iw_j)_{\infty}(q^{-s_j}w_iw_j)_{\infty}}{(q^{s_j}w_i-w_j)(q^{s_i}w_j-w_i)(q^{-s_i-s_j}w_iw_j)_{\infty}(w_iw_j)_{\infty}}\\ \times
	\prod_{i=1}^k \left[  \Gamma(-s_i)\Gamma(1+s_i)   \frac{\gqwhittun(w_i)}{\gqwhittun(q^{-s_i}w_i)}  
	 \frac{ (q^{-s_i}w_i^2)_{\infty} (-z)^{s_i}}{(w_i^2)_{\infty} (1-q^{s_i})w_i}\right],
	\label{eq:LaplaceqWhittaker}
	\end{multline}
	where the $w$ contours are small positively oriented circles enclosing the $a_i$ and no other singularity, and $\mathcal{D}_R = R+\I\R$ oriented upwards, as before.
	\label{cor:LaplaceqWhittaker}
\end{corollary}
\begin{proof}
	This is the $t=0$ degeneration of Theorem \ref{theo:MoumiLaplace}.
\end{proof}

\begin{remark}
	\label{rem:finitevsinfinitecontours}
Note that in light of similarities with $q$-Laplace transform formulas for the full-space $q$-Whittaker process -- see in particular \cite[Theorem 3.3]{borodin2015height} -- it is conceivable that \eqref{eq:LaplaceqWhittaker} holds as well if the contours for the $w_i$ variables contain not only the singularities $a_i$ but also $q^ja_i$ for all $1\leqslant j\leqslant k$ for an arbitrary integer $k$ (In \cite[Theorem 3.3]{borodin2015height}, the contour encloses all the singularities $1/(q^ja_i)$, but the formula would be valid as well if the contour would contain only the singularities $1/(q^ja_i)$ for $1\leqslant j\leqslant k$ where $k$ is an arbitrary positive integer). Transforming the contours in such a way may be convenient for later asymptotic analysis of the formula, so as to work with infinite contours in Section \ref{sec:Whittaker}. Presently, it is not clear how to justify such a contour deformation.
 \end{remark}

Both sides of \eqref{eq:LaplaceqWhittaker} are analytic in $z\in \C\setminus \R_{\geqslant 0}$ but not analytic at $z=0$, so one cannot extract coefficients as in the proof of Corollary \ref{cor:smallcontoursmomentsn}. Nevertheless, we have the following: 
\begin{corollary}
	For all $r\in\Z_{>0}$ such that $q^{r}> \big(\max \lbrace a_j \rbrace\big)^2$ and $q^{r}> \max \lbrace \alpha_j \rbrace \max \lbrace a_j \rbrace$, under the $q$-Whittaker measure $\PQWM_{(a_1,\ldots, a_n), ((\alpha_1, \dots, \alpha_{\ell}), \gamma)}$,
\begin{multline}
\EQWM[q^{-r\lambda_1}] = (q;q)_{r}  \sum_{{\tiny \begin{matrix} \mu\vdash r \\ \mu=1^{m_1}2^{m_2}\dots \end{matrix}}} \frac{1}{m_1! m_2! \dots }    \oint\frac{\mathrm{d}w_1}{2\I\pi} \dots  \oint\frac{\mathrm{d}w_{\ell(\mu)}}{2\I\pi} 
 \det\left[ \frac{1}{q^{\mu_i }w_i^{-1} - w_j^{-1}}  \right]_{i,j=1}^{\ell(\mu)} \\ \times 
 \prod_{1\leqslant a<b\leqslant \ell(\mu)} \frac{(q^{-\mu_a}w_a w_b)_{\infty}(q^{-\mu_b}w_a w_b)_{\infty}}{(w_a w_b)_{\infty}(q^{-\mu_a- \mu_b}w_a w_b)_{\infty}} \prod_{j=1}^{\ell(\mu)} \frac{(q^{-\mu_j}w_j^2)_{\infty}}{(w_j^2)_{\infty}}\frac{ \gqwhittun(w_j)}{\gqwhittun(q^{-\mu_j}w_j)}   \frac{\mathrm{d}w_j}{w_j^2},
\label{eq:smallcontourlambda1}
\end{multline}	
where the positively oriented contours enclose the $a_j$ and no other singularity. 
\label{cor:smallcontoursmoments1}
\end{corollary}
\begin{proof}
Let $D_q$ be the $q$-derivative operator defined by
$$ D_qf(z) = \frac{f(qz)-f(z)}{qz - z}.$$	
We have 
$$ D_q\left( \frac{1}{(z q^{-\la_1})_{\infty}}  \right) = \frac{q^{-\la_1}}{(1-q)(z q^{-\la_1})_{\infty}}.$$ 
For $z\in \C\setminus \R_{\geqslant 0}$, we can apply $D_q$ to the L.H.S of \eqref{eq:LaplaceqWhittaker}, and the $q$ derivative clearly commutes with the expectation, so that if we iterate the procedure, we get 
\begin{equation}
 (D_q)^r \  \EQWM\left[ \frac{1}{(z q^{-\la_1})_{\infty}}\right]  = \EQWM\left[ \frac{q^{-r\la_1}}{(1-q)^r(z q^{-\la_1})_{\infty}}\right].\label{eq:Dqr}
\end{equation}
If $\EQWM[q^{-r\la_1}]$ is finite, we may let $z$ tend to $0$ in \eqref{eq:Dqr} to obtain 
$$ \EQWM[q^{-r\la_1}] = (1-q)^r \lim_{z\to 0} \left\lbrace (D_q)^r \  \EQWM\left[ \frac{1}{(z q^{-\la_1})_{\infty}}\right]\right\rbrace.$$

Now we can apply $D_q$ to the R.H.S of \eqref{eq:LaplaceqWhittaker}. If $q^r> \big(\max \lbrace a_j \rbrace\big)^2$ and $q^r> \max \lbrace \alpha_j \rbrace \max \lbrace a_j \rbrace$, we may shift the contour $\mathcal{D}_{R}$ to $\mathcal{D}_{R+r}$ and we will encounter only the poles at $s=1, \dots, r$ during the contour deformation. The sum of the residues is a polynomial in the variable $z$.  We may rewrite the application of $(D_q)^r$ to \eqref{eq:LaplaceqWhittaker} as the sum of $(D_q)^r$ applied to this polynomial and $(D_q)^r$ applied to the integral remainder. We have $D_q (-z)^s = \frac{1-q^s}{1-q} (-z)^{s-1}$, so after applying $(D_q)^r$ to the integrand in \eqref{eq:LaplaceqWhittaker}, it gets multiplied by 
$$ \frac{1-q^{s_1+\dots + s_k}}{1-q}\frac{1-q^{s_1+\dots + s_k-1}}{1-q} \dots \frac{1-q^{s_1+\dots + s_k-r+1}}{1-q} (-z)^{-r}.$$
Hence, after shifting the contours and applying $(D_q)^r$,  the integral term goes to zero as $z\to 0$ because the factor $z^{-r}\prod_{i=1}^k(-z)^{s_i}$ with $\Real[s_i] = R+r$ goes to zero. For a polynomial $P(X) = \sum_{i=1}^N c_iX^i$, 
$$\lim_{z\to 0}\left\lbrace (D_q)^r P(z) \right\rbrace  = c_r \prod_{i=1}^r\frac{1-q^i}{1-q}.$$
In our context, $c_r$ corresponds to the contribution of residues at $s_1=\mu_1, \dots, s_k=\mu_k$ for some $\mu_1, \dots, \mu_k\in \Z_{>0}$ such that $\mu_1+\dots+\mu_k=r$. These residues are easy to compute and they can be written as in \eqref{eq:smallcontourlambda1}. 
\end{proof}
\begin{remark} Corollary \ref{cor:smallcontoursmoments1} could also be proved from \eqref{eq:qmomentslambda1} by shrinking the nested contours to a small contour around the $a_i$ (i.e., applying Proposition \ref{prop:contourshift}) as in the proof of  Corollary \ref{cor:smallcontoursmomentsn}. 
\end{remark}

\begin{remark}
Regarding the distribution of $q^{\la_n}$, the moment formula in Corollary \ref{cor:smallcontoursmomentsn} and the Laplace transform formula from Corollary \ref{cor:LaplaceqWhittakerlambdan} are essentially equivalent (for sufficiently small parameters $a_1, \dots, a_n$). Indeed, we have deduced the moment formulas from the Laplace transform but one can go backwards summing the $q$-moment generating series. 

However, turning to the distribution of  $q^{-\la_1}$, the moment formula from Corollary \ref{cor:smallcontoursmoments1} can be deduced from the Laplace transform formula from Corollary \ref{cor:LaplaceqWhittaker}, but one cannot go backwards (otherwise this would have been a much easier route to prove Corollary \ref{cor:LaplaceqWhittaker}). Let us see why.  Denoting the L.H.S in \eqref{eq:LaplaceqWhittaker} by $L(z)$ and the R.H.S by $R(z)$, Corollary \ref{cor:smallcontoursmoments1} implies that for all $n\geqslant 0$, 
$$  \lim_{z\to 0} \ (D_q)^n R(z) =\lim_{z\to 0} \ (D_q)^n L(z). $$
This is not sufficient to deduce   Corollary \ref{cor:LaplaceqWhittaker}. Indeed, as for the usual differential operator, a function $f(z)$ is in general not determined in a neighborhood around $0$ by the knowledge of $\lim_{z\to 0} \ (D_q)^n f(z)$ for all $n\in \Z_{\geqslant 0}$, unless it is analytic at $0$. As a counterexample, one may consider 
$$f(z) = \sum_{k=0}^{\infty} \frac{\theta^k (\theta)_{\infty}}{(z q^{-k})_{\infty} (q;q)_k}\ \ \text{ and }\ \  g(z) = \sum_{k=0}^{\infty} \frac{z^k (z)_{\infty}}{(\theta q^{-k})_{\infty} (q;q)_k},$$
and check that 
$$ \lim_{z\to 0} \ (D_q)^n f(z) = \frac{1}{(\theta q^{-n})_{\infty}(1-q)^n}  =\lim_{z\to 0} \ (D_q)^n g(z).$$
In the above counterexample, $f(z)$ should be thought of as the correct expression for the Laplace transform\footnote{We have that  $f(z)=\EQWM[\frac{1}{(zq^{-\lambda_1})_{\infty}}]$ where $\la_1$ is distributed according to the half-space $q$-Whittaker measure with specializations $\rhoup=(a_1)$ and $\rhodiag = (\alpha_1)$ with $a_1\alpha_1=\theta$.} while $g(z)$ is the wrong expression that one would obtain by summing moment formulas (or applying the operator $\Moumi^z$ instead of $\MoumiA^z$). 
\label{rem:subtleties}
\end{remark}

\subsection{Matveev-Petrov RSK-type dynamics}
\label{sec:rskdynamics}
Section \ref{sec:Macdyn} introduced a general scheme through which one can  grow half-space Macdonald processes using two types of operators $\U$ and $\Udiag$. The main condition required of these operators was that they satisfied the defining relations \eqref{eq:definingeq}, \eqref{eq:definingeq2}. Matveev-Petrov \cite{matveev2015q} provide four different choices for the bulk operator $\U$ which solve \eqref{eq:definingeq}, provided the partitions satisfy certain interlacing conditions and the specializations are chosen appropriately. The description of these operators is quite involved and overall unnecessary for our purposes. Instead, we will recall the relevant properties of these operators one by one.

We will use only two of the bulk $\U$ operators introduced in \cite{matveev2015q}, and  denote them as $\U_{\textrm{row}}[\alpha]$ and $\U_{\textrm{col}}[\alpha]$
where $\alpha$ is a generic positive real number. These transition operators correspond, in the Schur process degeneration, to dynamics on Gelfand-Tsetlin patterns induced by  the row  or column RSK insertion respectively (see \cite{knuth1970permutations} or \cite[Section 4.3]{matveev2015q} and references therein). 
\begin{remark}
It would be tempting to consider the two other possibilities  $\U_{\textrm{row}}[\hat\beta], \U_{\textrm{col}}[\hat\beta]$, which are dual analogues. However, this is not presently accessible due to our choice of $\Udiag$ which enforces the use of the same set of specializations on horizontal and vertical edges. 
\end{remark}
$\U_{\textrm{row}}[\alpha]$ and $ \U_{\textrm{col}}[\alpha]$ act from the subspace of $\kappa\otimes\mu\otimes \nu\in \Y_k\otimes \Y_k\otimes \Y_{k+1}$ such that 
$\mu\prech \nu$ and $\mu\prech \kappa$ to the subspace of $\kappa\otimes\pi\otimes \nu\in \Y_k\otimes \Y_{k+1}\otimes \Y_{k+1}$ such that $\kappa\prech\pi$ and $\nu\prech\pi$. 
For such partitions $\kappa,\mu,\nu,\pi$, we will use the notation  $\U(\pi|\kappa,\mu,\nu)$ as in Section \ref{sec:Macdyn}. 
For a partition $\lambda\in \Y_k$ and $1\leqslant j\leqslant k$, let $\Proj_j(\lambda)=(\lambda_1,\ldots,\lambda_j)$ be the projection of $\lambda$ onto its first\footnote{The orientation of arrows in $\Proj$ and $\Projtilde$ is chosen to be consistent with Figure \ref{fig:partitions} and interlacing arrays in Section \ref{sec:Whittakerdynamics}.} $j$ parts (now a partition in $\Y_j$). Similarly, let $\Projtilde_j(\lambda) =(\lambda_{k-j+1},\ldots, \lambda_k)$ be the projection of $\lambda$ onto its last $j$ parts. 
The main fact we use from \cite{matveev2015q} is that $\U_{\textrm{row}}[\alpha]$ and $ \U_{\textrm{col}}[\alpha]$  satisfy the defining relation \eqref{eq:definingeq}, and their projections under $\Proj_j$ and $\Projtilde_j$ are marginally Markov. 

We will continue to use the boundary transition operator $\Udiag$ of push-block type defined in Section \ref{sec:pushblock}, although other choices might be interesting to study and could lead to different dynamics. Recall that in the setting of Section \ref{sec:Macdyn}, the operators $\U$ and $\Udiag$ depend on specializations $\rho_{\circ}$ and $\rho_i$ for $i\geqslant 1$. 

\subsubsection{Boundary transition operator}

We recall that for $k\geqslant 0$ and partitions $\mu\in\Y_k$, $\kappa, \pi\in\Y_{k+1}$, 
$$ \Udiag_{k,k}(\pi\vert \kappa, \mu)  = \Udiag_{k,k}(\pi\vert \kappa) = \frac{P_{\pi/\kappa}(\rho_{k+1})\ve_{\pi}(\rho_{\circ})}{\ve_\kappa(\rho_{k+1}, \rho_{\circ})\Pi(\rho_{k+1}, \rho_{\circ})\Phi(\rho_{k+1})}.$$
The R.H.S. above can be computed more explicitly when the specializations are simply the evaluation into single variables. 
\begin{lemma} For $a,b\in \C$, we have 
	\begin{equation}
	\ve_{\mu}(a)P_{\mu/\lambda}(b) = a^{\sum_i \mu_{2i-1}-\mu_{2i}}b^{\sum_i \mu_i-\lambda_i} \prod_{i=1}^{\ell(\mu)} \frac{(q^{\mu_i-\lambda_i+1};q)_{\infty}(q^{\lambda_i-\mu_{i+1}+1};q)_{\infty}}{(q;q)_{\infty}^2}.
	\label{eq:transitiondiago}
	\end{equation}
	\label{lem:computeboundaryoperator}
\end{lemma}
\begin{proof} Using the combinatorial formula \eqref{eq:combinatorialQ}, 
	$Q_{\mu/\eta}(u) = u^{|\mu|-|\eta|} \varphi_{\mu/\eta} \mathds{1}_{\eta\prech \mu}$
	where
	$$
	\varphi_{\mu/\eta} = \prod_{i=1}^{\ell(\mu)} \frac{(q^{\mu_i-\eta_i+1};q)_{\infty} (q^{\eta_i-\mu_{i+1}+1};q)_{\infty}}{(q;q)_{\infty}(q^{\eta_i-\eta_{i+1}+1};q)_{\infty}}.
	$$
	
	We would like to evaluate $\ve_{\mu}(u)$. Since $u$ is a single usual specialization, $Q_{\mu/\eta}(u)$ is only nonzero when $\mu/\eta$ is a horizontal strip. Among those $\eta$, there is only one which has $\eta'$ even, and it  is given by $\eta^* = (\mu_2,\mu_2,\mu_4,\mu_4,\ldots)$. For that $\eta^*$,
	$$
	\varphi_{\mu/\eta^*} = \prod_{i=1}^{\ell(\mu)} \frac{(q^{\mu_i-\mu_{i+1}+1};q)_{\infty}}{(q;q)_{\infty}} \, \prod_{j=1}^{\lfloor \ell(\mu)/2\rfloor} \frac{(q;q)_{\infty}}{(q^{\mu_{2j}-\mu_{2j+2}+1};q)_{\infty}}.
	$$
	We also readily see that for $\eta^*$, the only boxes for which $b_{\eta^*}(\square)$ does not equal $1$ are those which have leg length 0 (otherwise the factor of $t^{\mathrm{leg}}$ in the denominator is zero). Thus, we find
	$$
	b^{\textrm{el}}_{\eta^*} = \prod_{j=1}^{\lfloor \ell(\mu)/2\rfloor} \frac{(q^{\mu_{2j}-\mu_{2j+2}+1};q)_{\infty}}{(q;q)_{\infty}}
	$$
	so that
	$$
	\ve_{\mu}(u) = b^{\textrm{el}}_{\eta^*} u^{\sum_i \mu_{2i-1}-\mu_{2i}} \varphi_{\mu/\eta^*} = u^{\sum_i \mu_{2i-1}-\mu_{2i}}  \prod_{i=1}^{\ell(\mu)} \frac{(q^{\mu_i-\mu_{i+1}+1};q)_{\infty}}{(q;q)_{\infty}}.
	$$
	
	By the combinatorial formula \eqref{eq:combinatorialP}, 
	$$
	P_{\mu/\lambda}(u) = u^{|\mu|-|\lambda|} \psi_{\mu/\lambda}
	$$
	where
	$$
	\psi_{\mu/\lambda} = \prod_{i=1}^{\ell(\lambda)} \frac{(q^{\mu_i-\lambda_i+1};q)_{\infty} (q^{\lambda_i-\mu_{i+1}+1};q)_{\infty}}{(q;q)_{\infty}(q^{\mu_i-\mu_{i+1}+1};q)_{\infty}}.
	$$
	
	Putting this all together and noting that $\ell(\lambda)+1=\ell(\mu)$ we arrive at \eqref{eq:transitiondiago}.
\end{proof}

\begin{definition}
A \emph{$q$-geometric} random variable with parameter $\theta\in (0,1)$, denoted $\mathrm{qGeom}(\theta)$, has distribution
$$
\PP(X=k)   = \frac{\theta^k}{(q;q)_{k}}(\theta;q)_{\infty} , \ \ \  k\in\Z_{\geqslant 0}.
$$
\label{def:qGeom}
\end{definition}
\begin{lemma} 
\label{lem:firstpartdynamics}
Assume that $\rho_{\circ}$ and $\rho_i$ are specializations into single variables $\diagq$ and $a_i$ respectively for all $i\geqslant 1$. 
\begin{enumerate}
\item For any  partitions $\mu\prech\kappa\in \Y_k$, the probability kernel $\Udiag_{k,k}(\pi\vert \kappa, \mu)$ is supported on $ \pi\in \Y_{k+1}$. Moreover, for any $j\leqslant k$, the push-forward of $\Udiag_{k,k}(\pi\vert \kappa, \mu)$ with respect to $\Proj_j$ (the marginal distribution of $\pi$ corresponding to its first $j$ parts) depends only on $\Proj_j(\kappa)$. 
\item The projection of the dynamics $\Udiag_{k,k}(\pi\vert \kappa, \mu)$ to the first part is given by 
$$ \pi_1 = \kappa_1 + \mathrm{qGeom}(\diagq a_{k+1}).$$
\end{enumerate}
\end{lemma}
\begin{proof}
For fixed $\mu\prech\kappa\in \Y_k$, $\Udiag(\pi\vert \kappa, \mu)$ is proportional to 
$$ P_{\pi/\kappa}(a_{k+1})\ve_{\pi}(\diagq)=\diagq^{\sum_i\pi_{2i-1}-\pi_{2i}} a_{k+1}^{\sum_i \pi_i-\kappa_i} \prod_{i=1}^{\ell(\pi)} \frac{1}{\qq{\pi_i-\kappa_i}\qq{\kappa_i-\pi_{i+1}}}\mathds{1}_{\kappa\prec\pi},$$
as follows from  Lemma \ref{lem:computeboundaryoperator}. By summing over $\pi_{j+1}, \pi_{j+2}, \dots$, we see that the distribution of $\Proj_j(\pi)$ depends only on $\Proj_j(\kappa)$. 
In particular, summing over $\pi_{2}, \dots ,\pi_{k+1}$, we find that the weight of $\pi_1$ is proportional to $$ \frac{\big(\diagq a_{k+1}\big)^{\pi_1}}{\qq{\pi_1-\kappa_1}} \mathds{1}_{\pi_1\geqslant \kappa_1},$$
which shows that $\pi_1-\kappa_1$ has the  $\mathrm{qGeom}(\diagq  a_{k+1})$ distribution. 
\end{proof}

\begin{definition}
We introduce the  \emph{$q$-inverse Gaussian} distribution (See Lemma \ref{lem:convqinversegaussian} for a justification of the name) with parameters $m\in \Z_{\geqslant 0}\cup\lbrace +\infty\rbrace$ and $\theta\in (0,1)$. A $q$-inverse Gaussian random variable $X$ is such that 
$$\PP(X=k) = \theta^k \frac{\qq{m}}{\qq{k}\qq{m-k}}\frac{1}{Z_m(\theta)},\ \ \ k\in \lbrace 0,  \dots, m\rbrace,$$
where $Z_m(\theta)$ is the $m$th Rogers-Szeg\H o polynomial
\begin{equation}
 Z_m(\theta) = \sum_{k=0}^{m} \theta^k \frac{\qq{m}}{\qq{k}\qq{m-k}}.
 \label{eq:RogersSzego}
\end{equation}
When $m=\infty$, the $q$-inverse Gaussian distribution degenerates to the $q$-geometric one. 
\label{def:qinverseGaussian}
\end{definition}

\begin{lemma}
\label{lem:lastpartdynamics}
Assume that $\rho_{\circ}$ and $\rho_i$ are specializations into single variables $\diagq$ and $a_i$, respectively, for all $i\geqslant 1$. 
\begin{enumerate}
\item For any $j\leqslant k$, and partitions $\mu\prech\kappa\in \Y_k$, the push-forward of  $\Udiag_{k,k}(\pi\vert \kappa, \mu)$ with respect to $\Projtilde_j$ (last $j$ parts of the partitions) depends only on $\Projtilde_j(\kappa)$. 
\item The projection of the dynamics $\Udiag_{k,k}(\pi\vert \kappa, \mu)$ on the last part is such that $\pi_{k+1}$ is a $q$-inverse Gaussian random variable with parameters $\kappa_{k}$ and $ \diagq^{(-1)^k}a_{k+1}$.  
\end{enumerate}
\end{lemma}
\begin{proof}
One proves (1) similarly to Lemma \ref{lem:firstpartdynamics}. Summing over $ \pi_1, \dots, \pi_k $ for a fixed $\kappa\in \Y_k$, the distribution of $\pi_{k+1}$ is proportional to
$$  \frac{\diagq^{(-1)^k\pi_{k+1}}a_{k+1}^{\pi_{k+1}}}{\qq{\pi_{k+1}} \qq{\kappa_k - \pi_{k+1}}},$$
which proves (2).
\end{proof}

\subsubsection{Bulk dynamics based on RSK row insertion}
\label{sec:rowinsertion}
For $j\geqslant 1$, consider four partitions $\kappa,\mu\in \Y_{j-1}$ and $\pi,\nu\in \Y_{j}$ such that
\begin{center}
\begin{tikzpicture}[scale=1.2]
\node (mu) at (0,0) {$\mu$};
\node (kappa) at (1,0) {$\kappa$};
\node (nu) at (0,1) {$\nu$};
\node (pi) at (1,1) {$\pi$};
\draw (0.5,0) node{$\prech$};
\draw (0.5,1) node{$\prech$};
\draw (0,0.5) node[rotate=90]{$\prech$};
\draw (1,0.5) node[rotate=90]{$\prech$};
\draw (1.5,0)node[gray]{$\ni$};
\draw (1.5,1)node[gray]{$\ni$};
\draw (2.2,0)node[gray]{$ \Y_{j-1} $};
\draw (2,1)node[gray]{$ \Y_{j} $};
\end{tikzpicture} 
\end{center}
Assume that these partitions appear on a section of a path $\pathh$ as in Figure \ref{URSK} so that a specialization $\rho_{i}$ lies on the horizontal edge and $\rho_j$ lies on the vertical edge. 
\begin{figure}
\begin{tikzpicture}[scale=0.8]
\begin{scope}[decoration={
	markings,
	mark=at position 0.5 with {\arrow{<}}}]
\draw[->, >=stealth', gray] (-1,0) -- (7, 0);
\draw[->, >=stealth', gray] (-1,0) -- (-1,4);
\fill (4,2) circle(0.05);
\fill (5,2) circle(0.05);
\fill (4,3) circle(0.05);
\fill (5,3) circle(0.05);
\draw (4,2) node[anchor = north]{$\mu$};
\draw (5,2) node[anchor = west]{$\kappa$};
\draw (4,3) node[anchor = south]{$\nu$};
\draw (5,3) node[anchor = south west]{$\pi$};
%\draw[thick] (3.5,3) -- (4,3) -- (4,2) -- (5,2) -- (5, 1.5);
\draw[fleche]  (4,3) -- (4,2);
\draw[fleche] (4,2) -- (5,2) ;
\draw[thick, dashed] (5, 2) -- (5,1) -- (5.5,1);
\draw[thick, dashed] (4, 3) -- (3, 3) -- (3,3.5);
\draw[dotted, thick,->, >=stealth'] (4,2) to[bend left] (5,3);
\draw[gray, dotted] (5, 3.5)  -- (5,0) node[anchor = north]{$i$};
\draw[gray, dotted] (-1,3) node[anchor= east]{$j$} -- (5.5,3) ;
\draw (4.5,2) node[anchor=north]{$ a_i$};
\draw (4,2.5) node[anchor=east]{$ a_j$};
\end{scope} 
\end{tikzpicture}
\caption{Action of the transition operators $\U_{\textrm{row}}(\pi|\kappa,\mu,\nu)$ and $\U_{\textrm{col}}(\pi|\kappa,\mu,\nu)$ on a corner formed by partitions $\kappa, \mu, \nu$ at positions $(i, j-1), (i-1, j-1), (i-1, j)$. The dashed part of the path has no influence on the distribution of $\pi$.}
\label{URSK}
\end{figure}
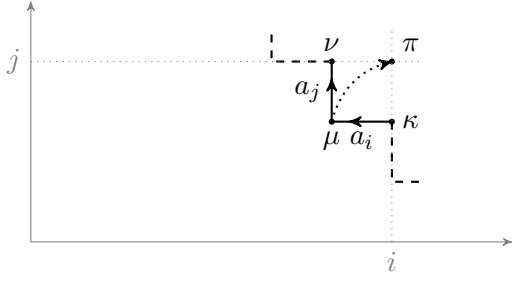
Assume that $\rho_i$ and $\rho_j$ are specializations into single variables $a_i$ and $a_j$. Then, \cite{matveev2015q} defines an operator that we will call
$\U_{\textrm{row}}(\pi|\kappa,\mu,\nu)$ which satisfies the relation in (\ref{eq:definingeq}).  We collect the following properties of $\U_{\textrm{row}}$ from \cite{matveev2015q}, wherein 
 $\U_{\textrm{row}}(\pi|\kappa,\mu,\nu)$ is denoted $\mathcal{U}_j(\nu\to \pi\vert \mu\to\kappa)$ when the sequence $\big(\mathcal{U}_j\big)$ defines the multivariate dynamics  $\mathcal{Q}^q_{\rm row}[\alpha] $, cf \cite[Section 6.2]{matveev2015q}.

\begin{proposition} With the above notation, 
\begin{enumerate}
\item $\U_{\textrm{row}}(\pi|\kappa,\mu,\nu)$ is supported on $\pi$ such that $\nu \prech \pi$ and $ \kappa\prec\pi $ \cite[Lemma 6.2]{matveev2015q}. 
\item $\U_{\textrm{row}}(\pi|\kappa,\mu,\nu)$ preserves the $q$-Whittaker process structure, in the sense of \eqref{eq:definingeq} \cite[Theorem 6.4]{matveev2015q}. 
\item The projection of the dynamics on the first $j$ parts of each partition is marginally Markov \cite[Section 6.3]{matveev2015q}. 
\end{enumerate}
\end{proposition}

\begin{definition}
For parameters $q,\xi,\eta \in \R_{>0}$, $y\in \{0,1,\ldots\}\cup \{+\infty\}$ and $s\in \{0,\ldots, y\}$, define the \emph{$q$-BetaBinomial} distribution, denoted $\mathrm{qBetaBin}(q, \xi, \eta, y)$,  by the weights
$$
\phidist_{q,\xi,\eta}(s|y) = \xi^s \frac{(\eta/\xi;q)_{s}(\xi;q)_{y-s}}{(\eta;q)_y}\frac{(q;q)_y}{(q;q)_s(q;q)_{y-s}}.
$$
\label{def:qbetabinomial}
\end{definition}
These  weights were introduced in \cite{gnedin2009analogue, povolotsky2013integrability}. There are several ranges of parameters for which the weights $\phidist_{q,\xi,\eta}(s|y)$ define a probability distribution on $\lbrace 0, \dots, y\rbrace$. The simplest choice is $0\leqslant q <1$ and $ 0\leqslant \eta<\xi<1 $, and the corresponding distribution is a natural $q$-analogue of the Beta-Binomial distribution \cite{gnedin2009analogue}. Another possibility, considered in \cite{matveev2015q}, is to use the weights $\phidist_{q^{-1},q^a,q^{a+b}}(s|y)$ where $a,b$ are non-negative integers such that $y\leqslant a+b$. We may consider also the degeneration when $b$ goes to infinity and denote the corresponding weights $\phidist_{q^{-1},q^a,0}(s|y)$. Note that $\phidist_{q,\xi,\eta}(s|y)$ degenerates to the $q$-Geometric distribution when $\eta=0$ and $y=+\infty$. 

The law of the first part marginals under the dynamics $\U_{\textrm{row}}$ can be described explicitly as follows. 
\begin{lemma}[{\cite[Section 6.3]{matveev2015q}}]
	Under the transition operator $\U_{\textrm{row}}(\pi|\kappa,\mu,\nu)$, 
	$$
	\pi_1 = \nu_1 + V + W
	$$
	where $V$ is distributed as a $q$-geometric random variable with parameter $a_i a_j$ and $W$ is distributed according to
	$$
	\PP(W=k) = \phidist_{q^{-1},q^{\nu_1-\mu_1},0}(k|\kappa_1-\mu_1);
	$$
$V$ and $W$ are independent. 
	\label{lem:bulkparticledynamicslambda1}
\end{lemma}

\subsubsection{Bulk dynamics based on RSK column insertion}
\label{sec:columninsertion}
We will also use another type  of dynamics  that we will denote 
$\U_{\textrm{col}}(\pi|\kappa,\mu,\nu)$, introduced in \cite{matveev2015q} wherein  it is denoted $\mathcal{U}_j(\nu\to \pi\vert \mu\to\kappa)$, when the sequence $\big(\mathcal{U}_j\big)$ defines the multivariate dynamics  $\mathcal{Q}^q_{\rm col}[\alpha] $, cf \cite[Section 6.4]{matveev2015q}. As for the dynamics $\U_{\textrm{row}}$, we will not provide the complete definition of $\U_{\textrm{col}}$ since we do not need it, but only list the following properties proved in \cite{matveev2015q}. 
\begin{proposition} Under the same assumptions as in Section \ref{sec:rowinsertion},
	\begin{enumerate}
		\item $\U_{\textrm{col}}(\pi|\kappa,\mu,\nu)$ is supported on $\pi$ such that $\nu \prech \pi$ and $ \kappa\prec\pi $ \cite[Lemma 6.6]{matveev2015q}. 
		\item $\U_{\textrm{col}}(\pi|\kappa,\mu,\nu)$ satisfies \eqref{eq:definingeq} \cite[Theorem 6.10]{matveev2015q}. 
		\item The projection of the dynamics on the last $j$ parts of each partition is marginally Markov \cite[Section 6.6]{matveev2015q}.
	\end{enumerate}
\end{proposition}

The dynamics on the last part of each partition under $\U_{\mathrm{col}}(\pi|\kappa,\mu,\nu)$ are explicit:
\begin{lemma}[{\cite[Section 6.6]{matveev2015q}}]
	Under the transition operator  $\U_{\mathrm{col}}(\pi|\kappa,\mu,\nu)$, 
	$$
	\pi_{k+1} = \nu_{k+1} +W
	$$
	where
	$$
	\PP(W= j) = \phidist_{q,a_i a_j,0}(j | \mu_{k}-\nu_{k+1}).
	$$
	\label{lem:bulkparticledynamicslambdan}
\end{lemma}
The distribution $\phidist_{q,\theta,0}(\cdot\vert m)$ can be seen as a truncation of the $q$-geometric distribution. It is a different truncation than the $q$-inverse Gaussian distribution (see Definition \ref{def:qinverseGaussian}).

\subsection{New exactly-solvable particle systems}
\label{sec:newprocesses}
The dynamics studied in Section \ref{sec:rskdynamics} suggest the definition of new exactly-solvable particle systems, which are ``half-space'' variants of the Geometric $q$-PushTASEP (introduced in \cite{matveev2015q}) and the discrete Geometric $q$-TASEP (introduced in \cite{borodin2015discrete}). For both particle systems, we can use the moment  formulas from Corollaries  \ref{cor:momentsqWhittaker}, \ref{cor:smallcontoursmoments1} and  \ref{cor:smallcontoursmomentsn} and the  Laplace transform formulas from Corollaries  \ref{cor:LaplaceqWhittaker} and \ref{cor:LaplaceqWhittakerlambdan}. This provides moment and Laplace transform formulas for the random variable $q^{\pm x_n(t)}$ where $x_n(t)$ will denote the position of an arbitrary  particle at any time $t$ in the models that we define now.

 We start with the particle system corresponding to the row insertion dynamics (Section \ref{sec:rowinsertion}). 
\begin{definition} The \emph{Geometric $q$-PushTASEP with particle creation} is a discrete-time Markov process on configurations of particles 
$$0=x_0(t)< x_1(t) <x_2(t) < \dots < x_t(t) < \infty.$$
At time $0$, there is only one particle at $x_0(0)=0$. From time $t$ to $ t+1$, particles' positions $x_n(t)$ for $n=1$ to $t$ are sequentially updated so that the particle at $ x_n(t)$ jumps to its new location $x_n(t+1)=x_n(t)+V_n(t)+ W_n(t)$ where $V_n(t)$ and $W_n(t)$ are independent,  $V_n(t)$ is a $q$-geometric random variable with parameter $a_n a_t$ (see Definition \ref{def:qGeom}), and $W_n(t)$ is distributed according to
$$
\PP(W_n(t)=k) = \phidist_{q^{-1},q^{\gap_n(t)},0}(k|x_{n-1}(t+1)- x_{n-1}(t)).
$$
where $\gap_n(t) = x_n(t)-x_{n-1}(t)-1$. Additionally, a new particle is created at location $x_{t+1}(t+1) = x_t(t+1) + V_{t+1}(t)+1$  where $V_{t+1}(t)$ is an independent $q$-geometric random variable with parameter $\diagq a_t$ (see Figure \ref{fig:GeometricPushTASEP}). Note that the strict ordering of particle locations is preserved by the dynamics.
\label{def:qPushTASEPwithcreation}
\end{definition}
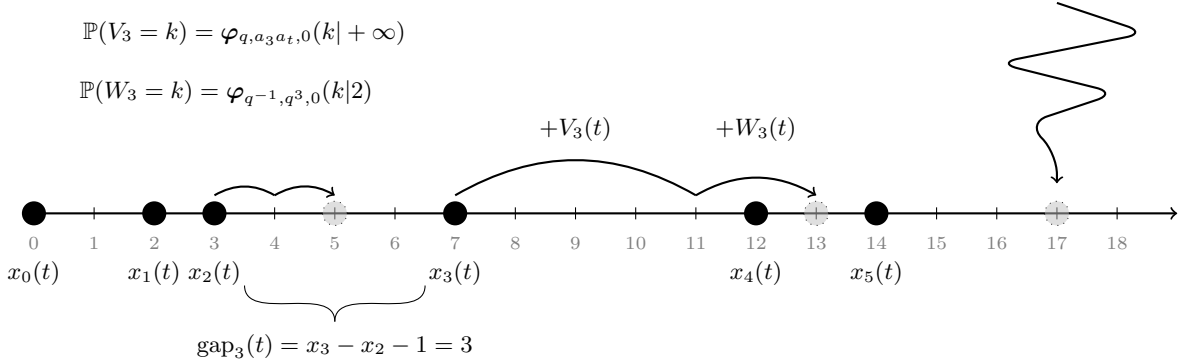
\begin{figure}
	\begin{center}
		\begin{tikzpicture}[scale=0.8]
		\draw[thick, ->] (0,0) -- (19,0);
		\foreach \k in {0, ..., 18}
		{\draw (\k, -0.1) -- (\k,0.1);
		\draw[gray] (\k, -0.5) node{{\tiny $\k$}};}
		\fill (0,0) circle(0.2);
		\fill (2,0) circle(0.2);
	    \fill (3,0) circle(0.2);	
		\fill (7,0) circle(0.2);
		\fill (12,0) circle(0.2);
		\fill (14,0) circle(0.2);
		
		\draw[dotted] (5,0) circle(0.2);
		\fill[black!25, opacity=0.5] (5,0) circle(0.2);
		
		\draw[dotted] (13,0) circle(0.2);
		\fill[black!25, opacity=0.5] (13,0) circle(0.2);
		
		\draw[dotted] (17,0) circle(0.2);
		\fill[black!25, opacity=0.5] (17,0) circle(0.2);
		
		\draw[thick] node{} (7,0.3) to[bend left] node{} (11,0.3);
		\draw[thick, ->] node{} (11,0.3) to[bend left] node{} (13,0.3);
		
		\draw[thick] node{} (3,0.3) to[bend left] node{} (4,0.3);
		\draw[thick, ->] node{} (4,0.3) to[bend left] node{} (5,0.3);
		
		\draw (9,1.4) node{ \footnotesize{$+V_3(t)$}};
		\draw (12,1.4) node{ \footnotesize{$+W_3(t)$}};
		\draw (3.5, -1.2) to[in=90, out=-90] (5,-1.8);
		\draw (6.5, -1.2) to[in=90, out=-90] (5,-1.8);
		
		\draw (5, -2.2) node{\footnotesize{$\mathrm{gap}_3(t)=x_3-x_2-1= 3$}};
		
		\draw (0,-1) node {\footnotesize{$ x_0(t)$}};
		\draw (2,-1) node {\footnotesize{$ x_{1}(t)$}};
		\draw (3,-1) node {\footnotesize{$ x_{2}(t)$}};
		\draw (7,-1) node {\footnotesize{$ x_{3}(t)$}};
		\draw (12,-1) node {\footnotesize{$ x_{4}(t)$}};
		\draw (14,-1) node {\footnotesize{$ x_{5}(t)$}};
		
		\draw (3.2,2) node{{\footnotesize $\PP(W_3=k) = \phidist_{q^{-1}, q^3, 0}(k\vert 2)$}};
		\draw (3.5,3) node{{\footnotesize $\PP(V_3=k) = \phidist_{q, a_3a_t, 0}(k \vert +\infty)$}};
		\draw[thick, ->, rounded corners=8pt] (17,3.5) -- (18.5,3) -- (16,2.5) -- (18,2) -- (16.5, 1.5) -- (17,1) -- (17,0.5);
		\end{tikzpicture}
	\end{center}
	\caption{Configuration of particles in the Geometric $q$-PushTASEP with particle creation, at time $t=5$. We illustrate possible jumps of the particles sitting at $x_2(t)$ and $x_3(t)$. Note that the particle sitting at $x_4(t)$ will necessary jump to a position on the right of $x_3(t+1)$ (by the construction of $W_4(t)$). A new particle will be created at $x_{t+1}(t+1)$($=17$ on the figure).}
\label{fig:GeometricPushTASEP}
\end{figure}

\begin{proposition}
	Let $\big(x_n(t)\big)_{1\leqslant n\leqslant t}$ denote the positions of particles in the Geometric $q$-PushTASEP with particle creation (Definition \ref{def:qPushTASEPwithcreation}). Let $\big(\lambda^{(t,n)}\big)_{1\leqslant n\leqslant t}$ be a sequence of random partitions distributed according to the Markovian growth  procedure described in Section \ref{sec:Macdyn}, using transition operators $\U_{\textrm{row}}$ and $\Udiag$ , and specializations $\rho_{\circ}$ and $\rho_i$ into single variables $\diagq$ and $a_i$, respectively, for all $i\geqslant 1$. In particular, $\lambda^{(t,n)}$ is distributed according to the half-space $q$-Whittaker measure $\PQWM_{(a_1, \dots, a_n),(\diagq, a_{n+1}, \dots, a_t)}$. Then, for any admissible path $\pathh\in \admpath$, we have 
		$$\big(x_n(t)\big)_{(t,n)\in \pathh}  \overset{(d)}{=}\big(\lambda_1^{(t,n)}+n\big)_{(t,n)\in \pathh}.$$ 
	\label{prop:qPushTASEPwithcreation}
\end{proposition}
\begin{proof}
The two families have the same dynamics, according to Lemmas \ref{lem:firstpartdynamics} and \ref{lem:bulkparticledynamicslambda1}. 
\end{proof}
\begin{remark}
The moments of $q^{-x_n(t)}$ are given by \eqref{eq:qmomentslambda1} in Corollary \ref{cor:momentsqWhittaker}. By definition of the dynamics, $x_n(t)$ can be bounded by a sum of $q$-Geometric random variables with parameters $\diagq a_{t'}$ and $a_{n'}a_{t'}$ over some range of $t', n'$.  However, for $X\sim \mathrm{qGeom}(\theta)$, $\EE[q^{-kX}]$ is finite as long as 
$q^k>\theta$. This explains why the moments of $q^{-\lambda_1}$ in Corollary \ref{cor:momentsqWhittaker} can exist only when  $q^{k}> \max\lbrace a_j\rbrace ^2 $ and $q^{k}> \max\{ \alpha_j\} \max\{a_j\}$. It turns out that these conditions are sufficient for the existence of moments. 
\end{remark}

We introduce now the particle system corresponding to the column insertion dynamics (Section \ref{sec:columninsertion}). 
\begin{definition} 
The \emph{Geometric $q$-TASEP with activation} is a discrete-time Markov process on infinite configurations of particles 
 $$ \dots <x_n(t)<\dots<x_1(t)<x_0(t)\equiv +\infty.$$
At time $0$, $x_n(0)=-n$ for all $n\geqslant 1$.
From time $t$ to time $t+1$ the first $t$ particles at locations $x_1(t), \dots, x_t(t)$ are updated in parallel so that the $n$th particle jumps to $x_n(t+1) = x_n(t)+V_n(t)$ where 
 $V_n(t)$ is distributed according to 
 $$
\PP(V_n(t)= j) = \phidist_{q,a_n a_t,0}\big(j | x_{n-1}(t) - x_{n}(t)-1\big).
$$
Additionally, the $(t+1)$th particle jumps to the location $x_{t+1}(t+1) = -t-1 +W(t)$
where $W(t)$ is distributed according to the $q$-inverse Gaussian distribution with parameters $m= x_t(t+1)+t$ and $\theta $ equals $ a_t \diagq $ or $ a_t/\diagq $ depending on the parity of $t$:
$$\PP(W(t)=k) = \big( \diagq^{(-1)^{t-1}}a_{t} \big)^k \frac{\qq{m}}{\qq{k}\qq{m-k}}\frac{1}{Z_m\big( \diagq^{(-1)^{t-1}}a_{t}\big)}.$$
All participating random variables are independent. The particle system is depicted in Figure \ref{fig:GeometricTASEP}.
\label{def:qTASEPwithactivation} 
\end{definition}
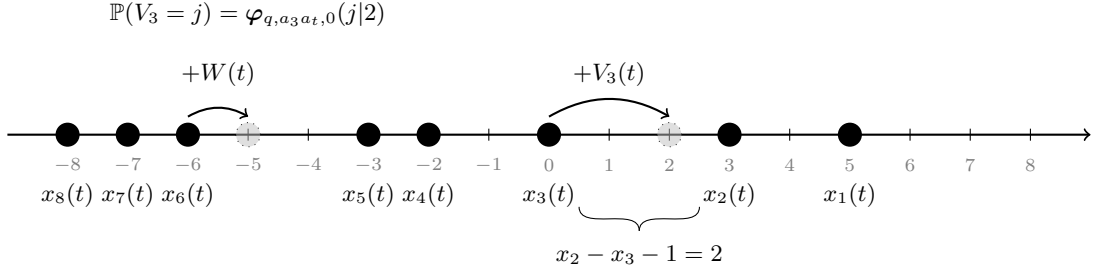
\begin{figure}
\begin{center}
	\begin{tikzpicture}[scale=0.8]
		\draw[thick, ->] (-9,0) -- (9,0);
		\foreach \k in {-8, ..., 8}
		{\draw (\k, -0.1) -- (\k,0.1);
		\draw[gray] (\k, -0.5) node{{\tiny $\k$}};}
		\fill (-8,0) circle(0.2);
		\fill (-7,0) circle(0.2);
		\fill (-6,0) circle(0.2);
		\fill (-3,0) circle(0.2);
		\fill (-2,0) circle(0.2);
		\fill (0,0) circle(0.2);
		\fill (3,0) circle(0.2);
		\fill (5,0) circle(0.2);
		
		\draw[dotted] (-5,0) circle(0.2);
		\fill[black!25, opacity=0.5] (-5,0) circle(0.2);
		
		\draw[dotted] (2,0) circle(0.2);
		\fill[black!25, opacity=0.5] (2,0) circle(0.2);

		\draw (0.5, -1.2) to[in=90, out=-90] (1.5,-1.6);
		\draw (2.5, -1.2) to[in=90, out=-90] (1.5,-1.6);
		\draw (1.5, -2.0) node{\footnotesize{$x_2-x_3-1= 2$}};
			
		\draw (-8,-1) node {\footnotesize{$ x_8(t)$}};
		\draw (-7,-1) node {\footnotesize{$ x_{7}(t)$}};
		\draw (-6,-1) node {\footnotesize{$ x_{6}(t)$}};
		\draw (-3,-1) node {\footnotesize{$ x_{5}(t)$}};
		\draw (-2,-1) node {\footnotesize{$ x_{4}(t)$}};
		\draw (0,-1) node {\footnotesize{$ x_{3}(t)$}};
		\draw (3,-1) node {\footnotesize{$ x_{2}(t)$}};
		\draw (5,-1) node {\footnotesize{$ x_{1}(t)$}};
		
		\draw[thick, ->] node{} (-6,0.3) to[bend left] node{} (-5,0.3);
		\draw (-5.5,1.0) node{ \footnotesize{$+W(t)$}};
		
		\draw[thick, ->] node{} (0,0.3) to[bend left] node{} (2,0.3);
		\draw (1,1.0) node{ \footnotesize{$+V_3(t)$}};
		
	%	\draw (-5,2) node{{\footnotesize $\PP(W_3=k) = \phidist_{q^{-1}, q^3, 0}(k\vert 2)$}};
		\draw (-5,2) node{{\footnotesize $\PP(V_3=j) = \phidist_{q, a_3a_t, 0}(j \vert 2)$}};
\end{tikzpicture}
\end{center}
\caption{Configuration of particles in the Geometric $q$-TASEP with activation at time $t=5$. We illustrate a possible jump of the particle sitting at $x_3(t)$ and the activation of the 6th particle. }
\label{fig:GeometricTASEP}
\end{figure}
\begin{proposition}
Let $\big(x_n(t)\big)_{1\leqslant n\leqslant t}$ denote the positions of particles in the Geometric $q$-TASEP with activation (Definition \ref{def:qTASEPwithactivation}). 
Let $\big(\lambda^{(t,n)}\big)_{1\leqslant n\leqslant t}$ be a sequence of random partitions distributed according to the Markovian growth  procedure described in Section \ref{sec:Macdyn}, using transition operators $\U_{\textrm{row}}$ and $\Udiag$ , and specializations $\rho_{\circ}$ and $\rho_i$ into single variables $\diagq$ and $a_i$, respectively, for all $i\geqslant 1$. In particular, $\lambda^{(t,n)}$ is distributed according to the half-space $q$-Whittaker measure $\PQWM_{(a_1, \dots, a_n),(\diagq, a_{n+1}, \dots, a_t)}$. Then, for any admissible path $\pathh\in \admpath$, we have 
$$\big(x_n(t)\big)_{(t,n)\in \pathh}  \overset{(d)}{=}\big(\lambda_n^{(t,n)}-n\big)_{(t,n)\in \pathh}.$$ 
\end{proposition}
\begin{proof}
	The two families have the same dynamics, according to Lemmas \ref{lem:firstpartdynamics} and \ref{lem:bulkparticledynamicslambdan}. 
\end{proof}

\thispagestyle{plain}
 \section{Half-space Hall-Littlewood process}
\label{sec:HL}

We assume now that $q=0$ and $t\in (0,1)$. We use again the same notations $P,Q,\ve$ as before. These ($P$ and $Q$) are now called the \emph{Hall-Littlewood symmetric functions}.
We define the \emph{half-space Hall-Littlewood measure} $\PHL_{\rhoup, \rhodiag}$ as the measure on partitions $\la\in \Y$ such that 
$$ \PHL_{\rhoup, \rhodiag}(\lambda) = \frac{P_{\lambda}(\rhoup) \ve_{\lambda}(\rhodiag)}{\Pi(\rhoup;\rhodiag)\Phi(\rhoup)},$$
and denote by $\PHL_{\pathh, \bm\uprho}$, for a sequence of specializations $\bm\uprho$, the \emph{half-space Hall-Littlewood process}, i.e.,  the $q=0$ degeneration of the half-space Macdonald process $\PMP_{\pathh, \bm\uprho}$.

\subsection{Observables and integral formulas}

Consider a Hall-Littlewood measure where $\rhoup=(a_1, \dots, a_n) \in( 0,1)^n$, and $\rhodiag= \rho(\alpha, \beta, \gamma)$ as defined in Section \ref{sec:specializations}. 
We further assume that all the parameters $\alpha_i$ are such that $\max\lbrace a_i\rbrace \max\lbrace \alpha_j\rbrace <1$ so that the measure is well-defined. We will be mostly interested in the distribution of the length of a Hall-Littlewood random partition $\lambda$, denoted $\ell(\la)$ in the following.

\begin{corollary}
	Let $r$ be a positive integer.  We have 
	\begin{multline} 
	\EHL_{(a_1,\ldots, a_n),\rho}\Big[e_r(t^{n-\ell(\la)-1},t^{n-\ell(\la)-2},  \dots, t^0)\Big]  \\  = 
	\frac{1}{ r!} \oint\frac{\mathrm{d}z_1}{2\I\pi}\cdots \oint\frac{\mathrm{d}z_r}{2\I\pi} \det\left[ \frac{1}{tz_k-z_l}\right]_{k,l=1}^r \prod_{1\leqslant i<j\leqslant r} \frac{1-tz_iz_j}{1-z_iz_j} 	\prod_{i=1}^r \frac{1-tz_i^2}{1-z_i^2} \\ \times \prod_{j=1}^{r} \left( \frac{1}{\Pi(z_j ; \rho)} \prod_{i=1}^n \left( \frac{tz_j-a_i}{z_j-a_i}  \frac{1-z_ja_i}{1-tz_ja_i}\right)\right),
	\label{eq:momentsHL}
	\end{multline}
	where the positively oriented contours encircle $\lbrace a_1, \dots, a_n\rbrace$ and no other singularity of the integrand. 
	\label{cor:tmomentsHL}
	\end{corollary}
	\begin{proof}
		This is the $q=0$ degeneration of Proposition \ref{prop:Dr}. Indeed, when $q=0$, 
		$$ e_r\big(q^{\lambda_1}t^{n-1} ,  q^{\lambda_2} t^{n-2} ,  \cdots,  q^{\lambda_n}\big) = e_r(t^{n-\ell(\la)-1}, \dots, t^0).$$
	\end{proof}
Notice that 
$$ \sum_{r=0}^n u^r e_r(t^{n-\ell-1}, \dots, t^0) = \prod_{i=0}^{n-\ell-1} (1+u t^i) = \frac{(-u;t)_{\infty}}{(-ut^{n-\ell};t)_{\infty}}. $$
The coefficient of $u^r$ can be extracted from the R.H.S. using the $q$-binomial theorem  \eqref{eq:qbinomial}:
$$ e_r(t^{n-\ell-1}, \dots, t^0) = (-1)^r(t^{n-\ell})^r\frac{(t^{\ell-n}; t)_r}{(t; t)_r}.$$

\begin{remark}
	We may take a generating series of \eqref{eq:momentsHL} and obtain a formula for $\EHL\left[\frac{(u;t)_{\infty}}{(ut^{n-\ell(\la)};t)_{\infty}}\right]$. However, this is not a convenient observable to study the distribution of $\ell(\la)$ in asymptotic regimes where  $n-\ell(\la)$ tends to $+\infty$, since it would require to scale $u$ to $+\infty$, and the prefactor $(u;t)_{\infty}$ would diverge. A similar issue was encountered in \cite{tracy2009asymptotics} wherein the formula their in Equation (2) is not directly amenable for asymptotic analysis (see also \cite[Theorem 5.5]{borodin2012duality}). In \cite{barraquand2018stochastic}, we show that in the special case where $\rhoup$ is trivial, it is possible to compute $\EHL\left[\frac{1}{(ut^{n-\ell(\la)};t^2)_{\infty}}\right]$ as a multiplicative functional of the Pfaffian Schur process, which can be written explicitly as contour integrals and the resulting formulas are amenable to asymptotic analysis. 
\end{remark}

  We need to slightly change contours so as to obtain an integral formula for a better observable, following similar lines to  \cite{borodin2012duality,borodin2016higher,dimitrov2016kpz}.
  For simplicity, we will focus on the case where the specialization $\rho$ is a pure alpha specialization (otherwise we would need restrictive assumptions on the parameters $\beta_i$) of the form $(\alpha_1, \dots, \alpha_r)$. 
  Define the function 
  $$ \mathcal{G}^t(w) = \prod_{i=1}^{r} (1-\alpha_i w) \prod_{i=1}^n\frac{1-w a_i}{1-w/a_i},$$
  so that (cf. \eqref{eq:momentsHL})
  $$ \frac{\mathcal{G}^t(w)}{\mathcal{G}^t(tw)} =\frac{1}{\Pi(w; \rho)} \prod_{i=1}^n \frac{tw-a_i}{w-a_i} \frac{1-w a_i}{1-t wa_i}.$$
  
\begin{proposition}For any $m\geqslant 1$, 
\begin{multline}
\EHL_{(a_1,\ldots, a_n),\rho}\left[ \left( t^{n-\ell(\la)} \right)^m \right] =   t^{\frac{m(m-1)}{2}} \oint_{C_1} \frac{\mathrm{d}z_1}{2\I\pi} \dots  \oint_{C_m} \frac{\mathrm{d}z_m}{2\I\pi}  \  \prod_{1\leqslant i<j\leqslant m}\frac{z_i-z_j}{z_i- tz_j} \frac{1-tz_iz_j}{1-z_iz_j}  \\ 	\times \prod_{j=1}^m \frac{1}{z_j} \frac{1-tz_j^2}{1-z_j^2}\frac{\mathcal{G}^t(z_j)}{\mathcal{G}^t(tz_j)}
\label{eq:momentsHLnested}
\end{multline}
where the positively oriented contours $C_1, \dots, C_m$ all enclose $0$ and the $a_i$ and are contained in the open disk of radius $1$ around zero, and the contours are nested in such a way that for $i<j$ the contour $C_i$ does not include any part of $t C_j$ (See Figure \ref{fig:nestedcontours6v} for a possible choice of such contours). 
\label{prop:momentsHL}
\end{proposition}
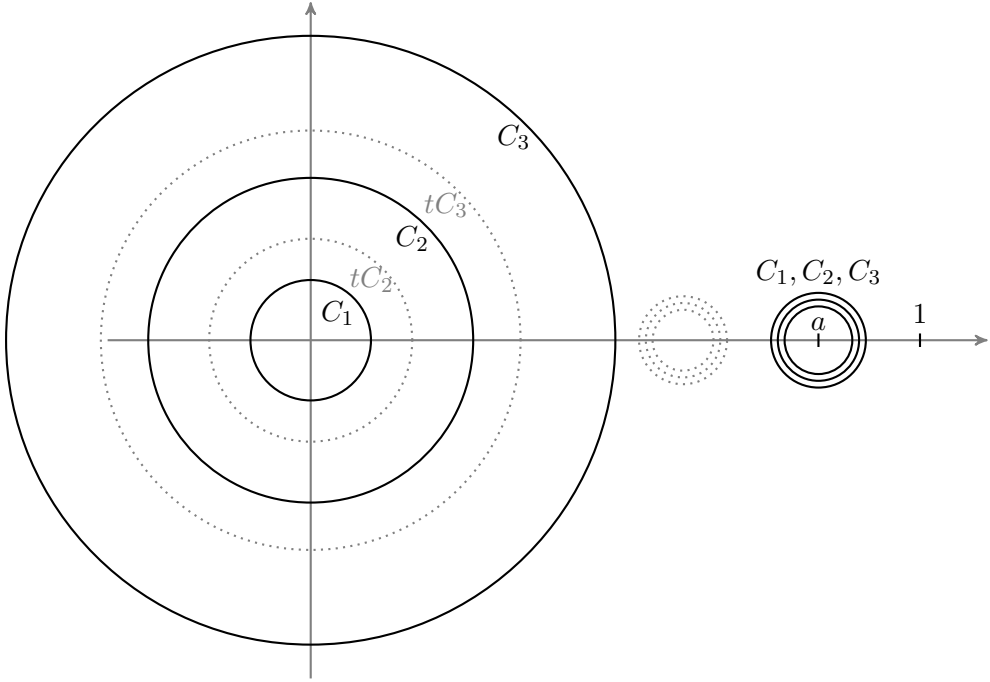
\begin{figure}
\begin{tikzpicture}[scale=0.9]
\draw[axis] (-3,0) -- (10,0);
\draw[axis] (0,-5) -- (0,5);
\draw[thick] (9,-.1) -- (9,.1) node[above]{$1$};
\draw[thick] (7.5,-.1) -- (7.5,.1);
\draw (7.5, 0.25) node{$a$};
\draw[thick] (0,0) circle(.89);
\draw (.4,.4) node{$C_1$};
\draw[thick, dotted, gray] (0,0) circle(1.5);
\draw[gray] (.9,.9) node{$tC_2$};
\draw[thick] (0,0) circle(2.4);
\draw (1.5,1.5) node{$C_2$};
\draw[thick, dotted, gray] (0,0) circle(3.1);
\draw[gray] (2,2) node{$tC_3$};
\draw[thick] (0,0) circle(4.5);
\draw (3,3) node{$C_3$};
\draw[thick] (7.5,0) circle(.5);
\draw[thick] (7.5,0) circle(.6);
\draw[thick] (7.5,0) circle(.7);
\draw[thick, gray, dotted] (5.5,0) circle(.45);
\draw[thick, gray, dotted] (5.5,0) circle(.55);
\draw[thick, gray, dotted] (5.5,0) circle(.65);
\draw (7.5,1) node{$C_1, C_2, C_3$};
\end{tikzpicture}
\caption{Possible choice of nested  contours  $C_1, C_2, C_3$ for Proposition \ref{prop:momentsHL} in a case where $a_1=a_2=a_3= a$. Each contour $C_i$ is the union of a circle around $0$ and a circle around $a$. The dotted contours are the images of $C_i$'s under $z\mapsto tz$.}
\label{fig:nestedcontours6v}
\end{figure}

\begin{proof}
Let $E_r= \EHL\left[ (-1)^r(t^{n-\ell(\la)})^r\frac{(t^{\ell(\la)-n}; t)_r}{(t; t)_r}\right]$. As a formal power series in $u$, we have that 
$$ \EHL\left[\frac{1}{(ut^{n-\ell(\la)};t)_{\infty}}\right] = \sum_{k=0}^{\infty} \frac{u^k}{(t;t)_k} \EHL\left[\frac{(u; t)_{\infty}}{(ut^{n-\ell(\la)};t)_{\infty}}\right] = \sum_{r=0}^n \sum_{k=0}^{\infty} u^{r+k}\frac{(-1)^rE_r}{(t;t)_k} = \sum_{m=0}^{\infty} u^m \sum_{k=0}^m \frac{(-1)^k E_k}{(t;t)_{m-k}}. $$
Using the Cauchy determinant evaluation in Corollary \ref{cor:tmomentsHL}, we may write that for all $k$,
$$ (t;t)_{k} E_k =  \frac{t^{k(k-1)/2} (t;t)_{k}(-1)^k}{ k! (1-t)^k} \oint\frac{\mathrm{d}z_1}{2\I\pi}\cdots \oint\frac{\mathrm{d}z_k}{2\I\pi} \prod_{i\neq j}\frac{z_i-z_j}{tz_i-z_j} \frac{F_k(z_1, \dots, z_k)}{z_1 \dots z_k},  $$
where $F_k$ is a symmetric meromorphic function in $k$ variables such that $F_k$ has no pole in any variable in some disk around zero and for all $n<m$, $F_m(z_1, \dots, z_n, 0, \dots, 0) = F_n(z_1, \dots, z_n)$, and the contours are as in the statement of the corollary. 
Using the symmetrization identity \eqref{eq:symidentity} to de-symmetrize the integrand, 
we have 
\begin{equation*}
(t;t)_{k} E_k =  (-1)^k t^{\frac{k(k-1)}{2}} \oint\frac{\mathrm{d}z_1}{2\I\pi}\cdots \oint\frac{\mathrm{d}z_k}{2\I\pi} \prod_{i<j}\frac{z_i-z_j}{tz_i- z_j} \frac{F_k(z_1, \dots, z_k)}{z_1 \dots z_k}, 
\end{equation*}
so that we can write 
\begin{multline}
\EHL\left[\frac{1}{(ut^{n-\ell(\la)};t)_{\infty}}\right]    \\ =   \sum_{m=0}^{\infty} u^m \frac{t^{m(m-1)/2}}{(t;t)_m}   \sum_{k=0}^m \binomt{m}{k} \frac{t^{k(k-1)/2}}{t^{m(m-1)/2}}   \oint_C\frac{\mathrm{d}z_1}{2\I\pi} \cdots \oint_C\frac{\mathrm{d}z_k}{2\I\pi} \prod_{i<j}\frac{z_i-z_j}{z_i-tz_j} \frac{F_k(z_1, \dots, z_k)}{z_1 \dots z_k},
\label{eq:formalseriesHL}
\end{multline} 
where the contour $C$ encloses $\lbrace a_1, \dots, a_n\rbrace$ and no other singularity. 
\begin{lemma}
For any $k\geqslant 0$, let $F_k$ be a symmetric meromorphic function in $k$ variables such that $F_k$ has no pole in any variable in some disk around zero, and for all $n<m$, $F_m(z_1, \dots, z_n, 0, \dots, 0) = F_n(z_1, \dots, z_n)$. 
Then we have 
\begin{multline*}
 \oint_{C_1} \frac{\mathrm{d}z_1}{2\I\pi} \dots  \oint_{C_m} \frac{\mathrm{d}z_m}{2\I\pi} \prod_{i<j}\frac{z_i-z_j}{z_i-tz_j} \frac{F_m(z_1, \dots, z_m)}{z_1 \dots z_m}  \\  = 
\sum_{k=0}^m \binomt{m}{k} t^{k(k-1)/2-m(m-1)/2} \oint_{C} \frac{\mathrm{d}z_1}{2\I\pi} \dots  \oint_{C} \frac{\mathrm{d}z_k}{2\I\pi} \prod_{i<j}\frac{z_i-z_j}{z_i-tz_j} \frac{F_k(z_1, \dots, z_k)}{z_1 \dots z_k}, 
 \end{multline*} 
 where $C$ is an arbitrary contour not encircling $0$, and the contours $C_1, \dots, C_m$ are defined as follows: Let $C_0$ be a small positively oriented circle around $0$, and let $r>t^{-1}$ be such that $t C$ does not intersect  $r^m C_0$ and $r^m C_0$ does not encircle poles of $F_m$. Then $C_j$ is defined as  the union of $r^j C_0$ and $C$.   
\label{lem:changecontours}
\end{lemma}
\begin{proof}
The statement is the same as \cite[Lemma 4.21]{borodin2012duality} except that the function $F$ was of the form $F_k(z_1, \dots, z_k) = \prod_{i=1}^kf(z_i)$. The proof extends to our case without any modification.  
\end{proof}
Hence, extracting the coefficient of $u^m$ in \eqref{eq:formalseriesHL} and using Lemma \ref{lem:changecontours}, we obtain  
$$ \EHL\left[ \left( t^{n-\ell(\la)} \right)^m \right] =   t^{\frac{m(m-1)}{2}} \oint_{C_1} \frac{\mathrm{d}z_1}{2\I\pi} \dots  \oint_{C_m} \frac{\mathrm{d}z_m}{2\I\pi} \prod_{i<j}\frac{z_i-z_j}{z_i- tz_j} \frac{F_m(z_1, \dots, z_m)}{z_1 \dots z_m},$$
which concludes the proof of Proposition \ref{prop:momentsHL}.
\end{proof}

\begin{proposition}For any $k\geqslant 1$, 
\begin{multline*}
	\EHL_{(a_1,\ldots, a_n),\rho}\left[ \left( t^{n-\ell(\la)} \right)^k \right]  =(t;t)_k \sum_{\mu \vdash k} \frac{1}{m_1!m_2!\dots}  \oint_{\mathcal C}\frac{\mathrm{d}w_1}{2\I\pi} \dots \oint_{\mathcal C}\frac{\mathrm{d}w_{\ell(\mu)}}{2\I\pi} \\  \times 
	\det\left( \frac{1}{w_j - w_it^{\mu_i}}\right)    \prod_{1\leqslant a<b\leqslant \ell(\mu)} \frac{(t^{\mu_a}w_a w_b)_{\infty}(t^{\mu_b}w_a w_b)_{\infty}}{(w_a w_b)_{\infty}(t^{\mu_a+ \mu_b}w_a w_b)_{\infty}} \\ \times 
	\prod_{j=1}^{\ell(\mu)} \frac{(t^{\mu_j}w_j^2)_{\infty}}{(w_j^2)_{\infty}}\frac{(tw_i^2; t^2)_{\infty}}{(w_i^2t^{2\mu_i+1}; t^2)_{\infty}} \frac{\mathcal{G}^t(w_j)}{\mathcal{G}^t(tw_j)},
\end{multline*}
where the positively oriented contour $\mathcal C$ contains $0$, the $a_i$'s and its image by multiplication by $t$, and is contained in the open disk of radius $1$ around zero.
\label{prop:momentsHLdet}
\end{proposition}
\begin{proof}
	For
	$$(z_1, \dots, z_k)=(w_1, tw_1, \dots, t^{\lambda_1-1}w_1, w_2, \dots, t^{\la_2-1}w_2, \dots, t^{\la_m-1}w_{m}), $$
	an adaptation of Lemma \ref{lem:evaluationstring} shows that 
	\begin{equation}
	\prod_{1\leqslant a<b\leqslant k}\frac{1 - t z_az_b}{1 - z_az_b} \prod_{i=1}^k \frac{1-tz_i^2}{1-z_i^2} =   \prod_{1\leqslant a<b\leqslant m} \frac{(t^{\la_a}w_a w_b; t)_{\infty}(t^{\la_b}w_a w_b; t)_{\infty}}{(w_a w_b; t)_{\infty}(t^{\la_a+ \la_b}w_a w_b; t)_{\infty}} \prod_{j=1}^{m} \frac{(t^{\lambda_j}w_j^2; t)_{\infty}}{(w_j^2; t)_{\infty}}\frac{(tw_i^2; t^2)_{\infty}}{(w_i^2t^{2\la_i+1}; t^2)_{\infty}},
	\label{eq:simplification2}
	\end{equation}
	and the L.H.S. is symmetric in the $z_i$. Then the formulas follows by expanding the nested contours to $C_{0,a}$,  using (a slight variant of) Proposition \ref{prop:contourshift} to collect the residues -- see also \cite[Proposition 5.2]{borodin2012duality}).
\end{proof}
\begin{remark}
Unlike the $q$-Whittaker case, it is not clear how to take a generating series of the above moment formulas. This is because the term $$\prod_{1\leqslant a<b\leqslant k} \frac{(t^{\la_a}w_a w_b)_{\infty}(t^{\la_b}w_a w_b)_{\infty}}{(w_a w_b)_{\infty}(t^{\la_a+ \la_b}w_a w_b)_{\infty}}$$
is typically of size $e^{c k^2}$ as $k$ grows. In the $q$-Whittaker case, a similar factor is present as well (cf \eqref{eq:expansionNoumiqWhittaker}), but one could argue that only finitely many integral terms are non-zero when taking the moment generating series. This would not be the case here. 
\end{remark}

\subsection{Half-space stochastic six-vertex model}

Recall the definition of the half-space stochastic six vertex model (Definition \ref{def:sixvertex}). The height function in the half-space stochastic six-vertex model is related to length of partitions in the half-space Hall-Littlewood processes in the following way. 
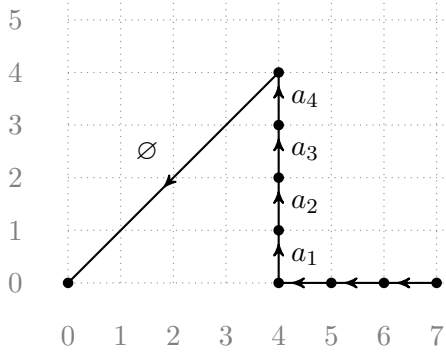
\begin{figure}
\begin{tikzpicture}[scale=0.7]
\foreach \x in {0, ..., 7} {
	\draw[gray] (\x,-1) node{$\x$};
}
\foreach \x in {0, ..., 5} {
	\draw[gray] (-1, \x) node{$\x$};
}
\begin{scope}[decoration={
	markings,
	mark=at position 0.5 with {\arrow{<}}}]
\draw[grille] (0,0) grid(7.2, 5.5);
\draw[fleche] (0,0) -- (4,4);
\draw[fleche] (4,4) -- (4,3);
\draw[fleche] (4,3) -- (4,2);
\draw[fleche] (4,2) -- (4,1);
\draw[fleche] (4,1) -- (4,0);
\draw[fleche] (4,0) -- (5,0);
\draw[fleche] (5,0) -- (6,0);
\draw[fleche] (6,0) -- (7,0);
\draw (4.5, 0.5) node{$a_1$};
\draw (4.5, 1.5) node{$a_2$};
\draw (4.5, 2.5) node{$a_3$};
\draw (4.5, 3.5) node{$a_4$};
\draw (1.5, 2.5) node{$\varnothing$};
\fill (0,0) circle(0.1);
\fill (4,4) circle(0.1);
\fill (4,3) circle(0.1);
\fill (4,2) circle(0.1);
\fill (4,1) circle(0.1);
\fill (4,0) circle(0.1);
\fill (5,0) circle(0.1);
\fill (6,0) circle(0.1);
\fill (7,0) circle(0.1);
\end{scope}
\end{tikzpicture}
\caption{The path $\omega_n$ for $n=4$. }
\label{fig:pathomegan}
\end{figure}
\begin{theorem}[{\cite[Corollary 4.5]{barraquand2018stochastic}}]
	Let $n\geqslant 1$ and  $\omega_n\in \Omega$ be the path in $\Omega$ which travels from $(+\infty, 0)$ to $(n,0)$, goes vertically to $(n,n)$ and diagonally to $(0,0)$ (see Figure \ref{fig:pathomegan}). Let $\uprho_n$ be  a sequence of specializations  such that edges $(n,i-1)\to( n,i)$ are labeled by the single variable specialization into $a_i$ and the diagonal specialization is empty.  Consider  a sequence of partitions $\labold$ distributed according to the half-space Hall-Littlewood process $\PHL_{\omega_n, \uprho_n}$. Then we have 
	$$ \big(\ell( \lambda^{v})\big)_{v\in \omega_n} \overset{(d)}{=} \big(\mathfrak{h}(v)\big)_{v\in \tilde \omega_n},$$
	where $\tilde \omega_n$ the path in $\Z^2$ obtained by reflecting $\omega$ with respect to the quadrant diagonal. 
	
	In particular, For any $1\leqslant x\leqslant y$, $\mathfrak{h}(x,y)$ has the same distribution as $\ell(\lambda)$ where $\lambda$ is distributed according to the half-space Hall-Littlewood measure $\PHL_{(a_1, \dots, a_x), (a_{x+1}, \dots, a_{y})}$. 
	\label{th:HL6V}
\end{theorem}
Note that we need to perform a reflection of the path $\omega_n$  simply because, in order to be consistent with earlier  literature,  we have defined our half-space stochastic six-vertex model  in the upper half-quadrant while half-space Macdonald processes are indexed by paths in the lower half-quadrant. Moreover, we expect the result should hold for any admissible path in $\Omega$ -- see \cite[Remark 4.6]{barraquand2018stochastic}.

\begin{corollary} For integers $1\leqslant x \leqslant y$  and $k\geqslant 1$,
\begin{multline}
\EHL\left[ t^{-k\mathfrak{h}(x,y)} \right] =   t^{\frac{k(k-1)}{2}} \oint_{C_1} \frac{\mathrm{d}z_1}{2\I\pi} \dots  \oint_{C_k} \frac{\mathrm{d}z_k}{2\I\pi}  \  \prod_{1\leqslant i<j\leqslant k}\frac{z_i-z_j}{z_i- tz_j} \frac{1-tz_iz_j}{1-z_iz_j}  \\ \times 	\prod_{j=1}^k\left(  \frac{1}{z_j} \frac{1-tz_j^2}{1-z_j^2} \prod_{i=1}^{y} \frac{1-a_iz_j}{1-ta_iz_j} \prod_{i=1}^x \frac{z_j-a_i/t}{z_j-a_i} \right),
\label{eq:moments6vnested}
\end{multline}
where the positively oriented contours $C_1, \dots, C_m$ all enclose $0$ and the $a_i$'s,  and are contained in the open disk of radius $1$ around zero, and the contours are nested in such a way that for $i<j$ the contour $C_i$ does not include any part of $t C_j$. 
\label{cor:moments6v}
\end{corollary}
\begin{proof}
	This is a direct consequence of Proposition \ref{prop:momentsHL} and Theorem \ref{th:HL6V}. 
\end{proof}

\subsection{Scaling limit to half-line ASEP}
\label{sec:ASEP}
\begin{definition}
	The \emph{half-line ASEP} is an interacting particle system on $\Z_{>0}$ where each site is occupied by at most one particle. It is a continuous time Markov process on the space of particle configurations, such that  each particle jumps by one to the right at rate $\p$ and to the left at rate $\q$, with $\q<\p$, provided the target site is empty. At the origin,  a reservoir of particles injects a particle at site $1$ (whenever it is empty) at rate $\ratealpha$ and removes a particle from site $1$ (whenever it is occupied) at rate $\rategamma$. We will further assume that initially all sites are empty and the parameters are chosen as $\p=1, \q=t$, $\ratealpha=1/2$, $\rategamma=t/2$ -- see Figure \ref{fig:halfASEP} for an illustration.  We refer to \cite{barraquand2018stochastic} for the reasons behind this specific choice of parameters. We define the current at site $x$ by
		$$ N_x(\tau) = \sum_{i=x}^{\infty} \eta_i(\tau),$$
		where $\eta_i(\tau)\in \lbrace 0,1\rbrace $ is the occupation variable at site $i$ and time $\tau$. 
	\label{def:halflineASEP}
\end{definition}

\begin{figure}
	\begin{tikzpicture}[scale=0.7]
	\draw[thick] (-1.2, 0) circle(1.2);
	\draw (-1.2,0) node{reservoir};
	\draw[thick] (0, 0) -- (12.5, 0);
	\foreach \x in {1, ..., 12} {
		\draw[gray] (\x, 0.15) -- (\x, -0.15) node[anchor=north]{\footnotesize $\x$};
	}
	
	\fill[thick] (3, 0) circle(0.2);
	\fill[thick] (6, 0) circle(0.2);
	\fill[thick] (7, 0) circle(0.2);
	\fill[thick] (10, 0) circle(0.2);
	\draw[thick, ->] (3, 0.3)  to[bend left] node[midway, above]{$1$} (4, 0.3);
	\draw[thick, ->] (6, 0.3)  to[bend right] node[midway, above]{$t$} (5, 0.3);
	\draw[thick, ->] (7, 0.3) to[bend left] node[midway, above]{$1$} (8, 0.3);
	\draw[thick, ->] (10, 0.3) to[bend left] node[midway, above]{$1$} (11, 0.3);
	\draw[thick, ->] (10, 0.3) to[bend right] node[midway, above]{$t$} (9, 0.3);
	\draw[thick, ->] (-0.1, 0.5) to[bend left] node[midway, above]{$1/2$} (0.9, 0.4);
	\draw[thick, <-] (0, -0.5) to[bend right] node[midway, below]{$t/2$} (0.9, -0.4);
	\end{tikzpicture}
	\caption{Jump rates in the half-line ASEP. }
	\label{fig:halfASEP}
\end{figure}
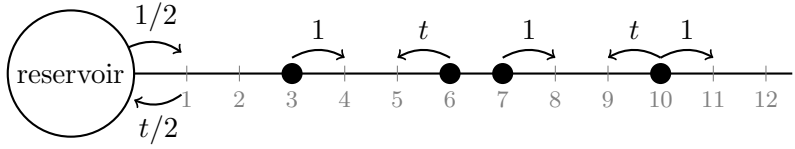

The stochastic six-vertex model in a half quadrant is a discretization of half-line ASEP in the following sense. 
Consider the six-vertex model in a half-quadrant where $a_x\equiv a$, and  scale $a$ as
$$ a=1-\frac{(1-t)\e}{2},\ \ \  \e\xrightarrow[\e>0]{}0,$$
so that to first order in $\e$,
\begin{equation}
\PP\left(\pathrr\right) \approx \e, \ \ \PP\left(\pathru\right) \approx 1-\e,\ \ \PP\left(\pathuu\right) \approx t\e, \ \ \PP\left(\pathur\right) \approx 1-t\e.
\label{eq:limitweights}
\end{equation}
Moreover, we rescale $n$ as $n=\tau \e^{-1}$ with finite $\tau>0$. 
\begin{proposition}[{\cite[Proposition 5.2]{barraquand2018stochastic}}]
	Under the scalings above and the boundary and initial conditions as in Definition \ref{def:halflineASEP}, for any $x\in\lbrace 1, 2, \dots\rbrace$,
	$$ n-x -\mathfrak{h}(n-x,n) \xRightarrow[\e \to 0]{} N_x(\tau),$$
	where $\mathfrak{h}$ is defined Definition \ref{def:sixvertex} and $N_x(\tau)$ is defined in Definition \ref{def:halflineASEP}.
	\label{prop:cv6vtoASEP}
\end{proposition}

	It would be interesting to take a limit of the moment formula from Corollary \ref{cor:moments6v} in order to deduce formulas for the current in half-line ASEP. However, the choice of contours above -- including the $a_i$ but excluding $1$ -- does not allow to take this limit directly. One would need to expand the contours in \eqref{eq:moments6vnested} to become a circle with radius larger than $1$. During this deformation of contours, one would encounter the same poles for $z_i=tz_j$ when $i<j$ as in the proof Proposition \ref{prop:momentsHLdet}, but also additional poles when $z_i=\pm 1$ and when $z_i=1/z_j$. It is not clear how to rearrange the contribution of all the corresponding residues and we do not know an analogue of Proposition \ref{prop:contourshift} in that setting.  It is worth mentioning that in the limit when $t$ goes to $1$, an analogue of Proposition \ref{prop:contourshift} was proposed in \cite{borodin2016directed} as Conjecture 5.2 with Pfaffians replacing the determinants.

 In the attempt to prove \cite[Conjecture 5.2]{borodin2016directed} proposed in \cite[Section 7.3]{borodin2016directed}, an essential step is to make use of symmetrization identities over the hyperoctaedral group $BC_k$ instead of the symmetric group $\mathcal{S}_k$ as in the proof of Proposition \ref{prop:contourshift} (see \cite[Section 7]{borodin2016directed} for details).  
	The following identity was  proved in \cite[Eq. (5)]{venkateswaran2015symmetric}, 
	\begin{equation}
	\sum_{\sigma\in BC_k} \sigma\left( \prod_{1\leqslant i<j \leqslant k} \frac{1-t\frac{z_j}{z_i}}{1-\frac{z_j}{z_i}}  \frac{1-t\frac{1}{z_iz_j}}{1-\frac{1}{z_iz_j}} \prod_{j=1}^{k} \frac{\left(1-a \frac{1}{z_j}\right)\left(1-b \frac{1}{z_j}\right)}{1-\frac{1}{z_j^2}} \right)  = \prod_{j=1}^k \frac{(1-t^j)(1-abt^{j-1})}{1-t}. 
	\label{eq:BCsymetrization}
	\end{equation} 
	where $\sigma$ is a signed permutation acting by permutation and inversion of variables $z_i$. Its degeneration as $t$ goes to $1$ was an important tool to arrive at \cite[Conjecture 5.2]{borodin2016directed}. 
	
	Taking $a=-b=\sqrt{t}$, and inverting variables in \eqref{eq:BCsymetrization}, we obtain 
	$$ \sum_{\sigma\in BC_k} \sigma\left( \prod_{1\leqslant i<j \leqslant k} \frac{z_j-tz_i}{z_j-z_i}  \frac{1-t z_iz_j}{1-z_iz_j} \prod_{j=1}^{k} \frac{1-tz_j^2}{1-z_j^2} \right)  =\frac{(t^2, t^2)_k}{(1-t)^k}.$$
	We expect that this BC-type symmetrization identity should play  a role when moving all contours in \eqref{eq:momentsHLnested} to a circle with radius larger than $1$. This further suggests that the appropriate moment generating series to consider is 
	$$ \sum_{k=0}^{+\infty}  \frac{u^k \EHL[t^{k(n-\ell(\la))}]}{(t^2, t^2)_k}  = \EHL\left[ \frac{1}{(ut^{n-\ell(\la)}; t^2)_{\infty}} \right].$$ 
	It was confirmed in  \cite[Proposition 3.3]{barraquand2018stochastic} that this quantity indeed admits a Fredholm Pfaffian representation at least in the special case where the specialisations of the half-space Hall-Littlewood measure are $\rhoup=(a_1, \dots, a_x)$ and $\rhodiag = \varnothing$. Allowing a more general $\rhodiag$ would allow  to study $\mathfrak{h}(x, y)$ for $y>x$, but we cannot presently generalize the results of \cite{barraquand2018stochastic}.

 \thispagestyle{plain}
 \section{Half-space Whittaker processes}
\label{sec:Whittaker}
\subsection{Whittaker functions}

Define the \emph{class-one $\mathfrak{gl}_{n}(\R)$-Whittaker functions} via their integral representations \cite{givental1997stationary} 
$$
\psi_{\lambda}(x)  = \int_{\R^{\frac{n(n-1)}{2}}} \prod_{k=1}^{n-1} \prod_{i=1}^{k}  d x_{k,i}\, \exp\left( {\mathcal{F}_{\lambda}(X)}\right)
$$
where $\lambda = (\lambda_1,\ldots,\lambda_n)\in \C^n$, $x=(x_1,\ldots,x_n)$, $X=(x_{k,i}:1\leqslant i\leqslant k\leqslant n)$, $x_{n,i} = x_i$, and
$$
\mathcal{F}_{\lambda}(X) = \I \sum_{k=1}^{n} \lambda_k \left(\sum_{i=1}^k x_{k,i} - \sum_{i=1}^{k-1} x_{k-1,i}\right) - \sum_{k=1}^{n-1} \sum_{i=1}^{k} \left(e^{x_{k,i}-x_{k+1,i}} + e^{x_{k+1,i+1}-x_{k,i}}\right).
$$
Note that unlike previous sections, $\lambda$ does not denote a partition but a vector in $\C^n$. Furthermore, we will see that  Whittaker functions play a role similar to  $q$-Whittaker and Macdonald functions in previous sections, but in order to be consistent with notations commonly used in the literature, the role of the index $\lambda$ and the variable $x$ are switched (for instance in Proposition \ref{prop:limitqWhittWhittwithplancherel} the index $\lambda$ in $\psi_{\lambda}(x)$ corresponds to the variable $x$ in $P_{\la}(x)$ and vice versa). 
Whittaker functions satisfy the following two integral identities.
\begin{proposition}[{\cite{stade2001mellin}, \cite[Corollaries 3.6 and 3.7]{o2014geometric}}]
Suppose $u>0$ and $\lambda,\nu\in\C^n$ with $\Re(\lambda_i+\nu_j)>0$ for all $1\leqslant i,j\leqslant n$. Then
\begin{equation}
\int_{\R^n} d x\, e^{-u e^{-x_n}} \psi_{\I\lambda}(x)\psi_{\I\nu}(x)  = u^{- \sum_{j=1}^{n} (\lambda_j+\nu_j)} \prod_{1\leqslant i,j\leqslant n} \Gamma( \lambda_i +  \nu_j),
\label{eq:Cauchywhittakern}
\end{equation}
and
\begin{equation}
\int_{\R^n} d x\, e^{-u e^{x_1}} \psi_{-\I\lambda}(x)\psi_{-\I\nu}(x)  = u^{- \sum_{j=1}^{n} (\lambda_j+\nu_j)} \prod_{1\leqslant i,j\leqslant n} \Gamma( \lambda_i +  \nu_j).
\label{eq:Cauchywhittakerone}
\end{equation}
\label{prop:Cauchywhittaker}
\end{proposition}
Define the \emph{Sklyanin measure} $\mskyl_n(\xi)$ as
\begin{equation}
    \mskyl_n(\xi)=\frac1{(2\pi)^n n!}\prod_{\substack{i,j=1\\i\neq j}}^n
    \frac1{\Gamma(\I \xi_i- \I \xi_j)}.
    \label{eq:defmskyl}
  \end{equation}
We have the following orthogonality relations, to be understood in a weak sense (see \cite[Theorem 2.1]{gerasimov2008baxter}, \cite{semenov1993quantization}). For all $\la, \mu, x,y\in\R^n$,
\begin{equation}
\int_{\R^n} \overline{\psi_{\la}(x)}\psi_{\mu}(x)\mathrm{d}x = \frac{1}{n!\mskyl_n(\la)}\sum_{\sigma\in \mathcal{S}_{n}}\delta\big(\la-\sigma(\mu)\big),
\label{eq:orthogonalityone}
\end{equation}
and
\begin{equation}
\int_{\R^n} \overline{\psi_{\la}(x)}\psi_{\la}(y)\mskyl_n(\la)\mathrm{d}\la = \delta(x-y). 
\label{eq:orthogonality2}
\end{equation} 
Then, Proposition \ref{prop:Cauchywhittaker} yields the following. Let  $\mathbb{H}$ denote the complex upper half plane $\mathbb{H} = \lbrace z\in \C: \Imag[z]>0\rbrace$. 
\begin{proposition}
For $u>0$ and $w\in \C^n$,
$w_i\in \mathbb{H}$, $1\leqslant i\leqslant n$,
\begin{equation}\label{eq:eignwhitN}
e^{-u e^{-x_n}} \psi_{w}(x) = \int_{\R^n}d\xi\, \mskyl_n(\xi)  u^{\I \sum_{k=1}^{n} ( w_i+\xi_i)}\prod_{1\leqslant i,j\leqslant n} \Gamma(-\I \xi_i - \I w_j) \psi_{-\xi}(x),
\end{equation}
and 
\begin{equation}\label{eq:eignwhitone}
e^{-u e^{x_1}} \psi_{-w}(x) = \int_{\R^n} d\xi\, \mskyl_n(\xi)  u^{\I \sum_{k=1}^{n} ( w_i+\xi_i)}\prod_{1\leqslant i,j\leqslant n} \Gamma(-\I \xi_i - \I w_j) \psi_{\xi}(x).
\end{equation}
\end{proposition}

Define an operator $\BBn_n^{u}$ acting on functions $ \mathbb{H}^n \to \C$ by
\begin{equation}
\BBn_n^{u} f(w) =  \int_{\R^n}d\xi\, \mskyl_n(\xi)  u^{\I \sum_{k=1}^{n} ( w_i+\xi_i)}\prod_{1\leqslant i,j\leqslant n} \Gamma(-\I \xi_i - \I w_j) f(-\xi). 
\label{eq:defBBn}
\end{equation}
Define as well a very similar operator $\BBone_n^u$ acting on functions $ (-\mathbb{H})^n \to \C$ by 
\begin{equation}
\BBone_n^{u} f(w) = \int_{\R^n} d\xi\, \mskyl_n(\xi)  u^{\I \sum_{k=1}^{n} ( - w_i+\xi_i)}\prod_{1\leqslant i,j\leqslant n} \Gamma(-\I \xi_i + \I w_j) f(\xi),
\label{eq:defBBone}
\end{equation}
i.e., $\BBone_n^{u}$ is a composition of $\BBn_n^{u}$ and $w\mapsto -w$. 
We may rewrite equations \eqref{eq:eignwhitN} and \eqref{eq:eignwhitone} as $  e^{-u e^{-x_n}} \psi_{w}(x) = \BBn^{u} \psi_{w}(x) $ and $  e^{-u e^{x_1}} \psi_{w}(x) = \BBone^{u}  \psi_{w}(x) $ (dropping the index $n$). The operator $\BBone^{u}$ is referred to as \emph{dual Baxter operator} in \cite{gerasimov2008baxter}. It was shown in \cite{borodin2015classical} that the eigenrelation \eqref{eq:eignwhitN} for $\BBn^{u}$ arises as a $t=0, q\to 1$ limit of the eigenrelation \eqref{eq:Noumieigenrelation} for Noumi's $q$-integral operator $\Noumi^z$. We expect that eigenrelation \eqref{eq:eignwhitone} for $\BBone^u$ similarly arises as the limit of the eigenrelation  for $\MoumiA^z$ from Proposition \ref{prop:analyticeigenrelation}. 

\subsection{Half-space Whittaker process with Plancherel specialization}
\label{sec:rigorousconvwithplancherel}

In this section, we are interested in the limit of the half-space $q$-Whittaker process when $q$ goes to $1$. For applications, it would be natural to focus on the case when specializations are all  pure alpha, that is the setting of Section \ref{sec:rskdynamics}. However, we know from the study of the full-space Whittaker process that some technical difficulties arise with this choice of specializations (see Section 4.2 in \cite{borodin2014macdonald}).  More precisely, some tail estimates about integrals involving Whittaker functions are necessary to justify the convergence of the $q$-Whittaker process, and these seem difficult to establish. However, all these estimates are much easier to prove when one includes a  Plancherel specialization (see Section 4.1 in \cite{borodin2014macdonald}). Thus, following an idea already present in \cite{borodin2015height}, we will consider a half-space $q$-Whittaker process where we add to the specialization $(\diagq, a_{n+1}, \dots , a_t)$ of $\ve_{\lambda}$ some Plancherel component $\gamma$. 

Consider a half-space $q$-Whittaker process indexed by a path $\pathh$ as in Figure \ref{pathWhittaker2}, with the following choice of specializations.  Fix $1\leqslant  n\leqslant t$. For all $i\leqslant t$ and any $j$,  assume that edges $ (i-1, j) \leftarrow (i,j) $ are labeled by single variable specialization $a_i$, for any $i, j$, edges $ (i, j-1) \rightarrow (i,j) $ are labeled by single variable specialization $a_j$, the diagonal edge is labeled by specialization $\diagq$ and -- unlike in Section \ref{sec:rskdynamics} -- edges $(t, j)\leftarrow (t+1, j)$ are labeled by the Plancherel specialization $\gamma$. 

\begin{figure}
\begin{tikzpicture}[scale=1]
\begin{scope}[decoration={
	markings,
	mark=at position 0.5 with {\arrow{<}}}]
\draw[->, >=stealth', gray] (0,0) -- (8.7, 0);
\draw[->, >=stealth', gray] (0,0) -- (0,5.5);
\draw[gray, dotted] (0,0) grid(8.5, 5.5); 
\draw[gray, dotted] (5, 3.5)  -- (5,0) node[anchor = north]{$i$};
\draw[gray, dotted] (0,3) node[anchor= east]{$j$} -- (5.5,3) ;
\draw[gray] (4.5,2) node[anchor=north]{$ a_i$};
\draw[gray] (4,2.5) node[anchor=east]{$ a_j$};
\draw[gray] (5,3) node[anchor =south west]{$(i,j)$};
\fill[gray] (5,3) circle(0.08);
\fill (7,5) circle(0.08);
\fill (8,5) circle(0.08);
\draw[] (7,0) node[anchor = north east]{$t$};
\draw[] (8,0) node[anchor = north]{$t+1$};
\draw[dotted] (0,5) node[anchor= east]{$n$} -- (7,5);
\draw[fleche] (0,0) -- (5,5);
\draw[fleche] (5,5) -- (7,5);
 \draw[fleche] (7,5) -- (8,5);
\draw[fleche] (8,5) -- (8, 0);
\draw (7,5) node[anchor =south east]{$(t,n)$};
\draw (8,5) node[anchor =south west]{$(t+1,n)$};
\draw[gray] (7.5,1) node[anchor=north]{$ \gamma$};
\draw[gray] (7.5,2) node[anchor=north]{$ \gamma$};
\draw[gray] (7.5,3) node[anchor=north]{$ \gamma$};
\draw[gray] (7.5,4) node[anchor=north]{$ \gamma$};
\draw[gray] (7.5,5) node[anchor=north]{$ \gamma$};
\draw[thick, gray, dashed] (7,-1) -- (7,6);
\end{scope} 
\end{tikzpicture}
\caption{The path $\pathh
$ considered in Section \ref{sec:rigorousconvwithplancherel}.}
\label{pathWhittaker2}
\end{figure}

Let us denote for $1\leqslant m\leqslant n$,  $\la^{(m)} :=\la^{(t+1, m)}$. 
The probability of the sequence $\la^{(1)}\prec \dots \prec \la^{(n)}$ is 
$$\PQWP_{\pathh, \bm\uprho}(\la^{(1)},  \dots ,  \la^{(n)}) = \frac{ \ve_{\la^{(n)}}\big((\diagq, a_{n+1}, \dots, a_t), \gamma\big)P_{\la^{(n)}/\la^{(n-1)}}(a_n)\dots P_{\la^{(1)}}(a_1)}{\Pi\big(a_1,\ldots, a_n; (\diagq, a_{n+1}, \dots ,a_{t}), \gamma\big)\Phi(a_1,\ldots, a_n)}.$$
In particular, $ \la^{(n)}$ is distributed according to the half-space $q$-Whittaker measure, 
$$
\PQWM_{(a_1,\ldots, a_n), ((\diagq, a_{n+1}, \dots ,a_{t}), \gamma)}(\la^{(n)}) = \frac{P_{\lambda^{(n)}}(a_1,\ldots, a_n) \ve_{\lambda^{(n)}}\big((\diagq, a_{n+1}, \dots, a_{t}), \gamma\big)}{\Pi\big(a_1,\ldots, a_n;(\diagq, a_{n+1}, \dots, a_t), \gamma\big) \Phi(a_1,\ldots, a_N)}.
$$
To go from $q$-Whittaker to Whittaker processes, we use the following scalings
\begin{equation}
q= e^{-\e},\qquad a_j = e^{-\e \alpha_j},\qquad \diagq = e^{-\e \diag}, \qquad \gamma=\tau\e^{-2}
\label{eq:scalings1}
\end{equation}
\begin{equation} \lambda^{(m)}_j =\tau\e^{-2}+ (t+m+1 - 2j) \e^{-1}\log \e^{-1} + \e^{-1} T^{(m)}_j, \ \ \forall \ 1\leqslant m\leqslant n.
\label{eq:scalings2}
\end{equation}
\begin{remark} Based on the analogy between half-space Macdonald processes and usual  Macdonald processes, our random variable  $T^{(m)}_j$ corresponds to the random variable denoted $T_{m,j} $ in the context of the $\alpha$-Whittaker process  \cite[Section 4.2]{borodin2014macdonald}.
\end{remark}
Let us recall some convergence results from \cite{borodin2014macdonald}.
\begin{lemma}[{\cite[Proposition 4.1.9]{borodin2014macdonald}}]For any $M\in \R$, 
\begin{equation}
\log \big[(q;q)_{\e^{-1}\log\e^{-1} + \e^{-1} y}\big]  = \A(\e) + e^{-y} + o(1)
\label{eq:qqestimate1}
\end{equation}
where
$
\A(\e) = -\e^{-1} \frac{\pi^2}{6} - \frac{1}{2} \log\frac{\e}{2\pi},
$
and for $k>1$,
\begin{equation}
\log \big[(q;q)_{k \e^{-1}\log\e^{-1} + \e^{-1} y}\big] = \A(\e)+ o(1),
\label{eq:qqestimate2}
\end{equation}
where the error $ o(1)$ in  \eqref{eq:qqestimate1} and \eqref{eq:qqestimate2} goes to $0$ uniformly for $y>M$ as $\e\to 0$.
Moreover, for any $y\in \R$ and $k\geqslant 1$, we have the inequality 
\begin{equation}
\log \qq{k\e^{-1}\log(\e^{-1}) + \e^{-1}y } \geqslant \A(\e)+ \e^{-1} e^{-y+k\log(\e)}  - c(\e), 
\label{eq:qqestimateuniform}
\end{equation}
where $c(\e)\to 0$ uniformly in $y$.   
\label{lem:qqestimates}
\end{lemma}
We recall that (with $t=0$) 
$$
\Pi(a_1,\ldots, a_n;b_1, \dots, b_k) = \prod_{i=1}^{n}\prod_{j=1}^k \frac{1}{(a_i b_j;q)_{\infty}}, \qquad \Phi(a_1,\ldots, a_n) = \prod_{1\leqslant i<j\leqslant n}\frac{1}{(a_ia_j;q)_{\infty}}.
$$
Noting that by the convergence of the $q$-Gamma function to the Gamma function (see \eqref{eq:qGamma}), 
$$
(e^{-\e \tilde x};e^{-\e})_{\infty} = \e^{1-\tilde{x}}  \frac{1}{\Gamma(\tilde x)} e^{\mathcal{A}(\e)+o(1)},
$$
so that (see \cite[Lemma 4.9]{borodin2015height}) with the scalings $a_i=q^{\alpha_i},  b_i=q^{\beta_i}$, 
\begin{multline}
\Pi\big(a_1,\ldots, a_n;(b_1, \dots, b_k), \gamma\big) \\ =  \left( e^{\tau t \e^{-2}} e^{-\e^{-1}\tau\sum_{j=1}^n \alpha_j} \prod_{i=1}^k\prod_{j=1}^{n}
\frac{1}{e^{\mathcal{A}(\epsilon)} \epsilon^{1- \beta_i-\alpha_j}}\right) 
e^{\tau \sum_{j=1}^n \alpha_j^2/2 }\prod_{i=1}^k\prod_{j=1}^{n}
\Gamma( \beta_i+\alpha_j)e^{o(1)},
\label{eq:limPi}
\end{multline}
where  the $o(1)$ error goes to zero as $\epsilon\to 0$. Similarly, we have 
\begin{multline}
\Phi(a_1,\ldots, a_n) = \prod_{1\leqslant i<j\leqslant n}\e^{\alpha_i+\alpha_j-1} e^{-\mathcal{A}(\e)} \Gamma(\alpha_i+\alpha_j)e^{o(1)}\\ 
= \e^{(n-1)\sum \alpha_i}\e^{-\frac{n(n-1)}{2}} e^{-\frac{n(n-1)}{2}\mathcal{A}(\e)} \prod_{1\leqslant i<j\leqslant n} \Gamma(\alpha_i+\alpha_j)e^{o(1)}.
\label{eq:limPhi}
\end{multline}
\begin{proposition}[{\cite[Theorem 4.1.7]{borodin2014macdonald}}]
Fix $1\leqslant n\leqslant t$,  the scalings $q=e^{-\e}$, $\la_j = \tau \e^{-2} + (t+n+1 - 2j)\e^{-1}\log \e^{-1} +\e^{-1} x_k$ and $z_j = e^{\I \e \nu_j}$ for $1\leqslant j\leqslant n$, and let
$$
\psi^{\e}_{\nu}(x) = \e^{\frac{n(n-1)}{2}} \e^{t\sum_{k=1}^{n}\I\nu_k}e^{\tau\e^{-1}\sum_{k=1}^{n}\I \nu_k} e^{\frac{n(n-1)}{2} \mathcal{A}(\e)} P_{\la}(z).
$$
Then, for all $\nu\in \C^n$,  as $\e\to 0$, $\psi^{\e}_{\nu}(x)$ converges to $\psi_{\nu}(x)$ uniformly for $x$ in a compact set.
\label{prop:limitqWhittWhittwithplancherel}
\end{proposition}
\begin{proof}
The scalings above are slightly different from the setting of  \cite[Theorem 4.1.7]{borodin2014macdonald}, but the result is obtained using the same arguments as in Step 2 of the proof of \cite[Lemma 4.1.25]{borodin2014macdonald}. 
\end{proof}
\begin{lemma}[{\cite[Lemma 4.8]{borodin2015height}}]
Fix any compact subset $D\subset\R^{n(n+1)/2}$. Then, under the scalings \eqref{eq:scalings1} and \eqref{eq:scalings2}, 
$$  P_{\la^{(n)}/\la^{(n-1)}}(a_n)P_{\la^{(n-1)}/\la^{(n-2)}}(a_{n-1})\dots P_{\la^{(1)}}(a_1)  = e^{-\frac{n(n-1)}{2}\A(\e)}\e^{t\sum_{k=1}^{n}\alpha_k} e^{-\e^{-1}\tau\sum_{j=1}^n \alpha_j} e^{\mathcal{F}_{ \I \alpha_1, \dots, \I \alpha_n}(T)}e^{o(1)},$$
where the error $o(1)$ goes to zero as $\e$ to $0$ uniformly with respect to $ (T_j^{m})_{1\leqslant j\leqslant m\leqslant n} \in D$.
\label{lem:productPtoWhittwithplancherel}
\end{lemma}
\noindent Proposition \ref{prop:limitqWhittWhittwithplancherel} and Lemma \ref{lem:productPtoWhittwithplancherel} correct a mistake in the exponent in front of $\A$ present in \cite{borodin2014macdonald}, as pointed out in \cite{borodin2015classical}.

We define now the Whittaker analogue of the symmetric function $\ve_{\la}$. 
\begin{definition}
We define for $\tau> 0$, $\alpha\in \R^k$ and $x\in \R^n$, 
\begin{equation}
 \mathcal{T}_{\alpha}^{\,\tau}(x)  =  \int_{(\R+\I a )^n} \psi_{-\nu}(x) e^{-\tau  \sum_{j=1}^n \nu_j^2/2 }\prod_{i=1}^k \prod_{j=1}^n \Gamma(\alpha_i - \I\nu_j) \prod_{1\leqslant i < j \leqslant n} \Gamma(-\I(\nu_i+\nu_j))
 \mskyl_n(\nu)\mathrm{d}\nu,
\label{eq:defT}
\end{equation}
where $a>0$ and $\alpha_i+a>0$ for all $1\leqslant i \leqslant k$.   
\end{definition}
\begin{proposition} Let $\tau>0$,  $n,k\in \Z_{>0}$, and fix a compact subset $ \mathbb{K} \subset \R^n $.  Then, under the scalings $q=e^{-\e}$, $ \la_j=\tau\e^{-2}+ (2n+k-2j)\e^{-1}\log(\e^{-1}) +\e^{-1}x_j$ for $ 1\leqslant j\leqslant n $, $a_i = q^{\alpha_i}$ for $1\leqslant i \leqslant k$ and $\gamma=\tau \e^{-2}$,
$$ \ve_{\lambda}\big((a_1, \dots, a_k), \gamma\big) =\e^n e^{\tau n\e^{-2}}\left(\prod_{i=1}^k \frac{1}{e^{\A(\e)}\e^{1-\alpha_i}}\right)^n\mathcal{T}_{\alpha}^{\,\tau}(x) e^{o(1)},$$
where $\ve_{\lambda}$ is specialized into the union of the pure alpha specialization $(a_1, \dots, a_k)  $ and a Plancherel specialization with parameter $\gamma$, and   the $o(1)$ error goes to zero uniformly as $\e\to 0$ for $x\in \mathbb{K}$.  
\label{prop:limitepsilonwithplancherel}
\end{proposition}
\begin{proof} We adapt the proof of  \cite[Lemma 4.10]{borodin2015height}. Recall that using the torus scalar product $\llangle \cdot, \cdot\rrangle $ from Section \ref{sec:orthogonality}, we may write for $\e>0$ and $a>0$, 
$$	 \ve_{\la}(x) = \frac{1}{\llangle P_{\la}, P_{\la}\rrangle} \int_{(e^{-\e a}\mathbb{T})^n} P_{\la}(z^{-1})\Pi(z, x)\Phi(z) \mskyl_{n}^{q,t}(z)\prod_{i=1}^n\frac{\mathrm{d}z_i}{z_i}.$$
  Under the scalings we consider, we have $\llangle P_{\la}, P_{\la} \rrangle= \langle P_{\la}, P_{\la} \rangle'=e^{o(1)} $ as in  \cite[Lemma 4.1.25]{borodin2014macdonald}. In the following we will use the change of variables $z_j = \exp(\I \e \nu_j)$.
Fix a compact subset $V\subset\R^N$. Then, 
$$\Pi\big(z_1,\ldots, z_n;(a_1, \dots, a_k), \gamma\big) = E_{\Pi} \prod_{i=1}^k\prod_{j=1}^{n}
\Gamma(\alpha_i-\I \nu_j)e^{-\tau  \sum_{j=1}^n \nu_j^2/2 } e^{o(1)}, $$
where $E_{\Pi} := \prod_{i=1}^k\prod_{j=1}^{n}
e^{-\mathcal{A}(\epsilon)} \epsilon^{-1+\alpha_i-\I \nu_j}$, 
$$ \Phi(z_1,\ldots, z_n) =  E_{\Phi}\prod_{1\leqslant i<j\leqslant n} \Gamma(-\I \nu_i -\I \nu_j)e^{o(1)},$$
where $E_{\Phi} := \e^{-\I(n-1)\sum_{j=1}^n  \nu_j}\e^{-\frac{n(n-1)}{2}} e^{-\frac{n(n-1)}{2}\mathcal{A}(\e)},$ and 
from Proposition \ref{prop:limitqWhittWhittwithplancherel},  
$$P_{\la}(1/z_1, \dots, 1/z_n) = \overline{E_P} \Psi_{-\nu_1, \dots, -\nu_n}(x_1, \dots, x_n)e^{o(1)},$$ where $\overline{E_{P}} = \e^{-\frac{n(n-1)}{2}}   \e^{\I (k+n-1)\sum_{j=1}^n\nu_j}e^{-\frac{n(n-1)}{2}\A(\e)}$.  In all the above asymptotics, the $o(1)$ errors go to zero uniformly for $x\in \mathbb{K}$ and $\nu\in V$. Further, we have 
$$\mskyl_{n}^{q,0}(z)\prod_{i=1}^n\frac{\mathrm{d}z_i}{z_i} = E_{\mskyl} \mskyl_n(\nu)\prod_{i=1}^n \mathrm{d}\nu_i e^{o(1)}, \ \ \ \ \ \ E_{\mskyl} = \e^{n^2}e^{n(n-1)\A(\e)},$$
 where the error goes to zero uniformly for $\nu\in V$.
 The above asymptotics altogether suggest that uniformly for $x\in \mathbb{K}$, 
\begin{equation} \llangle \Pi(z; a, \gamma), P_{\la}(z) \rrangle = E_{\Pi}E_{\Phi}\overline{E_P}E_m \mathcal{T}_{\alpha}^{\,\tau}(x)e^{o(1)},
\label{eq:whatwewant}
\end{equation}
since integrands on both sides of \eqref{eq:whatwewant} match when $\e\to0
$. However, the convergences above are valid for compact subsets of the integrand variable $\nu$. In order to justify that the integrand converge, one needs some tail decay estimate as $\vert \nu\vert \to \infty$, uniform with respect to $x\in \mathbb{K}$. For instance, it would be sufficient to prove that for 
$$ V_M = \lbrace z\in\mathbb{T}^n : z_k=e^{-\e a} e^{\I \e \nu_k} \text{ and }\vert \nu_k\vert >M\rbrace, $$ 
then 
\begin{equation}
\lim_{M\to\infty} \lim_{\e\to 0} \int_{z\in V_M} \frac{\Pi( z;  a, \gamma)\Phi( z)P_{\la}( z^{-1}) \mskyl_n^q( z)}{E_{\Pi}E_{\Phi} \overline{E_{P}} E_{\mskyl}}\prod_{i=1}^n\frac{\mathrm{d}z_i}{z_i}=0
\label{eq:uniformestimateneeded}
\end{equation}
uniformly for $x\in \mathbb{K}$. A similar estimate had already been proved in Step 4 of \cite[Lemma 4.1.25]{borodin2014macdonald}. One can estimate each of the quantities 
$$ \bigg\vert\frac{\Pi( z;  a, \tau)}{E_{\Pi}}\bigg\vert,\  \bigg\vert\frac{\Phi( z)}{E_{\Phi}}\bigg\vert ,\  \bigg\vert\frac{P_{\la}( z^{-1}) }{ \overline{E_{P}}}\bigg\vert\ \text{ and }\ \bigg\vert\frac{\mskyl_n^q( z)}{E_{\mskyl}}\bigg\vert.$$
While $ \bigg\vert\frac{\Pi( z;  a, \tau)}{E_{\Pi}}\bigg\vert$ has a Gaussian decay in $\nu$, all other quantities have at most exponential growth (this is the reason why it is essential for us to keep a positive Plancherel specialization). Therefore, the integrand in \eqref{eq:uniformestimateneeded} is bounded by a constant times $e^{-c\sum_{i=1}^n \nu_i^2}$, uniformly in $\e$ for $T\in D$. Thus, \eqref{eq:uniformestimateneeded} is established.  Finally, it is easy to check that $E_{\Pi}E_{\Phi}\overline{E_P}E_m$ matches the prefactor of $\mathcal{T}_{\alpha}^{\,\tau}(x)$ in the statement of the Proposition \ref{prop:limitepsilonwithplancherel}.
\end{proof}

\begin{definition}
For $\diag$ and $\alpha_1, \dots, \alpha_t\in \R_{>0}$, $\tau> 0$ and 
$1\leqslant n\leqslant t$, we define the \emph{(ascending) half-space Whittaker process} as a probability measure on $T:=(T_j^{(m)})_{1\leqslant j\leqslant m\leqslant n}\in \R^{\frac{n(n+1)}{2}}$ with density function given by 
$$\PWP(\mathrm{d}T) = \frac{\exp\left({\mathcal{F}_{ \I \alpha_1, \dots, \I \alpha_n}(T)}\right) \ \mathcal{T}_{\diag, \alpha_{n+1}, \dots, \alpha_t}^{\tau}(T^{(n)}_1, \dots,T^{(n)}_n) }{ e^{\tau\sum_{i=1}^n\alpha_i^2/2}\ \prod_{i=1}^n  \Gamma(\diag+\alpha_i)\prod_{j=n+1}^{t}\Gamma(\alpha_j+\alpha_i)\prod_{1\leqslant i<j\leqslant n}\Gamma(\alpha_j+\alpha_i)}\mathrm{d}T.$$
\label{def:PWM}
\end{definition}
Integrating the half-space Whittaker process over variables $T_j^{(m)}$ for $m<n$ defines the \emph{half-space Whittaker measure} $\PWM_{(\alpha_1, \dots, \alpha_n) ; (\diag, \alpha_{n+1}, \dots, \alpha_t), \tau}$ on $\big(T_j\big)_{1\leqslant j\leqslant n}\in \R^n$ with density 
$$\PWM_{(\alpha_1, \dots, \alpha_n) ;( \diag, \alpha_{n+1}, \dots, \alpha_t), \tau}(\mathrm{d}T) = \frac{  \psi_{\I \alpha_1, \dots, \I \alpha_n}(T_1, \dots,T_n) \ \mathcal{T}^{\,\tau}_{\diag, \alpha_{n+1}, \dots, \alpha_t}(T_1, \dots,T_n) }{e^{\tau\sum_{i=1}^n\alpha_i^2/2}\ \prod_{i=1}^n{\Gamma(\diag+\alpha_i)}\prod_{j=n+1}^{t}{\Gamma(\alpha_j+\alpha_i)}\prod_{1\leqslant i<j\leqslant n}{\Gamma(\alpha_j+\alpha_i)}} \mathrm{d}T.$$

The fact that these densities above define bona fide probability measures is not obvious. For $\diag \in \R$ and $\alpha_1, \dots, \alpha_t\in \R_{>0}$ with $\diag+\alpha_i>0$,  it follows from the limits of Propositions \ref{prop:limitqWhittWhittwithplancherel} and \ref{prop:limitepsilonwithplancherel} that the density is non-negative. 
 The next Proposition shows that  the density integrates to $1$.

\begin{proposition} Let $\tau>0$,  $n\leqslant t$ and $\alpha_1, \dots, \alpha_t>0$ such that $\alpha_i+\diag>0$ for all $1\leqslant i\leqslant n$. Then, 
\begin{multline}  
\int_{\R^n} \mathrm{d}x \  \psi_{\I\alpha_1, \dots, \I\alpha_n}(x) \mathcal{T}^{\,\tau}_{\diag, \alpha_{n+1}, \dots, \alpha_t}(x)   \\ =  e^{ \tau\sum_{i=1}^n \alpha_i^2/2}\ \prod_{i=1}^n{\Gamma(\diag+ \alpha_i)}\prod_{j=n+1}^{t}{\Gamma(\alpha_j+\alpha_i)}\prod_{1\leqslant i<j\leqslant n}{\Gamma(\alpha_i+\alpha_j)}.
\end{multline}
\label{prop:integratestoone}
\end{proposition}
This is formally the limit of generalized Littlewood summation identity \eqref{eq:CauchyLittlewood} via Proposition \ref{prop:limitqWhittWhittwithplancherel} and \ref{prop:limitepsilonwithplancherel}, though this does not constitute a proof because these limits hold only for $x$ in a compact set.  However, we will see that the result can be deduced from  the orthogonality of Whittaker functions along with certain bounds on the tails of the half-space Whittaker measure. We will first state these bounds and then prove Proposition \ref{prop:integratestoone}.  

We start with a general bound on the growth of Whittaker functions. 
\begin{lemma}
	For any fixed $y\geqslant 0$, there exists a constant  $C>0$ such that
	$$ \vert \psi_{\nu}(x)  \vert \leqslant C e^{y\sum_{i=1}^n\vert x_i\vert} \text{ for  all }\nu\in \C^n\text{ such that }\vert\Imag[\nu_j] \vert \leqslant y, \text{ for all }1\leqslant j\leqslant n.$$ 
	In particular, if $\alpha\in \C^n$ is such that $c_1 < \Real[\alpha_i] < c_2$  for all $1\leqslant i \leqslant n$, 
	$$  \vert \psi_{\I \alpha}(x)  \vert \leqslant C e^{-c_2 \sum_{i=1}^n x_i + (c_2-c_1)\sum_{i=1}^n\vert x_i\vert}.$$
	\label{lem:boundWhittaker}
\end{lemma}
\begin{proof}
	As \cite[Lemma 4.1.19]{borodin2014macdonald}, the first statement follows by recurrence on $n$ using the recursive structure in Givental's integral representation (see \cite[(4.2), (4.3)]{borodin2014macdonald}). The second part of the statement comes from the shift property of Whittaker functions $ \Psi_{\nu}(x)  = \Psi_{\nu_1-\I c, \dots, \nu_n-\I c}(x) e^{-c\sum_{i=1}^nx_i} $. Using the first part, we may write 
	$$  \vert \psi_{\I \alpha}(x)  \vert \leqslant  \vert \psi_{\I \alpha_1-\I c_2, \dots, \I \alpha_n-\I c_2}(x) e^{-c_2\sum_{i=1}^nx_i} \vert \leqslant C e^{-c_2 \sum_{i=1}^n x_i + (c_2-c_1)\sum_{i=1}^n\vert x_i\vert}.$$
\end{proof}
\begin{remark}
	In the Weyl chamber $\mathbb W_n:= \lbrace x\in \R^n :  x_1\geqslant x_2\geqslant \dots \geqslant x_n\rbrace $, one can  refine the estimate from Lemma \ref{lem:boundWhittaker} (see e.g. \cite[Corollary 2.3]{baudoin2011exponential}).   
	For any $n\geqslant 1$, there exists a constant $C>0$ such that for $x\in \mathbb W_n$ we have 
	$$ \vert \psi_{\nu}(x) e^{- \I \sum_{i=1}^n  \nu_i x_i}  \vert \leqslant C.$$ 
\end{remark}

When the index of Whittaker functions belongs to $\R^n$, Whittaker functions have doubly exponential decay away from the Weyl chamber. 
\begin{proposition}[{\cite[Proposition 4.1.3]{borodin2014macdonald}}] For $x\in \R^n$, define $\sigma(x) = \lbrace i: x_i-x_{i+1 }\leqslant 0\rbrace$. For each $\sigma\subset \lbrace 1, \dots, n\rbrace$, there exist a polynomial $A_{\sigma}$ such that for all $\nu\in \R^{n}$ and for all  $x\in \R^{n}$ with $\sigma(x) = \sigma$, we have 
	$$ \vert \psi_{\nu}(x)  \vert  \leqslant  A_{\sigma}(x) \prod_{i\in \sigma} \exp\left( -e^{-\frac{x_i-x_{i+1}}{2}}\right).$$
	\label{prop:doubleexponentialdecay}
\end{proposition}
This implies a  similar bound for $\mathcal{T}^{\,\tau}(x)$. 
\begin{proposition}Fix $n\geqslant 1$, $k\geqslant 0$, $\tau>0$ and $\alpha_1, \dots, \alpha_k\in \R$. Let $a>0$ be such that $a+\alpha_i>0$ for all $1\leqslant i\leqslant k$. Then,  there exists  a polynomial $A$  such that for all $x\in \R^n $ 
	$$\left\vert \mathcal{T}^{\,\tau}_{\alpha_{1}, \dots, \alpha_k}(x)\right\vert \leqslant  A(x) e^{a\sum_{i=1}^n x_i } \prod_{i=1}^{n-1} \exp\left( -e^{-\frac{x_i-x_{i+1}}{2}}\right).$$
	\label{prop:doubleexponentialdecayT}
\end{proposition}
\begin{proof}
We may use the change of variables $\nu=\tilde \nu+\I a$ in \eqref{eq:defT} so that the integration is over $\R^n$ and use Proposition \ref{prop:doubleexponentialdecay} to bound the Whittaker function inside the integral. 
\end{proof}
\begin{proposition}
Fix $n\geqslant 1$, $k\geqslant 0$, $\tau>0$ and $\alpha_1, \dots, \alpha_k\in \R$. Let  $R>0$ be such that $R+\alpha_i>0$ for all $1\leqslant i\leqslant k$. Then, there exists a constant $C>0$ such that for all $x\in \R^n$, 
$$ \vert  \mathcal{T}^{\,\tau}_{\alpha_{1}, \dots, \alpha_k}(x) \vert  \leqslant C e^{ R  \sum_{i=1}^n x_i}.$$
Moreover, the function $x\mapsto e^{- R  \sum_{i=1}^n x_i} \mathcal{T}^{\,\tau}_{\alpha_{1}, \dots, \alpha_k}(x)$ is in $\mathbb{L}^2(\R^n)$. 
\label{prop:boundT}
\end{proposition}
\begin{proof}
	Recall that in the definition of $\mathcal{T}^{\,\tau}_{\alpha_{1}, \dots, \alpha_k}(x)$ in \eqref{eq:defT}, the parameter $a>0$ can be taken arbitrarily as long as $a+\alpha_i>0$.
	 Using the shift property of Whittaker functions  $\psi_{\nu_1+\I a, \dots, \nu_n+\I a} = e^{-a \sum_{i=1}^n x_i}\psi_{\nu_1, \dots, \nu_n}(x)$ and the change of variables $ \nu_i  = \tilde\nu_i+\I a$ for all $1\leqslant i\leqslant n$, we obtain 
	\begin{multline}
	\vert \mathcal{T}^{\,\tau}_{\alpha_{1}, \dots, \alpha_k}(x) \vert   \leqslant  e^{a \sum_{i=1}^{n}x_i }   \\ 
	\times C \int_{\R^n}  \psi_{-\tilde \nu}(x) e^{-\tau  \sum_{j=1}^n \tilde \nu_j^2/2 } \left\vert \prod_{i=1}^k \prod_{j=1}^n \Gamma(\alpha_i - \I\tilde \nu_j+a) \prod_{1\leqslant i < j \leqslant n} \Gamma(-\I(\tilde \nu_i+\tilde \nu_j)+2a )
	\mskyl_n(\nu) \right\vert \mathrm{d}\tilde\nu. 
	\end{multline} 
	Using Lemma \ref{lem:boundWhittaker}, the integral can be bounded by a constant (depending on $a$ but not $x$), which yields the first part of the statement of Proposition \ref{prop:boundT} by choosing $a=R$. The fact that the integral above, seen as a function of $x$,  is bounded in $\mathbb{L}^2(\R^n)$ follows from the orthogonality of Whittaker functions \eqref{eq:orthogonalityone} and the Gaussian decay of the integrand.
\end{proof}
The estimate from Proposition \ref{prop:boundT} will be useful in Section \ref{sec:alternative} but it is not sufficiently sharp  for the proof of Proposition \ref{prop:integratestoone}. The next proposition is  a refinement.  
\begin{proposition} Fix $n\geqslant 2$, $k\geqslant 0$, $\tau>0$, $\alpha_1, \dots, \alpha_k\in \R$. Let  $a>0$ be such that $a+\alpha_i>0$ for all $1\leqslant i\leqslant k$. For any $R>a$, there exists a polynomial $C(x)$ such that for $x\in \R^n$, 
 	$$  \vert  \mathcal{T}^{\,\tau}_{\alpha_{1}, \dots, \alpha_k}(x) \vert  \leqslant C(x) \prod_{i: x_i<0} e^{ R  x_i}\prod_{i:x_i\geqslant 0} e^{a x_i}. $$
\label{prop:boundTchamber}
\end{proposition}
\begin{proof} 
 In the proof of Proposition \ref{prop:boundT}, we have seen that the ability to freely shift the integration contours can be translated into a decay bound. 
 In order to obtain a sharper bound we need to shift each contour separately by a distance depending on the sign of $x_i$. Since we may not use the shift property of Whittaker functions anymore,  it would be very convenient to decompose the Whittaker function $\psi_{\nu}(x)$ as linear combinations of $\prod_ie^{\I \nu_i x_i}$. Such a decomposition exists, in terms of the fundamental Whittaker functions introduced in \cite{hashizume1982whittaker}. We will be using notations from  \cite[Section 8]{o2012whittaker}. 
 Using  \cite[Proposition 3]{o2012whittaker} (see also \cite[Corollary 3]{o2012whittaker}) we may write 
  \begin{equation}
  \mathcal{T}^{\,\tau}_{\alpha}(x)  =   \int_{(\R+\I a )^n}    m_{-\I\nu}(x) e^{-\tau  \sum_{j=1}^n \nu_j^2/2 } Z(\alpha, \nu)
  \prod_{i<j} (\I\nu_i-\I\nu_j) \mathrm{d}\nu,
  \label{eq:alternativeT}
  \end{equation}
 where we have used the shorthand notation
 $$Z(\alpha, \nu)  = \prod_{i=1}^k \prod_{j=1}^n \Gamma(\alpha_i - \I\nu_j) \prod_{1\leqslant i < j \leqslant n} \Gamma(-\I(\nu_i+\nu_j)),$$ 
 and  $ m_{-\I \nu}(x)$ are fundamental Whittaker functions defined as follows. For $n\geqslant 2$ and $m\in (\Z_{\geqslant 0})^{n-1} $ we define analytic coefficients $a_{n,m}(\nu)$  by the recurrence \cite[Theorem 15]{ishii2007new} 
 $$ a_{2,m}(\nu) = \frac{1}{m! \Gamma(\nu_1 - \nu_2 + m+1)}, $$
 and for $n>2$, 
 \begin{equation}
 a_{n,m}(\nu)  = \sum_{k\in (\Z_{\geqslant 0})^{n-2}} a_{n-1, k}(\mu)  \prod_{i=1}^{n-1}\frac{1}{(m_i-k_i)!} \frac{1}{\Gamma(\nu_i-\nu_n+m_i-k_{i-1})}, 
 \label{eq:recurrenceanm}
 \end{equation}
 where the sum runs over $k\in (\Z_{\geqslant 0})^{n-2}$ such that $k_i\leqslant m_i$, we adopt the convention that $k_0=k_{n-1}=0$, and $\mu = (\nu_1 + \nu_n/(n-1), \dots, \nu_{n-1}+ \nu_n/(n-1), \nu_n) $.
 Then, fundamental Whittaker functions are defined by the  series 
 $$ m_{\nu}(x)  = \sum_{m\in (\Z_{\geqslant 0})^{n-1}} a_{n,m}( \nu) \exp\left( -\sum_{i=1}^{n-1} m_i(x_i -x_{i+1}) +\sum_{i=1}^{n} \nu_i x_i \right),$$
 which is absolutely convergent \cite[Lemma 4.6]{hashizume1982whittaker}. 
We will need the following estimate. 
\begin{lemma} Fix $a,R>0$ and $n\geqslant 2$. Let us define
	$$A_{N} =  \max_{\substack{m\in (\Z_{\geqslant 0})^{n-1}\\ \sum_{i=1}^{n-1} m_i=N}} \bigg\lbrace  \big\vert a_{n,m}(\I \nu) \big\vert  \bigg\rbrace . $$
There exist a positive constants $C,c, c'$ such that for  $\nu\in \C^n$ where for all $i$,  $\nu_i\in(\R+\I a)^n$ or $\nu_i\in(\R+\I (a+R))^n$, we have 
\begin{equation}
A_{N}  \leqslant   c e^{c' \sum_{i=1}^n\vert \Real[\nu_i]\vert}  \frac{C^N}{(N!)^2}.
\label{eq:estimateanm}
\end{equation}
\label{lem:estimatecoeffs}
\end{lemma}
\begin{proof}
We adapt the proof of \cite[Lemma 4.5]{hashizume1982whittaker} which concerns the case where $\nu$ is restricted to a compact set.  
The coefficients $a_{n,m}(\nu)$ satisfy another recurrence \cite[Theorem 15]{ishii2007new}: 
\begin{equation}
\left(\sum_{i=1}^{n-1} m_i^2 - \sum_{i=1}^{n-2} m_im_{i+1} +\sum_{i=1}^{n-1} (\nu_i-\nu_{i+1})m_i\right) a_{n,m}(\nu) = \sum_{i=1}^{n-1} a_{n, m-e_i}(\nu),
\label{eq:otherrecurrence}
\end{equation}
where $m-e_i = (m_1, \dots, m_i-1, \dots , m_{n-1})$, with $a_{n,m}=0$ if $m\not\in(\Z_{\geqslant 0})^{n-1}$ and 
$$ a_{n,0}(\nu)  = \prod_{i<j} \frac{1}{\Gamma(\nu_i-\nu_j+1)}. $$

We have 
$$ \sum_{i=1}^{n-1} m_i^2 - \sum_{i=1}^{n-2} m_im_{i+1} +\sum_{i=1}^{n-1} (\I\nu_i-\I\nu_{i+1})m_i = \sum_{i=0}^{n-1} \frac{1}{2} (m_i-m_{i+1})^2  +\sum_{i=1}^{n-1} (\I\nu_i-\I\nu_{i+1})m_i,$$
with the convention that $m_0=m_n=0$. 
Note that if $N=m_1+\dots + m_{n-1}$, we may write  $N=\sum_{i=0}^{n-1} (m_i-m_{i+1})(i+1)$, so that using Cauchy-Schwartz inequality, 
$$   N^2 \leqslant \sum_{i=0}^{n-1}  (m_i-m_{i+1})^2 \sum_{i=1}^{n} i^2 \leqslant  n^3 \sum_{i=0}^{n-1}  (m_i-m_{i+1})^2. $$
 Then, using $\vert z\vert \geqslant \vert \Real[z] \vert$, we can write 
$$ \left\vert \sum_{i=1}^{n-1} m_i^2 - \sum_{i=1}^{n-2} m_im_{i+1} +\sum_{i=1}^{n-1} (\I\nu_i-\I\nu_{i+1})m_i \right\vert > \frac{1}{2n^3} N^2 -N R,$$
where $N=m_1+\dots + m_{n-1}$. For $N>4 R n^3$, we have $N^2 -N R> N^2/(4n^3)$. In \eqref{eq:otherrecurrence}, we may  bound each term in the sum in the right hand side by  $A_{N-1}$, so that for $N>4 R n^3$, 
$$ A_{N} \leqslant \frac{4n^4}{N^2}  A_{N-1}.$$

Thus, if \eqref{eq:estimateanm} is true for $N\leqslant 4R n^3$, then it is true for all $N$, by choosing $C$ in \eqref{eq:estimateanm} larger than $4n^4$. Since $R$ and $n$ are fixed,  we may bound $A_N$ for all 
$N  \leqslant 4R n $ by $c e^{c' \sum_{i=1}^n\Real[\nu_i]}$ for some constants $c,c'$. Indeed, using the recurrence formula \eqref{eq:recurrenceanm}, coefficients $a_{n,m}$ are finite sums of finite products of Gamma factors which can be bounded by $c e^{c'  \sum_{i=1}^n\vert \Real[\nu_i]\vert }$. 
\end{proof}

Going back to the proof of Proposition \ref{prop:boundT}, we can rewrite \eqref{eq:alternativeT} as 
 \begin{multline}
  \mathcal{T}^{\,\tau}_{\alpha}(x)  =  \sum_{m\in (\Z_{\geqslant 0})^{n-1}} \int_{(\R+\I a )^n}  a_{n,m}(\I\nu)  \\ \times \exp\left( -\sum_{i=1}^{n-1} m_i(x_i -x_{i+1}) -\sum_{i=1}^{n} \I \nu_i x_i )\right) e^{-\tau  \sum_{j=1}^n \nu_j^2/2 } Z(\alpha, \nu)
  \prod_{i<j} (\I\nu_i-\I\nu_j) \mathrm{d}\nu.
  \label{eq:seriesforT}
 \end{multline}
 For all $i$, if $x_i<0$, we shift the contour by $\I (R-a)$ (that is from $\R+\I a$ to $\R+\I R$) and subsequently use the change of variables $\nu_i  = \tilde \nu_i+\I R$ so that $\tilde \nu_i\in \R$. If $x_i\geqslant 0$, we use the change of variables $\nu_i  = \tilde \nu_i+\I a$ so that $\tilde \nu_i\in \R$.  The factor $\exp\left(  -\sum_{i=1}^{n} \I \tilde \nu_i x_i \right)$ can be simply bounded by $1$.  We obtain that $\mathcal{T}^{\,\tau}_{\alpha}(x)$ is bounded by the product of three terms: 
 \begin{itemize}
 	\item $\prod_{i:x_i<0} e^{R x_i} \prod_{i:x_i\geqslant 0} e^{a x_i}$,
 	\item an absolutely convergent integral over $\tilde \nu\in \R$, which does not depend on $x$, 
 	\item the series  
 	$$ \sum_{m\in (\Z_{\geqslant 0})^{n-1} } \frac{ C^{m_1+\dots+m_{n-1}} \exp\left( -\sum_{i=1}^{n-1} m_i(x_i -x_{i+1})\right) }{((m_1+ \dots m_{n-1})!)^2},$$
 	where $C$ is the constant in Lemma \ref{lem:estimatecoeffs}. 
 \end{itemize}

Using that $(m_1+\dots + m_{n-1})!\geqslant m_1! \dots m_{n-1}!$  and the inequality  $\sum_{k=0}^{\infty} \frac{x^k}{(k!)^2} \leqslant \exp(2\sqrt{x})$ for positive $x$, we may bound the series as 
  	$$ \sum_{m\in (\Z_{\geqslant 0})^{n-1} } \frac{ C^{m_1+\dots+m_{n-1}} \exp\left( -\sum_{i=1}^{n-1} m_i(x_i -x_{i+1})\right) }{((m_1+ \dots m_{n-1})!)^2} \leqslant \prod_{i=1}^{n-1} e^{2C e^{\frac{-(x_i-x_{i+1})}{2}}}.$$
 
Thus we have arrived at 
\begin{equation}
 \vert  \mathcal{T}^{\,\tau}_{\alpha_{1}, \dots, \alpha_k}(x) \vert  \leqslant C'  \prod_{i: x_i<0} e^{ R  x_i}\prod_{i:x_i\geqslant 0} e^{a x_i} \prod_{i=1}^{n-1} e^{2 C e^{\frac{-(x_i-x_{i+1})}{2}}}. 
 \label{eq:intermediatebound}
\end{equation} 
Now we may combine this result with the bound of Proposition \ref{prop:doubleexponentialdecayT}.  Let us write 
$$ \vert  \mathcal{T}^{\,\tau}_{\alpha_{1}, \dots, \alpha_k}(x) \vert^{2C+1} = \vert  \mathcal{T}^{\,\tau}_{\alpha_{1}, \dots, \alpha_k}(x) \vert^{2C}\vert  \mathcal{T}^{\,\tau}_{\alpha_{1}, \dots, \alpha_k}(x) \vert,$$
where $C$ is the same constant as above.  
We bound the first factor with Proposition \ref{prop:doubleexponentialdecayT} and the second factor with \eqref{eq:intermediatebound}. The doubly exponential factors cancel and this concludes the proof of the proposition. 
\end{proof}

\begin{proof}[Proof of Proposition \ref{prop:integratestoone}]
	Let 
	$$ f(\alpha_1, \dots, \alpha_n)  = \int_{\R^n} \mathrm{d}x \  \psi_{\I\alpha_1, \dots, \I\alpha_n}(x) \mathcal{T}^{\,\tau}_{\diag, \alpha_{n+1}, \dots, \alpha_t}(x).$$
	Recall the definition of the function $\mathcal{T}^{\,\tau}$ in  \eqref{eq:defT}. Using orthogonality of Whittaker functions, we deduce that  when $\alpha_1, \dots, \alpha_n$ belong to $\I\R+a$ (where $a$ is the real part of the contour in \eqref{eq:defT}), 
	$$ f(\alpha_1, \dots, \alpha_n)  = e^{ \tau\sum_{i=1}^n \alpha_i^2/2}\ \prod_{i=1}^n{\Gamma(\diag+ \alpha_i)}\prod_{j=n+1}^{t}{\Gamma(\alpha_j+\alpha_i)}\prod_{1\leqslant i<j\leqslant n}{\Gamma(\alpha_i+\alpha_j)}.$$
	Recall that Whittaker functions are analytic functions in their index.  We wish to analytically continue the above identity for all $\alpha$ such that $\Real[\alpha_i]\geqslant a$, for some $a >0$ such that $a>-\diag$. As in the proof of \cite[Proposition 4.1.18]{borodin2014macdonald}, we prove that $f(\alpha_1, \dots, \alpha_n)$ is analytic  by showing  that  for any compact region $A\subset \lbrace z\in \C^n :  \Real[z_i]>a \rbrace$ and any $\e>0$, there exists a compact subset $\mathbb{K}\subset \R^n$ such that for all $\alpha\in A$, 
	\begin{equation}
	\int_{\R^n\setminus \mathbb{K}} \mathrm{d}x \  \left\vert \psi_{\I\alpha_1, \dots, \I\alpha_n}(x) \mathcal{T}^{\,\tau}_{\diag, \alpha_{n+1}, \dots, \alpha_t}(x) \right\vert  <\e.  
	\label{eq:tobound}
	\end{equation}
	Consider a compact set $A\subset \C^n$ such that for all $z\in A$, $\Real[z_i]>a$.   
	Let $\underline{a}>0$ be such that $\max\lbrace 0, -\diag\rbrace < \underline{a}< a$ and let $\bar a$ be such that $\Real[z_i] < \bar a$ for $\alpha\in A$. On the one hand, Proposition \ref{prop:boundTchamber}  shows that   for any $R>0$, there exists a polynomial $C(x)$ such that
	$$ \left\vert \mathcal{T}^{\,\tau}_{\diag, \alpha_{n+1}, \dots, \alpha_t}(x) \right\vert \leqslant  C(x) e^{\underline{a} \sum_{i=1}^n x_i} \prod_{i: x_i<0} e^{ R  x_i}.$$
	On the other hand, Lemma \ref{lem:boundWhittaker} shows that if $ a < \Real[\alpha_i] < \bar a$, 
	$$  \vert \psi_{\I \alpha}(x)  \vert \leqslant C e^{-\bar a \sum_{i=1}^n x_i + (\bar a -a)\sum_{i=1}^n\vert x_i\vert}.$$
	Combining these two bounds and using the fact that for $x<0$, $2x= x-\vert x\vert$, we obtain that for any $R>0$,
	\begin{eqnarray}
	\left\vert \psi_{\I \alpha}(x) \mathcal{T}^{\,\tau}_{\diag, \alpha_{n+1}, \dots, \alpha_t}(x) \right\vert &\leqslant& C(x) e^{-(a-\underline{a}) \sum_{i=1}^n x_i }e^{(\bar a -a)\sum_{i=1}^n \vert x_i\vert - x_i} \prod_{i: x_i<0} e^{ R  x_i},
	\nonumber \\ 
	&\leqslant& C(x) e^{-(a-\underline{a}) \sum_{i=1}^n x_i }e^{(\bar a -a-R/2)\sum_{i=1}^n \vert x_i\vert - x_i}.
	\label{eq:combining}
	\end{eqnarray}
	Let $\mathbb{A}^{(S)} = \lbrace x\in \R^n: \sum_{i=1}^n x_i \leqslant S\rbrace$. We have shown that there exist positive constants $C,c$ such that when $x\not \in (\R_{>-M})^n \cap \mathbb A^{(S)}$, 
	$$ \left\vert \psi_{\I \alpha}(x) \mathcal{T}^{\,\tau}_{\diag, \alpha_{n+1}, \dots, \alpha_t}(x) \right\vert \leqslant Ce^{-cS } + Ce^{-c M}.$$
Thus for any $\e>0$, taking the compact set $\mathbb{K}$ in the left-hand-side of  \eqref{eq:tobound} as   $\mathbb{K}= (\R_{>-M})^n \cap \mathbb A^{(S)}$ for $M,S$ large enough, \eqref{eq:tobound} holds. This concludes the proof of Proposition \ref{prop:integratestoone}. 
\end{proof}
\begin{proposition} Consider the scalings \eqref{eq:scalings1}, \eqref{eq:scalings2}. 
Assume that $\tau>0$ and $\alpha_1, \dots, \alpha_t>0$ are such that $\alpha_i+\diag>0$ for all $1\leqslant i\leqslant n$. Then the half-space $q$-Whittaker measure  $\PQWM_{(a_1,\ldots, a_n), (\diagq, a_{n+1}, \dots ,a_{t}), \gamma}$
 weakly converges to the half-space Whittaker measure 
 $\PWM_{(\alpha_1, \dots, \alpha_n) ,  (\diag, \alpha_{n+1}, \dots, \alpha_t), \tau}$. 
\label{prop:CVtoWhittmeasurewithplancherel}
\end{proposition}
\begin{proof}
The limits \eqref{eq:limPi}, \eqref{eq:limPhi},  Proposition \ref{prop:limitqWhittWhittwithplancherel} and  Proposition \ref{prop:limitepsilonwithplancherel} show that the density of the half-space $q$-Whittaker measure converge to the density of the half-space Whittaker measure for $x$ in a compact subset of $\R^n$. Proposition \ref{prop:integratestoone} shows that the half-space Whittaker measure is a probability measure. This is sufficient to deduce the weak convergence.
\end{proof}
\begin{remark}
 More  generally, the half-space $q$-Whittaker process also converges to the  half-space Whittaker process.
\end{remark}

\subsection{Observables and integral formulas}

Define the functions
$$  \mathcal{G}(v) = \frac{e^{-\tau v^2/2}}{\Gamma(\diag +v)}  \dfrac{\prod_{j=1}^{n} \Gamma(v-\alpha_j)   }{\prod_{j=1}^t\Gamma(\alpha_j+v)} \ \   \text{ and } \ \ 
\overline{\mathcal{G}}(v) = \frac{e^{-\tau v^2/2}}{\Gamma(\diag +v)}  \dfrac{\prod_{j=1}^{n} \Gamma(\alpha_j-v)   }{\prod_{j=1}^t\Gamma(\alpha_j+v)}   .$$

\subsubsection{Laplace transforms}
\begin{corollary}
	Let $t\geqslant n\geqslant 1$, and $\tau>0$. Assume that 
	\begin{enumerate}
		\item[(i)] The parameters $\alpha_1, \dots,  \alpha_t>0$, $\diag\in \R$ are chosen so that for all $1\leqslant i,j \leqslant n$, $\alpha_i-\alpha_j < \min_{1\leqslant i\leqslant n}\lbrace \alpha_i\rbrace$, $\alpha_i-\alpha_j<1$ and $\alpha_i+\diag>0$.  
		\item[(ii)] $R\in (0,1)$ is chosen so that $R<\min_{1\leqslant i\leqslant n}\lbrace   \alpha_i, \diag+ \alpha_i\rbrace$ and $R>\max_{1\leqslant i,j \leqslant n}\lbrace \alpha_i-\alpha_j\rbrace$. 
	\end{enumerate}
	Then, under the Whittaker measure $\PWM_{( \alpha_1,\ldots,  \alpha_n), (\diag,  \alpha_{n+1}, \dots,  \alpha_{t}, \tau)}$,
	\begin{multline}
	\EWM[e^{ u e^{ T_1}}] = \sum_{k=0}^{n} \ \frac{1}{k!}\ \int_{\mathcal{D}_{R}}\frac{\mathrm{d}s_1}{2\I\pi} \dots \int_{\mathcal{D}_{R}} \frac{\mathrm{d}s_k}{2\I\pi}\   \oint\frac{\mathrm{d}v_1}{2\I\pi} \dots \oint\frac{\mathrm{d}v_k}{2\I\pi}   \\  \times 
	\prod_{1\leqslant i<j\leqslant k} \frac{(s_i+v_j-s_j-v_i)(v_i-v_j)\Gamma(v_i+v_j)\Gamma(v_i+v_j-s_i-s_j)}{(v_j-s_j-v_i)(v_i-s_i-v_j)\Gamma(v_i+v_j-s_i)\Gamma(v_i+v_j-s_j)}\\ \times 
	\prod_{i=1}^k \left[ \Gamma(-s_i)\Gamma(1+s_i)   \frac{\overline{\mathcal{G}}(v_i)}{\overline{\mathcal{G}}(v_i-s_i)}  \frac{\Gamma(2v_i)}{\Gamma(2v_i-s_i)} \frac{(-u)^{s_i}}{s_i} \right],
	\label{eq:LaplaceWhittaker}
	\end{multline}
	where the contours for each variable $v_i$ are positively oriented circles enclosing the poles $\lbrace \alpha_j\rbrace_{1\leqslant j\leqslant n}$, and no other singularity of the integrand, and $\mathcal{D}_R=R+\I\R$ oriented upwards as before. 
	\label{cor:LaplaceWhittaker}
\end{corollary}
It should be noted that the assumption that the Plancherel component $\tau$ is positive is essential here. The R.H.S of \eqref{eq:LaplaceWhittaker} would simply diverge for $\tau=0$. 
\begin{proof}
	This is the $q\to 1$ limit of Corollary \ref{cor:LaplaceqWhittaker}. Consider the scalings \eqref{eq:scalings1},  \eqref{eq:scalings2}.  Recall the $q$-Gamma and the $q$-exponential functions from Section \ref{sec:qanalogues}. 
%	Recall that 
%	$$ \Pi(w; \diagq,  a_{n+1}, \dots,  a_{t}, \gamma)  = e^{\gamma w} \frac{1}{(\diagq w)_{\infty}}\prod_{i=n+1}^t \frac{1}{(w a_i)_{\infty}}.$$
	Setting $z = u (1-q)^{t+n}\exp(-\e^{-1} \tau)$ in Corollary \ref{cor:LaplaceqWhittaker} and using the change of variables $w_i=q^{v_i}$ yields 
	\begin{multline}
	\EWM\left[e_q\left( u e^{ T_1(q)} \right)\right] = \sum_{k=0}^{n} \ \frac{1}{k!}\ \int_{\mathcal{D}_{R}}\frac{\mathrm{d}s_1}{2\I\pi} \dots \int_{\mathcal{D}_{R}} \frac{\mathrm{d}s_k}{2\I\pi}\   \oint\frac{\mathrm{d}v_1}{2\I\pi} \dots \oint\frac{\mathrm{d}v_k}{2\I\pi}    \\   \times 
	\prod_{1\leqslant i<j\leqslant k} \frac{(s_i+v_j-s_j-v_i)(v_i-v_j)\Gamma_q(v_i+v_j)\Gamma_q(v_i+v_j-s_i-s_j)}{(v_j-s_j-v_i)(v_i-s_i-v_j)\Gamma_q(v_i+v_j-s_i)\Gamma_q(v_i+v_j-s_j)}\\ \times 
	\prod_{i=1}^k \left[ \Gamma(-s_i)\Gamma(1+s_i) \prod_{j=1}^{n} \left( \frac{\Gamma_q(\alpha_j-v_i)\Gamma_q(\alpha_j+v_i-s_i)}{\Gamma_q(\alpha_j+v_i)\Gamma_q(\alpha_j+s_i-v_i)}\right)  \frac{\Gamma_q(2v_i)}{\Gamma_q(2v_i-s_i)} \right. \\ \left. \times  \frac{\Gamma_q(\diag +v_i-s_i)}{\Gamma_q(\diag +v_i)}  \prod_{j=n+1}^t \left( \frac{\Gamma_q(\alpha_j+v_i-s_i)}{\Gamma_q(\alpha_j+v_i)} \right)\frac{(-u)^{s_i}e^{-\tau \e^{-1} s}e^{\gamma q^{v_-}(q^{-s_i}-1)} }{(1-q^{s_i})q^{v_i}} \right].
	\label{eq:LaplaceqWhittaker2}
	\end{multline}
	Indeed, consider first the limit of the L.H.S.  
	Using Proposition \ref{prop:CVtoWhittmeasurewithplancherel} and the convergence of the $q$-exponential function to the exponential on compacts, $\EWM[e_q( u e^{ T_1(q)})]$ converges to the L.H.S. of \eqref{eq:LaplaceWhittaker}. Now we consider the R.H.S.
	Note that for any $x\in \C \setminus \R_{<0}$, $ \Gamma_q(x) \xrightarrow[q\to 1]{} \Gamma(x),$
	and with $\gamma=\tau \e^{-2}$ and $w=q^v$, we have 
	$$ e^{-\tau \e^{-1} s}e^{\gamma(q^{-s}w-w)} \xrightarrow[\e \to 0]{} e^{-\tau v s +\tau s^2/2} = e^{\tau(v-s)^2/2 - \tau  v^2/2 }.$$
	This shows that the integrand in the R.H.S. in \eqref{eq:LaplaceqWhittaker2} converges pointwise to the integrand in the R.H.S. of \eqref{eq:LaplaceWhittaker}. To conclude using dominated convergence, we need to show that the integrands are uniformly (with respect to $q$) integrable.  We may first evaluate the integrals over the $v_i$.  The residues occur when $z_i = \alpha_{p(i)}$ for some choice of $p:\lbrace 1, \dots, k\rbrace \to \lbrace 1, \dots, n\rbrace$. Hence, 
	\begin{multline}
	\EWM\left[e_q\left( u e^{ T_1(q)} \right)\right] = \sum_{k=0}^{n} \ \frac{1}{k!}\ \int_{\mathcal{D}_{R}}\frac{\mathrm{d}s_1}{2\I\pi} \dots \int_{\mathcal{D}_{R}} \frac{\mathrm{d}s_k}{2\I\pi}\   \\ \sum_{p:\lbrace 1, \dots, k\rbrace \to \lbrace 1, \dots, n\rbrace }     
	\prod_{1\leqslant i<j\leqslant k} \frac{(s_i+\alpha_{p(j)}-s_j-\alpha_{p(i)})(\alpha_{p(i)}-\alpha_{p(j)})\Gamma_q(\alpha_{p(i)}+\alpha_{p(j)})\Gamma_q(\alpha_{p(i)}+\alpha_{p(j)}-s_i-s_j)}{(\alpha_{p(j)}-s_j-\alpha_{p(i)})(\alpha_{p(i)}-s_i-\alpha_{p(j)})\Gamma_q(\alpha_{p(i)}+\alpha_{p(j)}-s_i)\Gamma_q(\alpha_{p(i)}+\alpha_{p(j)}-s_j)}\\ \times 
	\prod_{i=1}^k \left[ \Gamma(-s_i)\Gamma(1+s_i) \prod_{j=1}^{n} \left( \frac{\Res{z=0}\lbrace \Gamma_q(z)\rbrace\Gamma_q(\alpha_j+\alpha_{p(i)}-s_i)}{\Gamma_q(\alpha_j+\alpha_{p(i)})\Gamma_q(\alpha_j+s_i-\alpha_{p(i)})}\right)  \frac{\Gamma_q(2\alpha_{p(i)})}{\Gamma_q(2\alpha_{p(i)}-s_i)} \right. \\ \left. \times \frac{\Gamma_q(\diag +\alpha_{p(i)}-s_i)}{\Gamma_q(\diag +\alpha_{p(i)})}  \prod_{j=n+1}^t \left( \frac{\Gamma_q(\alpha_j+\alpha_{p(i)}-s_i)}{\Gamma_q(\alpha_j+\alpha_{p(i)})} \right)\frac{(-u)^{s_i}e^{-\tau \e^{-1} s}e^{\gamma q^{-s}(q^{\alpha_{p(i)}}-1)} }{(1-q^{s_i})q^{\alpha_{p(i)}}} \right].
	\label{eq:LaplaceqWhittaker3}
	\end{multline}
	For any fixed $a,b>0$,  there exists a constant $C_1>0$ such that for any $y\in \R$ and $q\in (\frac 1 2,1)$, we have (see Lemma 2.7  in \cite{barraquand2017random}),
	\begin{equation}
	 \left| \frac{\Gamma_q(a+i y)}{\Gamma_q(b+iy)}\right|\leqslant C_1 \Big(\big\vert  y\big\vert^{\vert b-a\vert+1}+1\Big). 
	 \label{eq:estimateqGamma}
	 \end{equation}
	Moreover (see Lemma 5.11 in \cite{borodin2015height}), there exist constants $C_2, C_3$ such that for $x\in (a,b)$ and $y\in \R$, 
	$$ \left\vert  \Gamma_q(x+\I y) \right\vert <C_2 , \ \ \  \left\vert \frac{1}{\Gamma_q(x+ \I y)} \right\vert < e^{C_3 \vert y\vert}. $$ 
	By a Taylor approximation in $\e$, for $\e$ small enough, there exist a constant $C_4$ such that 
	$$ \vert e^{-\tau \e^{-1} s}e^{\gamma q^{\alpha_{i}}(q^{-s}-1)} \vert < e^{-C_4 \vert s\vert^2}, \;\; s\in\mathcal{D}_{R}.$$ 
	Using these estimates, each integrand in \eqref{eq:LaplaceqWhittaker3} can be bounded uniformly in $q\in (1/2,1)$ by 
	$$ \prod_{i=1}^k C_5 e^{C_6\vert s_i\vert - C_7 \vert s_i\vert^2}, $$
	for some constants $C_5,C_6, C_7>0$, 
	which is integrable over $\vec{s}\in \left(\mathcal{D}_{R}\right)^k$. Finally, the sum over $k$ is finite, so that dominated convergence can be applied. 
\end{proof}

\begin{corollary}
		Let $t\geqslant n\geqslant 1$, $\alpha_1, \dots,  \alpha_t>0$, $\diag\in \R$ such that $\alpha_i+\diag>0$ and $\tau>0$. Assume that the parameters  $ \alpha_1, \dots, \alpha_n$ are chosen such that for all $1\leqslant i,j \leqslant n $,  $\alpha_i-\alpha_j<1$,  and let $R\in (0,1)$ be such that $R>\alpha_i - \alpha_j$ for all $1\leqslant i,j \leqslant n$. 
		Under the Whittaker measure $\PWM_{( \alpha_1,\ldots,  \alpha_n), (\diag,  \alpha_{n+1}, \dots,  \alpha_{t}, \tau)}$,
		\begin{multline}
		\EWM\left[  e^{ u e^{- T_n}}\right]= \sum_{k=0}^{n} \frac{1}{k!} \int_{\mathcal{D}_{R}}\frac{\mathrm{d}s_1}{2\I\pi} \dots \int_{\mathcal{D}_{R}} \frac{\mathrm{d}s_k}{2\I\pi}\  
		\oint\frac{\mathrm{d}v_1}{2\I\pi} \dots \oint\frac{\mathrm{d}v_k}{2\I\pi}  \\  \times
		\prod_{1\leqslant i<j\leqslant k} \frac{(s_j+v_j-s_i-v_i)(v_i-v_j)\Gamma(v_i+v_j)\Gamma(s_i+s_j+v_i+v_j)}{(s_i+v_i-v_j)(s_j+v_j-v_i)\Gamma(s_i+v_i+v_j)\Gamma(s_j+v_i+v_j)} \\ 
		\times	\prod_{i=1}^k \left[  \Gamma(-s_i)\Gamma(1+s_i) \frac{\mathcal{G}(v_i)}{\mathcal{G}(v_i+s_i)}  \frac{\Gamma(2v_i)}{\Gamma(s_i+2v_i)} \frac{(-u)^{s_i}}{-s_i} \right],
		\label{eq:expansionNoumiWhittaker}
		\end{multline}
		where the contours for each variable $v_i$ are positively oriented circles enclosing the poles $\lbrace \alpha_j\rbrace_{1\leqslant j\leqslant n}$, and no other singularity of the integrand.
		\label{cor:LaplacecorollaryTn}
\end{corollary}
\begin{proof}
	The formula can be obtained from Corollary \ref{cor:LaplaceqWhittakerlambdan} similarly to the proof of Corollary \ref{cor:LaplaceWhittaker}.
\end{proof}

\subsubsection{Moment formulas}

\begin{proposition} Let $t\geqslant n\geqslant 1$, $\alpha_1, \dots,  \alpha_t>0$ and $\diag>0$ so that $\alpha_i+\diag>0$ for all $1\leqslant i \leqslant n$. 
	Under the Whittaker measure $\PWM_{( \alpha_1,\ldots,  \alpha_n), (\diag,  \alpha_{n+1}, \dots,  \alpha_{t}, \tau)}$, we have the following moment formulas. For any $k\in\Z_{>0}$, 
	\begin{equation}
	\EWM[e^{-k T_n}] = e^{k\tau/2} \oint\frac{\mathrm{d}w_1}{2\I\pi}\cdots \oint\frac{\mathrm{d}w_k}{2\I\pi} \prod_{1\leqslant a<b\leqslant k}\frac{w_a-w_b}{w_a-w_b-1}\,  \frac{1+ w_a+w_b}{w_a+w_b}
	\prod_{m=1}^{k} \frac{\mathcal{G}(w_m)}{ \mathcal{G}(w_m+1)}\frac{1}{2w_m},
	\label{eq:momentsTn}
	\end{equation}
	where the positively oriented contours are such that for all $1\leqslant c\leqslant k$, the contour for $w_c$ encloses $\lbrace  \alpha_j\rbrace_{1\leqslant j\leqslant n}$  and $\lbrace w_{c+1}+1, \dots, w_k+1\rbrace$, and excludes the pole of the integrand at $0$.

	For all $k\in\Z_{>0}$ such that $k< 2 \min\lbrace \alpha_i\rbrace$ and $k< \diag+ \min\lbrace \alpha_i\rbrace$,
	\begin{equation}
	\EWM[e^{k T_1}] =e^{k\tau/2}\oint\frac{\mathrm{d}w_1}{2\I\pi}\cdots \oint\frac{\mathrm{d}w_k}{2\I\pi} \prod_{1\leqslant a<b\leqslant k} \frac{w_a-w_b}{w_a-w_b-1}\, \frac{1+w_a+w_b}{2+ w_a+w_b}\\
	\prod_{m=1}^{k}  \frac{\overline{\mathcal{G}}(-w)}{\overline{\mathcal{G}}(-w-1)} (1+2w_m),
	\label{eq:momentsT1}
	\end{equation}
	where the contours are such that for all $1\leqslant c\leqslant k$, the contour for $w_c$ encloses $\lbrace - \alpha_j\rbrace_{1\leqslant j\leqslant n}$  and $\lbrace w_{c+1}+1, \dots, w_k+1\rbrace$, and excludes the poles of the integrand at $ \diag- 1$ and $ \alpha_j-1$ (for $1\leqslant j\leqslant t$). 
	\label{prop:momentsWhittaker}
\end{proposition}
\begin{remark}
Proposition \ref{prop:momentsWhittaker} corresponds to the $q\to 1$ limit of analogous formulas in the $q$-Whittaker case  stated as Corollary \ref{cor:momentsqWhittaker}. However, since weak convergence does not imply convergence of moments in general, Proposition  \ref{prop:momentsWhittaker} cannot be deduced from Corollary \ref{cor:momentsqWhittaker}. 
\end{remark}
\begin{proof}
	Observe that 
	$$\frac{\mathcal{G}(w)}{\mathcal{G}(w+1)} =  e^{\tau w} (w + \diag) \prod_{i=1}^t \left(w+\alpha_i\right) \prod_{j=1}^{n}\left( \frac{1}{w - \alpha_j}\right)$$
	and 
	$$ \frac{\overline{\mathcal{G}}(-w)}{\overline{\mathcal{G}}(-w-1)} = \frac{e^{\tau w}}{1+w-\diag}\prod_{i=1}^t \left( \frac{1}{\alpha_i-w-1} \right) \prod_{j=1}^{n} \left(\frac{1}{w + \alpha_j}\right).$$

We use a similar approach to the proof of Proposition \ref{prop:momentsMacdonald} using operators diagonalized by Whittaker functions. 
The eigenrelation \eqref{eq:eigenrelationD}  for Macdonald difference operator $\DD_n^1$ becomes the following in the Whittaker limit (see \cite[Lemma 4.1.36]{borodin2014macdonald}). 
\begin{equation}
\sum_{i=1}^n \prod_{j\neq i} \frac{1}{\alpha_i-\alpha_j} \Tshift_i \psi_{\I\alpha_1, \dots, \I\alpha_n}(x) = e^{-x_n} \psi_{\I\alpha_1, \dots, \I\alpha_n}(x),
\label{eq:diffoperatorWhitt}
\end{equation}
where $\Tshift_i$ acts on functions in the variables $\alpha_1, \dots , \alpha_n$ by shifting the $i$th coordinate by $1$. Using that $\psi_{w}(x_1, \dots, x_n) = \psi_{-w}(-x_{n}, \dots, -x_1)$, we obtain that 
\begin{equation}
\sum_{i=1}^n \prod_{j\neq i} \frac{1}{\alpha_j-\alpha_i} \Tshift^{-1}_i \psi_{\I\alpha_1, \dots, \I\alpha_n}(x) = e^{x_1} \psi_{\I\alpha_1, \dots, \I\alpha_n}(x),
\label{eq:diffoperatorWhittbis}
\end{equation}
where $\Tshift^{-1}_i$ acts on functions in the variables $\alpha_1, \dots , \alpha_n$ by shifting the $i$th coordinate by $-1$.

Let us iterate the relation \eqref{eq:diffoperatorWhitt} $k$ times, and use Proposition \ref{prop:integratestoone}. We obtain 
\begin{equation}
\EWM[e^{-k T_n}]  = \frac{1}{Z(\alpha_1, \dots, \alpha_n)} \sum_{i_1, \dots, i_k=1}^n \prod_{\ell=1}^k \left( \prod_{j\neq {i_{\ell}}} \frac{1}{\alpha_{i_{\ell}}-\alpha_j}\right) \Tshift_{i_{\ell}}   Z(\alpha_1, \dots, \alpha_n).
\label{eq:actionoperatrorTn}
\end{equation}
where 
$$	Z(\alpha_1, \dots, \alpha_n) = e^{ \tau\sum_{i=1}^n \alpha_i^2/2}\ \prod_{i=1}^n{\Gamma(\diag+ \alpha_i)}\prod_{j=n+1}^{t}{\Gamma(\alpha_j+\alpha_i)}\prod_{1\leqslant i<j\leqslant n}{\Gamma(\alpha_i+\alpha_j)}$$
	
Similarly, if  $\alpha_i>2k$ and $\alpha_i+\diag >k$ for all $1\leqslant i\leqslant n$, we may iterate the relation \eqref{eq:diffoperatorWhittbis} $k$ times and use Proposition \ref{prop:integratestoone}. We obtain,    	
\begin{equation}
\EWM[e^{k T_1}]  = \frac{1}{Z(\alpha_1, \dots, \alpha_n)} \sum_{i_1, \dots, i_k=1}^n \prod_{\ell=1}^k \left( \prod_{j\neq {i_{\ell}}} \frac{1}{\alpha_j - \alpha_{i_{\ell}}}\right) \Tshift^{-1}_{i_{\ell}}   Z(\alpha_1, \dots, \alpha_n).
\label{eq:actionoperatrorTone}
\end{equation}
The assumptions on the $\alpha_i$'s ensure that one can apply Proposition \ref{prop:integratestoone} in the summand after applying \eqref{eq:diffoperatorWhittbis} $k$ times. To conclude the proof of the proposition, we need to rewrite \eqref{eq:actionoperatrorTn} and \eqref{eq:actionoperatrorTone} as contour integrals, which can be done very similarly to Proposition \ref{prop:momentsMacdonald}. 
\end{proof}

\subsection{Limits of $q$-Whittaker dynamics}
\label{sec:Whittakerdynamics}

We may represent the ascending half-space Whittaker process $T$ from Definition \ref{def:PWM} as an array 
\begin{center}
	\begin{tikzpicture}[scale=0.4]
	\draw (0,0) node{$ T_1^{(1)}$};
	\draw (-2,2) node{$ T_2^{(2)}$};
	\draw (2,2) node{$ T_1^{(2)}$};
	\draw (-4, 4) node[rotate=-45]{$ \ldots$};
	\draw (4,4) node[rotate=45]{$ \ldots$};
	\draw (-7,7) node{$ T_{n-1}^{(n-1)}$};
	\draw (-3,7) node{$  T_{n-2}^{(n-1)}$};
	\draw (0,7)  node{$\ldots$};
	\draw (3,7) node{$  T_{2}^{(n-1)}$};
	\draw (7,7) node{$ T_{1}^{(n-1)}$};
	\draw (-9,9) node{$ T_{n}^{(n)}$};
	\draw (-5,9) node{$ T_{n-1}^{(n)}$};
	\draw (0,9)  node{$\ldots$};
	\draw (5,9) node{$  T_{2}^{(n)}$};
	\draw (9,9) node{$  T_{1}^{(n)}$};
	\end{tikzpicture} 
\end{center}
This triangular array is the (scaling) limit of the sequence of random $q$-Whittaker partitions $\la^{(t,1)}\prec \dots \prec \la^{(t,n)}$. From Sections \ref{sec:rowinsertion} and \ref{sec:columninsertion}, there exist (multivariate) Markov dynamics which map $\la^{(t,1)}\prec \dots \prec \la^{(t,n)}$ to $\la^{(t+1,1)}\prec \dots \prec \la^{(t+1,n)}$. In the $q\to 1$ limit, the triangular array is not interlacing anymore, but 
there should exist multivariate Markov dynamics transporting $T\in \R^{n(n+1)/2}$ distributed as an ascending half-space Whittaker process to  $T'\in \R^{n(n+1)/2}$ distributed as an ascending half-space Whittaker process with updated specializations, in such a way that the left edge and the right edge of the triangular array are both marginally Markov and correspond respectively to the limit of the dynamics from Sections \ref{sec:rowinsertion} and \ref{sec:columninsertion}.

In this section, we describe these limiting marginal dynamics. Connecting rigorously the resulting model to the law of $T_1^{(n)}$ and $T_{n}^{(n)}$ under the half-space Whittaker measure presents some technicalities (related to the fact that we must keep a positive Plancherel component when we work with half-space Whittaker measures) that we discuss in more details in Section \ref{sec:relationWhittakerpolymer}. 

Throughout this Section, we use the same scalings of parameters as in Section \ref{sec:rigorousconvwithplancherel}.
We will often use the letter $\theta$ to denote the parameter of $q$-deformed probability distributions, and we will always scale this parameter as $ \theta=e^{-\e \tilde{\theta}}$.
Consider a half-space $q$-Whittaker process indexed by a path $\pathh$ as in Figure \ref{pathWhittaker}, where the sequence of specializations $\bm\uprho$ is chosen as in Section \ref{sec:rskdynamics} (that is, edges $ (i-1, j) \leftarrow (i,j) $ are labeled by single variable specialization $a_i$, edges $ (i, j-1) \rightarrow (i,j) $ are labeled by single variable specialization $a_j$, and the diagonal edge is labeled by specialization $\diagq$). 
\begin{figure}
	\begin{tikzpicture}[scale=1]
	\begin{scope}[decoration={
		markings,
		mark=at position 0.5 with {\arrow{<}}}]
	\draw[->, >=stealth', gray] (0,0) -- (8.5, 0);
	\draw[->, >=stealth', gray] (0,0) -- (0,5.5);
	\draw[gray, dotted] (0,0) grid(8.5, 5.5); 
	\draw[gray, dotted] (5, 3.5)  -- (5,0) node[anchor = north]{$i$};
	\draw[gray, dotted] (0,3) node[anchor= east]{$j$} -- (5.5,3) ;
	\draw[gray] (4.5,2) node[anchor=north]{$ a_i$};
	\draw[gray] (4,2.5) node[anchor=east]{$ a_j$};
	\draw[gray] (5,3) node[anchor =south west]{$(i,j)$};
	\fill[gray] (5,3) circle(0.08);
	\fill (7,5) circle(0.08);
	\draw[] (7,0) node[anchor = north]{$t$};
	\draw[dashed] (0,5) node[anchor= east]{$n$} -- (7,5);
%	\draw[fleche]  (7,0) -- (7,1) -- (7,2) -- (7,3) -- (7,4) -- (7,5) -- (5,5) -- (0,0);
	\draw[fleche] (0,0) -- (5,5) ;
	\draw[fleche] (5,5) -- (7,5);
	\draw[fleche] (7,5) -- (7,0);
	\draw (7,5) node[anchor =south west]{$(t,n)$};
	\end{scope} 
	\end{tikzpicture}
	\caption{The path $\pathh \in \admpath$ considered in Section \ref{sec:Whittakerdynamics}}
	\label{pathWhittaker}
\end{figure}

\subsubsection{Right edge dynamics}
\label{sec:Whittakerdynamicsright}
We are interested here in the scaling (as in \cite[Theorem 4.2.4]{borodin2014macdonald} and \cite[Definition 8.4]{matveev2015q})

\begin{equation} \lambda_j^{(t,n)} = (t+n+1-2j)\e^{-1}\log \e^{-1} + \e^{-1}\log(R_j^{\e}(t,n)),
\label{eq:scalingrightedge}
\end{equation}
where $\lambda^{(t,n)}$ is distributed according to the half-space $q$-Whittaker measure $ \PQWM_{(a_1, \dots, a_n), (\diagq, a_{n+1}, \dots, a_t)} $. We will show that under the scaling of parameters  \eqref{eq:scalings1}, the  random variable $\big( R^{\e}_j(t,n)\big)_{1\leqslant j\leqslant n}$ weakly converges to some random variable $\big(R_j(t,n)\big)_{1\leqslant j\leqslant n}$.

\begin{definition}
The Gamma distribution with (shape) parameter $\tilde\theta$ (and scale parameter $1$), denoted $ \mathrm{Gamma}(\tilde\theta) $,  is the continuous probability distribution on $\R_{>0}$ with density 
$$ \frac{1}{\Gamma(\tilde\theta)}x^{\tilde\theta -1}e^{-x}.$$
Thus, if $X\sim \mathrm{Gamma}(\tilde\theta)$ then $X^{-1}$ has the inverse Gamma distribution, denoted $ \mathrm{Gamma}^{-1}(\tilde\theta) $,  and density 
$$ \frac{1}{\Gamma(\tilde\theta)}x^{-\tilde\theta -1}e^{-1/x}.$$
\label{def:Gammadist}
\end{definition}

Assume that we have sampled partitions $\la^{(v)}$ for $v$ belonging to or sitting below the path $\pathh$ of Figure \ref{pathWhittaker}, according to the rules of Section \ref{sec:rskdynamics} and using operators $\Udiag$ and $\U_{\rm row}$. 
For $1\leqslant m\leqslant s\leqslant t$ and $m \leqslant n$,  we have set 
$ \lambda_j^{(s,m)} = (s+m+1-2j)\e^{-1}\log \e^{-1} + \e^{-1}\log(R_j^{\e}(s,m)).$
Theorem 8.7 in \cite{matveev2015q} shows that in the situation of the $q$-Whittaker process in the full space case, the array $R^{\e}$ converges to some explicit limit $R$. In the course of proving this result, \cite{matveev2015q} shows that the dynamics $\U_{\rm row}$ have a limit in the following sense. 
\begin{lemma}[\cite{matveev2015q}]
	Assume that $\lambda_j^{(s-1,m)}, \lambda_j^{(s-1,m-1)}, \lambda_j^{(s,m-1)}$ are such that 
	there exist real random variables $R_j(s-1,m), R_j(s-1,m-1)$ and $R_j(s,m-1)$ so that we have the weak convergences
	\begin{align*}
	R_j^{\e}(s-1,m-1) &\xRightarrow[\e\to 0]{}R_j(s-1,m-1),\\
	R_j^{\e}(s-1,m)&\xRightarrow[\e\to 0]{}R_j(s-1,m),\\
	R_j^{\e}(s,m-1)&\xRightarrow[\e\to 0]{}R_j(s,m-1).
	\end{align*} 
	 If $\lambda^{(s,m)}$ is sampled according to the dynamics  $\U_{\rm row}(\lambda^{(s,m)} \vert \lambda^{(s,m-1)}, \lambda^{(s-1,m-1)}, \lambda^{(s-1,m)})$, then the sequence $R^{\e}_j(s,m)$ ($1\leqslant j\leqslant m$) converges weakly as $\e$ goes to zero to some sequence $\big(R_j(s,m)\big)_{1\leqslant j\leqslant m}$ whose distribution is explicit. In particular, $R_1(s,m)$ depends only on $R_1(s-1,m), R_1(s-1,m-1), R_1(s,m-1)$, and we have 
	$$ R_1(s,m) = d_{sm} \big( R_1(s, m-1) + R_1(s-1, m)\big)$$
	where $d_{sm}$ is a $\mathrm{Gamma}^{-1}(\alpha_m + \alpha_s)$ random variable, the random variables  $\big(d_{sm}\big)_{s,m}$ are independent, and we adopt the convention that $R_1(0,j)=R_1(t,0)=0$.
	\label{lem:limitqRSK}
\end{lemma}

In order to have a recurrence characterizing  the law of $R(t,j)$ completely,  we examine the degeneration of the boundary operator $\Udiag$ as $\e\to0$. We will only focus on the projection to the first coordinate coordinate $\lambda_1$ (i.e.,  $R_1$ in the limit), which is marginally Markov (by Lemma \ref{lem:firstpartdynamics}). 
\begin{lemma}[Lemma 8.15 \cite{matveev2015q}]
Let $X_\e$ be a $\Z_{\geqslant 0}$-family of random variables with $\mathrm{qGeom}(\theta)$ distribution and set 
$$X_\e = \e^{-1}\log(\Gamma_\e) + \e^{-1}\log(\e^{-1}).$$
Then, if $\theta=e^{-\e \tilde\theta}$, $\Gamma_{\e}$ weakly converges to a $\mathrm{Gamma}^{-1}(\tilde{\theta})$ random variable. 
\label{lem:limitqgeom}
\end{lemma}
Lemma \ref{lem:limitqgeom} applied to the result of Lemma \ref{lem:firstpartdynamics} implies that in the $ \e \to 0 $  limit, for $t\geqslant 2$
$$ R_1(t,t) = d_{tt} R_1(t,t-1)$$ 
where $d_{tt}$ is a $\mathrm{Gamma}^{-1}(\diag+\alpha_t)$ random variable, and $R(1,1)= d_{11}$.

\begin{proposition}
For any path $\pathh\in \admpath$, the sequence $ \left(R^{\e}_1(t,n) \right)_{(t,n)\in \pathh}$ defined by \eqref{eq:scalingrightedge} converges in distribution as $\e$ goes to zero to a sequence $\left(R(t,n) \right)_{(t,n)\in \pathh}$ such that 
$$ \left( Z(t,n) \right)_{(t,n)\in \pathh} \overset{(d)}{=} \left(R(t,n) \right)_{(t,n)\in \pathh},$$
where $ Z(t,n)$ is the partition function of the half-space log-gamma directed polymer (Definition \ref{def:LogGammapolymer}). 
In particular, under the notations above, 
$$ (1-q)^{t+n-1} q^{-\la_1^{(t,n)}} \xRightarrow[q\to 1]{} Z(t,n). $$
\label{prop:LogGammaqWhittakerlimit}
\end{proposition}
\begin{proof}
The convergence in distribution is proved by Lemmas \ref{lem:limitqRSK} and \ref{lem:limitqgeom}. 
Moreover, $Z(n,m)$ and $R(n,m)$ satisfy the same recurrence relation in law, and hence have the same distribution.
\end{proof}

\subsubsection{Left-edge dynamics}
\label{sec:Whittakerdynamicsleft}
We consider now the scaling 
$$ \lambda_{t}^{(t,j)} = (t-j+1)\e^{-1}\log \e^{-1} - \e^{-1}\log(L^{\e}(t,j)),$$
where  $\lambda^{(t,j)}$ is distributed according to the half-space $q$-Whittaker measure $ \PQWM_{(a_1, \dots, a_j), (\diagq, a_{j+1}, \dots, a_t)} $.  
We can show in a way similar to the right-edge case that  $L^{\e}(t,j)$ weakly converges to some $L(t,j)$, where the family of random variables $\lbrace L(t,j)\rbrace_{1\leqslant j\leqslant t}$ has an explicit recursive description. 
In a way similar to the right-edge dynamics, 
the proof of Theorem 8.8 in \cite{matveev2015q} implies that
when $ t>j$, 
$$ L(t,j) = L(t-1, j-1) + g_{t,j} L(t-1,j),$$
where $g_{t,j}$ is a $\mathrm{Gamma}(\alpha_t+\alpha_j)$ random variable and the family $ \big(g_{t,j}\big)$ consists of independent members. 

The limit of the boundary dynamics $\Udiag$ is more complicated. 
\begin{lemma}
\label{lem:convqinversegaussian}
Let $X_{\e}$ be a q-inverse Gaussian random variable with parameters $m$ and $\theta$ (see Definition \ref{def:qinverseGaussian}). Assume that $\theta=e^{-\e \tilde\theta}$ and 
$ m= 2\e^{-1}\log(\e^{-1}) - \e^{-1}\log(L)$.
Then, letting $ X_{\e} = \e^{-1}\log(\e^{-1}) - \e^{-1}\log(Y_\e)$, the random variable $Y_{\e}$ weakly converges to the continuous random variable on $\R_{>0}$ with density 
$$\frac{1}{2L^{\tilde\theta/2} K_{\theta}(2\sqrt{L})} x^{\tilde\theta -1}e^{-x-L/x},$$
where  $K_{\theta}$ is the modified Bessel function of the second kind\footnote{Sometimes $K_{\theta}$ are called Macdonald functions, although they  are completely different from Macdonald symmetric functions.}, 
which is a particular case of the generalized inverse Gaussian distribution. 
\end{lemma}
\begin{proof}
By definition, for $y$ such that $\e^{-1}\log(\e^{-1}) -\e^{-1}y\in\lbrace 0, 1, \dots, m\rbrace$,
\begin{align*}
 \PP\big(\log(Y_{\e}) = y\big) &= \PP\Big(X_\e = \e^{-1}\log(\e^{-1}) -\e^{-1}y\Big)\\
  &= \frac{1}{Z_m(\theta)}\theta^{\e^{-1}\log(\e^{-1}) -\e^{-1}y}\frac{\qq{2\e^{-1}\log(\e^{-1}) - \e^{-1}\log(L)}}{\qq{\e^{-1}\log(\e^{-1}) -\e^{-1}y} \qq{\e^{-1}\log(\e^{-1}) -\e^{-1}(\log(L)- y)}}\\
  &=\frac{ \e^{\tilde\theta}}{Z_m(\theta)}\exp\big(y\tilde\theta  -e^y - L/e^y - \A(\e)  +o(1)\big)
\end{align*}
where we have used the estimates \eqref{eq:qqestimate1} and \eqref{eq:qqestimate2} from Lemma \ref{lem:qqestimates} in the last equality (the error goes to zero uniformly for $y$ in a compact set). 
In order to conclude that $\log(Y_{\e})$ converges to the random variable with density proportional to $\exp\big(y\tilde\theta  -e^y - L/e^y\big)$, we need to prove some tail decay estimate when $\vert y\vert \to \infty$, uniformly in $\epsilon$, showing that the sequence $\log(Y_{\e})$ is tight. Using the inequality \eqref{eq:qqestimateuniform} in Lemma \ref{lem:qqestimates}, we have that for any $y\in \R$ and any fixed $L>0$, 
\begin{equation*}
 \PP\big(\log(Y_{\e}) = y\big) \leqslant \frac{ \e^{\tilde\theta}}{Z_m(\theta)}\exp\big(y\tilde\theta  -e^y - L/e^y - \A(\e)  +o(1)\big)
\end{equation*}
where the error $o(1)$ now goes to zero as $\e\to 0$ uniformly for any $y\in \R$. Thus, $\log(Y_\e)$ weakly converges to the distribution on $\R$ with density  
$$\frac{1}{C} e^{y\tilde\theta}e^{-e^y-L/e^y},$$
where $C$ is a normalizing constant so that the density integrates to $1$. By the change of variables $x=e^y$ it implies that $Y_\e$ converges to the continuous random variable on $\R$ with density $$
\frac{1}{C} x^{\tilde\theta -1}e^{-x-L/x}.$$
This distribution is known as a particular case of the generalized inverse Gaussian distribution, and it is well-known that the normalizing constant $C$ equals $2L^{\tilde\theta/2} K_{\theta}(2\sqrt{L})$ \cite{jorgensen2012statistical}.
\end{proof}

Lemma \ref{lem:convqinversegaussian} applied to the result of Lemma \ref{lem:lastpartdynamics}  implies that in the $ \e \to 0 $  limit, for $t\geqslant 2$, 
$ L(t,t)$ has the  generalized inverse Gaussian distribution with parameters  $L(t,t-1)$ and $(-1)^{t+1}\diag + \alpha_t$ as defined in Lemma \ref{lem:convqinversegaussian}. 

\begin{remark}
	One could define, as in Section \ref{sec:Whittakerdynamicsright}, a sort of directed polymer model such that its partition function satisfies the recurrence
	$$ \begin{cases}
	L(t,j) = L(t-1, j-1) + g_{t,j} L(t-1,j) & \text{ for }t>j,\\
	L(t,t) = h_t\left( L(t,t-1),(-1)^{t-1} \diag + \alpha_t \right) & \text{ for }t \geqslant 2,
	\end{cases} $$
where $\lbrace  g_{i,j} \rbrace_{i>j}$ is a family of independent  $\mathrm{Gamma}(\alpha_i+\alpha_j)$ distributed random variables, and $h_t(L, \alpha)$ denotes an inverse Gaussian random variable (independent for each $t$) with parameters $L$ and $\alpha$ as defined in Lemma \ref{lem:convqinversegaussian}. 
\end{remark} 

We may also  deduce the following asymptotics for Rogers-Szeg\H o polynomials defined in \eqref{eq:RogersSzego}: 
\begin{corollary}
Consider the $m$th Rogers-Szeg\H o polynomial $Z_m(\theta)$ under the scalings   $\theta=q^{\tilde\theta}$ and 
$q^m = L (1-q)^2$ for some fixed $L>0$. Then, we have that 
$$ Z_m(\theta)  \frac{(q;q)_{\infty}}{(1-q)^{\tilde \theta -1} } \xrightarrow[q \to 1]{} 2L^{\tilde \theta/2} K_{\tilde \theta}(2\sqrt{L}), $$
where  $K_{\theta}$ is the modified Bessel function of the second kind. 
\end{corollary}
\begin{proof}
The result is obtained by matching the normalizing constants in the proof of Lemma \ref{lem:convqinversegaussian}.
\end{proof}

\subsection{Whittaker measure and geometric RSK} 
\label{sec:gRSK}
Let us examine in more details the limit of the  half-space $q$-Whittaker measure as $q\to 1$,  when $\ve_{\la}$ is specialized into a single variable and the Plancherel component is set to $\gamma=0$. In that case we can define the  $\tau=0$ analogue of $\mathcal{T}^{\,\tau}_{\diag}$, which we will naturally denote $\mathcal{T}^{\,0}_{\diag}$. 

Assuming for simplicity that $n$ is even,
$$
\ve_{\lambda_1, \dots, \la_n}(\diagq) = \diagq^{\sum \lambda_{2i-1}-\lambda_{2i}} \prod_{i=1}^{n} \frac{1}{(q;q)_{\lambda_i-\lambda_{i+1}}} = \e^{n\alpha} e^{-\diag \sum_{i=1}^{n/2} T_{2i-1}-T_{2i}} \prod_{i=1}^{n} \frac{1}{(q;q)_{\lambda_i-\lambda_{i+1}}}.
$$
Thus, we have
$$
\ve_{\lambda}(\diagq) = \e^{n\diag } e^{-\diag \sum_{i=1}^{n/2} T_{2i-1}-T_{2i}} e^{-n\mathcal{A}(\e)} e^{-e^{-T_n}}.
$$
Taking into account a Jacobian of $\e^{-n}$, we arrive at
$$ \mathcal{T}^{\,0}_{\diag}(T) =  e^{-\diag \sum_{i=1}^{n/2} T_{2i-1}-T_{2i}} e^{-e^{-T_n}},$$
and  the half-space Whittaker measure can be extended to $\tau=0$ with 
$$\PWM(\mathrm{d}T) =\mathrm{d}T \,\frac{\psi_{\I\alpha_1, \dots, \I\alpha_n}(T) e^{-\diag \sum_{i=1}^{n/2} T_{2i-1}-T_{2i}} e^{-e^{-T_n}}}{\prod_{i=1}^{n} \Gamma(\alpha_i+\diag) \prod_{1\leqslant i<j\leqslant n} \Gamma(\alpha_i+\alpha_j)}.
$$
\begin{remark}
Using the identity 
$$ \ve_{\la}(a) = \frac{1}{\Phi(a)} \sum_{\mu} \ve_{\mu}(0) P_{\mu/\la}(a),$$
it should be possible to make sense of $\mathcal{T}^{\,\tau}_{\alpha}$ for $\tau=0$ and arbitrary $\alpha_1, \dots, \alpha_k>0$.
\label{rem:tauequalzero} 
\end{remark}

The following result from  \cite{o2014geometric} shows that $\PWM$ is a well-defined probability  measure.
\begin{theorem}[{\cite[Corollary 5.4]{o2014geometric}}]
For $\Real(\diag+\alpha_i)>0$ and $\Real(\alpha_i+\alpha_j)>0$ for all $i,j$, and $n$ even
	$$ \int_{\R^n}\psi_{\I\alpha_1, \dots, \I\alpha_n}(T) e^{-\diag \sum_{i=1}^{n/2} T_{2i-1}-T_{2i}} e^{-e^{-T_n}}\mathrm{d}T =  \prod_{i=1}^{n} \Gamma(\alpha_i+\diag) \prod_{1\leqslant i<j\leqslant n} \Gamma(\alpha_i+\alpha_j).$$
	\label{th:OSZ}
\end{theorem}
 Using \cite{o2014geometric} (in particular Equation (2.2)  and Corollary 5.3 therein), the measure $\PWM$ corresponds (modulo a  shift of variables by $\log(2)$) to the push-forward via the \emph{geometric RSK algorithm} of a symmetric matrix of weights $\lbrace \tilde w_{i,j}\rbrace_{1\leqslant i,j\leqslant n}$ having the following distribution: For $1\leqslant i<j\leqslant n$, $\tilde w_{i,j}$ is distributed as an inverse-Gamma random variable with shape parameter $ \alpha_i+\alpha_j$ and scale\footnote{We use  the convention that if $X$ is an inverse-Gamma random variable with scale parameter $1$, $kX$ is an inverse-Gamma random variable with scale parameter $k$ and the same shape parameter.} parameter $1$,  $\tilde w_{i,i}$ is distributed as an inverse-Gamma random variable with shape parameter $\alpha_i+\diag$ and scale parameter $1/2$, and the $\tilde w_{i,j}$ are independent modulo the symmetry constraint that for all $i,j$ we have $\tilde w_{i,j}=\tilde w_{j,i}$.

Using properties of the geometric RSK algorithm, \cite{o2014geometric} shows that $t_{n,n}$ (one of the components of the image of $\tilde w$ by the geometric RSK algorithm) has the law of the partition function $\tilde Z(n,n)$ of a polymer model with symmetrized weights $\lbrace w_{i,j}\rbrace_{1\leqslant i,j\leqslant n}$ as above. The change of scale parameter on the diagonal  compensates the fact that there are more paths in the symmetric model than in the half-space models. Counting the paths carefully,  we arrive at $2 \tilde{Z}(n,n) = Z(n,n)$ where $Z(n,n)$ is the partition function of the log-gamma polymer in a half-quadrant (Definition \ref{def:LogGammapolymer}). This also shows that 
$$ T_1  \overset{(d)}{=} \log Z(n,n),$$

The relation between our half-space Whittaker measure and the Whittaker measure from \cite{o2014geometric} holds not only for the marginal $T_1$. Indeed,  under the Whittaker measure, 
$$ \EWM[f(T)] = \dfrac{\int_{\R^n} f(e^{T_1}, \dots, e^{T_{n}}) \psi_{\I\alpha}(T) e^{-\diag \sum_{i=1}^{n/2} T_{2i-1}-T_{2i}} e^{-e^{-T_n}}\mathrm{d}T}{\prod_{i=1}^{n} \Gamma(\alpha_i+\diag) \prod_{1\leqslant i<j\leqslant n} \Gamma(\alpha_i+\alpha_j)  } .$$
Making the change of variables $T_i=\tilde T_i +\log(2)$ for all $i$, we obtain that 
$$ \EWM[f(T)]  = \dfrac{\int_{\R^n} f(2e^{T_1}, \dots, 2e^{T_{n}})\psi_{\I\alpha}(T) e^{-\diag \sum_{i=1}^{n/2} T_{2i-1}-T_{2i}} e^{-\frac{1}{2}e^{-T_n}}\mathrm{d}T}{2^{\sum_i \alpha_i} \prod_{i=1}^{n} \Gamma(\alpha_i+\diag) \prod_{1\leqslant i<j\leqslant n} \Gamma(\alpha_i+\alpha_j)}  =  \int f(2t_{1n}, \dots, 2t_{nn}) \tilde \nu_{\alpha, \zeta}(\mathrm{d}w) ,$$
where the measure $\tilde \nu_{\alpha, \zeta}(\mathrm{d}w)$ is defined in \cite[(5.10)]{o2014geometric}. 

We will see in the next section another way to relate the Whittaker measure with the distribution of directed polymers.

\subsection{Half-space directed polymers and Whittaker measure}
\label{sec:relationWhittakerpolymer}

In Section \ref{sec:Whittakerdynamics}, we have studied the limit of $q$-Whittaker dynamics in the scaling in which the $q$-Whittaker measure becomes the Whittaker measure. In particular we found that the limiting dynamics of $q^{-\lambda_1}$ are the same as that of the partition function of the log-gamma polymer. 
Hence, it is natural to expect that the marginal $T_1$, for $T$ distributed according to some Whittaker measure has the same distribution as $\log(Z(t,n))$ where $Z(t,n)$ is the log-gamma partition function (Definition \ref{def:LogGammapolymer}). 
%However, we cannot presently make this identification rigorous because   $\PWM_{(\alpha_1, \dots, \alpha_n),(\diag, \alpha_{n+1}, \dots, \alpha_t, \tau)}$ when $\tau\to 0$
However, we have proved the convergence of the $q$-Whittaker measure to the Whittaker measure only when the Plancherel component of the specialization is positive (Proposition \ref{prop:CVtoWhittmeasurewithplancherel}). Further, we have not defined the Whittaker measure corresponding to $Z(t,n)$ when $\tau=0$ (outside of the special case of Section \ref{sec:gRSK}) and it is not clear how to take a limit of \eqref{eq:defT} as $\tau \to 0$. 

There are two cases where it is possible to relate rigorously the Whittaker measure with the limit of $q$-Whittaker Markov dynamics as $q$ goes to $1$.

\subsubsection{Partition function on the boundary}

When $n=t$, we consider the half-space $q$-Whittaker measure $\PQWM_{(a_1, \dots, a_n), (\diagq)}$ and the Whittaker measure $\PWM_{(\alpha_1, \dots, \alpha_n),(\diag)}$, which takes a simple form as in Section \ref{sec:gRSK}. In that particular case, we know from Theorem \ref{th:OSZ} that the Whittaker measure is a well defined probability measure, and since the density of the $q$-Whittaker measure converges to the density of the Whittaker measure, we have the convergence
$$ \PQWM_{(a_1, \dots, a_n), (\diagq)} \xRightarrow[q\to 1]{} \PWM_{(\alpha_1, \dots, \alpha_n),(\diag)}.$$
Hence, using Proposition \ref{prop:LogGammaqWhittakerlimit}, we recover that for $T$ distributed according to $\PWM_{(\alpha_1, \dots, \alpha_n),(\diag)}$, we have that $ T_1  \overset{(d)}{=} \log Z(n,n),$
where $Z(n,n)$ is the partition function of the log-gamma polymer at point $(n,n)$.  
\subsubsection{Limits of $q$-Whittaker dynamics with additional Plancherel specialization}
\label{sec:WhittdynamicswithPlancherel}

Another approach is to extend the results of Section \ref{sec:Whittakerdynamics} to $q$-Whittaker processes including a Plancherel component. More precisely, we consider a half-space $q$-Whittaker process indexed by a path $\pathh$ as in Figure \ref{pathWhittaker2}.  For all $i\leqslant t$ and any $j$,  edges $ (i-1, j) \leftarrow (i,j) $ are labeled by single variable specialization $a_i$, for any $i, j$, edges $ (i, j-1) \rightarrow (i,j) $ are labeled by single variable specialization $a_j$, the diagonal edge is labeled by specialization $\diagq$ and now, edges $(t, j)\leftarrow (t+1, j)$ are labeled by the Plancherel specialization $\tau$.  We consider only the right-edge dynamics (the last parts of partitions). 

\begin{figure}
\begin{tikzpicture}[scale=0.8]
\draw[->, thick, >=stealth', gray] (0,0) -- (14.5, 0);
\draw[->, thick, >=stealth', gray] (0,0) -- (0, 6.2);
\draw[thick, gray]  (0,0) -- (6,6);
\draw (-.5, -.2) node{{\footnotesize $(1,1)$}};
\draw (9.6, 6.3) node{{\footnotesize $(t,n)$}};
\draw[->, >=stealth'] (4,-1) node[anchor=north]{{\footnotesize $w_{9,3} \sim \mathrm{Gamma}^{-1}(\alpha_9 + \alpha_3) $}} to[bend right] (8,2);
\draw[->, >=stealth'] (1,4) node[anchor=south]{{\footnotesize $w_{4,4} \sim \mathrm{Gamma}^{-1}(\diag + \alpha_4) $}} to[bend right] (3,3);

\draw[ultra thick] (0,0) -- (2,0)  -- (2,2) -- (5,2) -- (5,3) -- (9,3);
\draw[thick, gray, dashed] (9,-1) -- (9,7);
\draw[gray, thick] (9,0) -- (14,0);
\draw[gray, thick] (9,1) -- (14,1);
\draw[gray, thick] (9,2) -- (14,2);
\draw[gray, thick] (9,3) -- (14,3);
\draw[gray, thick] (9,4) -- (14,4);
\draw[gray, thick] (9,5) -- (14,5);
\draw[gray, thick] (9,6) -- (14,6);

\draw[ultra thick] (9,3) -- (10.7, 3) -- (10.7, 4) -- (13, 4) -- (13, 5) -- (13.3, 5) --(13.3, 6) -- (14,6);

\draw[thick, gray, dashed] (14,-1) -- (14,7);
\clip (0,0) -- (9,0) -- (9,6) -- (6, 6) -- (0,0);
\draw[dotted, gray] (0,0) grid (9,6); 
\end{tikzpicture}
\caption{An admissible path in the hybrid half space log-gamma /semi-discrete directed polymer.}
\label{halfspaceloggamma+plancherel}
\end{figure} 
\begin{definition}
	Let $(w_{i,j})_{1\leqslant i\leqslant j}$ be a family of inverse Gamma random variables as in Definition \ref{def:LogGammapolymer}. Let $(B^i(s))_{i\geqslant 1,s>0 }$ be a family of independent standard Brownian motions with drift  $b_i$ respectively. 
	
	Let $t\geqslant n\geqslant 1$. We define the \emph{hybrid half-space log-gamma/semi-discrete directed polymer} as a measure on the concatenation of two paths $\pi, \phi$ as follows: The path $\pi $ progress from $(1,1)$ by unit up-right steps in the half quadrant until a certain point $(t, n_0)$. The path $\phi$, depending on a parameter $\tau>0$, is a semi-discrete path  encoded by a set $t=\tau_0 <\tau_1< \dots <\tau_{n-n_0} <\tau_{n-n_0+1}=t+\tau$. The path progresses horizontally  from $(t, n_0)$ to $(\tau_1, n_0)$, jumps to $(\tau_1, n_0+1)$ and then progresses to $(\tau_2, n_0+1)$, etc. until it reaches $(\tau+t, n)$. 
	
The probability density of such a hybrid path $\pi, \phi$ is given by 
$$ 
\frac{1	}{Z(t,n,\tau)} \prod_{(i,j)\in \pi} w_{i,j}\ \ \exp\left(\int_0^{\tau} \mathrm{d} B^{\phi(t+s)}(s)\right), $$
where the partition function $Z(t,n,\tau)$ is defined by 
\begin{equation}
 Z(t,n,\tau)  = \sum_{n_0=1}^n \sum_{\pi:(1,1)\to (t,n_0)}  \prod_{(i,j)\in \pi} w_{i,j} \underset{{t<\tau_1 < \dots <\tau_{n-n_0-1} <t+\tau}}{\int} \exp\left(\int_0^{\tau} \mathrm{d} B^{\phi(t+s)}(s)\right) \mathrm{d}\tau_1 \dots \mathrm{d}\tau_{n-n_0-1}.
\label{eq:hybridpartition}
\end{equation}
\label{def:hybridpolymer}
\end{definition}
Note that $Z(t,n,\tau)$ can be constructed in a recursive way by constructing first $Z(t,m,0)$ for all $1\leqslant m\leqslant n$ as in Section \ref{sec:Whittakerdynamicsright}. Then \cite{o2012directed} showed that the vector of free energies  $\left(\log(Z(t,m,\tau))\right)_{1\leqslant m\leqslant n}$ is the solution  to the following system of stochastic differential equation with time parameter $\tau$  and initial data $\left(\log(Z(t,m,0))\right)_{1\leqslant m\leqslant n}$: 
\begin{equation}
\begin{cases}
\mathrm{d}T_1^{(1)} = \mathrm{d}B^1_{\tau}  \\ 
\mathrm{d}T_1^{(m)} = \mathrm{d}B^m_{\tau} + e^{T_1^{(m-1)}-T_1^{(m)}}\mathrm{d}\tau  \ \ \ \text{for all }2\leqslant m\leqslant n.
\end{cases}
\label{eq:SDE}
\end{equation}

The degeneration of the RSK dynamics $\U_{\mathrm{row}}$ from Section \ref{sec:Macdyn} where the horizontal  specialization becomes a Plancherel specialization were studied in \cite{borodin2013nearest}. Some of the arguments presented in \cite[Section 8.4]{borodin2013nearest} are only heuristics, and the convergence of $q$-Whittaker dynamics to SDE is not proved for the whole triangular array. However, the convergence is shown in \cite[Section 8.4.4]{borodin2013nearest} for the first coordinate marginal $T^{(k)}_1$.  More precisely, for any fixed $n$, under the  Whittaker dynamics which correspond to the $q\to 1$ limit of $\U_{\mathrm{row}}$, $\left( T_{1}^{(k)}\right)_{1\leqslant k\leqslant n}$ satisfies the SDE \eqref{eq:SDE} with the Brownian motion drifts $b_i$ chosen as $-\alpha_i$.  Thus, we arrive at the following. 
\begin{proposition}
	Let $T$ be distributed according to $\PWM_{(\alpha_1, \dots, \alpha_n),(\diag, \alpha_{n+1}, \dots, \alpha_t, \tau)}$ with $\tau>0$. Let $Z(t,n, \tau)$ be the partition function of the hybrid half-space log-gamma/semi-discrete directed polymer as in Definition \ref{def:hybridpolymer}, where the drifts in the hybrid log-gamma/semi-discrete model are chosen as $b_i=-\alpha_i$.  Then, we  have the distributional equality
	\begin{equation}
	T_1^{(n)}  \overset{(d)}{=} \log Z(t,n, \tau).
	\label{eq:equalityhybridWhittaker}
	\end{equation}
	\label{prop:equalityZwhittaker}
\end{proposition}

\begin{remark}
	We expect that \eqref{eq:equalityhybridWhittaker} should hold as well when $\tau = 0$ in some sense (see Remark \ref{rem:tauequalzero} about the case $\tau=0$). While it is clear from the definition of the polymer partition function that $Z(t,n, \tau)$ converges to $Z(t,n, 0)$,  additional arguments are necessary to control the tails of the half-space Whittaker measure as $\tau\to 0$. 
	\label{rem:connectWhittakerLogGamma}
\end{remark}

\begin{corollary} Let $t\geqslant n\geqslant 1$, $\alpha_1, \dots,  \alpha_t>0$ and $\diag>0$.  For all $k\in\Z_{>0}$ such that $k< \min\lbrace 2\alpha_i, \diag+\alpha_i\rbrace$,
		\begin{multline}
		\EWM[Z(t,n)^k] =\oint\frac{\mathrm{d}w_1}{2\I\pi}\cdots \oint\frac{\mathrm{d}w_k}{2\I\pi} \prod_{1\leqslant a<b\leqslant k} \frac{w_a-w_b}{w_a-w_b-1}\, \frac{1+w_a+w_b}{2+ w_a+w_b}\\ \times 
		\prod_{m=1}^{k}  \frac{1+2w_m}{1+w_m-\diag}\prod_{i=1}^t \left( \frac{1}{\alpha_i-w_m-1} \right) \prod_{j=1}^{n} \left(\frac{1}{w_m + \alpha_j}\right), 
		\label{eq:momentsZ}
		\end{multline}
		where the positively oriented contours are such that for all $1\leqslant c\leqslant k$, the contour for $w_c$ encloses $\lbrace - \alpha_j\rbrace_{1\leqslant j\leqslant n}$  and $\lbrace w_{c+1}+1, \dots, w_k+1\rbrace$, and excludes the poles of the integrand at $ \diag- 1$ and $ \alpha_j-1$ (for $1\leqslant j\leqslant t$). 
		\label{cor:momentsZ}
\end{corollary}
\begin{proof}
Combining the moment formula for $e^{T_1}$ from Proposition \ref{prop:momentsWhittaker} with Proposition  \ref{prop:equalityZwhittaker} yields a formula for $\EWM[Z(t,n, \tau)^k]$. Then, we may take the limit when $\tau$ goes to zero in the formulas, but we need to prove that the moments converge as well as $\tau\to 0$. The hybrid partition function $Z(t,n,\tau)$ converges almost surely as $\tau$ goes to zero to  the partition function of the half-space log-gamma polymer $Z(t,n)=Z(t,n,0)$ (if $Z(t,n,\tau)$ and $Z(t,n)$ are defined with the same variables $w_{i,j}$).  Let $\tau^* >0$ be an arbitrary positive real number. Using the system of SDE \eqref{eq:SDE}, we obtain that for $0\leqslant \tau <\tau^*$,
 \begin{equation}
 \log(Z(t,n, \tau^*))  - \log(Z(t,n, \tau)) = B^{(n)}_{\tau^*} -  B^{(n)}_{\tau} + \int_{\tau}^{\tau^*} \frac{Z(t,n-1,s)}{Z(t,n,s)} \mathrm{d}s.
 \label{eq:csqSDE}
 \end{equation}
Hence, since  $\frac{Z(t,n-1,s)}{Z(t,n,s)}\geqslant 0$, we have that for $\tau^*>0$ fixed and any $0\leqslant \tau <\tau^*$, 
\begin{equation}
0\leqslant Z(t,n,\tau)  \leqslant Z(t,n,\tau^*)\exp\left( \sup_{0\leqslant s\leqslant \tau^*}\lbrace B_s^{(n)}  \rbrace   - B^{(n)}_{\tau^*} \right).
\label{eq:inequality2}
\end{equation}
 For $k< \min\lbrace 2\alpha_i, \diag+\alpha_i\rbrace$,  we know from Proposition \ref{prop:momentsWhittaker} that $\EE[Z(t,n,\tau^*)]^k<\infty$. The random variable $ \exp\left( \sup_{0\leqslant s\leqslant \tau^*}\lbrace B_s^{(n)}  \rbrace  - B^{(n)}_{\tau^*}  \right)$ has finite moments of all orders, since $\sup_{0\leqslant s\leqslant \tau^*}\lbrace B_s^{(n)}  \rbrace  - B^{(n)}_{\tau^*} $ is distributed as the absolute value of a Gaussian (\cite[Chapter III,  Ex. 3.14]{revuz2013continuous}). Thus, the R.H.S. of \eqref{eq:inequality2} does not depend on $\tau$ and has finite moment of order $k$ as long as $k< \min\lbrace 2\alpha_i, \diag+\alpha_i\rbrace$. Dominated convergence implies that 
$$  \EE[Z(t,n,\tau)^k] \xrightarrow[\tau\to 0]{} \EE[Z(t,n,0)^k],$$
which concludes the proof. 
\end{proof}

\begin{remark}
The moment formula \eqref{eq:momentsZ} is reminiscent of nested contour integral solutions of the delta Bose gas in a half space, see for instance  \cite[Lemma 4.2]{borodin2016directed} and Proposition \ref{cor:momentsKPZ} in the present paper. We believe that the formula above could as well be obtained by solving the system of difference equations satisfied by the function 
$$ (n_1, \dots, n_k) \mapsto  \EWM[Z(t, n_1) \dots Z(t,n_k)]. $$
\end{remark}

\subsection{Alternative derivation of Laplace transform integral formulas}
 \label{sec:alternative}
We can also  use our Whittaker eigenrelations for $\BBn$ and $\BBone$ (defined  in \eqref{eq:defBBn} and  \eqref{eq:defBBone})  to directly compute Laplace transforms. The principle is to act  with $\BBn$ and $\BBone$ on both sides of the Whittaker analogue of the generalized Littlewood identity from  Proposition \ref{prop:integratestoone}, We will need to justify that one can exchange the action of the operators with the integrations so that we may use our eigenrelations \eqref{eq:eignwhitN} and \eqref{eq:eignwhitone}. For the latter, it is essential to keep a positive Plancherel specialization (we could not adapt the same approach using the identity from Theorem \ref{th:OSZ}, which would correspond to the $\tau=0$ case).

\begin{theorem} Fix $n\leqslant t$, $\tau >0$ and $\alpha_1, \dots, \alpha_t>0$ and $\diag\in \R$ such that $\diag+\alpha_i>0$ for $1\leqslant i\leqslant n$.  Under the Whittaker measure $\PWM_{(\alpha_1, \dots, \alpha_n), (\diag, \alpha_{n+1}, \dots, \alpha_t, \tau)}$, we have 
\begin{multline}
\EWM[e^{-u e^{T_1}}] = \int_{(\mathcal{D}_{-r})^n} \frac{\mathrm{d}z}{(2\I\pi)^n}  \mskyl_n(\I z)(2\pi)^n  \   \prod_{i=1}^{n}\left( u^{ \alpha_i + z_i} e^{\tau ( z_i^2-\alpha_i^2)/2}  \frac{\Gamma(\diag-z_i)}{\Gamma(\diag + \alpha_i)} \prod_{j=n+1}^t \frac{\Gamma(\alpha_j-z_i)}{\Gamma(\alpha_j+ \alpha_i)}\right) \\ \times
\prod_{i,j=1}^{n} \Gamma(-z_i - \alpha_j)
   \prod_{1\leqslant i<j\leqslant n} \frac{\Gamma(-z_i-z_j)}{\Gamma(\alpha_i+\alpha_j)},
\label{eq:nfoldwhit}
\end{multline}
 where $\mskyl_n$ is defined in \eqref{eq:defmskyl}, and  $r>0$ is such that  $r+\diag >0$ and   $r>\alpha_i$ for all $1\leqslant i \leqslant n$.
 
We also have, 
\begin{multline}
\EWM[e^{-u e^{-T_n}}]= 
\int_{(\mathcal{D}_{r})^n} \frac{\mathrm{d}z}{(2\I\pi)^n}  \mskyl_n(\I z)(2\pi)^n   \    \prod_{i=1}^{n} \left( u^{-\ a_i -  z_i} e^{\tau ( z_i^2-\alpha_i^2)/2}\frac{\Gamma(z_i+\diag)}{\Gamma(\alpha_i + \diag)} \prod_{j=n+1}^t\frac{\Gamma(z_i+\alpha_j)}{\Gamma(\alpha_i + \alpha_j)} \right) \\ \times 
 \prod_{i,j=1}^{n} \Gamma(-z_i + \alpha_j)
\prod_{1\leqslant i<j\leqslant n} \frac{\Gamma(z_i+z_j)}{\Gamma(\alpha_i+\alpha_j)}.
\label{eq:nfoldwhitn}
\end{multline}
where $r>0$ is such that $r+\diag >0$ and $r<\alpha_i$ for all $1\leqslant i\leqslant n$.
\label{th:proofOSZ}
\end{theorem}

\begin{proof}
Recall the statement of  Proposition \ref{prop:integratestoone}: for $w\in \mathbb{H}^n$, 
\begin{multline*}  
\int_{\R^n} \mathrm{d}x  \psi_{w}(x) \mathcal{T}^{\,\tau}_{\diag, \alpha_{n+1}, \dots, \alpha_t}(x) \\ =   e^{- \tau\sum_{i=1}^n w_i^2/2}\ \prod_{i=1}^n{\Gamma(\diag-\I w_i)}\prod_{j=n+1}^{t}{\Gamma(\alpha_j-\I w_i)}\prod_{1\leqslant i<j\leqslant n}{\Gamma(-\I (w_i+w_j))}.
\end{multline*}
We may substitute $w_i+\I r$ in place of $w_i$ above,  for $r>0$. The identity remains true as long as  $w_i+\I r\in \mathbb{H}$ for all $i$ and $\Re[\diag-\I w_i+r]>0$, but $w$ does not necessarily belong to $\mathbb{H}^n$ anymore. Using the shift property $\psi_{w+\I r}(x) = \psi_{w}(x) e^{ -r \sum x_i}$ we get
\begin{multline}  
\int_{\R^n} \mathrm{d}x  \psi_{w}(x) e^{- r \sum x_i} \mathcal{T}^{\,\tau}_{\diag, \alpha_{n+1}, \dots, \alpha_t}(x) \\ =   e^{- \tau\sum_{i=1}^n (w_i+\I r)^2/2}\ \prod_{i=1}^n \left({\Gamma(\diag-\I w_i + r)}\prod_{j=n+1}^{t}{\Gamma(\alpha_j-\I w_i+r)} \right)\prod_{1\leqslant i<j\leqslant n}{\Gamma(-\I (w_i+w_j)+2 r)}.
\label{eq:regularizedCauchyLitllewoodWhit}
\end{multline}
When $w\in (-\mathbb{H})^n$, we may act on both sides of  \eqref{eq:regularizedCauchyLitllewoodWhit}  with $\BBone_n^u$. 
Acting on the R.H.S. yields
\begin{multline*}
\int_{\R^n} \mathrm{d}\xi  \mskyl_n(\xi) u^{\I\sum_{i=1}^n- w_i + \xi_i} \prod_{i,j} \Gamma(-\I\xi_i + \I w_j) \\  \times 
  e^{- \tau\sum_{i=1}^n (\xi_i + \I r)^2/2}\ \prod_{i=1}^n\left( {\Gamma(\diag-\I\xi_i +r)}\prod_{j=n+1}^{t}{\Gamma(\alpha_j-\I\xi_i+r)}\right) \prod_{1\leqslant i<j\leqslant n}{\Gamma(-\I(\xi_i+\xi_j)+2r)}.
\end{multline*}
Acting on the L.H.S. with $\BBone_n^u$ yields 
$$ \int_{\R^n} \mathrm{d}x   e^{-u e^{x_1}} \psi_{w}(x) e^{ - r \sum x_i} \mathcal{T}^{\,\tau}_{\diag, \alpha_{n+1}, \dots, \alpha_t}(x).$$
We have exchanged the action of $\BBone_n^u$ and the integration above using Fubini theorem. It can be applied here because both $e^{-u e^{x_1}} \psi_{w}(x)$ and $e^{- r \sum x_i}\mathcal{T}^{\,\tau}_{\diag, \alpha_{n+1}, \dots, \alpha_t}(x)$ are bounded in $\mathbb{L}^2(\R^n)$, which has been proved respectively in \cite[Corollary 3.8]{o2014geometric} and Proposition \ref{prop:boundT} of the present paper.
We may substitute $w_i-\I r$ in place of $w_i$, and after a change of variables $z_i=\I\xi_i-r$, we obtain that 
\begin{multline}
\int_{\R^n} \mathrm{d}x   e^{-u e^{x_1}} \psi_{w}(x) \mathcal{T}^{\,\tau}_{\diag, \alpha_{n+1}, \dots, \alpha_t}(x) =
\int_{(\mathcal{D}_{-r})^n} -\mathrm{d}\I z \mskyl_n(\I z) u^{\sum_{i=1}^{n} -\I w_i + \sum_{i=1}^{n} z_i} \\  \times 
e^{\tau \sum_{i=1}^n z_i^2}\prod_{i,j=1}^{n} \Gamma(-z_i +\I w_j) \prod_{i=1}^{n}\left( \Gamma(\diag-z_i)\prod_{j=n+1}^t \Gamma(\alpha_j-z_i) \right)\prod_{1\leqslant i<j\leqslant n} \Gamma(-z_i-z_j).  
\label{eq:tocontinue}
\end{multline}
The expectation $\EWM[e^{-u e^{T_1}}]$ under the Whittaker measure $\PWM_{(\alpha_1, \dots, \alpha_n), (\diag, \alpha_{n+1}, \dots, \alpha_t, \tau)}$, corresponds to the left-hand side of \eqref{eq:tocontinue} for $w=(\I\alpha_1, \dots, \I\alpha_n)$, divided by the normalization constant, so that we have established \eqref{eq:nfoldwhit}. For this choice of $w$, we need that  $r>\alpha_i$ so that $w_i+\I r\in \mathbb{H}$ for all $i$, hence the hypothesis in the statement of Theorem \ref{th:OSZ}.

We turn now to the proof of \eqref{eq:nfoldwhitn}.  Again we start from \eqref{eq:regularizedCauchyLitllewoodWhit} where now $w\in \mathbb{H}^n$.  
Acting on both sides with $\BBn_n^u$ (the interchange of integration with respect to $x$ and action of $\BBn_n^u$ can be justified as in the previous case) yields for all $w\in \mathbb{H}^n$
\begin{multline*}
 \int_{\R^n} \mathrm{d}x   e^{-u e^{-x_n}} \psi_{w}(x) e^{- r \sum x_i} \mathcal{T}^{\,\tau}_{\diag, \alpha_{n+1}, \dots, \alpha_t}(x) = \int_{\R^n} \mathrm{d}\xi  \mskyl_n(\xi) u^{\I\sum_{i=1}^n w_i + \xi_i} \prod_{i,j} \Gamma(-\I\xi_i - \I w_j) \\ 
 \times e^{- \tau\sum_{i=1}^n (-\xi_i + \I r)^2/2}\ \prod_{i=1}^n\left( {\Gamma(\diag+ \I\xi_i +r)}\prod_{j=n+1}^{t}{\Gamma(\alpha_j+\I\xi_i+r)}\right) \prod_{1\leqslant i<j\leqslant n}{\Gamma(\I(\xi_i+\xi_j)+2r)}.
\end{multline*}
Now we may substitute back $w_i$ by $w_i-\I r$ (with $r$ not too large so that $w_i-\I r\in \mathbb{H}$), evaluate for $w=(\I\alpha_1, \dots, \I\alpha_n)$, and after a change of variables $z_i=\I\xi_i+r$, we obtain 
\begin{multline*}
\int_{\R^n} \mathrm{d}x   e^{-u e^{-x_n}} \psi_{\I\alpha_1, \dots, \I\alpha_n}(x) \mathcal{T}^{\,\tau}_{\diag, \alpha_{n+1}, \dots, \alpha_t}(x) = 
\int_{(\mathcal{D}_{r})^n} \mathrm{d} z \frac{\mskyl_n(\I z)}{\I} u^{\sum_{i=1}^{n} \alpha_i + \sum_{i=1}^{n} z_i} \\  \times 
e^{\tau \sum_{i=1}^n z_i^2}\prod_{i,j=1}^{n} \Gamma(-z_i + \alpha_j) \prod_{i=1}^{n}\left( \Gamma(\diag+z_i)\prod_{j=n+1}^t \Gamma(\alpha_j+z_i) \right)\prod_{1\leqslant i<j\leqslant n} \Gamma(z_i+z_j).  
\end{multline*}
We obtain \eqref{eq:nfoldwhitn} after dividing by the normalization constant. 
\end{proof}

\begin{remark}
We expect that one could deduce Theorem \ref{th:proofOSZ} from Corollaries \ref{cor:LaplaceWhittaker} and \ref{cor:LaplacecorollaryTn}. The correspondence between the two types of formulas should follow the same lines as  in \cite[Section 3]{borodin2013log}, where the equivalence between an $n$ fold contour integral such as \eqref{eq:nfoldwhit} and series expansions as in Corollary \ref{cor:LaplaceWhittaker} is explained in the context of the full-space log-gamma polymer.  Alternatively, we may prove $n$-fold contour integral formulas for the $(q, t)$-Laplace transform of general Macdonald measures by keeping the contour as $\mathcal{D}_{-\e}$ in the proof of Proposition \ref{prop:actionMoumiAonZ}. Such formulas would   degenerate to Theorem \ref{th:proofOSZ} in the Whittaker limit. 
\end{remark}

Note that in the $n$-fold Laplace transform formulas \eqref{eq:nfoldwhit} and \eqref{eq:nfoldwhitn}, we may let $\tau$ go to zero without encountering any singularity. As a consequence, we can prove a Laplace transform formula for the partition function of the half-space log-gamma polymer (without Plancherel specialization). 
\begin{corollary}
 Fix $n\leqslant t$,  $\alpha_1, \dots, \alpha_t>0$ and $\diag\in \R$ such that $\diag+\alpha_i>0$ for $1\leqslant i\leqslant n$. Let  $r>0$ such that  for all $1\leqslant i \leqslant n$,  $r>\alpha_i>0$. The partition function $Z(t,n)$ of the log-gamma polymer in a half-quadrant (Definition \ref{def:LogGammapolymer}) is characterized by the following. For any $u>0$, 
 \begin{multline}
\EWM[e^{-u Z(t,n)}] = \frac{1}{n!} \int_{\mathcal{D}_r}\frac{\mathrm{d}z_1}{2\I\pi} \dots \int_{\mathcal{D}_r}\frac{\mathrm{d}z_n}{2\I\pi} \prod_{i\neq j} \frac{1}{\Gamma(z_i-z_j)} \prod_{1\leqslant i<j\leqslant n} \frac{\Gamma(z_i+z_j)}{\Gamma(\alpha_i+\alpha_j)}
\prod_{i,j=1}^{n} \Gamma(z_i - \alpha_j)\\ \times 
\prod_{i=1}^{n}\left(u^{ \alpha_i -z_i} \frac{\Gamma(\diag+z_i)}{\Gamma(\diag + \alpha_i)} \prod_{j=n+1}^t \frac{\Gamma(\alpha_j+z_i)}{\Gamma(\alpha_j+ \alpha_i)}\right).
 \label{eq:nfoldwhitZ}
 \end{multline}
 \label{cor:OSZproof}
\end{corollary}
\begin{proof}
 We can compute $\EWM[e^{-u Z(t,n, \tau)}]$, where $Z(t,n, \tau)$ is the hybrid partition function from Definition \ref{def:hybridpolymer}, using the Laplace transform formula from Theorem \ref{th:proofOSZ} and the identity in law from Proposition \ref{prop:equalityZwhittaker}. Then, using the weak convergence $Z(t,n, \tau) \Rightarrow Z(t,n)$ as $\tau\to 0$, we obtain  
 \begin{multline}
 \EWM[e^{-u Z(t,n)}] = \int_{(\mathcal{D}_{-r})^n} \frac{\mathrm{d}z}{(2\I\pi)^n}  \mskyl_n(\I z)(2\pi)^n   u^{\sum_{i=1}^{n} \alpha_i + \sum_{i=1}^{n} z_i}
 \prod_{i,j=1}^{n} \Gamma(-z_i - \alpha_j)\\ \times 
 \prod_{i=1}^{n}\left(\frac{\Gamma(\diag-z_i)}{\Gamma(\diag + \alpha_i)} \prod_{j=n+1}^t \frac{\Gamma(\alpha_j-z_i)}{\Gamma(\alpha_j+ \alpha_i)}\right)  \prod_{1\leqslant i<j\leqslant n} \frac{\Gamma(-z_i-z_j)}{\Gamma(\alpha_i+\alpha_j)},
 \end{multline}
 which can be written equivalently as in the statement of the Proposition.  
\end{proof}

\subsection{Plancherel theory and comparison with O'Connell-Sepp\"al\"ainen-Zygouras's results}
\label{sec:relationwithOSZ} 

In this section, we shall compare our results, in particular Corollary \ref{cor:OSZproof},  with formal computations in \cite{o2014geometric}. 
The integral transform 
$$ \widehat f(\xi) = \int_{\R^n} f(x) \psi_{\xi}(x)\mathrm{d}x$$
defines an isometry from $\mathbb{L}^2(\R^n, \mathrm{d}x)$ to 	$\mathbb{L}^2(\R^n, \mskyl_n(\xi)\mathrm{d}\xi)$ restricted to symmetric functions. The associated Plancherel theorem \cite[Theorem 51]{semenov1993quantization} allows to compute the Laplace transform of observables of the Whittaker measure \cite{o2014geometric}. Indeed, we are interested in computing 
\begin{equation}
 \EWM[e^{-u Z(n,n)}] = \frac{1}{C}\int_{\R^n} e^{-ue^{T_1}} \psi_{\I\alpha_1, \dots, \I\alpha_n}(T) e^{-\diag \sum_{i=1}^{n/2} T_{2i-1}-T_{2i}} e^{-e^{-T_n}}\mathrm{d}T,
\label{eq:defLaplace}
\end{equation}
where $C$ is the normalizing constant and we assume that  $n$ is even for simplicity. Suppose that the integrand above can be written as $\overline{f(T)}g(T)$ where $f$ and $g$ are two functions in $\mathbb{L}^2(\R^n, \mathrm{d}x)$ such that  we are able to compute the integral transforms $\widehat f, \widehat g$. Then the Plancherel theorem yields 
$$\EWM[e^{-u Z(n,n)}] = \frac{1}{C} \int_{\R^n} \overline{\widehat f(\xi)} \widehat g(\xi) \mskyl_n(\xi)\mathrm{d}\xi.$$
However, it is not clear how to find such a decomposition of the integrand in \eqref{eq:defLaplace}. 
In \cite[Section 5]{o2014geometric}, in a remark titled ``A formal computation'', the authors propose to apply formally this scheme to functions that are not square integrable, in order to derive explicit integral formulas for $\EWM[e^{-u Z(n,n)}]$.  There may be several ways to decompose the integrand in \eqref{eq:defLaplace} into a product of functions whose integral transforms can be computed. We examine below two possibilities, the first one yields our formula from Corollary \ref{cor:OSZproof}, while the second one yields \cite[(5.15)]{o2014geometric}. Although the argument is formal, both approaches lead to the correct answer (though the form of the answer is different and we match them in Corollary \ref{cor:OSZproof2}).  

\subsubsection{Case 1: equivalent to using the operator $\BBone_n^u$}

Let us choose $f$ and $g$ as 
$$ f(T) = e^{-ue^{T_1}} \psi_{\I\alpha_1, \dots, \I\alpha_n}(T),  \ \ g(T) = e^{-\diag \sum_{i=1}^{n/2} T_{2i-1}-T_{2i}} e^{-e^{-T_n}}.$$
Then, using  \eqref{eq:Cauchywhittakerone}, 
$$ \widehat f(\xi) = u^{\sum_{j=1}^n\alpha_j-\I \xi_j}\prod_{1\leqslant i,j\leqslant n} \Gamma(\I \xi_j -\alpha_i),\ \ \ \  \Real[\I\xi_i]>\alpha_j \text{ for all } 1\leqslant i,j\leqslant n,$$
and using Theorem \ref{th:OSZ} (that is \cite[Corollary 5.4]{o2014geometric}),  
$$ \widehat g(\xi)  = \prod_{j=1}^n \Gamma(\diag - \I \xi_j) \prod_{1\leqslant i<j\leqslant n} \Gamma(-\I \xi_i-\I \xi_j) ,\ \ \ \  \Real[-\I\xi_i]>\min\lbrace 0, -\diag\rbrace  \text{ for all } 1\leqslant i \leqslant n.$$
Thus a formal application of the Plancherel theorem suggests that, 
\begin{equation}
\EWM[e^{-u Z(n,n)}] =  \int u^{\sum_{j=1}^n\alpha_j-\I \xi_j}\prod_{1\leqslant i,j\leqslant n} \Gamma(\I \xi_j -\alpha_i) \prod_{j=1}^n \frac{ \Gamma(\diag + \I \xi_j)}{\Gamma{\diag+\alpha_j}} \prod_{1\leqslant i<j\leqslant n} \frac{\Gamma(\I \xi_i+\I \xi_j)}{\Gamma{(\alpha_i+\alpha_j)}} \mskyl_n(\xi)\mathrm{d}\xi,
\label{eq:planchereltheorycorrect}
\end{equation}
where the integration contour is $\R-\I r$ with $r>\alpha_i$ for all   $1\leqslant i \leqslant n$. This formula is, in fact, correct as we have proved it in Corollary \ref{cor:OSZproof} (in the special case $t=n$), itself coming from Theorem \ref{th:proofOSZ}. The choice of contours can be justified more precisely by shifting variables as in the proof of Theorem \ref{th:proofOSZ}, however the derivation above is not rigorous due to the fact that the functions $f$ and  $g$ do not belong to $\mathbb{L}^2(\R^n, \mathrm{d}x)$.  This is why the presence of the Plancherel component in the proof of Theorem \ref{th:proofOSZ} is crucial. 
 
\subsubsection{Case 2: O'Connell-Sepp\"al\"ainen-Zygouras's approach}

Let us now choose $f$ and $g$ as
$$ f(T) = e^{-e^{-T_n}} \psi_{\I\alpha_1, \dots, \I\alpha_n}(T),  \ \ g(T) = e^{-\diag \sum_{i=1}^{n/2} T_{2i-1}-T_{2i}} e^{-ue^{T_1}}.$$
Then, using  \eqref{eq:Cauchywhittakern}, 
$$ \widehat f(\xi) = \prod_{1\leqslant i,j\leqslant n} \Gamma(\I \xi_j -\alpha_i), \ \ \ \  \Real[\I\xi_j]>\alpha_i  \text{ for all } 1\leqslant i,j \leqslant n.$$
Using $\psi_{z}(T_1, \dots, T_n) = \psi_{-z}(-T_n, \dots, -T_1)$ we have that for $n$ even, 
$$ \widehat g(\xi)  = \int_{\R^n} e^{-\diag \sum_{i=1}^{n/2} T_{2i-1}-T_{2i}} e^{-u e^{-T_n}}\psi_{-\xi}(T) \mathrm{d}T,$$
so that using  \cite[Corollary 5.4]{o2014geometric}),  
$$ \widehat g(\xi)  = \prod_{j=1}^n \Gamma(\diag + \I \xi_j) \prod_{1\leqslant i<j\leqslant n} \Gamma(\I \xi_i+\I \xi_j) u^{-\sum_{j=1}^n \I \xi_j}, \ \ \ \Real[\I \xi_j]>\min\lbrace 0, -\diag\rbrace \text{ for all } 1\leqslant j \leqslant n.$$
Thus the Plancherel theorem suggests that 
\begin{equation}
\EWM[e^{-u Z(n,n)}] =  \int_{\R^n} u^{-\sum_{j=1}^n\I \xi_j}\prod_{1\leqslant i,j\leqslant n}\Gamma(\I \xi_j +\alpha_i) \prod_{j=1}^n \frac{\Gamma(\diag + \I \xi_j)}{\Gamma(\diag+\alpha_j)} \prod_{1\leqslant i<j\leqslant n} \frac{\Gamma(\I \xi_i+\I \xi_j)}{\Gamma(\alpha_i+\alpha_j)} \mskyl_n(\xi)\mathrm{d}\xi,
\label{eq:OSZother}
\end{equation}
where the contour is  $\R-\I r$ with $r>\min\lbrace 0, -\diag\rbrace $. Again, the choice of contour could be justified more precisely by shifting variables, but the derivation above is not rigorous because the function  $g$ does not belong to $\mathbb{L}^2(\R^n, \mathrm{d}x)$. Since our partition function $Z(n,n)$ has the same distribution as the random variable $2 t_{nn}$ under the measure $\tilde \nu_{\alpha, \zeta}(\mathrm{d}w)$ as defined in \cite[(5.10)]{o2014geometric}, the expression \eqref{eq:OSZother} is exactly equivalent to \cite[(5.15)]{o2014geometric} (after making a change of variables $z_i=\I\xi_i$ and identify $u=r/2$). A similar argument can be adapted in the case where $n$ is odd.

Note that  \eqref{eq:planchereltheorycorrect} and \eqref{eq:OSZother} seem to be  quite different. There is an additional factor $u^{\sum_{j=1}^n\alpha_j}$ in \eqref{eq:planchereltheorycorrect}, and the factor $\prod_{1\leqslant i,j\leqslant n} \Gamma(\I \xi_j -\alpha_i)$ in \eqref{eq:planchereltheorycorrect} becomes  $\prod_{1\leqslant i,j\leqslant n} \Gamma(\I \xi_j +\alpha_i)$ in \eqref{eq:OSZother}.
The next corollary  shows that \eqref{eq:OSZother} (or equivalently  \eqref{eq:OSZ515} below after a change of variables) can be proved as well from our Corollary \ref{cor:OSZproof}. 
\begin{corollary}
 Fix  $\alpha_1, \dots, \alpha_n>0$ and $\diag\in \R$ such that $\diag+\alpha_i>0$ for $1\leqslant i\leqslant n$. Then, when $n$ is even,  for any $u>0$, 
\begin{multline}
\EWM[e^{-u Z(n,n)}] \\ 
 = \frac{1}{n!} \int_{\mathcal{D}_r}\frac{\mathrm{d}z_1}{2\I\pi} \dots \int_{\mathcal{D}_r}\frac{\mathrm{d}z_n}{2\I\pi} \prod_{i\neq j} \frac{1}{\Gamma(z_i-z_j)} \prod_{1\leqslant i<j\leqslant n}\frac{\Gamma(z_i+z_j)}{\Gamma(\alpha_i+\alpha_j)}
 \prod_{i,j=1}^{n} \Gamma(z_i + \alpha_j)
 \prod_{i=1}^{n}u^{ -z_i} \frac{\Gamma(\diag+z_i)}{\Gamma(\diag+\alpha_i)}, 
 \label{eq:OSZ515}
\end{multline}
where $r>0$ is such that $r+\diag>0$.

When $n$ is odd, for any $u>0$,  
\begin{multline}
\EWM[e^{-u Z(n,n)}] \\ 
=  \frac{ u^{\diag}}{n!} \int_{\mathcal{D}_r}\frac{\mathrm{d}z_1}{2\I\pi} \dots \int_{\mathcal{D}_r} \frac{\mathrm{d}z_n}{2\I\pi} \prod_{i\neq j} \frac{1}{\Gamma(z_i-z_j)} \prod_{1\leqslant i<j\leqslant n}\Gamma(z_i+z_j)
\prod_{i,j=1}^{n} \Gamma(z_i + \alpha_j)
\prod_{i=1}^{n}u^{-z_i} \Gamma(z_i-\diag),  
 \label{eq:OSZ516}
\end{multline}
where $r>0$ is such that $r-\diag>0$.
\label{cor:OSZproof2}
\end{corollary}
\begin{remark}
Equation \eqref{eq:OSZ515} corresponds exactly to \cite[(5.15)]{o2014geometric} given that  $ 2 t_{nn}$ in \cite[(5.15)]{o2014geometric} has the same distribution as our partition function $Z(n,n)$. In the case when $n$ is odd, our formula \eqref{eq:OSZ516} is slightly different than \cite[(5.16)]{o2014geometric}, we find an extra factor $u^{\diag}$. We believe that this is due to a typo in \cite[(5.16)]{o2014geometric} (otherwise the case $n=1$ would be in contradiction with \cite[Corollary 3.9]{o2014geometric}). 
\end{remark}
\begin{proof}
Consider the functions 
\begin{equation}
f_{\diag}(\alpha_1, \dots, \alpha_n) = \int_{\R^n} e^{-u e^{T_1}} \psi_{\I \alpha}(T) e^{-\diag \sum_{i=1}^{n/2} (T_{2i-1}- T_{2i})} e^{-e^{-T_n}}\mathrm{d}T,
\label{eq:deffalpha}
\end{equation}
and 
 \begin{multline}
g_{\diag}(\alpha_1, \dots, \alpha_n) =  \\ \frac{1}{n!} \int_{\mathcal{D}_r}\frac{\mathrm{d}z_1}{2\I\pi} \dots \int_{\mathcal{D}_r}\frac{\mathrm{d}z_n}{2\I\pi} \prod_{i\neq j} \frac{1}{\Gamma(z_i-z_j)} \prod_{1\leqslant i<j\leqslant n}\Gamma(z_i+z_j)
 \prod_{i,j=1}^{n} \Gamma(z_i - \alpha_j)
 \prod_{i=1}^{n}u^{ \alpha_i -z_i} \Gamma(\diag+z_i),
 \label{eq:integralforf}
 \end{multline}
where the contour $\mathcal{D}_r$ is such that $r>\alpha_i$ for all $1\leqslant i\leqslant n$ (the contour can be freely shifted to the right, so that $g_{\diag}(\alpha_1, \dots, \alpha_n)$ does not depend on $r$). We know from Corollary \ref{cor:OSZproof}  that for $\alpha\in (\R_{> 0})^n$ with $\alpha_i+\diag>0$, 
$$f_{\diag}(\alpha_1, \dots, \alpha_n)=g_{\diag}(\alpha_1, \dots, \alpha_n).$$

We will need to analytically continue this relation to negative values of the $\alpha_i$. The integral in \eqref{eq:integralforf} is analytic in each variable $\alpha_i$ as long as $\Re[\alpha_i]<r$, and $r$ can be taken arbitrarily large so that $g_{\diag}(\alpha_1, \dots, \alpha_n)$ is analytic in each variable $\alpha_i$ on $\R$. Before proving that $f_{\diag}(\alpha_1, \dots, \alpha_n)$ is analytic as well, let us see how this implies Corollary \ref{cor:OSZproof2}.

Using the change of variables $T_i=-\log(u) - \tilde T_{n-i+1}$ for all $1\leqslant i\leqslant n$, we obtain that when $n$ is even 
\begin{align}
f_{\diag}(\alpha_1, \dots, \alpha_n) &= \int_{\R^n} e^{-e^{-\widetilde T_n}} \psi_{-\I \alpha}(\widetilde T + \log(u)) e^{-\diag \sum_{i=1}^{n/2} \widetilde T_{2i-1}- \widetilde T_{2i}} e^{-u e^{\widetilde T_1}}\mathrm{d}\widetilde T\nonumber\\ &= u^{\sum_{i=1}^n \alpha_i}f_{\diag}(-\alpha_1, \dots, -\alpha_n).
\label{eq:functionalequation}
\end{align}
Similarly, when $n$ is odd, 
\begin{align}
f_{\diag}(\alpha_1, \dots, \alpha_n) &= \int_{\R^n} e^{-e^{-\widetilde T_n}} \psi_{-\I \alpha}(\widetilde T + \log(u)) e^{+\diag \sum_{i=1}^{n/2} \widetilde T_{2i-1}- \widetilde T_{2i}} e^{\diag \log(u)} e^{-u e^{\widetilde T_1}} \mathrm{d}\widetilde T \nonumber\\ &= u^{\diag + \sum_{i=1}^n \alpha_i} f_{-\diag}(-\alpha_1, \dots, -\alpha_n).
\label{eq:functionalequation2}
\end{align}

Using \eqref{eq:functionalequation} and the equality $ f_{\diag}(-\alpha_1, \dots, -\alpha_n)=g_{\diag}(-\alpha_1, \dots, -\alpha_n)$  for $\alpha_i\in \R$, we find when $n$ is even
\begin{multline*}
f_{\diag}(\alpha_1, \dots, \alpha_n) \\ = \frac{1}{n!} \int_{\mathcal{D}_r}\frac{\mathrm{d}z_1}{2\I\pi} \dots \int_{\mathcal{D}_r}\frac{\mathrm{d}z_n}{2\I\pi} \prod_{i\neq j} \frac{1}{\Gamma(z_i-z_j)} \prod_{1\leqslant i<j\leqslant n}\Gamma(z_i+z_j)
\prod_{i,j=1}^{n} \Gamma(z_i + \alpha_j)
\prod_{i=1}^{n}u^{-z_i} \Gamma(\diag+z_i).
\end{multline*}
We can finally freely move the contour as long as the real part stays positive to arrive at the statement of Corollary \ref{cor:OSZproof2}. Similarly, using \eqref{eq:functionalequation2} when $n$ is odd, 
\begin{multline*}
f_{\diag}(\alpha_1, \dots, \alpha_n)  \\ = u^{\diag}  \frac{1}{n!} \int_{\mathcal{D}_r}\frac{\mathrm{d}z_1}{2\I\pi} \dots \int_{\mathcal{D}_r} \frac{\mathrm{d}z_n}{2\I\pi} \prod_{i\neq j} \frac{1}{\Gamma(z_i-z_j)} \prod_{1\leqslant i<j\leqslant n}\Gamma(z_i+z_j)
\prod_{i,j=1}^{n} \Gamma(z_i + \alpha_j)
\prod_{i=1}^{n}u^{-z_i} \Gamma(-\diag+z_i), 
\end{multline*}
where the contour $\mathcal{D}_r$ is such that $r-\diag>0$. 

Now we turn to proving analyticity of $f_{\diag}(\alpha_1, \dots, \alpha_n)$ via the next two lemmas.  
\begin{lemma}
For any $\diag, \alpha_1, \dots, \alpha_n \in \R$, $u>0$, we have 
\begin{equation}
f_{\diag}(\alpha_1, \dots, \alpha_n)  = \int_{\R^{\frac{n(n+1)}{2}}} e^{-u\sum_{\pi} \prod_{(i,j)\in \pi} e^{x_{ij}}} \prod_{1\leqslant i<j \leqslant n} e^{-(\alpha_i+\alpha_j)x_{ij}}e^{-e^{-x_{ij}}}  \prod_{1\leqslant i\leqslant n} e^{-(\alpha_i+\diag )x_{ii}}e^{-e^{-x_{ii}}}\mathrm{d}x,
\label{eq:falternative}
\end{equation}
where the sum over $\pi$ in the first exponential term is a sum over all up-right paths from $(1,1)$ to $(n,n)$ in the lower half quadrant.
\end{lemma}
\begin{proof}
Notice that for $\alpha\in (\R_{>0})^n$,  $ f_{\diag}(\alpha_1, \dots, \alpha_n)$ is the unnormalized Laplace transform of $e^{T_1}$ under the half-space Whittaker measure, it is also the unnormalized Laplace transform of the polymer partition function which can be written as  in \eqref{eq:falternative}. More generally, for any $\diag, \alpha_1, \dots, \alpha_n \in \R$,  \eqref{eq:falternative} is obtained from \eqref{eq:deffalpha} via a change of variables corresponding to the geometric RSK map. This map is volume preserving \cite[Theorem 3.1]{o2014geometric}. It is shown in \cite[(3.8) and Corollary 3.3]{o2014geometric} how the integrand in \eqref{eq:deffalpha} becomes the integrand in \eqref{eq:falternative} via this change of variables. 
\end{proof}
\begin{lemma}
	For any $\diag\in \R$ and $u>0$, $f_{\diag}(\alpha_1, \dots, \alpha_n)$ is analytic in each variable $\alpha_i$ on $\R$. 
\end{lemma}
\begin{proof}
	We show that for $\alpha=(\alpha_1, \dots, \alpha_n)$ in a compact subset of $\R^n$, the integral in \eqref{eq:falternative} is absolutely convergent uniformly in $\alpha$. For $\alpha$ in a compact set, we may rewrite \eqref{eq:falternative} as 
	$$  f_{\diag}(\alpha_1, \dots, \alpha_n)  =\int_{\R^{\frac{n(n+1)}{2}}} C_{\alpha}(x) \exp\left( c\sum_{i,j} \vert x_{i,j}\vert  - \sum_{\pi} \left( u e^{\sum_{(i,j)\in \pi} x_{i,j}} + \sum_{(i,j)\in \pi}c_{i,j} e^{-x_{i,j}} \right)  \right)\mathrm{d}x,$$
	where we choose the constant $c$ large enough so that the prefactor $C_{\alpha}(x)$ is an integrable function over $\R^{\frac{n(n+1)}{2}}$ (uniformly for $\alpha$ in a compact),  and for each $i,j$, the constant $c_{i,j}$ is the inverse of the number of paths containing the vertex $(i,j)$. In order to prove uniform integrability, it is enough to show that each term in the sum over paths $\pi$  grows exponentially, thus compensating the term $c\sum_{i,j} \vert x_{i,j}\vert$. More precisely, we need to show that there exist some constants $k, k'$ such that for $x\not\in[-M,M]^{\frac{n(n+1)}{2}}$ and any path $\pi$ of fixed length, 
	$$   u e^{\sum_{(i,j)\in \pi} x_{i,j}} + \sum_{(i,j)\in \pi}c_{i,j} e^{-x_{i,j}} \geqslant   k e^{k' M}.$$
In order to prove this, we show that for any constant $c'>0$,  the auxiliary  function $g(x_1, \dots, x_m) = c' e^{\sum_{i=1}^m  x_i} + \sum_{i=1}^m e^{-x_i}$ satisfies the same bound. Indeed, take $x\in \R^m\setminus [-M,M]^m$, either there exists some $i$ for which $x_i<-M$ or  we have $x_i> M$ for all $i$. In the first case, $g(x_1, \dots, x_m) > e^M$ and in the second case $g(x_1, \dots, x_m) > c'e^{mM}$. This concludes the proof of the lemma. 
\end{proof} 
\end{proof}

\thispagestyle{plain}
 \section{KPZ equation on the half-line}
 \label{sec:KPZ}

 \subsection{KPZ equation on $\R_{\geqslant 0}$}
 
 The KPZ equation on $\R_{\geqslant 0}$ with Neumann boundary condition is the  a priori ill-posed stochastic partial differential equation
 \begin{equation}
 \begin{cases}
 \partial_{T} \mathscr H = \frac{1}{2}\Delta  \mathscr H + \frac{1}{2}\big(\partial_X \mathscr H\big)^2 + \dot{W},\\
 \partial_X \mathscr H (T,X)\big\vert_{X=0} = A \qquad (\forall T > 0),
 \end{cases}
 \label{eq:KPZequation}
 \end{equation}
 where $W$ is a space-time white noise. Following \cite{corwin2016open},  we say that $\mathscr H$ solves this equation in the Cole-Hopf sense with narrow-wedge initial condition when $\mathscr H= \log \mathscr Z$ and $\mathscr Z$ is a mild solution to the multiplicative stochastic heat equation with Robin boundary condition
 \begin{equation} \label{eq:mSHE}
 \begin{cases}
 \partial_{T} \mathscr Z = \tfrac12 \Delta \mathscr Z + \mathscr Z \dot W,\\
 \partial_X \mathscr Z (T,X)\big\vert_{X=0} = A \mathscr Z(T, 0)  \qquad (\forall \tau > 0),
 \end{cases}
 \end{equation}
 with delta initial condition. More precisely, a mild solution to \eqref{eq:mSHE} solves 
 \begin{equation}\label{e:SHE-mild}
 \mathscr Z(T,X) = \mathscr P^R_T(X,0)	 + \int_0^T \!\!\! \int_0^\infty \!\!\!  \mathscr P^R_{T-S} (X,Y) \,\mathscr Z(S,Y) \, \mathrm{d}W_S(\mathrm{d}Y),
 \end{equation}
 understood as an It\^o integral, where $\mathscr Z(T, \cdot)$ is adapted to the filtration $\sigma \{\mathscr Z(0,\cdot), W|_{[0,T]}\}$, 
 and $ \mathscr P^R$ is the  heat kernel satisfying the Robin boundary condition for all $T,Y>0$ 
 \begin{equation}
 \partial_X \mathscr P^R_T (X,Y)\Big\vert_{X=0} =A  \mathscr P^R_T (0,Y) \;.
 \end{equation}
 
 The KPZ equation and the multiplicative SHE arise as a limit of several stochastic processes in the KPZ universality class which can be divided into two classes: (1) Systems with a tunable asymmetry, such as ASEP, for which the exponential of the height function is expected to converge to the multiplicative SHE when time and space are rescaled and the asymmetry vanishes simultaneously. For the half-line ASEP, this was proved in \cite{corwin2016open, parekh2017kpz}.  
  (2) Systems with a temperature, or at least a parameter controlling the strengh of the noise, such as directed polymers. The partition function is expected to converge in general to the multiplicative SHE when time and space are rescaled and the temperature is sent to infinity simultaneously. In the full space, convergence of directed polymer partition functions to the multiplicative SHE is proved in \cite{alberts2014intermediate}. A half-space analogue is in preparation   \cite{wu2018intermediate}.

  In the following, we will see how some of our moment formulas obtained above degenerate in the scaling leading to the solution to the KPZ equation. We will focus on the log-gamma directed polymer in a half-quadrant as the parameters of the Gamma random variables go to infinity.
 
 \subsection{Log-gamma polymer at high temperature}
 \label{sec:KPZloggamma}
 Consider the half-space log-gamma polymer partition function $Z(t,n)$ and 
 scale parameters as $t=Tn+n^{1/2}X$,  $\alpha_i=n^{1/2}$,  $\diag \in \R$ stays unscaled. Define the rescaled partition function 
 \begin{equation}
 \mathcal Z_n(T,X) =C(T,X,n)   Z\left(\frac{T}{2}n+n^{1/2}X, \frac{T}{2}n\right),
 \label{eq:scalingsZKPZ}
 \end{equation}
 where the normalization factor is 
 $$  C(T,X,n) = \exp\left( \frac{T n \log(n) - (T-X\log(n))n^{1/2}}{2} - \frac{T}{8} - \frac{X}{2}\right).$$
 
 \begin{proposition} 
 	Let $T>0$, $X\geqslant 0$ and $\diag\in \R$. For all $k\in\Z_{>0}$, 
 	\begin{multline}
 	\lim_{n\to \infty} \EWM[\mathcal Z_n(T,X)^k] = 2^k \int_{r_1-\I\infty}^{r_1+\I\infty}\frac{\mathrm{d}z_1}{2\I\pi}\cdots \int_{r_k-\I\infty}^{r_k+\I\infty} \frac{\mathrm{d}z_k}{2\I\pi} \prod_{1\leqslant a<b\leqslant k} \frac{z_a-z_b}{z_a-z_b-1}\, \frac{z_a+z_b}{ z_a+z_b-1}\\ \times 
 	\prod_{m=1}^{k}  \frac{z_m}{z_m+\diag-1/2}\exp\left( \frac{T}{2}z_m^2 -z_m x\right), 
 	\label{eq:momentsKPZ}
 	\end{multline}
 	where 
 	$ r_1 > r_2+1 > \dots> r_k+k-1 $ and $r_k>-\diag+1/2$. 
 	\label{cor:momentsKPZ}
 \end{proposition}
 \begin{remark}
 	The R.H.S. of \eqref{eq:momentsKPZ} coincides with  moment formulas for the continuous directed polymer in a half-space  obtained\footnote{It was proved only assuming the uniqueness of solutions for the half-space delta Bose gas evolution equations.} in \cite[Eq. (2)]{borodin2016directed} using Bethe ansatz. Hence, we deduce that the parameter $\diag$ in the half-space log-gamma polymer is related to the boundary parameter $A$ in the stochastic PDEs \eqref{eq:KPZequation} and \eqref{eq:mSHE} via $\diag = A+1/2$. 
 \end{remark}
 \begin{proof}
 	Fix $k\in\Z_{>0}$. Then for $\alpha_1, \dots,  \alpha_t>0$ sufficiently large, \eqref{eq:momentsZ} yields (we have shifted all variables by $1/2$), 
 	\begin{multline*}
 	\EWM[Z(t,n)^k] =\oint\frac{\mathrm{d}w_1}{2\I\pi}\cdots \oint\frac{\mathrm{d}w_k}{2\I\pi} \prod_{1\leqslant a<b\leqslant k} \frac{w_a-w_b}{w_a-w_b-1}\, \frac{w_a+w_b}{ w_a+w_b+1}\\ \times 
 	\prod_{m=1}^{k}  \frac{2w_m}{w_m+ 1/2-\diag}\prod_{i=1}^t \left( \frac{1}{\alpha_i-w_m-1/2} \right) \prod_{j=1}^{n} \left(\frac{1}{w_m -1/2 + \alpha_j}\right), 
 	\end{multline*}
 	where the contours are such that for all $1\leqslant c\leqslant k$, the contour for $w_c$ encloses $\lbrace - \alpha_j+1/2\rbrace_{1\leqslant j\leqslant n}$  and $\lbrace w_{c+1}+1, \dots, w_k+1\rbrace$, and excludes the poles of the integrand at $ \diag- 1/2$ and $ \alpha_j-1/2$ (for $1\leqslant j\leqslant t$). Furthermore, if $t\leqslant n $ are large, there is no pole at infinity and we may assume that the contour for $w_i$ is the vertical line from  $r_i-\I\infty$ to $r_i+\I\infty$ where the $r_i$ are chosen so that 
 	$$ r_k+k-1 < \dots< r_2+1 < r_1 < \diag -1/2. $$
 	For $\alpha_i\equiv n^{1/2}$, 
 	\begin{multline*}
 	\prod_{i=1}^{Tn/2 + n^{1/2}X} \left( \frac{1}{\alpha_i-w-1/2} \right) \prod_{j=1}^{Tn/2} \left(\frac{1}{w -1/2 + \alpha_j}\right) \\ =  \exp\left(\frac{-T n \log(n) + (T-X\log(n))n^{1/2}}{2} + \frac{T}{8} + \frac{X}{2} + \frac{T}{2}w^2 +w x + o(1) \right).
 	\end{multline*} 
 	Moreover, the convergence holds uniformly in $w$. This is because for any fixed  $a\in \R$,  the convergence 
 	$$ \left(  \frac{1}{1-z/n} \right)^n \xrightarrow[n\to\infty]{} e^z $$
 	holds uniformly for $z$ in the set $\lbrace z\in \C: \Re[z]< a\rbrace$. It then suffices to apply this convergence twice, for $z=\frac{T}{2} w^2$ and $z=w x$, and observe that along the contours that we have chosen, the real part of $w_i$ and $w_i^2$ stays  bounded. Using the scalings \eqref{eq:scalingsZKPZ}, we obtain that the limit of $	\EWM[\mathcal Z_n(T,X)^k] $ is 
 	\begin{multline*} 2^k \int_{r_1-\I\infty}^{r_1+\I\infty}\frac{\mathrm{d}w_1}{2\I\pi}\cdots \int_{r_k-\I\infty}^{r_k+\I\infty} \frac{\mathrm{d}w_k}{2\I\pi} \prod_{1\leqslant a<b\leqslant k} \frac{w_a-w_b}{w_a-w_b-1}\, \frac{w_a+w_b}{ w_a+w_b+1}\\\times 
 	\prod_{m=1}^{k}  \frac{w_m}{w_m+ 1/2-\diag}\exp\left( \frac{T}{2}w_m^2 +w_m x\right).
 	\end{multline*}
 	The final result is obtained by applying the change of variables $w_i = -z_{k-i+1}$ for all $i$.  
 \end{proof}
 
 Even if one could identify \eqref{eq:momentsKPZ} with the moments of $\mathscr Z(T,X)$ (the solution to \eqref{eq:mSHE}), this would not determine completely the distribution since the moments grow too fast. In order to fully characterize the distribution of $\mathscr Z(T,X)$, we would need to analyze the limit of the Laplace transform formula from Corollary \ref{cor:LaplaceWhittaker} under the scalings \eqref{eq:scalingsZKPZ}. It is not obvious how to take this limit rigorously and we leave this for future consideration. 
 
 Note that in the special case where $A=-1/2$, the distribution of $\mathscr Z(T,0)$ if fully characterized in \cite[Theorem B]{barraquand2018stochastic} (see also \cite[Corollary 1.3]{parekh2017kpz}), and it would be interesting to compare those formulas with the limit of Corollary \ref{cor:LaplaceWhittaker} in the case $\diag=0$.

 \thispagestyle{plain}
 \section{Tracy-Widom asymptotics for the log-gamma polymer partition function}
\label{sec:asymptotics}
The aim of this section is to explain how to manipulate our Laplace transform formulas in order to derive limit theorems for the partition function of the half-space log-gamma polymer $Z(n,m)$ in various ranges of parameters. In this derivation we will perform several non-rigorous steps. Most of them can be made rigorous with some additional technical arguments. However, there is one important obstruction to rigor that we cannot presently overcome: we are unable to show that the infinite series that we manipulate are uniformly summable as $n,m \to \infty$, which suggests that we miss structural cancellations hidden in the formulas.

It is reasonable to expect that the fluctuations of the free energy $\log(Z(n,m))$ are of the same nature as the fluctuations of last-passage percolation in a half-quadrant, studied in \cite{baik2001asymptotics, baik2018pfaffian, betea2018free}, which corresponds to the zero temperature limit. In particular, the fluctuations of $\log(Z(n,m))$ should be the same as in the full-space case when $n \gg m\gg 1$, i.e., converge to the Tracy-Widom GUE distribution. We will focus below on the more interesting case of fluctuations close to the boundary for which we expect a phase transition as the boundary parameter varies.

Consider the half-space log-gamma polymer partition function $Z(n,n)$ as in Definition \ref{def:LogGammapolymer} with $\alpha_1=\dots=\alpha_n=\alpha>0$. 
Let us scale $u=-e^{-nf-n^{1/3}\sigma x}$, where $f, \sigma$ are constants to be determined later. 
If we have the pointwise convergence for every $x\in \R$ 
$$ \EWM[e^{ u Z(n,n)}]  \xrightarrow[n\to \infty]{} F(x),$$
where $F(x)$ is the distribution function of a certain probability distribution, then it follows (see \cite[Lemma 4.1.39]{borodin2014macdonald}) that 
$$ \lim_{n\to \infty} \PP\left(\frac{\log(Z(n,n)) - fn}{\sigma n^{1/3}} \leqslant x\right)  = F(x). $$
It is not clear how to take asymptotics from the formula for the  Laplace transform of $Z(n,n)$ in Corollary \ref{cor:OSZproof} or \ref{cor:OSZproof2}. However, we have a formula for the Laplace transform of  $Z(n,n,\tau)$ from Corollary \ref{cor:LaplaceWhittaker}, which seems more adapted to asymptotic analysis. 
Moreover, $Z(n,n,\tau)$ should be close to $Z(n,n, 0) = Z(n,n)$ for fixed $\tau$ in the sense that $n^{-1/3} \left\vert \log(Z(n,n, \tau))  - \log(Z(n,n, 0)) \right\vert   $ converges to zero in probability as $n$ goes to infinity.

Thus, we are left with studying the asymptotics of  $\log(Z(n,n,\tau))$ as $n$ goes to infinity for fixed $\tau$. 
Using the change of variables $z_i=s_i-v_i$ in the statement of Corollary \ref{cor:LaplaceWhittaker}, we have 
\begin{multline}
\EWM[e^{ u Z(n,n, \tau)}] = \sum_{k=0}^{n} \ \frac{1}{k!}\ \int_{\mathcal{D}_{R}}\frac{\mathrm{d}z_1}{2\I\pi} \dots \int_{\mathcal{D}_R} \frac{\mathrm{d}z_k}{2\I\pi}\   \oint\frac{\mathrm{d}v_1}{2\I\pi} \dots \oint\frac{\mathrm{d}v_k}{2\I\pi}   \\  \times 
\prod_{1\leqslant i<j\leqslant k} \frac{(z_i-z_j)(v_i-v_j)\Gamma(v_i+v_j)\Gamma(-z_i-z_j)}{(-v_i-z_j)(-v_j-z_i)\Gamma(v_j-z_i)\Gamma(v_i-z_j)}\prod_{i=1}^k \frac{\Gamma(2v_i)}{\Gamma(v_i-z_i)}  \\ \times 
\prod_{i=1}^k \left[ \frac{\pi}{\sin(\pi(v_i+z_i))} e^{n\left(G(v_i)+G(z_i)\right)- n^{1/3} x \sigma (v_i+z_i) }   \frac{\Gamma(\diag -z_i)}{\Gamma(\diag +v_i)}  \frac{e^{\tau z_i^2/2 - \tau v_i^2/2}}{v_i+z_i} \right],
\label{eq:LaplaceZ2}
\end{multline}
where the contour for the variables $v_i$ is a small positively oriented circle around $\alpha$,  the contour $\mathcal{D}_R=R+\I\R $ is such that 
$-\alpha < R < \min\lbrace 0, \diag, 1-\alpha \rbrace$ and 
\begin{equation}
G(z) = \log(\Gamma(\alpha-z)) - \log(\Gamma(\alpha+z)) - f z.
\label{eq:defGG}
\end{equation}
We may deform the integration contours as long as we do not cross poles. In particular, we must ensure that  $ \Real[z_i+v_i] \in(0,1)$ and that the poles when $z_i=\diag$ and when $z_i=-z_j$ lie to the right of the integration contour of $z_i$. 

The asymptotic behavior as $n$ goes to infinity of expansions such as \eqref{eq:LaplaceZ2} is usually analyzed using the saddle-point method. One needs to study the function $G(z)$, deform the contours so that they go through a critical point of $G$, and justify that the main asymptotic contributions of the integral is localized in a neighborhood of this critical point where one may use straightforward approximations. 

Notice that 
$$ G'(z) = -f-\Psi(\alpha-z) - \Psi(\alpha+z), \ \ G''(z) = \Psi_1(\alpha-z)-\Psi_1(\alpha+z),$$
where 
$$ \Psi(z) = \frac{\rm d}{\mathrm{d} z} \log(\Gamma(z)), \ \   \Psi_n(z) = \frac{\mathrm{d}^n}{\mathrm{d}z^n} \Psi(z).$$ 
If we set $f=-2\Psi(\alpha)$, then $G'(0)=G''(0)=0$ and we can use the Taylor expansion around zero
\begin{eqnarray}
G(z) = \sigma^3 z^3/3+\mathcal{O}(z^4),
\label{eq:Taylor}
\end{eqnarray} 
where $\sigma^3= G'''(0)/2$. Hence we find that $G$ has a double critical point at zero, which means that when $\diag\geqslant 0$, we may use a saddle point method around this critical point. For the Laplace's method to work, one also need to control the decay of $\Real [G(z)]$ along the contours.  However when $\diag<0$, this is not possible and the asymptotic behaviour of \eqref{eq:LaplaceZ2} will be quite different.

We will first provide probabilistic heuristics explaining why we should expect a phase transition at $\diag=0$. Then we show that a formal asymptotic analysis  of \eqref{eq:LaplaceZ2} --  following ideas similar to \cite{borodin2012free, borodin2013log, krishnan2018tracy} though these works are rigorous -- confirms all these heuristics.  

\subsection{The Baik-Rains phase transition for directed polymers}
\label{sec:BaikRainsexplained}

Consider the partition function of the half-space log-gamma polymer with weights $w_{i,j}\sim \mathrm{Gamma}^{-1}(2\alpha)$ for $i>j$ and $w_{i,i}\sim \mathrm{Gamma}^{-1}(\diag+ \alpha)$.
In order to understand why the asymptotic behavior of $Z(n,n)$ is different whether $\diag\geqslant 0$ or $\diag<0$, let us examine the asymptotics of the free energy 
$$  \log(Z(n,n))  =    \log\left(\sum_{\text{path } \pi}\prod_{(i,j)\in \pi}w_{i,j}\right) . $$
By KPZ universality, we expect that this quantity has a deterministic first order linear in $n$ and fluctuations on the scale $n^{1/3}$.  
Only a small fraction of admissible paths contribute to the first order of the free energy. Let $H(\pi) = \sum_{(i,j)\in \pi} \log(w_{i,j})$ be the energy of a path. The paths contributing to the first order are those with (close to) maximal energy. The expectation of the energy of a path depends on the number of times the path hits the boundary. It is reasonable to expect that if $\diag\geqslant \alpha$, the weights on the diagonal will be not larger than the bulk weights (in average) and will not influence much the limiting behavior. However, if $\diag$ is very small, the weights on the boundary will be typically very large, so that the path with maximal energy will typically take $\mathcal{O}(n)$ weights on the boundary. This will increase  the first order of the free energy and the Gaussian fluctuations of those boundary weights will imply Gaussian fluctuations for the free energy.   

By analogy with the zero temperature case (see \cite[Section 6.1]{baik2018pfaffian}), we expect that\footnote{This may also be derived, at least formally, from saddle point asymptotics of \eqref{eq:LaplaceZ2}.} for homogeneous weights, the first order of the free energy is the same in a half or a full quadrant. That is, for $\diag=\alpha$  (and consequently $\diag\geqslant \alpha$), we expect that (using \cite[Theorem 2.4]{seppalainen2012scaling}), 
\begin{equation}
\frac{ \log(Z(ns,nt))}{n} \xrightarrow[n\to \infty]{} -s \Psi(\theta) -t\Psi(2\alpha-\theta).\label{eq:homogeneouscase}
\end{equation} 
where $\theta$ is such that 
$$ \frac{s}{t} = \frac{\Psi_1(2\alpha-\theta)}{\Psi_1(\theta)}.$$

However, for $\diag$ sufficiently small,  the situation may be different. For instance, when $\diag$ is close to $-\alpha$, the boundary weights become huge and will dominate the asymptotic behaviour of the partition function. 
In order to predict the precise value of $\diag$ where the transition arises, 
the following proposition will be useful. A similar argument is presented in the zero temperature limit (exponential last passage percolation in a half-quadrant) in \cite[Section 6.1]{baik2018pfaffian}. 
\begin{proposition} Let $a,b>0$. 
Let $Z_1(n,n)$ be the partition function of the half-space log-gamma polymer as in Definition \ref{def:LogGammapolymer}, where the parameters are chosen so that  $\alpha_1=\dots  = \alpha_n=a/2>0$ and $\diag=b-a/2$ is arbitrary. Let $Z_2(n,n)$ be the partition function of the half-space log-gamma polymer where the parameters are chosen as $\alpha_1=b-a/2$, $\alpha_2 = \dots =\alpha_n=a/2$ and $\diag = a/2$. Then we have 
$$ Z_1(n,n) \overset{(d)}{=} Z_2(n,n).$$
\label{prop:identitypartitions}
\end{proposition}
This proposition means that for the half-space log-gamma polymer, $Z(n,n)$  has the same distribution whether the boundary weights are $\mathrm{Gamma}^{-1}(b)$ and the bulk weights are $\mathrm{Gamma}^{-1}(a)$; or when the weights on the first row are $\mathrm{Gamma}^{-1}(b)$ while all other weights are $\mathrm{Gamma}^{-1}(a)$. This is a consequence of the more general identity in law in Proposition \ref{prop:identityinlaw}. 
\begin{proof}
Let $\lambda$ be distributed as 	$\PQWM_{(q^{a/2}, \dots, q^{a/2}), q^{b-a/2}}$ and $\pi$ be distributed as $\PQWM_{(q^{b-a/2}, q^{a/2}, \dots, q^{a/2}), 0}$, where $q^{a/2}$ appears $n$ times in both measures. From Proposition \ref{prop:identityinlaw}, we know that the distributions of $\lambda_1$ and $\pi_1$ are the same. We will exploit this fact in the $q \to 1$ limit. 

By Proposition \ref{prop:LogGammaqWhittakerlimit}, the law of $(1-q)^{2n-1} q^{-\lambda_1}$  under $\PQWM_{(q^{a/2}, \dots, q^{a/2}), q^{b-a/2}}$ converges to that of $Z_1(n,n)$. Let $\kappa$ be distributed as $\PQWM_{(q^{b-a/2}, q^{a/2}, \dots, q^{a/2}), q^{a/2}}$, where $q^{a/2}$ appears $n-1$ times in the specialization  $(q^{b-a/2}, q^{a/2}, \dots, q^{a/2})$. One can couple $\pi$ and $\kappa$, by sampling first $\kappa$ and then $\pi$ according to the transition operator $\Udiag(\pi\vert \kappa)$ from Section \ref{sec:rskdynamics}. By Lemma \ref{lem:firstpartdynamics}, we know that $\pi_1 = \kappa_1+\mathrm{qGeom}(0) = \kappa_1$. Moreover by Proposition \ref{prop:LogGammaqWhittakerlimit}, the law of $(1-q)^{2n-1} q^{-\kappa_1}$  under $\PQWM_{(q^{a/2}, \dots, q^{a/2}), q^{b-a/2}}$ converges to that of $Z_2(n,n)$. Since $(1-q)^{2n-1} q^{-\lambda_1} \overset{(d)}{=} (1-q)^{2n-1} q^{-\kappa_1},$ it suffices to let $q$ tend to $1$ on both sides to conclude the proof. 
\end{proof}

Consider now the partition function $Z_2(n,n)$ of the half-space Log-gamma polymer with weights $w_{i,j}\sim \mathrm{Gamma}^{-1}(2\alpha)$ for $i\geqslant j>1$ and $w_{i,1}\sim \mathrm{Gamma}^{-1}(\diag+ \alpha)$. By Proposition \ref{prop:identitypartitions}, $ Z(n,n) \overset{(d)}{=} Z_2(n,n)$. In the latter model,  the energy of a path can be conveniently decomposed as the energy collected along the first row plus the energy collected after the path has departed the first row, that is  
$$ H(\pi)  = \sum_{i=1}^k \log (w_{i1})  + \sum_{(i,j)\in \pi; j>1}\log (w_{i,j})$$
Assuming that $k$ (the number of steps in the first row) is roughly the same for all paths which contribute to the first order of the free energy, we get
$$ \log(Z_2(n,n))  \approx \sum_{i=1}^k \log(w_{i1}) + \log\left(  \sum_{\pi':(k,2)\to (n,n)} \prod_{(i,j)\in \pi'}w_{i,j} \right).$$
By the previous discussion on the case with homogeneous weights \eqref{eq:homogeneouscase}, the second term can be approximated  by 
$$  \frac{1}{n}\log\left(  \sum_{\pi':(k,2)\to (n,n)} \prod_{(i,j)\in \pi'}w_{i,j} \right) \approx - (1-k/n) \Psi(\theta) -\Psi(2\alpha-\theta) $$
where $\theta$ is such that $1-k/n = \frac{\Psi_1(2\alpha-\theta)}{\Psi_1(\theta)}$. Hence, since $\EE[\log(w_i,1)] = \Psi(\diag+\alpha)$, 
$$\frac{ \EE \log(Z(n,n)) }{n} \approx \max_k\left\lbrace  -k/n \Psi(\diag+\alpha)  - (1-k/n) \Psi(\theta) -\Psi(2\alpha-\theta) \right\rbrace.$$
A quick study of the function 
$$ x \mapsto  -x \Psi(\diag+\alpha)  - (1-x) \Psi(\theta) -\Psi(2\alpha-\theta) \ \text{under the constraint }  1-x = \frac{\Psi_1(2\alpha-\theta)}{\Psi_1(\theta)}$$
when $x\in [0,1]$ shows that when $\diag\geqslant 0$, the maximum arises for $x=0$ while for  $\diag<0$, the maximum arises for $x$ such that $\theta=\alpha+\diag$.

Thus, we expect that when $\diag\geqslant 0$,  
$$\log(Z(n,n))/n \xrightarrow[n\to\infty]{} -2 \Psi(\alpha)$$
 with $n^{-2/3}$ fluctuations. However, when $\diag<0$, we expect that 
 $$\log(Z(n,n))/n \xrightarrow[n\to\infty]{} - \Psi(\alpha +\diag) - \Psi(\alpha-\diag),$$ with Gaussian fluctuations on the scale $n^{-1/2}$. Since $\mathrm{Var}[\log\left( \mathrm{Gamma}^{-1}(\theta) \right)] = \Psi_1(\theta)$, one can even predict that the variance of the Gaussian should be 
$$ x \Psi_1(\alpha+\diag) = \left(1-\frac{\Psi_1(2\alpha-\theta)}{\Psi_1(\theta)}\Big\vert_{\theta=\alpha+\diag} \right)\Psi_1(\alpha+\diag) = \Psi_1(\alpha+\diag) - \Psi_1(\alpha- \diag). $$  

 \subsection{Fredholm Pfaffians and Tracy-Widom distributions}
 \label{sec:pfaffians}
 
 The Pfaffian of a skew-symmetric  $2k\times 2k$ matrix $A$ is defined by
 \begin{equation}
 \Pf(A) = \frac{1}{2^k k!} \sum_{\sigma\in\mathcal{S}_{2k}} \mathrm{sgn}(\sigma) a_{\sigma(1)\sigma(2)}a_{\sigma(3)\sigma(4)} \dots a_{\sigma(2k-1)\sigma(2k)},
 \label{eq:defPfaffian}
 \end{equation}
 where $\mathrm{sgn}(\sigma)$ is the signature of the permutation $\sigma$. Schur's Pfaffian identity states that for any $x_1, x_2, \dots, x_{2n}$, 
 \begin{equation}
 \Pf\left( \frac{x_i-x_j}{x_i+x_j} \right)  =  \prod_{i<j} \frac{x_i-x_j}{x_i+x_j}.
 \label{eq:SchurPfaffian}
 \end{equation}
 Let $(\mathbb{X}, \mu) $ be a measure space. 
 For a $2\times 2$-matrix valued  skew-symmetric kernel,
 $$ \kernel(x,y) = \begin{pmatrix}
 \kernel_{11}(x,y) & \kernel_{12}(x,y)\\
 \kernel_{21}(x, y) & \kernel_{22}(x,y)
 \end{pmatrix},\ \ x,y\in \mathbb{X},$$
 we define the Fredholm Pfaffian (introduced in \cite[Section 8]{rains2000correlation}) by
 \begin{equation}
 \Pf\big[\mathsf{J}+\kernel\big]_{L^2(\mathbb{X}, \mu)} = 1+\sum_{k=1}^{\infty} \frac{1}{k!}
 \int_{\mathbb{X}} \dots \int_{\mathbb{X}}  \Pf\Big( \kernel(x_i, x_j)\Big)_{i,j=1}^k \mathrm{d}\mu(x_1) \dots \mathrm{d}\mu(x_k),
 \label{eq:defFredholmPfaffian}
 \end{equation}
 provided the series converges. The kernel $\mathsf{J}$ is defined by
 $$ \mathsf{J}(x,y)  = \delta_{x=y}
 \begin{pmatrix}
 0 & 1\\-1 &0
 \end{pmatrix} .$$
 In what follows, $\mu$ is the Lebesgue measure and we write $\Pf\big[\mathsf{J}+K\big]_{\mathbb{L}^2(\mathbb{X})}$.

  For $a\in \R$ and $\theta\in (0\pi/2)$, let  $\mathcal{C}_a^{\theta}$ be the contour formed by the union of two semi-infinite rays departing $a$ with angles $\theta$ and $-\theta$, oriented from $a+\infty e^{-\I\theta}$ to $a+\infty e^{\I\theta}$. 
 \begin{definition}
 	The GSE Tracy-Widom distribution $\mathcal{L}^{\rm GSE}$ is a continuous probability distribution on the real line such that for $X\sim \mathcal{L}^{\rm GSE}$, 
 	$ F_{\rm GSE}\left( x\right) :=\PP(X\leqslant x)  = \Pf\big( \mathsf{J}- \kernel^{\rm GSE}\big)_{\mathbb{L}^2(x, \infty)} ,$
 	where $\kernel^{\rm GSE}$ is a $2\times 2$-matrix valued kernel
 	defined by
 	\begin{align*}
 	\kernel_{11}^{\rm GSE}(x,y) &=  \int_{\mathcal{C}_{1}^{\pi/3}}\frac{\mathrm{d}z}{2\I\pi}\int_{\mathcal{C}_{1}^{\pi/3}}\frac{\mathrm{d}w}{2\I\pi} \frac{z-w}{4zw(z+w)} e^{z^3/3 + w^3/3 - xz -yw} ,\\
 	\kernel_{12}^{\rm GSE}(x,y) &= -\kernel_{21}^{\rm GSE}(y,x) =  \int_{\mathcal{C}_{1}^{\pi/3}}\frac{\mathrm{d}z}{2\I\pi}\int_{\mathcal{C}_{1}^{\pi/3}}\frac{\mathrm{d}w}{2\I\pi} \frac{z-w}{4z(z+w)} e^{z^3/3 + w^3/3 - xz -yw} ,\\
 	\kernel_{22}^{\rm GSE}(x,y) &= 
 	\int_{\mathcal{C}_{1}^{\pi/3}}\frac{\mathrm{d}z}{2\I\pi}\int_{\mathcal{C}_{1}^{\pi/3}}\frac{\mathrm{d}w}{2\I\pi} \frac{z-w}{4(z+w)} e^{z^3/3 + w^3/3 - xz -yw} .
 	\end{align*}
 	\label{def:GSEdistribution}
 \end{definition}
 
 Notice that the kernel $\kernel^{\rm GSE}$ has the form
 $$ \kernel^{\rm GSE}(x,y)  = \begin{pmatrix}
 A(x,y) & -\partial_y A(x, y) \\
 -\partial_x A(x,y) & \partial_x\partial_y A(x,y)
 \end{pmatrix} $$
 where $A(x,y)$ is the smooth and antisymmetric kernel $\kernel^{\rm GSE}_{11}(x,y)$. We define another kernel 
 $$ \kernel^{\infty}(x,y) = \begin{pmatrix}
 A(x,y) & -2\partial_y A(x, y) \\
 -2\partial_x A(x,y) & 4\partial_x\partial_y A(x,y) + \updelta'(x,y)
 \end{pmatrix} $$
 where $\updelta'$ is a distribution on $\R^2$ such that
 \begin{equation} \int \int f(x,y)\updelta'(x,y)\mathrm{d}x\mathrm{d}y =  \int  \left.\big(\partial_y f(x,y) - \partial_x f(x,y)\big)\right|_{y=x}\mathrm{d}x,\label{eq:defdeltaprime}
 \end{equation}
 for smooth and compactly supported test functions $f$. It is shown in \cite[Section 5]{baik2018pfaffian} that the Fredholm Pfaffian  of $\kernel^{\infty}$ is well defined and we have 
 \begin{equation}
 \Pf[\mathsf{J}-\kernel^{\rm GSE}]_{\mathbb{L}^2(x, +\infty)}= \Pf[\mathsf{J}-\kernel^{\infty}]_{\mathbb{L}^2(x, +\infty)}.
 \label{eq:equivalentpfaffians}
 \end{equation}
 
 Moreover, if a sequence of kernels $\kernel_n$ converges to $\kernel^{\infty}$ in a certain sense (see \cite[Proposition 5.7]{baik2018pfaffian}), then 
 $$ \lim_{n\to\infty} \Pf[\mathsf{J}-\kernel_n]_{\mathbb{L}^2(x, +\infty)} = \Pf[\mathsf{J}-\kernel^{\rm GSE}]_{\mathbb{L}^2(x, +\infty)}.$$

 \begin{definition}
 	The GOE Tracy-Widom distribution $\mathcal{L}^{\rm GOE}$ is a continuous probability distribution on the real line such that for $X\sim \mathcal{L}_{\rm GOE}$, 
 	$ F_{\rm GOE}(x):=\PP(X\leqslant x) =  \Pf\big( \mathsf{J}- \kernel^{\rm GOE}\big)_{\mathbb{L}^2(x, \infty)},$
 	where $\kernel^{\rm GOE}$ is the $2\times 2$ matrix valued kernel defined by
 	\begin{align*}
 	\kernel_{11}^{\rm GOE}(x,y) &=  \int_{\mathcal{C}_{1}^{\pi/3}}\frac{\mathrm{d}z}{2\I\pi}\int_{\mathcal{C}_{1}^{\pi/3}}\frac{\mathrm{d}w}{2\I\pi} \frac{z-w}{z+w} e^{z^3/3 + w^3/3 - xz -yw},\\
 	\kernel_{12}^{\rm GOE}(x,y) &= -\kernel_{21}^{\rm GOE}(y,x) = \int_{\mathcal{C}_{1}^{\pi/3}}\frac{\mathrm{d}z}{2\I\pi}\int_{\mathcal{C}_{-1/2}^{\pi/3}}\frac{\mathrm{d}w}{2\I\pi} \frac{w-z}{2w(z+w)} e^{z^3/3 + w^3/3 - xz -yw} ,\\
 	\kernel_{22}^{\rm GOE}(x,y) &=  \int_{\mathcal{C}_{1}^{\pi/3}}\frac{\mathrm{d}z}{2\I\pi}\int_{\mathcal{C}_{1}^{\pi/3}}\frac{\mathrm{d}w}{2\I\pi}  \frac{z-w}{4zw(z+w)} e^{z^3/3 + w^3/3 - xz -yw} \\
 	& + \int_{\mathcal{C}_{1}^{\pi/3}}\frac{\mathrm{d}z}{2\I\pi} \frac{e^{z^3/3-zx}}{4z} -   \int_{\mathcal{C}_{1}^{\pi/3}}\frac{\mathrm{d}z}{2\I\pi}  \frac{e^{z^3/3-zy}}{4z} -\frac{\sgn{(x-y)}}{4},
 	\end{align*}
 	where we adopt the convention that
 	$$\sgn(x-y) = \mathds{1}_{x>y} - \mathds{1}_{x<y}. $$
 	\label{def:GOEdistribution}
 \end{definition}

\subsection{Formal saddle-point asymptotics leading to the Baik-Rains transition}
In this section we explain how a non rigorous asymptotic analysis of \eqref{eq:LaplaceZ2} leads to the following:
\begin{formal*} Let $Z(n,n)$ be the half-space log-gamma partition function (Definition \ref{def:LogGammapolymer}) with $\alpha_1=\dots=\alpha_n=\alpha>0$ and $\diag\in \R$. Let us define
	$$f =  -2\Psi(\alpha), \ \ \ \sigma = \sqrt[3]{\Psi_2(\alpha)}.$$
Modulo several non-rigorous steps in the (attempted) proof presented below, we have the following weak limits:
\begin{itemize}
	\item When $\diag>0$, 
	 $$ \lim_{n\to \infty} \PP\left(\frac{\log(Z(n,n)) -f n}{\sigma n^{1/3}} \leqslant x\right)  = F_{\rm GSE}(x).$$
	\item When $\diag=0$, we have 
	$$ \lim_{n\to \infty} \PP\left(\frac{\log(Z(n,n)) -f n}{\sigma n^{1/3}} \leqslant x\right)  = F_{\rm GOE}(x).$$
	\item When $\diag<0$, 
	$$ \lim_{n\to \infty} \PP\left(\frac{\log(Z(n,n)) -f_{\alpha} n}{\sigma_{\alpha} n^{1/2}} \leqslant x\right)  = \int_{-\infty}^x \frac{e^{-t^2/2}}{\sqrt{2\pi}}\mathrm{d}t, $$
	where $f_{\diag} = -\Psi(\alpha-\diag) -\Psi(\alpha+\diag)$ and $\sigma_{\alpha} = \sqrt{\Psi_1(\alpha+\diag) -\Psi_1(\alpha-\diag)}$.
\end{itemize}

\end{formal*}

We use Laplace's method and rescale variables around the critical point of the function $G$ in \eqref{eq:LaplaceZ2}. Note that the same function $G$ as in \eqref{eq:LaplaceZ2} also appears in the asymptotic analysis of the full-quadrant log-gamma polymer \cite{borodin2013log, krishnan2018tracy} and can be controlled along certain contours. Since the asymptotic analysis performed in this section is not rigorous anyway, we will not  write the relevant decay estimates of $\Real [G(z)]$.  Thus, let us rescale the variables around zero as $z_i = \sigma^{-1} n^{-1/3}\tilde z_i$ and $w_i = \sigma^{-1} n^{-1/3}\tilde w_i$ in \eqref{eq:LaplaceZ2} (we will drop the tildes in the following formulas).  
 There are now three cases to consider. If $\diag > 0$, then 
 $$ \frac{\Gamma(\diag -n^{-1/3}\sigma^{-1} z_i)}{\Gamma(\diag +n^{-1/3 }\sigma^{-1} v_i)} \xrightarrow[]{} 1.$$
 If $\diag = 0$, then 
 $$ \frac{\Gamma(\diag -n^{-1/3}\sigma^{-1} z_i)}{\Gamma(\diag +n^{-1/3 }\sigma^{-1} v_i)} \xrightarrow[]{} \frac{-v_i}{z_i}.$$
 If $\diag<0$, then there is a pole at $-\diag$ for the variable $z$ that prevents from using the saddle point method around $0$, the scalings will be different, and we will use the saddle point method around $\diag$.

\subsubsection{Case $\diag>0$}
\label{sec:heuristicsGSE}
Using the fact that 
$$ \Gamma(z) \approx \frac{1}{z}, \ \ \ \ \ \ \frac{\pi}{\sin(\pi z)}\approx \frac 1 z$$
for $z$ close to zero, 
the integrand in \eqref{eq:LaplaceZ2} becomes, as $n$ goes to infinity, 
\begin{multline*}
\prod_{1\leqslant i<j\leqslant k} \frac{(z_i-z_j)(v_i-v_j)(z_i-v_j)(v_i-z_j)}{(v_i+z_j)(v_j+z_i)(v_i+v_j)(z_i+z_j)}\prod_{i=1}^k \frac{v_i-z_i}{v_i+z_i}  
\prod_{i=1}^k \left[  e^{v_i^3/3+z_i^3/3 - x z_i-x v_i}  \frac{1}{v_i+z_i} \frac{1}{2v_i} \right].
\end{multline*}
Notice that for $(u_1, \dots, u_{2k}) = (v_1, z_1, \dots, v_k, z_k)$, and using Schur's Pfaffian identity \eqref{eq:SchurPfaffian}, 
$$ \prod_{1\leqslant i<j\leqslant k} \frac{(z_i-z_j)(v_i-v_j)(z_i-v_j)(v_i-z_j)}{(v_i+z_j)(v_j+z_i)(v_i+v_j)(z_i+z_j)}\prod_{i=1}^k \frac{v_i-z_i}{v_i+z_i}  = (-1)^k\Pf\left( \frac{u_i-u_j}{u_i+u_j} \right).$$
 We use that for $\Real[z_i+v_i]>0$,
 $$ \frac{1	}{z_i+v_i} = \int_0^{\infty} e^{-r_i (z_i+v_i)}\mathrm{d}r_i.$$
 We may bring the integrations inside the Pfaffian to
 find that 
 \begin{equation*}
 \lim_{n\to \infty} \EWM[e^{ u Z(t,n)}] = \sum_{k=0}^{\infty} \frac{(-1)^k}{k!} \int_x^{\infty} \mathrm{d}r_1 
 \dots \int_x^{\infty} \mathrm{d}r_k \Pf\left(K(r_i,r_j)\right)_{i,j=1}^k  = \Pf[\mathsf{J}-\kernel ]_{\mathbb{L}^2(x, \infty)},
 \end{equation*}
where $\kernel$ is the $2\times 2$ matrix antisymmetric kernel 
\begin{subequations}
\begin{eqnarray}
\kernel_{11}(r,s) &=& \frac{1}{4} \int_{\mathcal{D}_1} \frac{\mathrm{d}v}{2\I\pi}  \int_{\mathcal{D}_1} \frac{\mathrm{d}v'}{2\I\pi} \frac{v-v'	}{v v'(v+v')} e^{v^3/3 +v'^3/3  -rv -sv'},\\
\kernel_{12}(r,s) &=& \frac{1}{2} \int_{\mathcal{D}_1} \frac{\mathrm{d}v}{2\I\pi}  \int_{\mathcal{D}_{-1/2}} \frac{\mathrm{d}z}{2\I\pi} \frac{v-z	}{v(v+z)} e^{v^3/3 +z^3/3 -rv -sz} ,\label{eq:K12limit}\\ 
\kernel_{22} (r,s) &=&  \int_{\mathcal{D}_{-1}} \frac{\mathrm{d}z}{2\I\pi}  \int_{\mathcal{D}_{-1}} \frac{\mathrm{d}z'}{2\I\pi} \frac{z-z'	}{z+z'} e^{z^3/3 +z'^3/3 -rz -sz'}.\label{eq:K22limit}
\end{eqnarray}
\end{subequations}
Note that the formulas above do not make sense because the integrand in \eqref{eq:K12limit} and \eqref{eq:K22limit} is not absolutely convergent on ${\mathcal{D}_{-1}}$ and  ${\mathcal{D}_{-1/2}}$ respectively. The integration over ${\mathcal{D}_{-1/2}}$ in \eqref{eq:K12limit} is not a real issue since the contour could have been freely deformed earlier to ${\mathcal{D}_{1}}$ (the only pole in $z$ is at $z=-v$ which lies on the left of ${\mathcal{D}_{-1/2}}$). However, the integrations over ${\mathcal{D}_{-1}}$ in \eqref{eq:K22limit}  are a real issue because the contours cannot be deformed due to the pole at $z=-z'$. However, formally, one can write 
$$ \kernel_{22} (x,y) = \int_{0}^{\infty} \big(\Ai(x-u)\Ai'(y-u) -  \Ai'(x-u)\Ai(y-u)\big)\mathrm{d}u,$$
and using the formal identity 
$$ \int_{\R} \Ai(x+u)\Ai(y+u) = \delta(x-y),$$
one can write $\kernel_{22}$ as 
$$ \kernel_{22} (x,y) =   \int_{\mathcal{D}_{1}} \frac{\mathrm{d}z}{2\I\pi}  \int_{\mathcal{D}_{1}} \frac{\mathrm{d}z'}{2\I\pi} \frac{z-z'	}{z+z'} e^{z^3/3 +z'^3/3 -xz -yz'} \ \ + \ \ \updelta'(x,y),$$
where $\updelta'$ is a distribution on $\R^2$ defined by \eqref{eq:defdeltaprime}. 
Then it follows from \eqref{eq:equivalentpfaffians} in Section \ref{sec:pfaffians}  that 
$$ \Pf[\mathsf{J}-\kernel ]_{\mathbb{L}^2(x, \infty)} = F_{\rm GSE}(x).$$
Finally, we obtain -- modulo several non-rigorous steps above -- that when $\diag>0$ and $\tau>0$, 
$$ \lim_{n\to \infty} \PP\left(\frac{\log(Z(n,n, \tau)) -f n}{\sigma n^{1/3}} \leqslant x\right)  = F_{\rm GSE}(x), $$
where $f=-2\Psi(\alpha)$ and $\sigma = \sqrt[3]{\Psi_2(\alpha)}$.

\subsubsection{Case $\diag=0$}

Following the same steps as in Section \ref{sec:heuristicsGSE}, we obtain when $\diag=0$ that 
 \begin{equation*}
 \lim_{n\to \infty} \EWM[e^{ u Z(t,n)}] = \sum_{k=0}^{\infty} \frac{(-1)^k}{k!} \int_x^{\infty} \mathrm{d}r_1 
 \dots \int_x^{\infty} \mathrm{d}r_k \Pf\left(\kernel(r_i,r_j)\right)_{i,j=1}^k  = \Pf[J+\kernel ]_{\mathbb{L}^2(x, \infty)}.
 \end{equation*}
 where $K$ is the $2\times 2$ matrix antisymmetric kernel 
 \begin{subequations}
 	\begin{eqnarray}
 	\kernel_{11}(r,s) &=& \frac{1}{4} \int_{\mathcal{D}_1} \frac{\mathrm{d}v}{2\I\pi}  \int_{\mathcal{D}_1} \frac{\mathrm{d}v'}{2\I\pi} \frac{v-v'	}{(v+v')} e^{v^3/3 +v'^3/3  -rv -sv'},\\
 	\kernel_{12}(r,s) &=& \frac{1}{2} \int_{\mathcal{D}_1} \frac{\mathrm{d}v}{2\I\pi}  \int_{\mathcal{D}_{-1/2}} \frac{\mathrm{d}z}{2\I\pi} \frac{v-z	}{z(v+z)} e^{v^3/3 +z^3/3 -rv -sz} ,\label{eq:K12limitGOE}\\ 
 	\kernel_{22} (r,s) &=&  \int_{\mathcal{D}_{-1}} \frac{\mathrm{d}z}{2\I\pi}  \int_{\mathcal{D}_{-1}} \frac{\mathrm{d}z'}{2\I\pi} \frac{z-z'	}{z z' (z+z')} e^{z^3/3 +z'^3/3 -rz -sz'}.\label{eq:K22limitGOE}
 	\end{eqnarray}
 \end{subequations}
The formula for $\kernel_{22}$ in \eqref{eq:K22limitGOE} does not make sense but if one shifts the vertical contours to the right so that they become $\mathcal{D}_1$ and compute the residues, then the resulting formula does make sense. Moreover,  $\Pf[J+\kernel ]_{\mathbb{L}^2(x, \infty)} = \Pf[\mathsf{J}-\tilde \kernel]_{\mathbb{L}^2(x, \infty)}$, where $\kernel_{11}=\tilde \kernel_{11}$, $\kernel_{22}=\tilde \kernel_{22}$, $\kernel_{12}=-\tilde \kernel_{12}$ and $\kernel_{21}=-\tilde \kernel_{21}$, and it can be shown (see \cite[Lemma 2.6]{baik2018pfaffian}) that  
$$ \Pf[\mathsf{J}-\tilde \kernel]_{\mathbb{L}^2(x, \infty)} = F_{\rm GOE}(x).$$
Thus when $\diag =0$ and $\tau>0$ -- modulo several non-rigorous steps --   
$$ \lim_{n\to \infty} \PP\left(\frac{\log(Z(n,n, \tau)) -f n}{\sigma n^{1/3}} \leqslant x\right)  = F_{\rm GOE}(x), $$
where $f=-2\Psi(\alpha)$ and $\sigma = \sqrt[3]{\Psi_2(\alpha)}$. 

\begin{remark}
If we scale $\diag$ close to the critical point as $\diag = n^{-1/3}\sigma^{-1} \varpi$, the limiting distribution $F_{\rm GOE}$ would be replaced by a crossover distribution $F(x; \varpi)$, originally introduced in \cite[Definition 4]{baik2001asymptotics} in the context of half-space last passage percolation with geometric weights. This distribution is such that  $F(x; 0) = F_{\rm GOE}(x)$ and $\lim_{\varpi\to\infty} F(x; \varpi) = F_{\rm GSE}(x)$.  The cumulative  probability distribution $F(x; \varpi)$ can be written as the Fredholm Pfaffian of a crossover kernel introduced  in \cite[(1.14)]{forrester2006correlation}, which is also a special case of the more general crossover kernel $\kernel^{\rm cross}$ in \cite[Theorem 1.7 and Definition 2.9]{baik2018pfaffian}. 
\label{rem:crossover}
\end{remark}

\subsubsection{Case $\diag<0$}

In that case one cannot deform the contours so that they go through a neighborhood of zero as in Sections \ref{sec:heuristicsGSE}, because there is a pole in \eqref{eq:LaplaceZ2} at $z_i=\diag <0$. Instead, we will scale $u$ in a different way and apply the saddle point method in a neighborhood of $\diag$. 

Let $f_{\diag} = -\Psi(\alpha-\diag) -\Psi(\alpha+\diag). $
Letting $u=-e^{-nf_{\diag}-n^{1/2}\sigma_{\diag} x}$ (where $\sigma_{\diag}$ is a constant to determine later), we may write 
\begin{multline}
\EWM[e^{ u Z(n,n)}] = \sum_{k=0}^{n} \ \frac{1}{k!}\ \int_{\mathcal{D}_{R}}\frac{\mathrm{d}z_1}{2\I\pi} \dots \int_{\mathcal{D}_R} \frac{\mathrm{d}z_k}{2\I\pi}\   \oint\frac{\mathrm{d}v_1}{2\I\pi} \dots \oint\frac{\mathrm{d}v_k}{2\I\pi}   \\  \times 
\prod_{1\leqslant i<j\leqslant k} \frac{(z_i-z_j)(v_i-v_j)\Gamma(v_i+v_j)\Gamma(-z_i-z_j)}{(-v_i-z_j)(-v_j-z_i)\Gamma(v_j-z_i)\Gamma(v_i-z_j)}\prod_{i=1}^k \frac{\Gamma(2v_i)}{\Gamma(v_i-z_i)}  \\ \times 
\prod_{i=1}^k \left[ \frac{\pi}{\sin(\pi(v_i+z_i))} e^{n\left(G_{\diag}(v_i)+G_{\diag}(z_i)\right)- n^{1/2} x \sigma (v_i+z_i) }   \frac{\Gamma(\diag -z_i)}{\Gamma(\diag +v_i)}  \frac{1}{v_i+z_i} \right],
\label{eq:LaplaceZGaussian}
\end{multline}
where 
$$ G_{\diag}(z) = \log(\Gamma(\alpha-z)) - \log(\Gamma(\alpha+z)) - f_{\diag} z $$
satisfies $G_{\diag}'(\diag) = 0$ and $G_{\diag}''(\diag) = \Psi_1(\alpha-\diag) -\Psi_1(\alpha+\diag) =:-\sigma_{\alpha}^2$. Assuming that the saddle point method would work without any issue, we rescale variables $z_i$ as $z_i=\diag +\sigma_{\alpha} n^{-1/2}\tilde z_i$ and variables $v_i$ as $v_i = -\diag+\sigma_{\alpha} n^{-1/2}\tilde v_i$, and obtain 
 \begin{multline}
  \EWM[e^{ u Z(n,n)}] \approx \sum_{k=0}^{n} \ \frac{1}{k!}\ \int_{\mathcal{D}_{-1/2}}\frac{\mathrm{d}z_1}{2\I\pi} \dots \int_{\mathcal{D}_{-1/2}} \frac{\mathrm{d}z_k}{2\I\pi}\   \int_{\mathcal{D}_1}\frac{\mathrm{d}v_1}{2\I\pi} \dots \int_{\mathcal{D}_1}\frac{\mathrm{d}v_k}{2\I\pi}  \\   \times 
  \prod_{1\leqslant i<j\leqslant k} \frac{(z_i-z_j)(v_i-v_j)}{(-v_i-z_j)(-v_j-z_i)} 
  \prod_{i=1}^k \left[ \frac{1}{v_i+z_i} e^{-(z_i^2/2 -w_i^2/2)- x (v_i+z_i) } \frac{-v_i}{z_i} \frac{1}{v_i+z_i} \right].
  \label{eq:LaplaceZGaussian2}
  \end{multline}
  Then using the Cauchy determinant formula $$\prod_{1\leqslant i<j\leqslant k} \frac{(z_i-z_j)(v_i-v_j)}{(-v_i-z_j)(-v_j-z_i)} = \det\left( \frac{1}{z_i+v_i} \right)_{i,j=1}^k $$ 
  we may recognize that 
  $ \eqref{eq:LaplaceZGaussian2} \approx \det(I -\kernel^G)_{\mathbb{L}^2(x, +\infty)},$
  where 
 $$  \kernel^G(r,s)  = \int_{\mathcal{D}_{-1/2}}\frac{\mathrm{d}z}{2\I\pi} \int_{\mathcal{D}_1}\frac{\mathrm{d}v}{2\I\pi}  \frac{v}{z(z+v)} e^{-z^2/2+v^2/2 -rv -s z}  = \frac{1}{\sqrt{2\pi}} e^{-r^2/4 -s^2/4}, $$ 
so that
we would obtain 
$$ \lim_{n\to \infty} \PP\left(\frac{\log(Z(n,n, \tau)) -f_{\alpha} n}{\sigma_{\alpha} n^{1/2}} \leqslant x\right)  = \int_{-\infty}^x \frac{e^{-t^2/2}}{\sqrt{2\pi}}\mathrm{d}t, $$
where $f_{\diag} = -\Psi(\alpha-\diag) -\Psi(\alpha+\diag)$ and $\sigma_{\alpha} = \sqrt{\Psi_1(\alpha+\diag) -\Psi_1(\alpha-\diag)}$.

\renewcommand{\emph}[1]{\textit{#1}}
\bibliography{mainbiblio.bib}
\bibliographystyle{mystyle-amsalpha}

\end{document}